\documentclass[10pt,a4paper,english]{article}

\usepackage[T1]{fontenc}
\usepackage[utf8]{inputenc}
\usepackage[english]{babel}

\usepackage{newtxtext}
\usepackage{mathtools}
\usepackage{amsthm}
\usepackage{newtxmath}
\usepackage{microtype}

\usepackage[
a4paper,
textwidth=13.5cm,
textheight=21cm,
centering,
headheight=12pt,
headsep=16pt,
footskip=25pt
]{geometry}

\usepackage[colorinlistoftodos,prependcaption,color=yellow,textsize=tiny,textwidth=2cm]{todonotes}

\usepackage{titlesec}

\titleformat{\section}
{\normalfont\large\bfseries}
{\thesection}{0.7em}{}

\titleformat{\subsection}
{\normalfont\normalsize\bfseries}
{\thesubsection}{0.7em}{}

\titleformat{\subsubsection}
{\normalfont\normalsize\itshape}
{\thesubsubsection}{0.7em}{}

\titlespacing*{\section}
{0pt}
{3ex plus 1ex minus 0.5ex}
{1.3ex plus 0.3ex}

\titlespacing*{\subsection}
{0pt}
{2.5ex plus 0.8ex minus 0.4ex}
{1ex plus 0.2ex}

\titlespacing*{\subsubsection}
{0pt}
{2ex plus 0.5ex minus 0.3ex}
{0.8ex plus 0.2ex}

\usepackage{titling}

\pretitle{\begin{center}\Large\bfseries}
	\posttitle{\par\end{center}\vskip 0.5em}

\preauthor{%
	\begin{center}
		\normalsize
		\begin{tabular}[t]{c}%
		}
		\postauthor{%
		\end{tabular}
		\par\end{center}%
}

\date{}

\usepackage[numbers,square,sort&compress]{natbib}
\usepackage{url}
\usepackage{xargs}[2008/03/08]

\numberwithin{equation}{section}
\numberwithin{figure}{section}

\theoremstyle{plain}
\newtheorem{thm}{Theorem}[section]
\newtheorem{lem}[thm]{Lemma}
\newtheorem{cor}[thm]{Corollary}
\newtheorem{prop}[thm]{Proposition}

\theoremstyle{definition}
\newtheorem{defn}[thm]{Definition}
\newtheorem{example}[thm]{Example}

\theoremstyle{remark}
\newtheorem{rem}[thm]{Remark}

\newcommand{\cdl}{c\`{a}dl\`{a}g }

\allowdisplaybreaks[2]

\usepackage[hidelinks]{hyperref}

\begin{document}
	\global\long\def\norm#1{\|#1\|}%
	
	\global\long\def\bnorm#1{\left\Vert #1\right\Vert }%
	
	\global\long\def\of#1{(#1)}%
	
	\global\long\def\eps{\varepsilon}%
	
	\global\long\def\R{\mathbb{R}}%
	
	\global\long\def\T{\mathbb{T}}%
	
	\global\long\def\Z{\mathbb{Z}}%
	
	\global\long\def\N{\mathbb{N}}%
	
	\global\long\def\p{\mathbb{P}}%
	
	\global\long\def\E{\mathbb{E}}%
	
	\global\long\def\I{\mathbb{I}}%
	
	\global\long\def\H{\mathrm{H}}%
	
	\global\long\def\rE{\mathrm{E}}%
	
	\global\long\def\Cf{C}%
	
	\global\long\def\cL{\mathcal{L}}%
	
	\global\long\def\F{\mathcal{F}}%
	
	\global\long\def\cN{\mathcal{N}}%
	
	\global\long\def\cS{\mathcal{S}}%
	
	\global\long\def\cH{\mathcal{H}}%
	
	\global\long\def\rM{\mathcal{\mathrm{M}}}%
	
	\global\long\def\rR{\mathcal{\mathrm{R}}}%
	
	\global\long\def\inn#1{\langle#1\rangle}%
	
	\global\long\def\binn#1{\left\langle #1\right\rangle }%
	
	\global\long\def\lrarr{\leftrightarrow}%
	
	\global\long\def\var{\mathcal{\mathrm{Var\,}}}%
	
	\global\long\def\cK{\mathcal{K}}%
	\global\long\def\cM{\mathcal{M}}%
	\global\long\def\pr{\mathrm{pr}}%
	
	\global\long\def\C{\mathbb{C}}%
	\global\long\def\Leb{\mathrm{Leb}}%
	\global\long\def\i{\mathfrak{i}}%
	\global\long\def\ex{\mathcal{E}}%
	\global\long\def\divv{\mathrm{div}}%
	\global\long\def\cG{\mathcal{G}}%
	\global\long\def\Tr{\mathrm{Tr}}%
	\global\long\def\cov{\mathrm{Cov}}%
	\newcommandx\Hrho[1][usedefault, addprefix=\global, 1=]{H_{#1}(\T^{d};[0,1])}%
	\newcommandx\Hrhoint[1][usedefault, addprefix=\global, 1=]{H_{#1}(\T^{d};(0,1))}%
	\global\long\def\supp{\mathrm{supp}}%
	\global\long\def\Td{\T^{d}}%
	
	\global\long\def\Tnd{\T^{d}_{n}}%
	\global\long\def\Znd{\mathbb{Z}^{d}_{n}}%
	
	\global\long\def\EEP{\{0,1\}^{\Tnd}}%
	\global\long\def\EHE{[0,1]^{\Tnd}}%
	\global\long\def\L{L^{2}}%
	\global\long\def\Ln{L^{2}_{n}}%
	
	\global\long\def\id{\mathrm{id}}%
	\newcommandx\Cd[1][usedefault, addprefix=\global, 1=]{\mathrm{C}^{#1}(\T^{d})}%
	\renewcommandx\H[1][usedefault, addprefix=\global, 1=J]{H^{#1}}%
	\global\long\def\CL{\mathcal{CL}}%
	\newcommandx\MLO[2][usedefault, addprefix=\global, 1=, 2=]{\cL_{#1}(#2)}%
	\newcommandx\CF[2][usedefault, addprefix=\global, 1=, 2=]{\Cf^{#1}(#2)}%
	\newcommandx\CFl[2][usedefault, addprefix=\global, 1=, 2=]{\Cf^{#1}_{l}(#2)}%
	\newcommandx\normC[2][usedefault, addprefix=\global, 1=, 2=]{\|#2\|_{\mathrm{C}^{#1}}}%
	\newcommandx\bnormC[2][usedefault, addprefix=\global, 1=, 2=]{\left\Vert #2\right\Vert _{\mathrm{C}^{#1}}}%
	\newcommandx\normL[1][usedefault, addprefix=\global, 1=]{\|#1\|_{L_{2}}}%
	\newcommandx\bnormL[1][usedefault, addprefix=\global, 1=]{\left\Vert #1\right\Vert _{L_{2}}}%
	\newcommandx\normH[2][usedefault, addprefix=\global, 1=, 2=]{\|#2\|_{H_{#1}}}%
	\newcommandx\bnormH[2][usedefault, addprefix=\global, 1=, 2=]{\left\Vert #2\right\Vert _{H_{#1}}}%
	\global\long\def\tr{\mathrm{tr}}%
	\global\long\def\Image{\mathrm{Im}}%
	\global\long\def\law{\mathrm{Law}}%
	\global\long\def\K{\mathbb{K}}%
	\global\long\def\cP{\mathcal{P}}%
	\global\long\def\Zd{\Z^{d}}%
	\global\long\def\lr{\leftrightarrow}%
	\global\long\def\offdiag{\text{2-off-diag}}%
	\global\long\def\offdiagone{\text{1-off-diag}}%
	\global\long\def\Ber{\textrm{Ber}}%
	\global\long\def\rM{\mathrm{M}}%
	\global\long\def\spann{\mathrm{span}}%
	\global\long\def\cA{\mathcal{A}}%
	\global\long\def\CC{\mathrm{CC}}%
	\global\long\def\cC{\mathcal{C}}%
	\global\long\def\ZdM{\Z^{d}_{M}}%
	\global\long\def\cT{\mathcal{T}}%
	\global\long\def\cD{\mathcal{D}}%
	\global\long\def\bB{\mathbb{B}}%
	\global\long\def\cB{\mathcal{B}}%
	\global\long\def\t{\mathfrak{t}}%
	\global\long\def\cF{\mathcal{F}}%
	\global\long\def\d{\mathrm{d}}%
	\global\long\def\rhoini{\varrho_{*}}%
	\global\long\def\xiini{\xi_{*}}%
	\global\long\def\zetaini{\zeta_{*}}%
	
	\title{Higher-Order Gaussian Approximations on Path Space for the Simple
		Symmetric Exclusion Process}
	
	\author{%
		Benjamin Gess%
		\thanks{%
			Institut f{\"u}r Mathematik,
			Technische Universit{\"a}t Berlin,
			Stra{\ss}e des 17. Juni 136,
			10623 Berlin, Germany.%
		}\, 
		\thanks{%
			Max Planck Institute for Mathematics in the Sciences,
			04103 Leipzig, Germany. \\
			E-mail: \protect\url{benjamin.gess@tu-berlin.de}%
		}
		\and
		Vitalii Konarovskyi%
		\thanks{%
			Faculty of Mathematics, Informatics and Natural Sciences,
			University of Hamburg,
			20146 Hamburg, Germany.\\
			E-mail: \protect\url{vitalii.konarovskyi@uni-hamburg.de}%
		}%
	}
	
	\maketitle
	\begin{abstract}
		Quantitative Gaussian approximation estimates on path space
		are established for the nonequilibrium density fluctuations of the
		simple symmetric exclusion process (SSEP) on a \(d\)-dimensional
		discrete torus of mesh size \(1/n\), started from a product measure
		with a smooth density profile. A novel \(n\)-dependent modified
		Ornstein--Uhlenbeck process is constructed that incorporates
		higher-order lattice information while remaining linear and Gaussian.
		It is shown that, for sufficiently high modification orders,
		this modification yields quantitative Gaussian approximation error bounds
		of order \(n^{-\min\{d/2,2\}}\) for a class of smooth functionals of
		the entire fluctuation trajectory. In addition, it is shown that
		retaining an explicit microscopic correlation correction leaves a
		remainder of order \(n^{-d/2}\), thereby distinguishing contributions
		to the bound due to correlations from those due to spatial
		approximation. For nonconstant initial profiles in dimensions
		\(d\ge3\), a separate argument yields matching upper and lower bounds
		of order \(n^{-2}\) for the direct weak error of suitable quadratic
		test functionals, provided that the modification orders are
		sufficiently high.\\
		The proof is based on a novel path-dependent Trotter formula. Its derivation builds on the introduction of a concept of horizontal absolute continuity for non-anticipative functionals and a corresponding martingale formula for Markov jump processes. This formula is subsequently combined with an integral backward Kolmogorov equation for Gaussian continuation functionals.  A second ingredient of the proof is a refinement of the exchangeable-pair argument introduced by the authors in a previous work, thereby improving the central-limit rate $n^{-\min\{d/2,1\}}$ to \(n^{-d/2}\) for the initial empirical fluctuations.

		\medskip
		
		
		\medskip
		
	\end{abstract}
	\tableofcontents{}

	\section{Introduction and main results}
	A central aim in the study of fluctuating hydrodynamics is to develop
	effective mesoscopic descriptions of interacting particle systems.
	Typically formulated as stochastic partial differential equations,
	such descriptions aim to capture the evolution of the density and its
	fluctuations with quantitative control of the approximation error.
	Hydrodynamic limits and central limit theorems identify the leading
	deterministic and Gaussian contributions. The development of more accurate continuum
	descriptions requires understanding which features of the microscopic
	dynamics are lost in these limits and how their effects can be retained.
	
	In this paper, we pursue this programme for the simple symmetric
	exclusion process (SSEP) on a $d$-dimensional discrete torus of mesh
	size $1/n$, started from a product measure with a smooth density
	profile. Building on the previous quantitative central limit
	theorem~\citep{Gess_Konarovskyi2026}, we analyse the mechanisms contributing
	to the approximation error and use this analysis to construct
	novel $n$-dependent modified Gaussian dynamics. The main contributions are a higher-order Gaussian approximation incorporating lattice information into its drift and noise coefficients, an error decomposition that isolates the microscopic correlation
	contribution, and quantitative weak approximation estimates for smooth functionals of the entire fluctuation trajectory.
	
	We establish quantitative weak approximation estimates for smooth
	functionals of the entire fluctuation trajectory. This extends the
	fixed-time comparison in our previous work to path-dependent
	observables. To obtain these estimates, we develop a framework
	combining weak horizontal calculus with a path-dependent Trotter
	generator comparison for distribution-valued particle fluctuations
	and Gaussian dynamics. We establish an integral backward Kolmogorov
	equation for Gaussian continuation functionals directly on every
	admissible càdlàg history, using horizontal absolute continuity,
	and prove a corresponding martingale formula for Markov jump
	processes. These results make the backward equation available along
	the microscopic jump trajectories. Together with regularity estimates
	in negative Sobolev spaces that are uniform in the lattice parameter,
	they provide quantitative control of the weak error on path space.

	\subsection{The model and the approximation problem}
	
	Fix $d\in\N:=\{1,2,\ldots\}$. For an integer $n\ge3$, consider
	$\Tnd:=n^{-1}\{0,\ldots,n-1\}^d\subset\Td:=\R^d/\Z^d$,
	a periodic lattice of mesh size $1/n$. Each site contains at most
	one particle. A configuration is therefore a function
	$\eta:\Tnd\to\{0,1\}$, where $\eta(x)=1$ means that $x$ is
	occupied. In the simple symmetric exclusion process, each particle
	attempts to jump to each of its $2d$ nearest neighbours at rate $1/2$.
	A jump is carried out if the target site is empty and is
	suppressed otherwise. Thus the exclusion rule is the only interaction
	between particles.
	
	We accelerate time for the particle system by $n^2$. Then the resulting process $(\eta^n_t)_{t\ge 0}$ has generator
	\begin{equation}
		\cL^{n,0}f(\eta)
		=\frac{n^2}{2}\sum_{j=1}^d\sum_{x\in\Tnd}
		\left[f\left(\eta^{x\lrarr x+e_j^n}\right)-f(\eta)\right],
		\label{eq:introduction_ssep_generator}
	\end{equation}
	for $f:\{0,1\}^{\Tnd}\to\R$. Here $e_j^n:=e_j/n$, where
	$(e_j)_{j=1}^d$ is the canonical basis of $\R^d$, and
	$\eta^{x\lrarr y}$ is obtained by interchanging the occupations
	at $x$ and $y$. Exchanging an occupied and an empty site moves a
	particle, while exchanging equal occupations has no effect.
	
	We assume that the initial occupations are independent, with
	$\p\{\eta_0^n(x)=1\}=\rhoini(x)$ for a smooth profile
	$\rhoini:\Td\to[0,1]$.
	The exact mean density $\rho_t^n(x):=\E\eta_t^n(x)$ solves the
	discrete heat equation
	$\partial_t\rho_t^n=\frac12\Delta_n\rho_t^n$, with
	$\rho_0^n=\rhoini|_{\Tnd}$, where $\Delta_n$ is the usual discrete Laplacian.
	
	The empirical density $n^{-d}\sum_{x\in\Tnd}\eta_t^n(x)\delta_x$
	assigns mass $n^{-d}$ to each occupied site, with $\delta_x$ the
	unit point mass at $x$. As $n\to\infty$, this random measure
	converges to a deterministic density $\varrho_t(x)\,dx$ called the
	hydrodynamic limit of the SSEP~\citep[see, e.g.,]{Kipnis_Landim:1999}.
	The profile solves $\partial_t\varrho_t=\frac12\Delta\varrho_t$ with
	$\varrho_0=\rhoini$. To describe the remaining randomness,
	we center by the exact microscopic mean and use the central-limit
	scaling for spatial averages over $n^d$ sites. Setting
	$\zeta_t^n:=n^{d/2}(\eta_t^n-\rho_t^n)$, $t\ge 0$, we obtain the empirical
	fluctuation field
	\begin{equation}
		\bar\zeta_t^n
		=\frac{1}{n^{d/2}}\sum_{x\in\Tnd}
		\bigl(\eta_t^n(x)-\rho_t^n(x)\bigr)\delta_x.
		\label{eq:introduction_empirical_fluctuation_field}
	\end{equation}
	The centering by $\rho^n$, rather than $\varrho$, separates random
	fluctuations from the deterministic error in approximating the mean.
	
	The field $\bar\zeta^n$ converges to a generalized
	Ornstein--Uhlenbeck process $(\xi_t)_{t\ge 0}$
	\citep{GaKiSp,Ravishankar:1992}, which satisfies the affine SPDE
	\begin{equation}
		\partial_t\xi_t
		=\frac12\Delta\xi_t
		+\divv\!\left(
		\sqrt{\varrho_t(1-\varrho_t)}\,\dot W_t
		\right),\qquad
		\xi_0=\xiini.
		\label{eq:introduction_limiting_ou}
	\end{equation}
	Here $\dot W$ is a $d$-dimensional space--time white noise on $\Td$,
	and $\xiini$ is a centered Gaussian random distribution,
	independent of $\dot W$, with covariance
	\begin{equation}
		\E\left[\binn{\xiini,\varphi}\binn{\xiini,\psi}\right]
		=\int_{\Td}\rhoini(x)(1-\rhoini(x))\varphi(x)\psi(x)\,dx,
		\qquad \varphi,\psi\in\Cf^\infty(\Td),
		\label{eq:introduction_initial_covariance}
	\end{equation}
	where $dx$ is normalized Lebesgue measure on $\Td$.
	The pairing $\inn{\xi_t,\varphi}$ describes a spatial average
	of the distribution-valued fluctuation field. We also use
	$\inn{\cdot,\cdot}$ for the $L^2(\Td)$ inner product.
	The fluctuation limit can be formulated in
	$D([0,T],\H[-I])$ for $I>d/2+1$, where
	$\H[J]:=H^J(\Td)$ and $D([0,T],E)$ denotes the space of
	$E$-valued \cdl paths.
	
	Our previous work~\citep[Theorem~1.1]{Gess_Konarovskyi2026}
	gives, in particular, a quantitative comparison of the fluctuation field
	with these limiting dynamics at a fixed time. For admissible smooth
	test functions depending only on the fluctuation field, the estimate
	takes the following form in the present notation:
	\begin{equation}
		\sup_{t\in[0,T]}
		\left|\E f(\ex_n\zeta_t^n)-\E f(\xi_t)\right|
		\le C_{T,\rhoini,f}\,n^{-\min\{d/2,1\}},
		\label{eq:introduction_previous_rate}
	\end{equation}
	where $\ex_n$ is the trigonometric interpolation introduced in Section~\ref{sec:notation}.
	
	The preceding bound reflects distinct sources of error. For spatial
	averages over $n^d$ sites, $n^{-d/2}$ is the natural benchmark suggested
	by quantitative central limit theorems. Comparison with a continuum
	process also involves approximating the discrete spatial operators
	and the microscopic noise covariance. In the previous estimate, these
	spatial approximation errors are bounded at order $n^{-1}$ and
	determine the rate when $d>2$. Understanding the distinct nature of the contributions to the approximation error is the first key aim  of this work, and provides a route to developing more accurate effective dynamics.
	
	We follow this route by incorporating higher-order lattice information
	into an $n$-dependent Gaussian approximation. For sufficiently high
	approximation orders, we obtain a weak-error bound of order
	$n^{-\min\{d/2,2\}}$ relative to this modified Gaussian reference. At the same time, we extend the comparison from fixed-time observables to smooth functionals of the entire fluctuation trajectory.
	
	Microscopic correlations give a further contribution to the approximation error. In order to develop a definitive decomposition of this error, we retain this
	term as an explicit correction
	in the weak-error formula. By doing so, we prove that the remaining error is of order $n^{-d/2}$, provided the approximation orders are sufficiently high. Therefore, this yields an explicit and clear decomposition of the total approximation error, and separates spatial contributions from microscopic correlation contributions.

	\subsection{The modified Gaussian dynamics}
	
	To retain higher-order information about the lattice, choose
	approximation orders $m_\Delta,m_\partial\in\N$ and
	$m_\tau\in\N_0:=\N\cup\{0\}$, and define the differential operators
	\begin{align*}
		\Delta_n^{(2m_\Delta)}
		&:=2\sum_{j=1}^d\sum_{k=1}^{m_\Delta}
		\frac{n^{2-2k}}{(2k)!}\partial_j^{2k},\\
		\partial_{n,j}^{(m_\partial)}
		&:=\sum_{k=1}^{m_\partial}
		\frac{n^{1-k}}{k!}\partial_j^k,
		\qquad
		\tau_{n,j}^{(m_\tau)}
		:=\sum_{k=0}^{m_\tau}
		\frac{n^{-k}}{k!}\partial_j^k.
	\end{align*}
	These operators are obtained from Taylor
	approximations of the discrete Laplacian, forward discrete derivative, and nearest-neighbour
	shift, respectively. We take $m_\Delta$ odd so that the modified
	Laplacian generates a contraction semigroup on each Sobolev space.
	
	The noise coefficient is based on the mean squared occupation
	difference across a bond under a Bernoulli product measure.
	If the occupations at the two endpoints have means $r$ and $s$,
	this expectation equals $r+s-2rs$. Replacing the spatial shift
	by its Taylor approximation gives
	\begin{equation}
		\Gamma_{n,j}^{(m_\tau)}(\varrho)
		:=\varrho+\tau_{n,j}^{(m_\tau)}\varrho
		-2\varrho\,\tau_{n,j}^{(m_\tau)}\varrho,
		\qquad j\in[d]:=\{1,\ldots,d\}.
		\label{eq:definition_of_Gamma_m_tau-1}
	\end{equation}
	Writing
	\[
	\Phi_j^n(t):=
	\sqrt{\bigl(\Gamma_{n,j}^{(m_\tau)}(\varrho_t^n)\bigr)^+},
	\qquad \Phi^n(t):=(\Phi_j^n(t))_{j\in[d]},
	\]
	we define the modified density profile $\varrho^n$ and fluctuation
	process $\xi^n$ by
	\begin{align}
		\partial_t\varrho_t^n
		&=\frac12\Delta_n^{(2m_\Delta)}\varrho_t^n,
		\label{eq:modified_heat_equation-1}\\
		\partial_t\xi_t^n
		&=\frac12\Delta_n^{(2m_\Delta)}\xi_t^n
		+\frac{1}{\sqrt2}\divv_n^{(m_\partial)}
		\left(\Phi^n(t)\dot W_t\right),
		\label{eq:SPDE_for_OU_process_with_Gamma-1}
	\end{align}
	with $\varrho_0^n=\rhoini$ and $\xi_0^n=\xiini$, where the initial
	Gaussian field has covariance
	\eqref{eq:introduction_initial_covariance} and is independent of
	$\dot W$. Here $a^+:=\max\{a,0\}$, multiplication by $\Phi^n(t)$
	is componentwise, and
	\[
	\divv_n^{(m_\partial)}h
	:=\sum_{j=1}^d\left(\partial_{n,j}^{(m_\partial)}\right)^*h_j,
	\qquad h=(h_j)_{j\in[d]},
	\]
	is the formal $L^2(\Td)$ adjoint of the modified gradient,
	in the convention used below.
	
	The profile $\varrho^n$ approximates the exact lattice mean $\rho^n$
	using the modified heat equation. Since it is deterministic, the
	equation for $\xi^n$ is linear with additive noise and defines an
	$\H[-I]$-valued Gaussian process for $I>d/2+m_\partial$.
	The positive part ensures a nonnegative noise covariance, since
	the higher-order approximations need not preserve positivity.
	No separation of the initial profile from $0$ or $1$ is required.
	
	For $m_\Delta=m_\partial=1$ and $m_\tau=0$, these equations recover
	the hydrodynamic profile $\varrho$ and, in law, the limiting
	Ornstein--Uhlenbeck process in
	\eqref{eq:introduction_limiting_ou}.
	In this case $\divv_n^{(1)}=-\divv$, and the sign is absorbed into
	the white noise. Higher approximation orders retain additional lattice
	information while preserving the Gaussian structure.
	
	\subsection{Main result}
	
	We consider smooth functionals of the fluctuation trajectory.
	For a Banach space $E$, continuity and differentiation on
	$D([0,T],E)$ are understood with respect to the uniform norm,
	while measurability refers to the canonical $\sigma$-algebra
	generated by evaluations. We denote by
	$\cP^3(D([0,T],E))$ the class of measurable continuous functionals
	with bounded, left-continuous vertical path derivatives of orders
	one to three. Here $\nabla^k_{[t,T]}F(\gamma)$ is the $k$-th
	Fr\'echet derivative at $h=0$ of $h\mapsto F(\gamma+h)$,
	for $h\in D([t,T],E)$ extended by zero before $t$.
	The precise left-continuity condition is given in
	Section~\ref{subsec:vertical_derivatives}. We equip this class
	with the seminorm
	\[
	\|F\|_{\cP^3(D([0,T],E))}
	:=
	\sum_{k=1}^3
	\sup_{\gamma\in D([0,T],E)}
	\sup_{t\in[0,T]}
	\|\nabla^k_{[t,T]}F(\gamma)\|,
	\]
	where the norms on the right are multilinear operator norms.
	For $E=\H[-M_1]$, examples include smooth functions of finitely
	many observations $\inn{\gamma_{t_i},\varphi_i}$ and time integrals
	$\int_0^T\inn{\gamma_s,\varphi_i}\,ds$, with smooth spatial tests
	and bounded derivatives of the outer function up to order three.
	
	The correction in our approximation is expressed through the
	Gaussian continuation of $F$. Let $\xi_{t,s}^n(g)$ denote the
	solution of \eqref{eq:SPDE_for_OU_process_with_Gamma-1} at time
	$s\ge t$, started from $g$ at time $t$, and define
	\[
	U_t^n(\gamma)
	:=
	\E F\left(
	\gamma^{[0,t)}
	+\bigl(\xi_{t,s}^n(\gamma_t)\bigr)_{s\in[t,T]}
	\right),
	\qquad \gamma\in D([0,t],\H[-M_1]).
	\]
	The sum denotes concatenation: the observed past is retained,
	and the future follows the modified Gaussian dynamics.
	We write $X^{K}:=(X_s)_{s\in K}$, $K\subseteq[0,T]$, for a restricted trajectory.
	
	The correlation correction couples second derivatives with
	products of fluctuations at neighbouring sites. For a sufficiently
	smooth state function $G$ on $\H[-M_1]$, write
	$D^2G(g;x,y):=D^2G(g)[\delta_x,\delta_y]$ and set
	\[
	\cC^nG(\bar\zeta)
	:=
	\frac{1}{2n^d}\sum_{j=1}^d
	\binn{
		\left(\partial_{n,j}^{(m_\partial)}\right)^{\odot2}
		D^2G(\bar\zeta;\cdot),
		\zeta\,\tau_{n,j}\zeta
	}_n.
	\]
	Here $\bar\zeta:=n^{-d}\sum_{x\in\Tnd}\zeta(x)\delta_x$,
	$\inn{h,k}_n:=n^{-d}\sum_{x\in\Tnd}h(x)k(x)$, and
	$\tau_{n,j}\zeta(x):=\zeta(x+e_j^n)$.
	The superscript $\odot2$ means that
	$\partial_{n,j}^{(m_\partial)}$ acts on each spatial argument
	of the kernel, followed by restriction to the diagonal.
	For a history $\gamma$ with $\gamma_t=\bar\zeta$,
	$\cC^nU_t^n(\gamma)$ denotes the application of this operator
	to the terminal section
	$g\mapsto U_t^n(\gamma^{[0,t)},g)$ at $g=\gamma_t$.
	Thus the past is held fixed when taking these derivatives.
	
	\begin{thm}
		\label{thm:main_result}
		Let $T>0$, $m_\Delta,m_\partial\in\N$, $m_\tau\in\N_0$, and
		$I>d/2+m_\partial$. Assume that $m_\Delta$ is odd and choose
		$M_1,M_2\in\N$ such that
		\[
		M_1>\frac d2+\max\left\{
		\lceil I\rceil+2m_\Delta+2,\,
		m_\partial+\lceil I\rceil,\,
		m_\partial+d+1\right\}
		\]
		and
		\[
		M_2\ge\max\left\{
		m_\tau+2+\lceil d/2\rceil,\,
		m_\partial+1+\lceil d/2\rceil,\,
		\lceil I\rceil+2m_\Delta+d+3\right\}.
		\]
		Let $\rhoini\in\Cf^{M_2}(\Td;[0,1])$ and
		$F\in\cP^3(D([0,T],\H[-M_1]))$, and let $\bar\zeta^n$ and $\xi^n$
		be defined by \eqref{eq:introduction_empirical_fluctuation_field}
		and \eqref{eq:SPDE_for_OU_process_with_Gamma-1}, respectively.
		There exists a constant $C$, depending only on
		$d,T,I,m_\Delta,m_\partial,m_\tau,M_1,M_2$, such that
		\begin{align}
			&\left|\E F(\bar\zeta^{n,[0,T]})-\E F(\xi^{n,[0,T]})
			+\int_0^T\E\left[\cC^nU_t^n(\bar\zeta^{n,[0,t]})\right]dt\right|
			\label{eq:main_estimate_with_correlaton_term}\\
			&\qquad\le Cn^{-\alpha}
			\left(1+\|\rhoini\|_{\Cf^{M_2}(\Td)}^2\right)
			\|F\|_{\cP^3(D([0,T],\H[-M_1]))},\nonumber
		\end{align}
		for all $n\ge3$, where
		\[
		\alpha:=\min\left\{\frac d2,\,2m_\Delta,\,m_\partial,\,m_\tau+1\right\}.
		\]
		In particular,
		\begin{align}
			&\left|\E F(\bar\zeta^{n,[0,T]})-\E F(\xi^{n,[0,T]})\right|
			\label{eq:main_estimate_without_correlation}\\
			&\qquad\le C\left(n^{-\alpha}+n^{-\min\{2,d/2\}}\right)
			\left(1+\|\rhoini\|_{\Cf^{M_2}(\Td)}^2\right)
			\|F\|_{\cP^3(D([0,T],\H[-M_1]))}
			\nonumber
		\end{align}
		for all $n\ge3$.
	\end{thm}
	
	The exponent $\alpha$ separates the central-limit contribution
	$d/2$ from the approximation orders of the Laplacian, derivative,
	and shift. Writing
	$\beta:=\min\{2m_\Delta,m_\partial,m_\tau+1\}$ for the spatial
	approximation order, the choice $\beta\ge d/2$ yields a corrected
	error of order $n^{-d/2}$. The direct comparison with the modified
	Gaussian process then has rate $n^{-\min\{d/2,2\}}$.
	It therefore attains the central-limit scale in dimensions $d\le4$,
	while in higher dimensions the correlation contribution determines
	the rate in this bound.
	
	For the quadratic observable
	$F(\gamma)=\frac12\inn{\gamma_{t_0},\varphi}^2$, the weak error
	is a difference of second moments of a spatial average at time $t_0$.
	Its first derivative grows linearly, but it can be treated directly
	since its Gaussian continuation is again quadratic. The generator
	comparison then has no third-order remainder, the terms linear in
	the centered field have zero expectation, and the initial comparison
	involves only covariances. After the explicit correlation correction,
	only spatial approximation errors remain, giving $O(n^{-\beta})$; see
	Proposition~\ref{prop:quadratic_weak_error}. For every nonconstant
	initial profile, a suitable choice of $\varphi\in\Cf^\infty(\Td)$
	and $t_0\in(0,T]$ makes the correction at least of order $n^{-2}$
	by Proposition~\ref{prop:bound_from_belov_of_correlation_correction}.
	When $\beta>2$, this yields a direct weak error of order $n^{-2}$
	in dimensions $d\ge3$, including $d=4$.
	
	The main theorem goes beyond such single-time observations: it allows
	smooth functionals of the whole trajectory, including observations
	at several times and time integrals. For these observables the
	comparison must keep track of the history already observed, as well
	as the current state. We explain this aspect of the proof next.
	
	\subsection{A path-dependent generator comparison}
	
	Established quantitative comparisons of effective mesoscopic continuum models, e.g. in the context of Dean--Kawasaki dynamics \cite{CoFi2023,CoFiInRa2023}  control smooth observables at
	finitely many prescribed times. Here we seek to develop a general toolbox to directly establish error estimates on full path space, that is, for smooth functionals of the entire SSEP fluctuation
	trajectory. Extending the generator comparison to these observables
	requires retaining the history already observed, since the current
	state alone does not determine the value of a path-dependent observable.
	
	The continuation $U_t^n$ resolves this difficulty by retaining the
	observed history and averaging $F$ over a future generated by the
	modified Gaussian dynamics. Along a microscopic trajectory,
	$U_0^n(\bar\zeta_0^n)$ averages $F$ over the future Gaussian noise
	with the microscopic initial field held fixed, whereas
	$U_T^n(\bar\zeta^{n,[0,T]})=F(\bar\zeta^{n,[0,T]})$.
	As $t$ increases, more of the Gaussian continuation is replaced by
	the actual microscopic path. We express the expected change in this
	quantity through the difference of the two generators. A separate
	initial-field comparison accounts for the difference between the
	microscopic and Gaussian initial laws.
	
	In contrast to the finite-dimensional diffusion setting of
	\citep{Kovacs_Lindner2016}, the fluctuation fields considered here take
	values in infinite-dimensional negative Sobolev spaces, and their
	microscopic trajectories may contain jumps. The Gaussian reference
	has continuous paths, so a backward equation established only on the
	support of its law does not cover these discontinuous histories.
	Moreover, smoothness of the path functional does not ensure the
	classical horizontal regularity of its Gaussian continuation.
	We therefore establish horizontal absolute continuity and an integral
	backward Kolmogorov equation directly for every admissible càdlàg
	history. We also establish a martingale formula for horizontally
	absolutely continuous non-anticipative functionals of Markov jump
	processes. Combining these results yields the path-dependent Trotter
	generator comparison. Its quantitative application requires regularity
	and moment estimates in negative Sobolev spaces that are uniform in $n$.
	The relevant results are Theorem~\ref{thm:backward_kolmogorov_equation} and Proposition~\ref{prop:matingale_proplem_for_nonanticipative_functional}.
	
	More generally, let $X$ and $Y$ be \cdl Markov processes in a
	Banach space $E$, with generators $\cL_t$ and $\cG_t$, and let
	$U_t^F(\gamma)$ retain the history before $t$ and average $F$
	over a continuation from $\gamma_t$ following $Y$.
	Under the regularity and integrability assumptions needed for the
	backward equation and the martingale formula, this argument gives
	\begin{equation}
		\begin{aligned}
			\left|\E F(X^{[0,T]})-\E F(Y^{[0,T]})\right|
			\le{}&
			\left|\E U_0^F(X_0)-\E U_0^F(Y_0)\right|\\
			&+\int_0^T
			\E\left|
			(\cL_t-\cG_t)U_t^F(X^{[0,t]})
			\right|\,dt .
		\end{aligned}
		\label{eq:introduction_path_space_weak_error}
	\end{equation}
	Here the generators act on
	$x\mapsto U_t^F(\gamma^{[0,t)},x)$.
	The first term compares the initial laws, while the integral measures
	the accumulated difference between the dynamics. This is a
	path-dependent version of the Trotter generator comparison used in
	our previous work.
	
	For the SSEP and the modified Gaussian dynamics, the signed comparison
	identity, before taking absolute values, is
	\eqref{eq:path_dependent_trotter_formula}.
	Expanding the microscopic generator isolates the correlation
	correction $\cC^n$. The remaining terms are controlled by the
	approximation properties of the modified spatial operators, estimates
	on microscopic correlations, and a quantitative Gaussian
	approximation of the initial fluctuation field.
	
	Our pathwise differential calculus is based on the horizontal and
	vertical derivatives of Dupire~\citep{Dupire:2019} and Cont and
	Fourni\'e~\citep{Cont_Fournie2010,Cont_Fournie:2013}; see also
	\citep{Bally_Cont:2016}. Related path-dependent weak-error
	representations were developed for finite-dimensional diffusions by
	Kov\'acs and Lindner~\citep{Kovacs_Lindner2016}. Weak approximation
	of SPDE trajectories under spectral Galerkin discretization was
	studied by Br\'ehier, Hairer and
	Stuart~\citep{Brehier_Hairer_Stuart2018}. For stochastic Volterra
	equations and parabolic SPDEs, Andersson, Kov\'acs and
	Larsson~\citep{Andersson_Kovacs_Larsson2016} considered observables
	built from integrals of the path against finite measures, including
	observations at finitely many times.
	
	\subsection{Related work}
	
	The fluctuation theory of interacting particle systems goes back to
	the distribution-valued martingale approach of Holley and
	Stroock~\citep{Holley_Stroock:1978,Holley_Stroock:1979} and to the
	Boltzmann--Gibbs principle, which reduces local fluctuations to those
	of the conserved quantities~\citep{BrRo1984}; see the monographs
	\citep{Kipnis_Landim:1999,KoLaOl2012} for comprehensive accounts.
	Nonequilibrium fluctuation limits have subsequently been established
	in many settings, including weakly asymmetric exclusion
	\citep{DiGa1991}, boundary-driven systems and slow bonds
	\citep{LaMiOl2008,ErFrGoNeTa2020}, and tagged particles
	\citep{Jara:2006}. Second-order Boltzmann--Gibbs principles provide
	a further mechanism for deriving fluctuation limits in weakly
	asymmetric exclusion-type systems; see, for instance,
	Gon\c{c}alves, Jara and Simon~\citep{Goncalves_Jara_Simon2017}.
	
	Quantitative results for local particle systems require sufficiently
	precise control of the microscopic correlations. Quantitative
	Boltzmann--Gibbs estimates obtained through polynomial duality appear
	in~\citep{AyCaRe2018} for independent walkers and suitable interacting
	systems, while correlation estimates for exclusion
	processes go back at least to~\citep{Ferrari_Presutti:1991}. More
	recently, Xu and Zhao~\citep{Xu2025} derived quantitative estimates
	for correlations arising from inhomogeneous product measures in their
	study of occupation-time fluctuations. Their correlation bounds are
	used in the present work to estimate the term left after the Gaussian
	generator comparison. Although the correction itself contains only
	products of fluctuations at neighbouring sites, controlling its
	path-dependent coefficient also requires higher-order correlation
	estimates.
	
	Related quantitative questions have been studied for other models
	and observables. At the hydrodynamic level, Marahrens, Menegaki and
	Mouhot~\citep{Marahrens_Menegaki_Mouhot2024} obtain quantitative
	zero-range limits in dimensions one and two, in Monge--Kantorovich
	distance and relative entropy. For interacting diffusions in continuous
	space at equilibrium, Gu, Mourrat and Nitzschner~\citep{GuMoNi2024} quantify
	two-point density correlations and the Green--Kubo formula, building
	on quantitative homogenization results such as
	\citep{GiGuMo2022}. Dirr, Wu and Zimmer~\citep{Dirr_Wu_Zimmer2026}
	study the quantitative recovery of macroscopic mobility from
	microscopic fluctuations, including applications to the SSEP.
	These results concern hydrodynamic limits, correlations, or transport
	coefficients, whereas the present work gives a quantitative Gaussian
	approximation for nonlinear smooth functionals of the fluctuation
	trajectory. For mean-field systems, fixed-time weak expansions for
	smooth functionals of empirical measures can be developed in the inverse
	particle number using differential calculus on spaces of probability
	measures
	\citep{Chassagneux_Szpruch:2022}.
	
	A closely related direction replaces the microscopic particle system
	by nonlinear conservative SPDEs. For the SSEP, Dirr, Fehrman and
	Gess~\citep{Dirr_Fehrman:2020} construct stochastic continuum
	approximations with the same Gaussian fluctuation limit and dynamical
	large-deviation rate function as the particle system. Cornalba and
	Fischer~\citep{CoFi2023}, and Cornalba, Fischer, Ingmanns and
	Raithel~\citep{CoFiInRa2023}, obtain high-order weak approximations
	for independent and weakly interacting diffusing particles,
	respectively, using discretized Dean--Kawasaki equations. These
	results concern observables at finitely many observation times and
	separate the particle-number and spatial discretization errors.
	Fixed-time weak comparisons with regularized Dean--Kawasaki dynamics
	were obtained for independent particles in~\citep{DjKrPe2024} and
	extended to smooth mean-field systems in~\citep{DjJiPe2026}. At the
	continuum level, Gess, Wu and Zhang~\citep{GeWuZh2026} establish
	higher-order fluctuation expansions in joint small-noise and singular
	limits, including conservative SPDEs related to exclusion models.
	In contrast, the approximation used here remains linear and Gaussian:
	higher-order lattice information is incorporated into its spatial
	operators, while the microscopic correlation contribution is retained
	explicitly in the weak-error expansion.
	
	Finally, our estimates belong to the broader theory of quantitative
	Gaussian approximation on function spaces. Stein's method has been
	developed in this setting by, among others,
	Barbour~\citep{Ba1990}, Shih~\citep{Sh2011}, and Coutin and
	Decreusefond~\citep{CoDe2013}. Barbour, Ross and
	Zheng~\citep{Barbour_Ross_Zheng2023} obtain Gaussian approximation
	bounds for finite-dimensional \cdl processes, and their smoothing
	results~\citep{BaRoZh2024} relate estimates for smooth functionals to
	larger classes of observables, with additional regularity estimates
	for the paths and a loss in the rate. In the present argument,
	quantitative normal approximation enters at the initial time through
	the independent
	occupations, using the exchangeable-pairs framework of
	\citep{Meckes:2009,Reinert_Roellin:2009}. The subsequent comparison
	is dynamical and is carried out through the Gaussian continuation and
	the microscopic generator expansion.
	
	\paragraph{Organization of the paper.}
	In Section~\ref{sec:notation}, we fix the functional-analytic notation
	and the discrete and modified differential operators.
	Section~\ref{sec:microscopic_and_modified_gaussian_dynamics} develops
	the microscopic generator expansion, the modified density equation,
	and the Gaussian dynamics. In Section~\ref{sec:non-anticipative_functionals}, we
	establish the path-dependent calculus, the jump-process martingale
	formula, and the regularity and backward equation for the continuation
	functional. The final main section proves
	Theorem~\ref{thm:main_result}, gives examples of admissible path
	functionals, and establishes the correlation lower bound and the
	quadratic weak-error estimate.

	\section{Notation and preliminaries\label{sec:notation}}
	
	In this section, we introduce the notation and collect some preliminary
	material used throughout the paper. In particular, we recall the notation
	for multilinear operators and Fréchet derivatives, introduce the functional-analytic
	framework in which the fluctuation fields take values, and fix the
	notation for the discrete and modified differential operators used
	below.
	
	\subsection{General notation and function spaces}
	
	We write $\N:=\left\{ 1,2,\ldots\right\} $, $\N_{0}:=\N\cup\{0\}$
	and $[d]=\left\{ 1,\ldots,d\right\} $ for $d\in\N$. For $a,b\in\R$,
	we set $a\wedge b:=\min\left\{ a,b\right\} $ and $a\vee b:=\max\{a,b\}$. 
	
	For a sequence $\{s_{k}\}_{k\in\N}\subset\R$ and $s\in\R$, we write
	$s_{k}\to s-$ (resp. $s_{k}\to s+$) if $s_{k}\to s$ and $s_{k}<s$
	(resp. $s_{k}>s$) for every $k$. We write $s_{k}\uparrow s$ (resp.
	$s_{k}\downarrow s$) if $s_{k}\to s$ and $\{s_{k}\}_{k\in\N}$ is
	non-decreasing (resp. non-increasing).
	
	If $E$ is a metric space, we denote its metric by $\d_{E}$. If $E$
	is a normed or an inner product space, we denote its norm and inner
	product by $\norm{\cdot}_{E}$ and $\inn{\cdot,\cdot}_{E}$, respectively.
	The Borel $\sigma$-algebra on $E$ is denoted by $\cB(E)$. If $X$
	is an $E$-valued random variable, $\law(X)$ denotes its law.

	Throughout the paper, $\Td\coloneqq\R^{d}/\Z^{d}$ denotes the $d$-dimensional
	flat torus. For $I\in\R$, we denote by $\H[I]$ the Sobolev space
	of order $I$ on $\Td$, equipped with the norm 
	\[
	\norm f^{2}_{\H[I]}\coloneqq\sum_{k\in\Z^{d}}\left(1+|k|^{2}\right)^{I}\left|\inn{f,\varsigma_{k}}\right|^{2},
	\]
	where $|\cdot|$ is the Euclidean norm on $\R^{d}$ and
	\[
	\varsigma_{k}(x)\coloneqq e^{2\i\pi k\cdot x},\quad x\in\Td,\quad k\in\Zd.
	\]
	We write $\L=L^{2}(\Td)$ for the space of square-integrable real-valued
	functions on $\Td$, whose inner product is denoted by $\inn{\cdot,\cdot}$.
	
	For $I>0$, we identify $\H[-I]$ with the dual of $\H[I]$ and denote
	the corresponding duality pairing again by $\inn{\cdot,\cdot}$. For
	$I>\frac{d}{2}$, signed measures on $\Td$ with bounded total variation
	are naturally identified with elements of $\H[-I]$ through the duality
	relation 
	\[
	\inn{\varphi,\mu}\coloneqq\int_{\Td}\varphi(x)\,\mu(dx),\quad\varphi\in\H[I].
	\]
	
	The sites of the particle systems are located on the $d$-dimensional
	discrete torus $\Tnd\coloneqq\left\{ \frac{k}{n}:k\in\{0,\ldots,n-1\}^{d}\right\}$,
	viewed as a subset of $\Td$. We denote by $\Ln$ the space of all
	functions from $\Tnd$ to $\R$, equipped with the inner product 
	\[
	\binn{f,g}_{n}\coloneqq\frac{1}{n^{d}}\sum_{x\in\Tnd}f(x)g(x).
	\]
	The corresponding norm is denoted by $\norm{\cdot}_{n}$.
	
	The family $\{\varsigma_k\}_{k\in\Zd}$ is the standard Fourier basis
	of the complexification of $L^2(\Td)$. By a slight abuse of notation,
	we use the same convention for $L^2(\Td)$ and, analogously, for $\Ln$.
	
	For a function $f\in\Ln$ the associated empirical distribution is defined by
	\[
	\bar{f}:=\frac{1}{n^{d}}\sum_{x\in\Tnd}f(x)\delta_{x}.
	\]

	Let $\K$ denote either $\Td$ or $\R^{d}$. For $m\in\N_{0}\cup\{\infty\}$,
	the space of $m$ times continuously differentiable functions on $\K$
	is denoted, as usual, by $\Cf^{m}(\K)$.
	For $A\subseteq\R$, we write
	$\Cf^{m}(\K;A)$ for the functions in $\Cf^{m}(\K)$ taking values in $A$.
	We also write $\Cf(\K)\coloneqq\Cf^{0}(\K)$. 
	We denote by $\Cf_{b}(\R^{d})$ the space of bounded continuous functions
	on $\R^{d}$, and by $\Cf^{m}_{b}(\R^{d})$ the space of functions
	in $\Cf^{m}(\R^{d})$ whose derivatives up to order $m$ are bounded.
	The spaces $\Cf^{m}(\Td)$ and $\Cf^{m}_{b}(\R^{d})$ are equipped
	with their usual uniform norms for $m\in\N_{0}$.
	
	The letter $C$, possibly with subscripts, denotes a positive constant
	whose value may change from line to line. Its dependence on relevant
	parameters is indicated by subscripts. Since the dimension $d$ is
	fixed throughout the paper, we suppress the dependence of constants
	on $d$.
	
	\subsection{Multilinear operators and Fréchet derivatives}
	
	Let $E,\tilde{E}$ be Banach spaces and $m\in\N_{0}$. We denote by
	$\Cf^{m}(E;\tilde{E})$ the space of $m$-times continuously Fréchet
	differentiable functions from $E$ to $\tilde{E}$. We write $\Cf(E;\tilde{E}):=C^{0}(E;\tilde{E})$,
	$\Cf^{m}(E):=\Cf^{m}(E;\R)$ and $\Cf(E):=C^{0}(E)$.
	
	For $F\in\Cf^{m}(E;\tilde{E})$, we denote by $D^{k}F$ its $k$-th
	Fréchet derivative. According to \citep[Theorem 5.3.1]{Cartan:1971},
	$D^{k}F(z)$ can be identified with a continuous symmetric $k$-linear
	operator from $E^{k}$ to $\tilde{E}$. We denote the space of all
	such operators by $\cL_{k}(E;\tilde{E})$ and write $\cL_{k}(E)\coloneqq\cL_{k}(E;\R).$
	The space $\cL_{k}(E;\tilde{E})$, equipped with the norm 
	\[
	\norm A_{\cL_{k}(E;\tilde{E})}\coloneqq\sup\left\{ \bnorm{A[z_{1},\ldots,z_{k}]}_{\tilde{E}}:\norm{z_{i}}_{E}\leq1,\ i\in[k]\right\} ,
	\]
	is a Banach space.
	
	If $E=\R^{d}$ or $E=\Td$ and $\alpha=(\alpha_{j})_{j\in[d]}\in\N^{d}_{0}$
	is a multi-index, we use the notation
	\[
	D^{\alpha}F(x)\coloneqq D^{|\alpha|_{1}}F(x)\left[e^{\otimes\alpha_{1}}_{1},\ldots,e^{\otimes\alpha_{d}}_{d}\right],
	\]
	where $(e_{j})_{j\in[d]}$ denotes the canonical basis of $\R^{d}$,
	$|\alpha|_{1}\coloneqq\sum^{d}_{j=1}\alpha_{j}$, and $z^{\otimes k}=(z,\ldots,z)\in E^{k}$.
	Note that if $\tilde{E}=\R$ then $D^{\alpha}$ coincides with the
	usual partial derivative $\partial^{\alpha}\coloneqq\partial^{\alpha_{1}}_{1}\cdots\partial^{\alpha_{d}}_{d}$,
	where $\partial_{j}$ is the partial derivative in the $j$-th coordinate.
	
	For $m\geq1$, we denote by $\Cf^{m}_{\text{der}}(E;\tilde{E})$ the
	space of functions $F\in\Cf^{m}(E;\tilde{E})$ whose derivatives of
	orders $1,\ldots,m$ are bounded, and write $\Cf^{m}_{\text{der}}(E)\coloneqq\Cf^{m}_{\text{der}}(E;\R).$
	We equip $\Cf^{m}_{\text{der}}(E;\tilde{E})$ with the seminorm 
	\[
	\norm F_{\Cf^{m}_{\text{der}}(E;\tilde{E})}\coloneqq\max_{k\in[m]}\sup_{z\in E}\bnorm{D^{k}F(z)}_{\cL_{k}(E;\tilde{E})}.
	\]
	
	For $F\in\Cf^{m}(E)$ we recall Taylor's formula
	\citep[Sections 1.5.5--1.5.6]{Cartan:1971}
	\begin{align}
		F(z) & =F(z_{0})+\sum^{m-1}_{k=1}\frac{1}{k!}D^{k}F(z_{0})\left[(z-z_{0})^{\otimes k}\right]+\frac{1}{m!}D^{m}F(\tilde{z})\left[(z-z_{0})^{\otimes m}\right]\label{eq:taylors_formula}
	\end{align}
	where $\tilde{z}\in E$ belongs to the segment joining $z_{0}$ and
	$z$.
	
	We will often consider functions $F\in\Cf^{\ell}(\H[-I])$, where
	$I>\frac{d}{2}+m$ for some $m\in\N_{0}$. In this case, $DF(f)\in\H[I]$
	for every $f\in\H[-I]$. By the Sobolev embedding theorem, $DF(f)$
	admits a representative $DF(f;\cdot)\in\Cf^{m}(\Td)$ such that
	\[
	DF(f)[\mu]=\binn{DF(f;\cdot),\mu}
	\]
	for every finite signed measure $\mu$ on $\Td$. 
	
	More generally, for $k\in\N$ and $m_{1},\ldots,m_{k}\in\N_{0}$,
	we denote by $\Cf^{m_{1},\ldots,m_{k}}((\Td)^{k})$ the space of functions
	$K:(\Td)^{k}\to\R$ such that $\partial^{\alpha_{1}}_{x_{1}}\cdots\partial^{\alpha_{k}}_{x_{k}}K$
	exists and is continuous for each multi-index $\alpha_{i}\in\N^{d}_{0}$
	satisfying $|\alpha_{i}|_{1}\le m_{i}$, $i\in[k]$. We equip this
	space with the norm
	\[
	\bnorm K_{\Cf^{m_{1},\ldots,m_{k}}((\Td)^{k})}
	:=
	\max\left\{
	\bnorm{
		\partial^{\alpha_{1}}_{x_{1}}\cdots
		\partial^{\alpha_{k}}_{x_{k}}K
	}_{\Cf((\Td)^{k})}
	:
	\alpha_{i}\in\N^{d}_{0},\
	|\alpha_{i}|_{1}\le m_{i},\
	i\in[k]
	\right\}.
	\]
	
	As we will show in Section \ref{subsec:kernels_of_multilinear_functionals},
	for every $f\in\H[-I]$ and $k\in[\ell]$, the symmetric multilinear
	functional $D^{k}F(f)\in\cL_{k}(\H[-I])$ admits a kernel in $\Cf^{m,\ldots,m}((\Td)^{k})$.
	We denote this kernel by
	\[
	D^{k}F(f;x_{1},\ldots,x_{k}),\quad x_{1},\ldots,x_{k}\in\Td.
	\]
	Thus, for finite signed measures $\mu_{1},\ldots,\mu_{k}$ on $\Td$,
	\[
	D^{k}F(f)[\mu_{1},\ldots,\mu_{k}]=\int_{(\Td)^{k}}D^{k}F(f;x_{1},\ldots,x_{k})\prod^{k}_{i=1}\mu_{i}(dx_{i}).
	\]

	\subsection{Trigonometric interpolation}
	
	We next introduce the trigonometric interpolation operator that will
	be used throughout the paper to identify functions on the discrete
	torus $\Tnd$ with smooth functions on $\Td$. We also collect the
	approximation estimates needed below.
	\begin{rem}\label{rem:choice_of_n}
		The results are stated for all integers $n\ge3$. To simplify
		notation in the Fourier calculations, we present the arguments
		for $n=2N+1$, with $N\in\N$, and set
		\[
		\Znd\coloneqq\{-N,\ldots,N\}^{d}.
		\]
		For even $n=2N$, we use the standard real-valued trigonometric
		interpolation convention at the highest frequencies, splitting
		the corresponding coefficients equally between the frequencies
		$N$ and $-N$ in each coordinate. The interpolation and Sobolev
		estimates used below remain valid with the same powers of $n$
		and with constants independent of $n$. We omit the corresponding
		minor modifications of the Fourier formulas.
	\end{rem}
	
	We define the trigonometric interpolation operator
	\[
	\ex_{n}:\Ln\to\Cf^{\infty}(\Td)
	\]
	by
	\begin{equation}
		\ex_{n}f(x)=\sum_{k\in\Znd}\binn{f,\varsigma_{k}}_{n}\varsigma_{k}(x),\quad x\in\Td.\label{eq:ex_definition}
	\end{equation}
	Equivalently,
	\[
	\ex_{n}f=\sum_{k\in\Znd}\inn{\bar{f},\varsigma_{k}}\varsigma_{k}.
	\]
	If $f$ is defined on $\Td$, we interpret $\ex_{n}f$ as the application
	of $\ex_{n}$ to the restriction $f\big|_{\Tnd}$.
	
	\begin{lem}
		\label{lem:approximation_for_ex}Let $m\in\N$. Then there exists
		a constant $C_{m}>0$ such that, for every $f\in\Cf^{m}(\Td)$,
		\[
		\bnorm{f-\ex_{n}f}_{\Cf(\Td)}\le\frac{C_{m}\left(1+\ln^{d}n\right)}{n^{m}}\bnorm f_{\Cf^{m}(\Td)}
		\]
		and 
		\[
		\left|\binn{f,\varsigma_{k}}-\binn{f,\varsigma_{k}}_{n}\right|\le\frac{C_{m}}{n^{m}}\bnorm f_{\Cf^{m}(\Td)}
		\]
		for all $k\in\Znd$.
	\end{lem}
	
	\begin{proof}
		We set 
		\[
		V_{n}:=\spann\left\{ \varsigma_{k},k\in\Znd\right\} \subset\Cf(\Td).
		\]
		Then 
		\begin{align*}
			\ex_{n}f(x) & =\frac{1}{n^{d}}\sum_{y\in\Tnd}\sum_{k\in\Znd}f(y)\bar{\varsigma}_{k}(y)\varsigma_{k}(x)\\
			&=\sum_{y\in\Tnd}\left(\frac{1}{n^{d}}\sum_{k\in\Znd}\varsigma_{k}(x-y)\right)f(y)=\sum_{y\in\Tnd}\phi_{n}(x,y)f(y),
		\end{align*}
		where
		\[
		\phi_{n}(x,y)=\frac{1}{n^{d}}\sum_{k\in\Znd}\varsigma_{k}(x-y).
		\]
		
		Using the fact that $\ex_{n}g=g$ for $g\in V_{n}$, we can estimate
		\begin{align*}
			\norm{f-\ex_{n}f}_{\Cf(\Td)} & \le\norm{f-g}_{\Cf(\Td)}+\norm{g-\ex_{n}f}_{\Cf(\Td)}\le\left(1+\norm{\ex_{n}}_{op}\right)\norm{f-g}_{\Cf(\Td)}.
		\end{align*}
		The operator norm $\norm{\ex_n}_{op}$ in $\Cf(\Td)$ is the Lebesgue
		constant associated with trigonometric interpolation and satisfies
		\[
		\norm{\ex_n}_{op}\le C(1+\ln^d n).
		\]
		For $d=1$, the corresponding logarithmic estimate for equidistant
		trigonometric interpolation is classical; see
		\citep[pp.~453--456]{Jackson1913}. See also
		\citep[pp.~30--31]{Fejer1910} for the closely related logarithmic
		estimate of the Fourier Lebesgue constants. The multidimensional
		estimate follows by tensorization.

		Taking the infimum over $g\in V_{n}$, we obtain 
		\begin{align*}
			\norm{f-\ex_{n}f}_{\Cf(\Td)} & \le C(1+\ln^{d}n)\inf_{g\in V_{n}}\norm{f-g}_{\Cf(\Td)}.
		\end{align*}
		Combining this estimate with Jackson's theorem yields
		\[
		\bnorm{f-\ex_{n}f}_{\Cf(\Td)}\le\frac{C_{m}\left(1+\ln^{d}n\right)}{n^{m}}\bnorm f_{\Cf^{m}(\Td)}.
		\]
		
		We now prove the second inequality in the statement of the lemma.
		Let $g\in V_{n}$ and $k\in\Znd$. Then
		\[
		\left|\langle f,\varsigma_{k}\rangle_{n}-\langle f,\varsigma_{k}\rangle\right|\le\left|\langle f-g,\varsigma_{k}\rangle_{n}\right|+\left|\langle g-f,\varsigma_{k}\rangle\right|\le2\norm{f-g}_{C(\Td)}.
		\]
		Taking the infimum over $g\in V_{n}$ and applying Jackson's theorem
		gives
		\[
		\left|\langle f,\varsigma_{k}\rangle_{n}-\langle f,\varsigma_{k}\rangle\right|\le\frac{C_{m}}{n^{m}}\norm f_{C^{m}(\Td)}.
		\]
		This completes the proof of the lemma.
	\end{proof}
	
	\begin{cor}
		\label{cor:ex_n_approximation_in_H_I_norm}Let $I\ge0$ and $m\in\N$
		satisfy $m\ge I+\frac{d}{2}.$ Then there exists a constant $C_{I,m}>0$
		such that, for every $f\in\Cf^{m}(\Td)$,
		\[
		\bnorm{f-\ex_{n}f}_{\H[I]}\le\frac{C_{I,m}}{n^{m-I-\frac{d}{2}}}\bnorm f_{\Cf^{m}(\Td)}.
		\]
	\end{cor}
	
	\begin{proof}
		Using the Fourier representation of $\ex_{n}f$ and Lemma \ref{lem:approximation_for_ex},
		we obtain
		\begin{align*}
			\bnorm{f-\ex_{n}f}^{2}_{\H[I]} & =\sum_{k\in\Znd}(1+|k|^{2})^{I}\left|\binn{f,\varsigma_{k}}-\binn{f,\varsigma_{k}}_{n}\right|^{2}+\sum_{k\notin\Znd}(1+|k|^{2})^{I}\left|\binn{f,\varsigma_{k}}\right|^{2}\\
			& \le\frac{C_{I,m}}{n^{2m}}\bnorm f^{2}_{\Cf^{m}(\Td)}\sum_{k\in\Znd}(1+|k|^{2})^{I}+\sum_{k\notin\Znd}(1+|k|^{2})^{I}\left|\binn{f,\varsigma_{k}}\right|^{2}.
		\end{align*}
		Since
		\[
		\sum_{k\in\Znd}(1+|k|^{2})^{I}\le C_{I}n^{d+2I},
		\]
		the first term is bounded by $\frac{C_{I,m}}{n^{2m-d-2I}}\bnorm f^{2}_{\Cf^{m}(\Td)}.$
		Moreover, for $k\notin\Znd$ one has $|k|\ge cn$. Hence
		\[
		\sum_{k\notin\Znd}(1+|k|^{2})^{I}\left|\binn{f,\varsigma_{k}}\right|^{2}\le\frac{C_{I,m}}{n^{2(m-I)}}\bnorm f^{2}_{\H[m]}\le\frac{C_{I,m}}{n^{2(m-I)}}\bnorm f^{2}_{\Cf^{m}(\Td)}.
		\]
		
		Combining the two estimates yields
		\[
		\bnorm{f-\ex_{n}f}^{2}_{\H[I]}\le\frac{C_{I,m}}{n^{2\left(m-I-\frac{d}{2}\right)}}\bnorm f^{2}_{\Cf^{m}(\Td)}.
		\]
		This proves the claim.
	\end{proof}
	
	\begin{cor}
		\label{cor:bound_of_ex}Let $I\ge0$ and let $m\in\N$ satisfy $m\ge\frac{d}{2}+I$.
		Then there exists a constant $C_{I,m}>0$ such that, for every $f\in\Cf^{m}(\Td)$,
		\[
		\bnorm{\ex_{n}f}_{\H[I]}\le C_{I,m}\bnorm f_{\Cf^{m}(\Td)}.
		\]
	\end{cor}
	
	\begin{proof}
		By the triangle inequality and Corollary (\ref{cor:ex_n_approximation_in_H_I_norm}),
		\[
		\bnorm{\ex_{n}f}_{\H[I]}\le\bnorm f_{\H[I]}+\bnorm{f-\ex_{n}f}_{\H[I]}\le C\bnorm f_{\Cf^{m}(\Td)}.
		\]
		This proves the claim.
	\end{proof}
	
	The following lemma relates the empirical distribution associated
	with a function on the discrete torus to its trigonometric interpolation.
	In particular, it allows us to transfer estimates for $\ex_{n}f$
	to $\bar{f}$ in negative Sobolev spaces.
	\begin{lem}
		Let $I>\frac{d}{2}$. Then there exists a constant $C_{I}>0$ such
		that
		\[
		\bnorm{\bar{f}}_{\H[-I]}\le C\bnorm{\ex_{n}f}_{\H[-I]}
		\]
		for all $f\in\Ln.$
	\end{lem}
	
	\begin{proof}
		Using the definition of Sobolev norm, we get
		\[
		\bnorm{\bar{f}}^{2}_{\H[-I]}=\sum_{k\in\Zd}\left(1+|k|^{2}\right)^{-I}\left|\binn{\bar{f},\varsigma_{k}}\right|^{2}=\sum_{k\in\Zd}\left(1+|k|^{2}\right)^{-I}\left|\binn{f,\varsigma_{k}}_{n}\right|^{2}.
		\]
		Since $\Z^{d}_{n}=\{-N,\ldots,N\}^{d}$ is a complete set of representatives
		of $\Zd/n\Zd$, every $k\in\Zd$ has a unique representation
		\[
		k=q+n\ell,\quad q\in\Z^{d}_{n},\ \ell\in\Zd.
		\]
		Moreover
		\[
		\inn{f,\varsigma_{q+n\ell}}_{n}=\inn{f,\varsigma_{q}}_{n},
		\]
		since $\varsigma_{q+n\ell}=\varsigma_{q}$ on $\T^{d}_{n}$. Hence
		\[
		\bnorm{\bar{f}}^{2}_{\H[-I]}=\sum_{q\in\Z^{d}_{n}}\left|\langle f,\varsigma_{q}\rangle_{n}\right|^{2}\sum_{\ell\in\Zd}\left(1+|q+n\ell|^{2}\right)^{-I}.
		\]
		Note that for some constant $C>0$ and all $q\in\Znd$, $\ell\in\Zd\backslash\{0\}$
		one has
		\[
		|q+n\ell|\ge Cn|\ell|\ge C(1+|q|)(1+|\ell|).
		\]
		Since $I>\frac{d}{2}$, we obtain
		\[
		\sum_{\ell\in\Zd}\left(1+|q+n\ell|^{2}\right)^{-I}\le C_{I}\left(1+|q|^{2}\right)^{-I},
		\]
		uniformly in $q\in\Z^{d}_{n}$ and $n$. Therefore,
		\[
		\bnorm{\bar{f}}^{2}_{\H[-I]}\le C_{I}\sum_{q\in\Z^{d}_{n}}\left(1+|q|^{2}\right)^{-I}\left|\binn{f,\varsigma_{q}}_{n}\right|^{2}=C_{I}\bnorm{\ex_{n}f}^{2}_{\H[-I]}.
		\]
		Taking square roots completes the proof.
	\end{proof}

	\subsection{Discrete and differential operators}
	
	We next introduce the differential and discrete operators used in
	the paper. In particular, we define finite-order differential approximations
	of the discrete operators and collect several estimates for their
	approximation errors.
	
	We start with the definition of the usual differential operators and
	their discrete counterparts. The partial derivative in the $j$-th
	coordinate direction is denoted by $\partial_{j}$, $j\in[d]$. For
	a multi-index $\alpha=(\alpha_{j})_{j\in[d]}\in\N^{d}_{0}$, we set
	\[
	|\alpha|\coloneqq\sum^{d}_{j=1}\alpha_{j},\quad\partial^{\alpha}\coloneqq\partial^{\alpha_{1}}_{1}\cdots\partial^{\alpha_{d}}_{d}.
	\]
	We also set 
	\[
	\nabla f\coloneqq(\partial_{j}f)_{j\in[d]},\quad\Delta f\coloneqq\sum^{d}_{j=1}\partial^{2}_{j}f.
	\]
	
	For $j\in[d]$, let $e^{n}_{j}\coloneqq e_{j}/n,$ where $(e_{j})_{j\in[d]}$
	denotes the canonical basis of $\R^{d}$. For $f\in\Ln$, we define
	the discrete shift and derivative operators by 
	\[
	\tau_{n,j}f\coloneqq f(\cdot+e^{n}_{j}),\quad\partial_{n,j}f\coloneqq n\bigl(f(\cdot+e^{n}_{j})-f\bigr),\quad j\in[d].
	\]
	Set
	\[
	\nabla_{n}f\coloneqq(\partial_{n,j}f)_{j\in[d]}
	\]
	and 
	\begin{equation}
		\Delta_{n}f\coloneqq-\sum^{d}_{j=1}\partial^{*}_{n,j}\partial_{n,j}f=n^{2}\sum^{d}_{j=1}\left(f(\cdot+e^{n}_{j})+f(\cdot-e^{n}_{j})-2f\right),\label{eq:definition_of_Delta_n}
	\end{equation}
	where $\partial^{*}_{n,j}$ denotes the adjoint operator of $\partial_{n,j}$
	on $\Ln$.
	
	If $f:(\Tnd)^{2}\to\R$, we define the mixed discrete derivative by
	\[
	\begin{aligned}\partial^{\otimes2}_{n,j}f(x,y)\coloneqq n^{2}\big[ & f(x+e^{n}_{j},y+e^{n}_{j})-f(x,y+e^{n}_{j})\\
		& -f(x+e^{n}_{j},y)+f(x,y)\big],\quad x,y\in\Tnd,
	\end{aligned}
	\]
	and its restriction to the diagonal by 
	\[
	\partial^{\odot2}_{n,j}f(x)\coloneqq\partial^{\otimes2}_{n,j}f(x,x),\quad x\in\Tnd.
	\]
	Higher-order tensor derivatives $\partial^{\otimes k}_{n,j}$ are
	defined analogously, and $\partial^{\odot k}_{n,j}$ denotes the restriction
	of $\partial^{\otimes k}_{n,j}$ to the diagonal.
	We use the same notation for the corresponding operators on functions
	on $\Td$, defined by the same formulas. Likewise,
	$\partial^{\otimes k}_{n,j}$ and $\partial^{\odot k}_{n,j}$ are
	defined on functions on $(\Td)^k$ by the same expressions.
	
	The discrete operators admit the formal representation
	\[
	\tau_{n,j}=e^{\partial_{j}/n},\quad\partial_{n,j}=n(e^{\partial_{j}/n}-1)\quad\text{and}\quad\Delta_{n}=n^{2}\sum^{d}_{j=1}\left(e^{\partial_{j}/n}+e^{-\partial_{j}/n}-2\right),
	\]
	for $j\in[d]$. These representations suggest approximating the discrete
	operators by replacing the exponential function with its Taylor polynomial.
	For $m\in\N_{0}$, let
	\[
	P_{m}(z):=\sum^{m}_{k=0}\frac{z^{k}}{k!}.
	\]
	We then define the corresponding modified differential operators by
	\[
	\tau^{(m)}_{n,j}:=P_{m}\left(\frac{\partial_{j}}{n}\right),\quad\partial^{(m)}_{n,j}:=n\left[P_{m}\left(\frac{\partial_{j}}{n}\right)-1\right]
	\]
	and
	\begin{equation}
		\Delta^{(2m)}_{n}:=n^{2}\sum^{d}_{j=1}\left[P_{2m}\left(\frac{\partial_{j}}{n}\right)+P_{2m}\left(-\frac{\partial_{j}}{n}\right)-2\right].\label{eq:modified_laplacian}
	\end{equation}
	
	\begin{lem}
		\label{lem:approximation_of_discrete_derivatives} Let $m\in\N_{0}$,
		$\ell\in\N_{0}$ and $\varphi\in\Cf^{m+\ell+1}(\Td)$. Then there
		exists a constant $C_{m,\ell}>0$ such that, for every $n\in\N$ and
		$j\in[d]$
		\begin{align*}
			\bnorm{\tau_{n,j}\varphi-\tau^{(m)}_{n,j}\varphi}_{\Cf^{\ell}(\Td)} & \le\frac{C_{m,\ell}}{n^{m+1}}\norm{\partial^{m+1}_{j}\varphi}_{\Cf^{\ell}(\Td)},
		\end{align*}
		and
		\[
		\bnorm{\partial_{n,j}\varphi-\partial^{(m)}_{n,j}\varphi}_{\Cf^{\ell}(\Td)}\le\frac{C_{m,\ell}}{n^{m}}\bnorm{\partial^{m+1}_{j}\varphi}_{\Cf^{\ell}(\Td)}.
		\]
		Moreover, if $\varphi\in\Cf^{2m+\ell+2}(\Td)$, then
		\[
		\bnorm{\Delta_{n}\varphi-\Delta^{(2m)}_{n}\varphi}_{\Cf^{\ell}(\Td)}\le\frac{C_{m,\ell}}{n^{2m}}\sum^{d}_{j=1}\bnorm{\partial^{2m+2}_{j}\varphi}_{\Cf^{\ell}(\Td)}.
		\]
	\end{lem}
	
	\begin{proof}
		The statement follows directly from Taylor's expansion.
	\end{proof}
	
	We recall that the family $\left\{ \varsigma_{k},k\in\Znd\right\} $
	forms an orthonormal basis of $\Ln$. Since
	\[
	\partial_{j}\varsigma_{k}=2\pi\i k_{j}\varsigma_{k},
	\]
	we obtain that $\varsigma_{k}$, $k\in\Znd$, are eigenfunctions of
	$\Delta_{n}$ with corresponding eigenvalues
	\begin{align*}
		\lambda_{k,n}: & =n^{2}\sum^{d}_{j=1}\left(e^{2\pi\i k_{j}/n}+e^{-2\pi\i k_{j}/n}-2\right)=2n^{2}\sum^{d}_{j=1}\left[\cos\frac{2\pi}{n}k_{j}-1\right],\quad k\in\Znd.
	\end{align*}
	Similarly, $\varsigma_{k}$, $k\in\Zd$, are eigenfunctions of $\Delta^{(2m)}_{n}$
	associated with eigenvalues 
	\begin{align*}
		\lambda^{(2m)}_{k,n} & :=n^{2}\sum^{d}_{j=1}\left[P_{2m}\left(\frac{2\pi\i k_{j}}{n}\right)+P_{2m}\left(-\frac{2\pi\i k_{j}}{n}\right)-2\right]\\
		& =2n^{2}\sum^{d}_{j=1}\left[p_{m}\left(\frac{2\pi k_{j}}{n}\right)-1\right],\quad k\in\Zd,
	\end{align*}
	where 
	\[
	p_{m}(z)\coloneqq\sum^{m}_{l=0}\frac{(-1)^{l}z^{2l}}{(2l)!}
	\]
	is the Taylor polynomial of $\cos z$ of degree $2m$. 
	
	The following lemma quantifies the approximation of the eigenvalues
	of $\Delta_{n}$ by those of $\Delta^{(2m)}_{n}$.
	\begin{lem}
		\label{lem:estimate_of_eigenvalues} For every $m\in\N$ there exists
		a constant $C_{m}>0$ such that for all $n\in\N$ and $k\in\Z^{d}_{n}$,
		\[
		\left|\lambda_{k,n}-\lambda^{(2m)}_{k,n}\right|\le\frac{C_{m}}{n^{2m}}|k|^{2m+2}.
		\]
		Moreover, if $m$ is odd, then 
		\[
		\lambda^{(2m)}_{k,n}\le0,\quad k\in\Zd,
		\]
		and for every $n\in\N$
		\[
		\lim_{|k|\to\infty}\frac{-\lambda^{(2m)}_{k,n}}{\sum^{d}_{j=1}|k_{j}|^{2m}}=\frac{2(2\pi)^{2m}}{(2m)!n^{2m-2}}.
		\]
	\end{lem}
	
	\begin{proof}
		The first estimate follows from Taylor's theorem applied to $\cos z-1$.
		
		For odd $m$, the Taylor polynomial $p_{m}$ of $\cos z$ of degree
		$2m$ lies below $\cos z$ on $\mathbb{R}$. Hence
		\[
		p_{m}(z)-1\le\cos z-1\le0.
		\]
		This implies that
		\[
		\lambda^{(2m)}_{k,n}=2n^{2}\sum^{d}_{j=1}\left[p_{m}\left(\frac{2\pi k_{j}}{n}\right)-1\right]\le0.
		\]
		
		The last part of the lemma follows from the expression of $\lambda^{(2m)}_{k,n}$.
		The proof of the lemma is complete.
	\end{proof}

	\subsection{Kernels of multilinear functionals on Sobolev spaces\label{subsec:kernels_of_multilinear_functionals}}
	
	We conclude Section \ref{sec:notation} with two auxiliary results
	concerning Dirac distributions and kernels associated with multilinear
	functionals on negative Sobolev spaces. They will allow us to represent
	Fréchet derivatives of functionals on $\H[-I]$ by sufficiently regular
	functions on finite products of $\Td$.
	
	For $I>\frac{d}{2}$, we identify the Dirac measure $\delta_{x}$,
	$x\in\Td$, with an element of $\H[-I]$ and denote by
	
	\[
	\delta_{\cdot}:\Td\to\H[-I],\quad x\mapsto\delta_{x},
	\]
	the corresponding map. We first establish the regularity of this map
	as a function taking values in a negative Sobolev space.
	\begin{lem}
		\label{lem:differentiability_of_delta} Let $m\in\N_{0}$ and $I>\frac{d}{2}+m$.
		Then $\delta_{\cdot}\in\Cf^{m}(\Td;\H[-I]).$ Moreover, for every
		multi-index $\alpha\in\N^{d}_{0}$ with $|\alpha|_{1}\leq m$ and
		every $x\in\Td$, 
		\[
		\binn{D^{\alpha}\delta_{x},\varphi}=\partial^{\alpha}\varphi(x),\quad\varphi\in\H[I],
		\]
		and $\bnorm{D^{\alpha}\delta_{x}}_{\H[-I]}$ is finite and independent
		of $x$. 
	\end{lem}
	
	\begin{proof}
		For a multi-index $\alpha\in\N^{d}_{0}$ with $|\alpha|_{1}\le m$,
		define $\delta^{(\alpha)}_{x}\in\H[-I]$ by 
		\[
		\binn{\delta^{(\alpha)}_{x},\varphi}:=\partial^{\alpha}\varphi(x),\quad\varphi\in\H[I].
		\]
		The map $\delta^{(\alpha)}_{x}$ is an element of $\H[-I]$ due to
		\[
		\begin{aligned}\bnorm{\delta^{(\alpha)}_{x}}^{2}_{\H[-I]} & =\sum_{k\in\Zd}\frac{|\partial^{\alpha}\varsigma_{k}(x)|^{2}}{(1+|k|^{2})^{I}} =(2\pi)^{2|\alpha|_{1}}\sum_{k\in\Zd}\frac{|k^{\alpha}|^{2}}{(1+|k|^{2})^{I}}<\infty,
		\end{aligned}
		\]
		since $I>\frac{d}{2}+m$. In particular, this norm is independent
		of $x$.
		
		Moreover, for $x,y\in\Td$, 
		\[
		\bnorm{\delta^{(\alpha)}_{x}-\delta^{(\alpha)}_{y}}^{2}_{\H[-I]}=(2\pi)^{2|\alpha|_{1}}\sum_{k\in\Zd}\frac{|k^{\alpha}|^{2}|\varsigma_{k}(x)-\varsigma_{k}(y)|^{2}}{(1+|k|^{2})^{I}}.
		\]
		Hence, by dominated convergence, 
		\[
		\bnorm{\delta^{(\alpha)}_{x}-\delta^{(\alpha)}_{y}}_{\H[-I]}\longrightarrow0\quad\text{as }y\to x.
		\]
		Thus $x\mapsto\delta^{(\alpha)}_{x}$ is continuous.
		
		We now show by induction over $|\alpha|_{1}$ that 
		\[
		D^{\alpha}\delta_{x}=\delta^{(\alpha)}_{x}.
		\]
		The assertion is immediate for $\alpha=0$. Suppose that it holds
		for some $\alpha$ with $|\alpha|_{1}<m$, and let $j\in[d]$. Then
		\[
		\begin{aligned} & \frac{1}{\eps^{2}}\bnorm{\delta^{(\alpha)}_{x+\eps e_{j}}-\delta^{(\alpha)}_{x}-\eps\delta^{(\alpha+e_{j})}_{x}}^{2}_{\H[-I]}=\sum_{k\in\Zd}\frac{|\partial^{\alpha}\varsigma_{k}(x)|^{2}}{(1+|k|^{2})^{I}}\left|\frac{e^{2\pi\i k_{j}\eps}-1}{\eps}-2\pi\i k_{j}\right|^{2}.\end{aligned}
		\]
		The right-hand side converges to zero by the dominated convergence
		theorem. Consequently, 
		\[
		D_{j}\delta^{(\alpha)}_{x}=\delta^{(\alpha+e_{j})}_{x}.
		\]
		The induction therefore yields 
		\[
		D^{\alpha}\delta_{x}=\delta^{(\alpha)}_{x}
		\]
		for every $|\alpha|_{1}\le m$. Hence $\delta_{\cdot}\in\Cf^{m}(\Td;\H[-I])$
		and 
		\[
		\binn{D^{\alpha}\delta_{x},\varphi}=\partial^{\alpha}\varphi(x).
		\]
		This completes the proof of the statement.
	\end{proof}
	
	\begin{rem}
		If $\partial^{\alpha}\delta_{x}$ denotes the distributional derivative
		of $\delta_{x}$, then
		\[
		D^{\alpha}\delta_{x}=(-1)^{|\alpha|_{1}}\partial^{\alpha}\delta_{x}.
		\]
		Indeed,
		\[
		\binn{D^{\alpha}\delta_{x},\varphi}=\partial^{\alpha}\varphi(x)=(-1)^{|\alpha|_{1}}\binn{\partial^{\alpha}\delta_{x},\varphi}.
		\]
	\end{rem}
	
	The regularity of $\delta_{\cdot}$ allows us to associate with every
	continuous multilinear functional on $\H[-I]$ a regular kernel on
	a finite product of $\Td$.
	\begin{lem}
		\label{lem:existence_of_smooth_kernel} Let $k\in\N,$ $m\in\N_{0}$,
		$I>\frac{d}{2}+m$ and $A\in\cL_{k}(\H[-I])$. Then the function
		\[
		K_{A}(x_{1},\ldots,x_{k}):=A[\delta_{x_{1}},\ldots,\delta_{x_{k}}],\quad x_{i}\in\Td,\ i\in[k],
		\]
		belongs to $\Cf^{m,\ldots,m}\left((\T^{d})^{k}\right)$ and is a
		kernel associated with $A$, i.e. for finite signed measures $\mu_{1},\ldots,\mu_{k}$
		it holds 
		\begin{equation}
			A[\mu_{1},\ldots,\mu_{k}]=\int_{(\T^{d})^{k}}K_{A}(x_{1},\ldots,x_{k})\mu_{1}(dx_{1})\ldots\mu_{k}(dx_{k}).\label{eq:expression_for_A_via_kernel}
		\end{equation}
		Moreover, for every multi-index $\alpha_{i}\in\N_{0}^d$ satisfying
		$|\alpha_{i}|_{1}\le m$,
		\[
		\partial^{\alpha_{1}}_{x_{1}}\ldots\partial^{\alpha_{k}}_{x_{k}}K_{A}(x_{1},\ldots,x_{k})=A\left[D^{\alpha_{1}}\delta_{x_{1}},\ldots,D^{\alpha_{k}}\delta_{x_{k}}\right],\quad x_{i}\in\Td,\ i\in[k],
		\]
		and there exists a constant $C_{m,I,k}>0$ independent of $A$ such
		that
		\[
		\bnorm{K_{A}}_{\Cf^{m,\ldots,m}\left((\T^{d})^{k}\right)}\le C_{m,I,k}\bnorm A_{\cL_{k}(\H[-I])}.
		\]
	\end{lem}
	
	\begin{proof}
		The continuous differentiability of $K_{A}$ follows from Lemma \ref{lem:differentiability_of_delta}
		and \citep[Theorems 2.2.1, 2.4.3]{Cartan:1971}. By the Hahn decomposition
		theorem, it is enough to prove (\ref{eq:expression_for_A_via_kernel})
		only for finite measures $\mu_{i}$, $i\in[k]$, on $\Td$. 
		
		Since the function $\delta_{\cdot}$ is continuous it is measurable
		and, hence, strongly measurable due to the separability of $\H[-I]$.
		Moreover, $\norm{\delta_{x}}_{\H[-I]}$, $x\in\Td$, is constant.
		Consequently, $\delta_{\cdot}$ is Bochner integrable for each finite
		measure $\mu$ on $\Td$ and 
		\[
		\int_{\Td}\delta_{x}\mu(dx)=\mu.
		\]
		Indeed, the Bochner integrability follows from \citep[Proposition 1.2.2]{Veraar:2016}.
		Moreover, for every $\varphi\in\H[I]$, one has
		\[
		\binn{\int_{\Td}\delta_{x}\mu(dx),\varphi}=\int_{\Td}\inn{\delta_{x},\varphi}\mu(dx)=\int_{\Td}\varphi(x)\mu(dx)=\inn{\varphi,\mu}.
		\]
		
		Since $A$ is a bounded multilinear functional, we get
		\begin{align*}
			A[\mu_{1},\ldots,\mu_{k}] & =A\left[\int_{\Td}\delta_{x_{1}}\mu_{1}(dx_{1}),\ldots,\int_{\Td}\delta_{x_{k}}\mu_{k}(dx_{k})\right]\\
			& =\int_{\Td}\ldots\int_{\Td}A\left[\delta_{x_{1}},\ldots,\delta_{x_{k}}\right]\mu_{1}(dx_{1})\ldots\mu_{k}(dx_{k})\\
			& =\int_{\Td}\ldots\int_{\Td}K_{A}(x_{1},\ldots,x_{k})\mu_{1}(dx_{1})\ldots\mu_{k}(dx_{k}),
		\end{align*}
		by \citep[Theorem 1.2.4]{Veraar:2016}.
		
		For a multi-index $\alpha_{i}\in\N^{d}_{0}$ denote by $\partial^{\alpha_{i}}_{x_{i}}K_{A}(x_{1},\ldots,x_{k})$
		the derivative with respect to variable $x_{i}\in\Td$ for $i\in[k]$.
		Then for each $\alpha_{i}\in\N^{d}_{0}$ satisfying $|\alpha_{i}|_{1}\le m$
		we get $D^{\alpha_{i}}\delta_{x_{i}}\in\H[-I]$ and
		\[
		C_{\alpha_{i},I,d}:=\bnorm{D^{\alpha_{i}}\delta_{x_{i}}}_{\H[-I]}
		\]
		is independent of $x_{i}$. Thus, by the continuity of the multilinear
		operator $A$ and the differentiability of $\delta_{\cdot}$, we get
		\begin{align*}
			\left|\partial^{\alpha_{1}}_{x_{1}}\ldots\partial^{\alpha_{k}}_{x_{k}}K_{A}(x_{1},\ldots,x_{k})\right| & =\left|A\left[D^{\alpha_{1}}\delta_{x_{1}},\ldots,D^{\alpha_{k}}\delta_{x_{k}}\right]\right|\\
			& \le\bnorm A_{\cL_{k}(\H[-I])}\bnorm{D^{\alpha_{1}}\delta_{x_{1}}}_{\H[-I]}\ldots\bnorm{D^{\alpha_{k}}\delta_{x_{k}}}_{\H[-I]}\\
			& \le\prod^{k}_{i=1}C_{\alpha_{i},I,d}\bnorm A_{\cL_{k}(\H[-I])}.
		\end{align*}
		This implies that there exists a constant $C_{m,k,I}$ such that
		\[
		\bnorm{K_{A}}_{\Cf^{m,\ldots,m}\left((\T^{d})^{k}\right)}\le C_{m,k,I}\bnorm A_{\cL_{k}(\H[-I])}.
		\]
		This completes the proof of the lemma.
	\end{proof}
	
	\begin{cor}
		Let $k\in\N$, $m\in\N_{0}$, and $I>\frac{d}{2}+m$. Then the map
		\[
		\cK_{k}:\cL_{k}(\H[-I])\to\Cf^{m,\ldots,m}\left((\Td)^{k}\right),\quad A\mapsto K_{A},
		\]
		is a bounded linear operator.
	\end{cor}
	
	\begin{proof}
		Linearity follows directly from the definition of $K_{A}$, while
		boundedness follows from Lemma \ref{lem:existence_of_smooth_kernel}.
	\end{proof}

	\section{Microscopic and modified Gaussian dynamics\label{sec:microscopic_and_modified_gaussian_dynamics}}
	
	In this section, we introduce the microscopic fluctuation dynamics
	and construct the modified Gaussian process that will be used to approximate
	them. We first derive a decomposition of the generator of the SSEP
	fluctuation field into its drift and covariance contributions, together
	with a higher-order remainder. We then approximate the discrete operators
	and the density-dependent coefficients appearing in this decomposition
	by their modified continuum counterparts. This leads naturally to
	a modified Ornstein-Uhlenbeck process whose generator reproduces the
	leading terms of the microscopic generator up to explicitly controlled
	errors.
	
	The results of this section provide the link between the microscopic
	particle system and the Gaussian dynamics used in the sequel. In particular,
	the generator comparison obtained below will serve as the main input
	for the quantitative comparison of their distributions.
	
	\subsection{SSEP and the fluctuation field}
	
	In this subsection, we introduce the symmetric simple exclusion process
	and the associated fluctuation field together with its generator.
	
	Recall that the generator of the SSEP $(\eta^{n}_{t})_{t\ge0}$ is
	given by 
	\[
	\cL^{n,0}F(\eta)=\frac{n^{2}}{2}\sum^{d}_{j=1}\sum_{x\in\Tnd}\left[F\left(\eta^{x\lrarr x+e^{n}_{j}}\right)-F(\eta)\right],
	\]
	where $\eta^{x\lrarr y}$ denotes the configuration obtained from
	$\eta$ by interchanging the occupation variables at the sites $x$
	and $y$. In particular, for every function $F:\{0,1\}^{\Tnd}\to\R$,
	the process 
	\[
	F(\eta^{n}_{t})-F(\eta^{n}_{0})-\int^{t}_{0}\cL^{n,0}F(\eta^{n}_{s})\,ds,\quad t\ge0,
	\]
	is a martingale with respect to the filtration $(\cF^{n}_{t})_{t\ge0}$
	generated by $(\eta^{n}_{t})_{t\ge0}$.
	
	We denote the associated deterministic density profile by 
	\[
	\rho^{n}_{t}(x):=\E\eta^{n}_{t}(x),\quad x\in\Tnd,\quad t\ge0.
	\]
	It solves the discrete heat equation 
	\[
	\partial_{t}\rho^{n}_{t}=\frac{1}{2}\Delta_{n}\rho^{n}_{t},\quad t\ge0.
	\]
	
	The fluctuation field associated with $(\eta^{n}_{t})_{t\ge0}$ is
	defined by 
	\begin{equation}
		\zeta^{n}_{t}:=n^{d/2}\left(\eta^{n}_{t}-\rho^{n}_{t}\right),\quad t\ge0.\label{eq:fluctuation_field_of_eta}
	\end{equation}
	For $I>\frac{d}{2}$, we consider its empirical representation 
	\[
	\bar{\zeta}^{n}_{t}=\frac{1}{n^{d/2}}\sum_{x\in\Tnd}\left(\eta^{n}_{t}(x)-\rho^{n}_{t}(x)\right)\delta_{x},\quad t\ge0,
	\]
	as an $\H[-I]$-valued process.
	
	We next introduce the generator of the fluctuation field. For $\rho\in\Ln$,
	$\eta\in\{0,1\}^{\Tnd}$, and $\zeta=n^{d/2}(\eta-\rho)$, set 
	\[
	\bar{\zeta}^{x\lrarr y}:=n^{d/2}\left(\overline{\eta^{x\lrarr y}}-\bar{\rho}\right).
	\]
	For $F\in\Cf^{1}_{\text{der}}(\H[-I])$, define 
	\begin{align}
		\cL^{n}_{\rho}F(\bar{\zeta}):= & \frac{n^{2}}{2}\sum^{d}_{j=1}\sum_{x\in\Tnd}\left[F\left(\bar{\zeta}^{x\lrarr x+e^{n}_{j}}\right)-F(\bar{\zeta})\right]-\frac{n^{d/2}}{2}\binn{\Delta_{n}DF(\bar{\zeta};\cdot),\rho}_{n}.\label{eq:definition_of_cL_rho}
	\end{align}
	Then, for every $F\in\Cf^{1}_{\text{der}}(\H[-I])$, the process 
	\[
	F(\bar{\zeta}^{n}_{t})-F(\bar{\zeta}^{n}_{0})-\int^{t}_{0}\cL^{n}_{\rho^{n}_{s}}F(\bar{\zeta}^{n}_{s})\,ds,\quad t\ge0,
	\]
	is a martingale with respect to $(\cF^{n}_{t})_{t\ge0}$. In Section
	\ref{subsec:martingale_problem_for_non-anticipative_functionals},
	this martingale problem will be extended to the class of non-anticipative
	functionals.
	
	\subsection{Expansion of the microscopic generator}
	
	The following proposition provides the basic decomposition of the
	microscopic fluctuation generator into its drift and covariance contributions
	and a third-order remainder.
	\begin{prop}[Generator decomposition]
		\label{prop:expansion_of_generator}Let $I>\frac{d}{2}+3$ and $F\in\Cf^{3}_{\text{der}}\left(\H[-I]\right)$.
		Then for every $n\ge3$, $\eta\in\{0,1\}^{\Tnd}$, $\rho\in[0,1]^{\Tnd}$
		and $\zeta=n^{d/2}(\eta-\rho)$
		\[
		\cL^{n}_{\rho}F(\bar{\zeta})=\frac{1}{2}\binn{\Delta_{n}DF(\bar{\zeta};\cdot),\bar{\zeta}}+\frac{1}{4}\sum^{d}_{j=1}\binn{\partial^{\odot2}_{n,j}D^{2}F(\bar{\zeta};\cdot),(\tau_{n,j}\eta-\eta)^{2}}_{n}+R_{n}(F,\eta,\rho)
		\]
		with
		\[
		\left|R_{n}(F,\eta,\rho)\right|\le\frac{C_{I}}{n^{\frac{d}{2}+1}}\sup_{f\in\H[-I]}\bnorm{D^{3}F(f)}_{\cL_{3}(\H[-I])}
		\]
		
	\end{prop}
	
	\begin{proof}
		Using Taylor's expansion (\ref{eq:taylors_formula}), we get for $\eta\in\left\{ 0,1\right\} ^{\Tnd}$,
		$\rho\in[0,1]^{\Tnd}$ and $\zeta=n^{d/2}(\eta-\rho)$
		\begin{align*}
			\cL^{n}_{\rho}F(\bar{\zeta}) & =\frac{n^{2}}{2}\sum^{d}_{j=1}\sum_{x\in\Tnd}\left[F\left(\bar{\zeta}^{x\lr x+e_{j}^n}\right)-F\left(\bar{\zeta}\right)\right]-\frac{n^{d/2}}{2}\binn{\Delta_{n}DF\left(\bar{\zeta};\cdot\right),\rho}_{n}\\
			& =\frac{n^{2}}{2}\sum^{d}_{j=1}\sum_{x\in\Tnd}DF(\bar{\zeta})\left[\bar{\zeta}^{x\lr x+e_{j}^n}-\bar{\zeta}\right]-\frac{n^{d/2}}{2}\binn{\Delta_{n}DF\left(\bar{\zeta};\cdot\right),\rho}_{n}\\
			& +\frac{n^{2}}{4}\sum^{d}_{j=1}\sum_{x\in\Tnd}D^{2}F(\bar{\zeta})\left[\left(\bar{\zeta}^{x\leftrightarrow x+e_{j}^n}-\bar{\zeta}\right)^{\otimes2}\right]\\
			& +\frac{n^{2}}{12}\sum^{d}_{j=1}\sum_{x\in\Tnd}D^{3}F(\tilde{\zeta}_{x,j})\left[\left(\bar{\zeta}^{x\leftrightarrow x+e_{j}^n}-\bar{\zeta}\right)^{\otimes3}\right],
		\end{align*}
		where $\tilde{\zeta}_{x,j}$ are elements of $\H[-I]$. Denote the
		terms in the right hand side of the equality above by $S_{i}$, $i\in[4]$.
		The terms $S_{1}+S_{2}$ correspond to the drift contribution, $S_{3}$
		corresponds to the covariance structure of the fluctuation field and
		$S_{4}$ is the microscopic third-order remainder. 
		
		According to the equality
		\[
		n^{2}\binn{\varphi,\bar{\zeta}^{x\leftrightarrow x+e_{j}^n}-\bar{\zeta}}=-\frac{n^{d/2}}{n^{d}}\partial_{n,j}\varphi(x)\partial_{n,j}\eta(x),
		\]
		we get
		\begin{align*}
			S_{1}+S_{2} & =-\frac{n^{d/2}}{2n^{d}}\sum^{d}_{j=1}\sum_{x\in\Tnd}\partial_{n,j}DF(\bar{\zeta};x)\partial_{n,j}\eta(x)-\frac{n^{d/2}}{2}\binn{\Delta_{n}DF(\bar{\zeta};\cdot),\rho}_{n}\\
			& =\frac{n^{d/2}}{2}\binn{\Delta_{n}DF(\bar{\zeta};\cdot),\eta}_{n}-\frac{n^{d/2}}{2}\binn{\Delta_{n}DF(\bar{\zeta};\cdot),\rho}_{n}\\
			& =\frac{1}{2}\binn{\Delta_{n}DF(\bar{\zeta}),\bar{\zeta}}.
		\end{align*}
		Using the kernel representation of $D^{2}F(\bar{\zeta})$ and the
		definition of the diagonal mixed derivative, we obtain
		\begin{align*}
			S_{3} & =\frac{n^{d}}{4n^{2d+2}}\sum^{d}_{j=1}\sum_{x\in\Tnd}\partial^{\otimes2}_{n,j}D^{2}F(\bar{\zeta};x,x)\left(\partial_{n,j}\eta(x)\right)^{2}\\
			& =\frac{1}{4}\sum^{d}_{j=1}\binn{\partial^{\odot2}_{n,j}D^{2}F(\bar{\zeta};\cdot),(\tau_{n,j}\eta-\eta)^{2}}_{n}.
		\end{align*}
		The last term $S_{4}$ can be estimated by 
		\begin{align*}
			|S_{4}|
			&\le
			\frac{1}{12n^{d/2+1}}
			\sum_{j=1}^{d}
			\sup_{x\in\Tnd}
			\bnorm{
				\partial^{\odot3}_{n,j}
				D^{3}F(\tilde{\zeta}_{x,j};\cdot)
			}_{\Cf(\Td)}
			\\
			&\le
			\frac{d}{12n^{d/2+1}}
			\sup_{f\in\H[-I]}
			\bnorm{D^{3}F(f;\cdot)}_{\Cf^{3}((\Td)^{3})}
			\\
			&\le
			\frac{C_{I}}{n^{d/2+1}}
			\sup_{f\in\H[-I]}
			\bnorm{D^{3}F(f)}_{\cL_{3}(\H[-I])}.
		\end{align*}
		due to 
		\[
		\left|\partial_{n,j}\eta(x)\right|\le n
		\]
		and Lemma \ref{lem:existence_of_smooth_kernel}. This completes the
		proof of the proposition.
	\end{proof}

	\subsection{Modified density profile}
	
	We next introduce a modified density profile that will provide a higher-order
	continuum approximation of the discrete density profile. Recall that
	\[
	\rho^{n}_{t}(x):=\E\eta^{n}_{t}(x),\quad x\in\Tnd,\quad t\ge0,
	\]
	solves the discrete heat equation 
	\begin{equation}
		\partial_{t}\rho^{n}_{t}=\frac{1}{2}\Delta_{n}\rho^{n}_{t},\quad t\ge0.\label{eq:discrete_heat_equation}
	\end{equation}
	Motivated by the approximation of $\Delta_{n}$ by the modified differential
	operator $\Delta^{(2m)}_{n}$ introduced in (\ref{eq:modified_laplacian}),
	we consider the corresponding modified heat equation 
	\begin{equation}
		\partial_{t}\varrho^{n}_{t}=\frac{1}{2}\Delta^{(2m)}_{n}\varrho^{n}_{t},\quad t\ge0.\label{eq:modified_heat_equation}
	\end{equation}
	
	For odd $m\in\N$, $I\in\R$ and $f\in\H[I]$, define 
	\[
	\cS^{n,m}_{t}f:=\sum_{k\in\Zd}e^{\frac{1}{2}\lambda^{(2m)}_{k,n}t}\binn{f,\varsigma_{k}}\varsigma_{k},\quad t\ge0.
	\]
	Recall from Lemma \ref{lem:estimate_of_eigenvalues} that $\lambda^{(2m)}_{k,n}\le0$
	for odd $m$. The following lemma collects the basic properties of
	$(\cS^{n,m}_{t})_{t\ge0}$.
	\begin{lem}
		\label{lem:modified_heat_semigroup} Let $I\in\R$ and let $m\in\N$
		be odd. Then $(\cS^{n,m}_{t})_{t\ge0}$ is a strongly continuous contraction
		semigroup on $\H[I]$. Its infinitesimal generator is
		\[
		A_{n,m}=\frac{1}{2}\Delta^{(2m)}_{n}
		\]
		with domain $D(A_{n,m})=\H[I+2m].$ Consequently, for every $f\in\H[I+2m]$,
		\[
		\cS^{n,m}_{\cdot}f\in\Cf^{1}([0,\infty);\H[I])\cap\Cf([0,\infty);\H[I+2m])
		\]
		and
		\[
		\frac{d}{dt}\cS^{n,m}_{t}f=\frac{1}{2}\Delta^{(2m)}_{n}\cS^{n,m}_{t}f=\frac{1}{2}\cS^{n,m}_{t}\Delta^{(2m)}_{n}f,\quad t\ge0.
		\]
		Moreover, the inequality
		\[
		\bnorm{\cS^{n,m}_{t}f}_{\H[I]}\le\bnorm f_{\H[I]},
		\]
		holds for all $f\in\H[I]$ and $t\ge0$.
	\end{lem}
	
	\begin{proof}
		Since $\lambda^{(2m)}_{k,n}\le0$ for every $k\in\Zd$, we have 
		\begin{align*}
			\bnorm{\cS^{n,m}_{t}f}^{2}_{\H[I]} & =\sum_{k\in\Zd}e^{\lambda^{(2m)}_{k,n}t}(1+|k|^{2})^{I}\left|\binn{f,\varsigma_{k}}\right|^{2}\\
			& \le\sum_{k\in\Zd}(1+|k|^{2})^{I}\left|\binn{f,\varsigma_{k}}\right|^{2}=\bnorm f^{2}_{\H[I]}.
		\end{align*}
		Thus $\cS^{n,m}_{t}$ is a contraction on $\H[I]$. The semigroup
		property follows directly from its Fourier representation.
		
		For $s,t\ge0$, 
		\[
		\bnorm{\cS^{n,m}_{t}f-\cS^{n,m}_{s}f}^{2}_{\H[I]}=\sum_{k\in\Zd}\left|e^{\frac{1}{2}\lambda^{(2m)}_{k,n}t}-e^{\frac{1}{2}\lambda^{(2m)}_{k,n}s}\right|^{2}(1+|k|^{2})^{I}\left|\binn{f,\varsigma_{k}}\right|^{2}.
		\]
		Since $\lambda^{(2m)}_{k,n}\le0$, the factor in absolute value is
		bounded by $2$ and converges to zero for every $k\in\Zd$ as $s\to t$.
		Hence the dominated convergence theorem yields 
		\[
		\bnorm{\cS^{n,m}_{t}f-\cS^{n,m}_{s}f}_{\H[I]}\to0\quad\text{as }s\to t.
		\]
		Therefore $(\cS^{n,m}_{t})_{t\ge0}$ is a strongly continuous contraction
		semigroup on $\H[I]$.
		
		We note that 
		\begin{equation}
			\sum_{k\in\Zd}(1+|k|^{2})^{I}\left|\lambda^{(2m)}_{k,n}\right|^{2}\left|\binn{f,\varsigma_{k}}\right|^{2}<\infty\quad\iff\quad\sum_{k\in\Zd}(1+|k|^{2})^{I+2m}\left|\binn{f,\varsigma_{k}}\right|^{2}<\infty,\label{eq:equivalence_for_series_convergence}
		\end{equation}
		according to Lemma \ref{lem:estimate_of_eigenvalues}. Thus, for each
		$f\in\H[I+2m]$,
		\begin{align*}
			&\lim_{t\to0+}\bnorm{\frac{\cS^{n,m}_{t}f-f}{t}-\frac{1}{2}\Delta^{(2m)}_{n}f}^{2}_{\H[I]}\\
			&\qquad=\lim_{t\to0+}\sum_{k\in\Zd}\left|\frac{e^{\frac{1}{2}\lambda^{(2m)}_{k,n}t}-1}{t}-\frac{1}{2}\lambda^{(2m)}_{k,n}\right|^{2}(1+|k|^{2})^{I}\left|\binn{f,\varsigma_{k}}\right|^{2}=0,
		\end{align*}
		by the dominated convergence theorem and (\ref{eq:equivalence_for_series_convergence}).
		Hence
		\[
		\H[I+2m]\subseteq D(A_{n,m}).
		\]
		
		Next, let $f\in D(A_{n,m})$ and set
		\[
		g:=\lim_{t\to0+}\frac{\cS^{n,m}_{t}f-f}{t}
		\]
		in $\H[I]$. Since the Fourier coefficient maps are continuous on
		$\H[I]$, for every $k\in\Zd$,
		\[
		\binn{g,\varsigma_{k}}=\frac{1}{2}\lambda^{(2m)}_{k,n}\binn{f,\varsigma_{k}}.
		\]
		Therefore, we have
		\[
		\bnorm g^{2}_{\H[I]}=\frac{1}{4}\sum_{k\in\Zd}(1+|k|^{2})^{I}\left|\lambda^{(2m)}_{k,n}\right|^{2}\left|\binn{f,\varsigma_{k}}\right|^{2}<\infty.
		\]
		By (\ref{eq:equivalence_for_series_convergence}), $f\in\H[I+2m]$.
		Hence
		\[
		D(A_{n,m})=\H[I+2m]
		\]
		and
		\[
		A_{n,m}=\frac{1}{2}\Delta^{(2m)}_{n}.
		\]
		The remaining assertions follow from \citep[Theorem 1.2.4(c)]{Pazy:1983}
		or \citep[Proposition A.2]{Prato:2014}.
	\end{proof}
	
	As a direct consequence of the preceding estimate and the Sobolev
	embedding theorem, we obtain the following uniform bound in spaces
	of continuously differentiable functions.
	\begin{cor}
		\label{cor:regularity_of_modified_semigroup} Let $m\in\N$ be odd,
		$\ell\in\N_{0}$ and $\tilde{\ell}\in\N$ such that $\tilde{\ell}>\ell+\frac{d}{2}.$Then
		there exists a constant $C_{\tilde{\ell}}>0$ such that 
		\[
		\bnorm{\cS^{n,m}_{t}f}_{\Cf^{\ell}(\Td)}\le C_{\tilde{\ell}}\bnorm f_{\Cf^{\tilde{\ell}}(\Td)}
		\]
		for all $f\in\Cf^{\tilde{\ell}}(\Td)$, $t\ge0$ and $n\ge3$.
	\end{cor}
	
	\begin{proof}
		By the Sobolev embedding theorem and the preceding lemma, 
		\[
		\bnorm{\cS^{n,m}_{t}f}_{\Cf^{\ell}(\Td)}\le C_{\tilde{\ell}}\bnorm{\cS^{n,m}_{t}f}_{\H[\tilde{\ell}]}\le C_{\tilde{\ell}}\bnorm f_{\H[\tilde{\ell}]}\le C_{\tilde{\ell}}\bnorm f_{\Cf^{\tilde{\ell}}(\Td)}.
		\]
	\end{proof}
	
	\begin{rem}
		If $\varrho_{0}=\ex_{n}\rho_{0}$ for some $\rho_{0}\in\Ln$, then
		the Fourier representation of $\ex_{n}$ yields 
		\begin{equation}
			\cS^{n,m}_{t}\varrho_{0}=\sum_{k\in\Znd}e^{\frac{1}{2}\lambda^{(2m)}_{k,n}t}\binn{\rho_{0},\varsigma_{k}}_{n}\varsigma_{k},\quad t\ge0.\label{eq:solution_to_modified_HE_started_from_ex_rho}
		\end{equation}
		In particular, $\cS^{n,m}_{t}\varrho_{0}$ belongs to the space $V_{n}$
		of trigonometric polynomials with frequencies in $\Znd$ for every
		$t\ge0$. 
	\end{rem}
	
	In addition to strong continuity on $\H[I]$, we will need continuity
	at zero in the operator norm if an arbitrarily small loss of Sobolev
	regularity is allowed. Let $\id$ denote the identity operator.
	\begin{lem}
		\label{lem:operator_norm_continuity_modified_semigroup} Let $m\in\N$
		be odd and $I\in\R$. Then, for every $\delta>0$, 
		\[
		\sup_{n\ge3}\bnorm{\id-\cS^{n,m}_{t}}_{\cL_{1}(\H[I];\H[I-\delta])}\to0\quad\text{as }t\to0+.
		\]
	\end{lem}
	
	\begin{proof}
		For $f\in\H[I]$, we have 
		\begin{align*}
			\bnorm{f-\cS^{n,m}_{t}f}^{2}_{\H[I-\delta]} & =\sum_{k\in\Zd}\left|1-e^{\frac{1}{2}\lambda^{(2m)}_{k,n}t}\right|^{2}\left(1+|k|^{2}\right)^{I-\delta}\left|\binn{f,\varsigma_{k}}\right|^{2}\\
			& \le\sup_{k\in\Zd}\frac{\left|1-e^{\frac{1}{2}\lambda^{(2m)}_{k,n}t}\right|^{2}}{\left(1+|k|^{2}\right)^{\delta}}\bnorm f^{2}_{\H[I]}.
		\end{align*}
		Consequently, 
		\[
		\bnorm{\id-\cS^{n,m}_{t}}^{2}_{\cL_{1}(\H[I];\H[I-\delta])}\le\sup_{k\in\Zd}\frac{\left|1-e^{\frac{1}{2}\lambda^{(2m)}_{k,n}t}\right|^{2}}{\left(1+|k|^{2}\right)^{\delta}}.
		\]
		Since $\lambda^{(2m)}_{k,n}\le0$, the numerator is bounded by $1$.
		Given $R>0$, we therefore have 
		\[
		\sup_{|k|>R}\frac{\left|1-e^{\frac{1}{2}\lambda^{(2m)}_{k,n}t}\right|^{2}}{\left(1+|k|^{2}\right)^{\delta}}\le\frac{1}{(1+R^{2})^{\delta}}
		\]
		for all $n\ge3.$ On the other hand, for fixed $R$, the set $\{k\in\Zd:|k|\le R\}$
		is finite, and hence 
		\[
		\sup_{n\ge3}\sup_{|k|\le R}\left|1-e^{\frac{1}{2}\lambda^{(2m)}_{k,n}t}\right|\to0\quad\text{as }t\to0+,
		\]
		by Lemma \ref{lem:estimate_of_eigenvalues}. Letting first $t\to0+$
		and then $R\to\infty$ proves the claim.
	\end{proof}
	
	We now compare the solution of the discrete heat equation with that
	of its modified counterpart. The following estimate quantifies the
	error introduced by replacing the discrete Laplacian $\Delta_{n}$
	with its finite-order approximation $\Delta^{(2m)}_{n}$.
	\begin{prop}
		\label{prop:comparison_of_rho_and_varrho} Let $m\in\N$ be odd, let
		$(\rho^{n}_{t})_{t\ge0}$ be a solution to (\ref{eq:discrete_heat_equation}),
		and let $(\varrho^{n}_{t})_{t\ge0}$ be the solution to (\ref{eq:modified_heat_equation})
		with initial condition 
		\[
		\varrho^{n}_{0}=\ex_{n}\rho^{n}_{0}.
		\]
		Then, for every $I\in\R$, there exists a constant $C_{m}>0$ such
		that 
		\[
		\bnorm{\ex_{n}\rho^{n}_{t}-\varrho^{n}_{t}}_{\H[I]}\le\frac{C_{m}t}{n^{2m}}\bnorm{\ex_{n}\rho^{n}_{0}}_{\H[I+2m+2]},\quad t\ge0.
		\]
		In particular, for every $k\in\N_{0}$ there exists a constant $C_{m,k}>0$
		such that 
		\[
		\bnorm{\ex_{n}\rho^{n}_{t}-\varrho^{n}_{t}}_{\Cf^{k}(\Td)}\le\frac{C_{m,k}t}{n^{2m}}\bnorm{\ex_{n}\rho^{n}_{0}}_{\H[k+\frac{d}{2}+2m+3]},\quad t\ge0.
		\]
	\end{prop}
	
	\begin{proof}
		The Fourier representation of the solution to the discrete heat equation
		gives 
		\begin{equation}
			\ex_{n}\rho^{n}_{t}=\sum_{k\in\Znd}e^{\frac{1}{2}\lambda_{k,n}t}\binn{\rho^{n}_{0},\varsigma_{k}}_{n}\varsigma_{k},\quad t\ge0.\label{eq:formula_for_ex_n_rho}
		\end{equation}
		Combining this with (\ref{eq:solution_to_modified_HE_started_from_ex_rho})
		and the identity 
		\[
		\binn{\rho^{n}_{0},\varsigma_{k}}_{n}=\binn{\ex_{n}\rho^{n}_{0},\varsigma_{k}},\quad k\in\Znd,
		\]
		we obtain 
		\begin{align*}
			\bnorm{\ex_{n}\rho^{n}_{t}-\varrho^{n}_{t}}^{2}_{\H[I]} & =\sum_{k\in\Znd}(1+|k|^{2})^{I}\left|e^{\frac{1}{2}\lambda_{k,n}t}-e^{\frac{1}{2}\lambda^{(2m)}_{k,n}t}\right|^{2}\left|\binn{\ex_{n}\rho^{n}_{0},\varsigma_{k}}\right|^{2}.
		\end{align*}
		Since $\lambda_{k,n}\le0$ and, by Lemma \ref{lem:estimate_of_eigenvalues},
		$\lambda^{(2m)}_{k,n}\le0$ for odd $m$, the mean value theorem yields
		\[
		\left|e^{\frac{1}{2}\lambda_{k,n}t}-e^{\frac{1}{2}\lambda^{(2m)}_{k,n}t}\right|\le\frac{t}{2}\left|\lambda_{k,n}-\lambda^{(2m)}_{k,n}\right|.
		\]
		Hence, again by Lemma \ref{lem:estimate_of_eigenvalues}, 
		\[
		\left|e^{\frac{1}{2}\lambda_{k,n}t}-e^{\frac{1}{2}\lambda^{(2m)}_{k,n}t}\right|\le\frac{C_{m}t}{n^{2m}}|k|^{2m+2}.
		\]
		Therefore, 
		\begin{align*}
			\bnorm{\ex_{n}\rho^{n}_{t}-\varrho^{n}_{t}}^{2}_{\H[I]} & \le\frac{C^{2}_{m}t^{2}}{n^{4m}}\sum_{k\in\Znd}(1+|k|^{2})^{I+2m+2}\left|\binn{\ex_{n}\rho^{n}_{0},\varsigma_{k}}\right|^{2}\\
			& =\frac{C^{2}_{m}t^{2}}{n^{4m}}\bnorm{\ex_{n}\rho^{n}_{0}}^{2}_{\H[I+2m+2]}.
		\end{align*}
		This proves the first estimate.
		
		For the second estimate, take 
		\[
		I=k+\frac{d}{2}+1.
		\]
		Since $I>k+\frac{d}{2}$, the Sobolev embedding theorem and the first
		estimate yield 
		\begin{align*}
			\bnorm{\ex_{n}\rho^{n}_{t}-\varrho^{n}_{t}}_{\Cf^{k}(\Td)} & \le C_{k}\bnorm{\ex_{n}\rho^{n}_{t}-\varrho^{n}_{t}}_{\H[k+\frac{d}{2}+1]}\\
			& \le\frac{C_{m,k}t}{n^{2m}}\bnorm{\ex_{n}\rho^{n}_{0}}_{\H[k+\frac{d}{2}+2m+3]}.
		\end{align*}
		This completes the proof.
	\end{proof}
	
	We will later assume that the initial density $\rho^{n}_{0}$ is obtained
	by restricting a sufficiently regular function $\rhoini:\Td\to\R$
	to $\Tnd$. In order to apply Proposition \ref{prop:comparison_of_rho_and_varrho},
	we therefore need a bound on $\ex_{n}\rhoini$ in Sobolev spaces that
	is uniform in $n$. Combining Proposition \ref{prop:comparison_of_rho_and_varrho}
	with Corollary~\ref{cor:bound_of_ex} yields the following corollary.
	\begin{cor}
		\label{cor:comparison_of_rho_and_varrho} Let $m\in\N$ be odd, $\ell\in\N_{0}$,
		and set $\tilde{\ell}\coloneqq\ell+d+2m+3.$ Let $\rhoini\in\Cf^{\tilde{\ell}}(\Td)$,
		let $(\rho^{n}_{t})_{t\ge0}$ be the solution to (\ref{eq:discrete_heat_equation})
		with initial condition 
		\[
		\rho^{n}_{0}=\rhoini\big|_{\Tnd},
		\]
		and let $(\varrho^{n}_{t})_{t\ge0}$ be the solution to (\ref{eq:modified_heat_equation})
		with initial condition 
		\[
		\varrho^{n}_{0}=\ex_{n}\rho^{n}_{0}=\ex_{n}\rhoini.
		\]
		Then there exists a constant $C_{m,\ell}>0$ such that 
		\[
		\bnorm{\ex_{n}\rho^{n}_{t}-\varrho^{n}_{t}}_{\Cf^{\ell}(\Td)}\le\frac{C_{m,\ell}t}{n^{2m}}\bnorm{\rhoini}_{\Cf^{\tilde{\ell}}(\Td)},\quad t\ge0.
		\]
	\end{cor}
	
	\begin{proof}
		By Proposition \ref{prop:comparison_of_rho_and_varrho}, 
		\[
		\bnorm{\ex_{n}\rho^{n}_{t}-\varrho^{n}_{t}}_{\Cf^{\ell}(\Td)}\le\frac{C_{m,\ell}t}{n^{2m}}\bnorm{\ex_{n}\rhoini}_{\H[\ell+\frac{d}{2}+2m+3]}.
		\]
		Since 
		\[
		\tilde{\ell}=\ell+d+2m+3\ge\frac{d}{2}+\left(\ell+\frac{d}{2}+2m+3\right),
		\]
		Corollary~\ref{cor:bound_of_ex} yields 
		\[
		\bnorm{\ex_{n}\rhoini}_{\H[\ell+\frac{d}{2}+2m+3]}\le C_{m,\ell}\bnorm{\rhoini}_{\Cf^{\tilde{\ell}}(\Td)}
		\]
		uniformly in $n$. Combining the two estimates proves the claim.
	\end{proof}
	
	We will assume that the initial density satisfies 
	\[
	0\le\rhoini(x)\le1,\quad x\in\Td.
	\]
	The discrete heat equation preserves the interval $[0,1]$, and hence
	$0\le\rho^{n}_{t}\le1$ for every $t\ge0$. The modified heat equation
	does not, in general, satisfy a maximum principle. Nevertheless, the
	preceding comparison estimate shows that its solution remains close
	to the range of the initial density.
	\begin{lem}
		\label{lem:state_bound_of_varrho}Let $m\in\N$ be odd and $\rhoini\in\Cf^{2m+d+3}(\Td)$.
		Let $(\varrho^{n}_{t})_{t\ge0}$ be the solution to (\ref{eq:modified_heat_equation})
		with initial condition $\varrho^{n}_{0}=\ex_{n}\rhoini$. Then there
		exists a constant $C_{m}>0$ such that, for all $x\in\Td$ and $t\ge0$,
		\[
		\min_{y\in\Td}\rhoini(y)-\frac{C_{m}(1+t)}{n^{2m}}\bnorm{\rhoini}_{\Cf^{2m+d+3}(\Td)}\le\varrho^{n}_{t}(x)
		\]
		and 
		\[
		\varrho^{n}_{t}(x)\le\max_{y\in\Td}\rhoini(y)+\frac{C_{m}(1+t)}{n^{2m}}\bnorm{\rhoini}_{\Cf^{2m+d+3}(\Td)}.
		\]
	\end{lem}
	
	\begin{proof}
		By the expression (\ref{eq:formula_for_ex_n_rho}),
		\[
		\ex_{n}\rho^{n}_{t}=\sum_{k\in\Znd}e^{\frac{1}{2}\lambda_{k,n}t}\binn{\rho^{n}_{0},\varsigma_{k}}_{n}\varsigma_{k}
		\]
		on $\Td$. Therefore,
		\[
		\partial_{t}\ex_{n}\rho^{n}_{t}=\frac{1}{2}\Delta_{n}\ex_{n}\rho^{n}_{t}
		\]
		on $\Td$ and $\ex_{n}\rho^{n}_{0}=\ex_{n}\rhoini.$ The semigroup
		generated by $\frac{1}{2}\Delta_{n}$ preserves positivity and constants.
		Hence 
		\begin{equation}
			\min_{y\in\Td}\ex_{n}\rhoini(y)\le\ex_{n}\rho^{n}_{t}(x)\le\max_{y\in\Td}\ex_{n}\rhoini(y),\quad x\in\Td,\ t\ge0.\label{eq:maximum_principle_ex}
		\end{equation}
		
		By Lemma \ref{lem:approximation_for_ex}, 
		\begin{equation}
			\bnorm{\rhoini-\ex_{n}\rhoini}_{\Cf(\Td)}\le\frac{C_{m}}{n^{2m}}\bnorm{\rhoini}_{\Cf^{2m+1}(\Td)}.\label{eq:bound_for_initial_interpolation}
		\end{equation}
		Moreover, Corollary \ref{cor:comparison_of_rho_and_varrho}, applied
		with $\ell=0$, gives 
		\[
		\bnorm{\varrho^{n}_{t}-\ex_{n}\rho^{n}_{t}}_{\Cf(\Td)}\le\frac{C_{m}t}{n^{2m}}\bnorm{\rhoini}_{\Cf^{2m+d+3}(\Td)}.
		\]
		Consequently, 
		\begin{align*}
			\varrho^{n}_{t}(x) & \le\ex_{n}\rho^{n}_{t}(x)+\bnorm{\varrho^{n}_{t}-\ex_{n}\rho^{n}_{t}}_{\Cf(\Td)}\\
			& \le\max_{y\in\Td}\rhoini(y)+\bnorm{\rhoini-\ex_{n}\rhoini}_{\Cf(\Td)}+\bnorm{\varrho^{n}_{t}-\ex_{n}\rho^{n}_{t}}_{\Cf(\Td)}\\
			& \le\max_{y\in\Td}\rhoini(y)+\frac{C_{m}(1+t)}{n^{2m}}\bnorm{\rhoini}_{\Cf^{2m+d+3}(\Td)}.
		\end{align*}
		The lower bound follows analogously.
	\end{proof}
	
	The preceding estimates concern the modified heat equation started
	from the trigonometric interpolation $\ex_{n}\rhoini$. For the application
	one can prefer to start the modified equation directly from $\rhoini$.
	The following lemma shows that this change of initial condition produces
	an error of the same order.

	\begin{lem}
		\label{lem:bound_for_varrho_and_rho}Let $m\in\N$ be odd, $\ell\in\N_{0},$
		and set $\tilde{\ell}\coloneqq\ell+d+2m+3$, and let $\rhoini\in\Cf^{\tilde{\ell}}(\Td)$.
		Let $(\rho^{n}_{t})_{t\ge0}$ be the solution to the discrete heat
		equation (\ref{eq:discrete_heat_equation}) with initial condition
		\[
		\rho^{n}_{0}=\rhoini\big|_{\Tnd},
		\]
		and let 
		\[
		\varrho^{n}_{t}=\cS^{n,m}_{t}\rhoini,\quad t\ge0,
		\]
		be the solution to the modified heat equation (\ref{eq:modified_heat_equation})
		with initial condition $\rhoini$. Then there exists a constant $C_{m,\ell}>0$
		such that 
		\begin{equation}
			\bnorm{\ex_{n}\rho^{n}_{t}-\varrho^{n}_{t}}_{\Cf^{\ell}(\Td)}\le\frac{C_{m,\ell}(1+t)}{n^{2m}}\bnorm{\rhoini}_{\Cf^{\tilde{\ell}}(\Td)},\quad t\ge0.\label{eq:comparion_of_rho_t_and_varrho_t_full}
		\end{equation}
		Moreover, for every $x\in\Td$ and $t\ge0$, 
		\[
		\min_{y\in\Td}\rhoini(y)-\frac{C_{m}(1+t)}{n^{2m}}\bnorm{\rhoini}_{\Cf^{2m+d+3}(\Td)}\le\varrho^{n}_{t}(x)
		\]
		and 
		\[
		\varrho^{n}_{t}(x)\le\max_{y\in\Td}\rhoini(y)+\frac{C_{m}(1+t)}{n^{2m}}\bnorm{\rhoini}_{\Cf^{2m+d+3}(\Td)}.
		\]
	\end{lem}
	
	\begin{proof}
		Set 
		\[
		\widetilde{\varrho}^{n}_{t}\coloneqq\cS^{n,m}_{t}\ex_{n}\rhoini,\quad t\ge0.
		\]
		By the triangle inequality, 
		\begin{align*}
			\bnorm{\ex_{n}\rho^{n}_{t}-\varrho^{n}_{t}}_{\Cf^{\ell}(\Td)}\le & \bnorm{\ex_{n}\rho^{n}_{t}-\widetilde{\varrho}^{n}_{t}}_{\Cf^{\ell}(\Td)}+\bnorm{\widetilde{\varrho}^{n}_{t}-\varrho^{n}_{t}}_{\Cf^{\ell}(\Td)}.
		\end{align*}
		Using Corollary \ref{cor:comparison_of_rho_and_varrho}, we obtain
		\[
		\bnorm{\ex_{n}\rho^{n}_{t}-\widetilde{\varrho}^{n}_{t}}_{\Cf^{\ell}(\Td)}\le\frac{C_{m,\ell}t}{n^{2m}}\bnorm{\rhoini}_{\Cf^{\tilde{\ell}}(\Td)}.
		\]
		
		Set 
		\[
		I\coloneqq\ell+\left\lfloor \frac{d}{2}\right\rfloor +1\quad\text{and}\quad M\coloneqq I+2m+\left\lceil \frac{d}{2}\right\rceil =\ell+d+2m+1.
		\]
		Then $I>\ell+\frac{d}{2}$. Hence, by the Sobolev embedding theorem,
		the contraction property of $\cS^{n,m}_{t}$ and Corollary \ref{cor:ex_n_approximation_in_H_I_norm},
		one has
		\begin{align*}
			\bnorm{\widetilde{\varrho}^{n}_{t}-\varrho^{n}_{t}}_{\Cf^{\ell}(\Td)} & =\bnorm{\cS^{n,m}_{t}(\ex_{n}\rhoini-\rhoini)}_{\Cf^{\ell}(\Td)}\\
			& \le C_{\ell}\bnorm{\cS^{n,m}_{t}(\ex_{n}\rhoini-\rhoini)}_{\H[I]}\\
			& \le C_{\ell}\bnorm{\ex_{n}\rhoini-\rhoini}_{\H[I]}\\
			& \le\frac{C_{m,\ell}}{n^{2m}}\bnorm{\rhoini}_{\Cf^{M}(\Td)}\\
			& \le\frac{C_{m,\ell}}{n^{2m}}\bnorm{\rhoini}_{\Cf^{\tilde{\ell}}(\Td)}.
		\end{align*}
		Combining the two estimates proves (\ref{eq:comparion_of_rho_t_and_varrho_t_full}).
		
		It remains to prove the pointwise bounds. By Lemma \ref{lem:state_bound_of_varrho},
		\[
		\widetilde{\varrho}^{n}_{t}(x)\le\max_{y\in\Td}\rhoini(y)+\frac{C_{m}(1+t)}{n^{2m}}\bnorm{\rhoini}_{\Cf^{2m+d+3}(\Td)}.
		\]
		On the other hand, the estimate above with $\ell=0$ yields 
		\[
		\bnorm{\widetilde{\varrho}^{n}_{t}-\varrho^{n}_{t}}_{\Cf(\Td)}\le\frac{C_{m}}{n^{2m}}\bnorm{\rhoini}_{\Cf^{2m+d+3}(\Td)}.
		\]
		Consequently, 
		\[
		\varrho^{n}_{t}(x)\le\max_{y\in\Td}\rhoini(y)+\frac{C_{m}(1+t)}{n^{2m}}\bnorm{\rhoini}_{\Cf^{2m+d+3}(\Td)}.
		\]
		The lower bound follows analogously.
	\end{proof}

	\subsection{Modified covariance coefficients}
	
	The second-order term in the generator decomposition of Proposition
	\ref{prop:expansion_of_generator} still depends explicitly on the
	microscopic configuration $\eta$ and therefore does not yield a closed
	generator for an approximation of the fluctuation field. In this subsection,
	we introduce the corresponding deterministic covariance coefficient
	and collect several estimates needed to approximate it in terms of
	the modified density profile and the differential approximation of
	the discrete shift.
	
	Let $j\in[d]$ be fixed. Since $\eta(x)\in\{0,1\}$, we have 
	\[
	(\tau_{n,j}\eta-\eta)^{2}=\eta+\tau_{n,j}\eta-2\eta\tau_{n,j}\eta.
	\]
	This motivates, for $\rho\in\Ln$, the definition 
	\begin{equation}
		\Gamma_{n,j}(\rho):=\rho+\tau_{n,j}\rho-2\rho\tau_{n,j}\rho.\label{eq:definition_of_Gamma}
	\end{equation}
	We extend $\Gamma_{n,j}$ to functions on $\Td$ by the same formula.
	
	For sufficiently regular $\varrho$, we further introduce its modified
	counterpart 
	\begin{equation}
		\Gamma^{(m)}_{n,j}(\varrho):=\varrho+\tau^{(m)}_{n,j}\varrho-2\varrho\tau^{(m)}_{n,j}\varrho.\label{eq:definition_of_Gamma_m_tau}
	\end{equation}
	Since the latter need not be nonnegative, whereas $(\tau_{n,j}\eta-\eta)^{2}\ge0$,
	the covariance coefficient used below will be $\bigl(\Gamma^{(m)}_{n,j}(\varrho)\bigr)^{+}$,
	where $x^{+}:=\max\{x,0\}$.
	
	The following lemma collects the basic estimates for these covariance
	coefficients that will be used in the comparison of generators in
	Section \ref{subsec:generator_comparison}.
	\begin{lem}
		\label{lem:bounds_for_Gamma}Let $\ell,m\in\N_{0}$ and $j\in[d]$.
		Then the following statements hold.
		\begin{enumerate}
			\item[(i)] There exists $C_{\ell}>0$ such that, for every $f,g\in\Cf^{\ell}(\Td)$,
			\[
			\bnorm{\Gamma_{n,j}(f)-\Gamma_{n,j}(g)}_{\Cf^{\ell}(\Td)}\le C_{\ell}\left(1+\norm f_{\Cf^{\ell}(\Td)}+\norm g_{\Cf^{\ell}(\Td)}\right)\norm{f-g}_{\Cf^{\ell}(\Td)}.
			\]
			\item[(ii)] There exists $C_{m,\ell}>0$ such that, for every $\varrho\in\Cf^{m+\ell+1}(\Td)$,
			\[
			\bnorm{\Gamma_{n,j}(\varrho)-\Gamma^{(m)}_{n,j}(\varrho)}_{\Cf^{\ell}(\Td)}\le\frac{C_{m,\ell}}{n^{m+1}}\norm{\varrho}_{\Cf^{m+\ell+1}(\Td)}\left(1+\norm{\varrho}_{\Cf^{\ell}(\Td)}\right).
			\]
			\item[(iii)] There exists $C_{m}>0$ such that, for every $\varphi,\varrho\in\Cf^{m}(\Td)$,
			\[
			\left|\binn{\varphi,\Gamma_{n,j}(\varrho)}_{n}-\binn{\varphi,\Gamma_{n,j}(\varrho)}\right|\le\frac{C_{m}}{n^{m}}\norm{\varphi}_{\Cf^{m}(\Td)}\norm{\varrho}_{\Cf^{m}(\Td)}\left(1+\norm{\varrho}_{\Cf^{m}(\Td)}\right).
			\]
			\item[(iv)] There exists $C>0$ such that, for every $\varrho\in\Cf(\Td)$, 
			\[
			\bnorm{\Gamma_{n,j}(\varrho)^{-}}_{\Cf(\Td)}\le C\left(\bnorm{\varrho^{-}}_{\Cf(\Td)}+\bnorm{(1-\varrho)^{-}}_{\Cf(\Td)}\right)\left(1+\norm{\varrho}_{\Cf(\Td)}\right),
			\]
			where $x^{-}:=\max\{-x,0\}$.
		\end{enumerate}
	\end{lem}
	
	\begin{proof}
		We first note that the shift operator preserves the $\Cf^{\ell}$-norm,
		that is, 
		\[
		\bnorm{\tau_{n,j}f}_{\Cf^{\ell}(\Td)}=\bnorm f_{\Cf^{\ell}(\Td)}.
		\]
		
		For (i), using 
		\[
		\Gamma_{n,j}(f)=f+\tau_{n,j}f-2f\tau_{n,j}f,
		\]
		we obtain 
		\begin{align*}
			\Gamma_{n,j}(f)-\Gamma_{n,j}(g)= & (f-g)+\tau_{n,j}(f-g)\\
			& -2(f-g)\tau_{n,j}f-2g\tau_{n,j}(f-g).
		\end{align*}
		Hence, by the standard product estimate in $\Cf^{\ell}(\Td)$, 
		\[
		\bnorm{\Gamma_{n,j}(f)-\Gamma_{n,j}(g)}_{\Cf^{\ell}(\Td)}\le C_{\ell}\left(1+\bnorm f_{\Cf^{\ell}(\Td)}+\bnorm g_{\Cf^{\ell}(\Td)}\right)\bnorm{f-g}_{\Cf^{\ell}(\Td)}.
		\]
		
		For (ii), we have 
		\[
		\Gamma_{n,j}(\varrho)-\Gamma^{(m)}_{n,j}(\varrho)=(1-2\varrho)\left(\tau_{n,j}\varrho-\tau^{(m)}_{n,j}\varrho\right).
		\]
		Therefore, 
		\begin{align*}
			\bnorm{\Gamma_{n,j}(\varrho)-\Gamma^{(m)}_{n,j}(\varrho)}_{\Cf^{\ell}(\Td)} & \le C_{\ell}\left(1+\bnorm{\varrho}_{\Cf^{\ell}(\Td)}\right)\bnorm{\tau_{n,j}\varrho-\tau^{(m)}_{n,j}\varrho}_{\Cf^{\ell}(\Td)}\\
			& \le\frac{C_{m,\ell}}{n^{m+1}}\bnorm{\varrho}_{\Cf^{m+\ell+1}(\Td)}\left(1+\bnorm{\varrho}_{\Cf^{\ell}(\Td)}\right),
		\end{align*}
		where the last inequality follows from Lemma \ref{lem:approximation_of_discrete_derivatives}.
		
		We next prove (iii). The case $m=0$ is immediate. Let now $m\ge1$. Set 
		\[
		h:=\varphi\Gamma_{n,j}(\varrho).
		\]
		Then 
		\[
		\binn{\varphi,\Gamma_{n,j}(\varrho)}_{n}-\binn{\varphi,\Gamma_{n,j}(\varrho)}=\binn{h,\varsigma_{0}}_{n}-\binn{h,\varsigma_{0}}.
		\]
		Consequently, Lemma \ref{lem:approximation_for_ex} yields 
		\[
		\left|\binn{\varphi,\Gamma_{n,j}(\varrho)}_{n}-\binn{\varphi,\Gamma_{n,j}(\varrho)}\right|\le\frac{C_{m}}{n^{m}}\bnorm h_{\Cf^{m}(\Td)}.
		\]
		Since 
		\[
		\bnorm{\Gamma_{n,j}(\varrho)}_{\Cf^{m}(\Td)}\le C_{m}\bnorm{\varrho}_{\Cf^{m}(\Td)}\left(1+\bnorm{\varrho}_{\Cf^{m}(\Td)}\right),
		\]
		the product estimate gives 
		\[
		\bnorm h_{\Cf^{m}(\Td)}\le C_{m}\bnorm{\varphi}_{\Cf^{m}(\Td)}\bnorm{\varrho}_{\Cf^{m}(\Td)}\left(1+\bnorm{\varrho}_{\Cf^{m}(\Td)}\right),
		\]
		which proves (iii).
		
		Finally, for (iv), observe that 
		\[
		\Gamma_{n,j}(\varrho)=\varrho(1-\tau_{n,j}\varrho)+(1-\varrho)\tau_{n,j}\varrho.
		\]
		Using 
		\[
		(a+b)^{-}\le a^{-}+b^{-}
		\]
		and 
		\[
		(ab)^{-}\le a^{-}|b|+b^{-}|a|,
		\]
		we obtain 
		\begin{align*}
			\Gamma_{n,j}(\varrho)^{-}\le & \varrho^{-}\left|1-\tau_{n,j}\varrho\right|+(1-\tau_{n,j}\varrho)^{-}|\varrho|\\
			& +(1-\varrho)^{-}|\tau_{n,j}\varrho|+(\tau_{n,j}\varrho)^{-}|1-\varrho|.
		\end{align*}
		Since 
		\[
		(\tau_{n,j}\varrho)^{-}=\tau_{n,j}(\varrho^{-})
		\]
		and 
		\[
		(1-\tau_{n,j}\varrho)^{-}=\tau_{n,j}\bigl((1-\varrho)^{-}\bigr),
		\]
		taking the supremum over $\Td$ gives 
		\[
		\bnorm{\Gamma_{n,j}(\varrho)^{-}}_{\Cf(\Td)}\le C\left(\bnorm{\varrho^{-}}_{\Cf(\Td)}+\bnorm{(1-\varrho)^{-}}_{\Cf(\Td)}\right)\left(1+\bnorm{\varrho}_{\Cf(\Td)}\right).
		\]
		This completes the proof.
	\end{proof}

	\subsection{Modified Ornstein--Uhlenbeck process\label{subsec:Modified-Ornstein=002013Uhlenbeck-proc}}
	
	The goal of this subsection is to introduce the modified Ornstein--Uhlenbeck
	dynamics that will be used to approximate the fluctuation field. Its
	definition is motivated by the generator expansion obtained in Proposition
	\ref{prop:expansion_of_generator} and by the higher-order approximations
	of the discrete operators established in Lemma \ref{lem:approximation_of_discrete_derivatives}.
	
	Fix $n\ge3$, approximation orders $m_{\Delta},m_{\partial}\in\N$,
	where $m_{\Delta}$ is odd, and a measurable locally bounded function
	\[
	\Phi=(\Phi_{j})_{j\in[d]}:[0,\infty)\to L^{2}(\Td;\R^{d}).
	\]
	Later we will take 
	\[
	\Phi_{j}(t)=\sqrt{\left(\Gamma^{(m)}_{n,j}(\varrho_{t})\right)^{+}},\quad j\in[d],
	\]
	where $(\varrho_{t})_{t\ge0}$ is a solution to the modified heat
	equation (\ref{eq:modified_heat_equation}).
	
	We also write 
	\[
	A\coloneqq\frac{1}{2}\Delta^{(2m_{\Delta})}_{n},\quad\cS_{t}\coloneqq\cS^{n,m_{\Delta}}_{t}.
	\]
	
	The covariance structure of the noise is encoded by the operator 
	\[
	B(t):L^{2}(\Td;\R^{d})\to\H[-I]
	\]
	defined by 
	\begin{equation}
		B(t)h\coloneqq\frac{1}{\sqrt{2}}\sum^{d}_{j=1}\left(\partial^{(m_{\partial})}_{n,j}\right)^{*}\left(\Phi_{j}(t)h_{j}\right)=\frac{1}{\sqrt{2}}\divv^{(m_{\partial})}_{n}\left(\Phi(t)h\right)\label{eq:definition_of_B}
	\end{equation}
	for $h=(h_{1},\ldots,h_{d})$, where 
	\[
	\left(\partial^{(m_{\partial})}_{n,j}\right)^{*}=n\left[P_{m_{\partial}}\left(-\frac{\partial_{j}}{n}\right)-1\right]
	\]
	denotes the adjoint operator to $\partial^{(m_{\partial})}_{n,j}$
	in $\L$.
	The differential operators in this definition are understood in the
	distributional sense. Indeed, for $h\in L^{2}(\Td;\R^{d})$ one has
	$\Phi_{j}(t)h_{j}\in L^{1}(\Td)$, and the assumption
	$I>\frac{d}{2}+m_{\partial}$ ensures that the resulting distributions
	belong to $\H[-I]$.

	We denote by $\cL^{2}(H_{1};H_{2})$ the space of Hilbert--Schmidt
	operators from a Hilbert space $H_{1}$ to a Hilbert space $H_{2}$.
	The following estimate shows that $B(t)$ is Hilbert--Schmidt whenever
	$I>\frac{d}{2}+m_{\partial}$.
	\begin{lem}
		\label{lem:norm_of_diffusion_term_for_SPDE} Let $I>\frac{d}{2}+m_{\partial}$
		and let $B(t)$ be defined by (\ref{eq:definition_of_B}). Then 
		\[
		B(t)\in\cL^{2}\bigl(L^{2}(\Td;\R^{d});\H[-I]\bigr)
		\]
		for every $t\ge0$, and there exists a constant $C_{I,m_{\partial}}>0$
		such that 
		\[
		\bnorm{B(t)}_{\cL^{2}(L^{2}(\Td;\R^{d});\H[-I])}\le C_{I,m_{\partial}}\norm{\Phi(t)}_{L^{2}(\Td;\R^{d})},\quad t\ge0.
		\]
	\end{lem}
	
	\begin{proof}
		We first note that the family $h_{k,j}=\varsigma_{k}e_{j}$, $k\in\Zd$,
		$j\in[d]$, is an orthonormal basis in $L^{2}(\Td;\R^{d})$. Thus,
		we can estimate
		\begin{align*}
			\bnorm{B(t)}^{2}_{\cL^{2}(L^{2}(\Td;\R^{d});\H[-I])} & =\frac{1}{2}\sum^{d}_{j=1}\sum_{k\in\Zd}\bnorm{\left(\partial^{(m_{\partial})}_{n,j}\right)^{*}\left(\Phi_{j}(t)\varsigma_{k}\right)}^{2}_{\H[-I]}\\
			& \le\frac{1}{2}\sum^{d}_{j=1}\sum_{k\in\Zd}\sum_{l\in\Zd}\frac{1}{(1+|l|^{2})^{I}}\left|\binn{\left(\partial^{(m_{\partial})}_{n,j}\right)^{*}\left(\Phi_{j}(t)\varsigma_{k}\right),\varsigma_{l}}\right|^{2}\\
			& =\frac{1}{2}\sum^{d}_{j=1}\sum_{k\in\Zd}\sum_{l\in\Zd}\frac{1}{(1+|l|^{2})^{I}}\left|\binn{\Phi_{j}(t)\varsigma_{k},\partial^{(m_{\partial})}_{n,j}\varsigma_{l}}\right|^{2},
		\end{align*}
		Note that
		\[
		\partial^{(m_{\partial})}_{n,j}\varsigma_{l}=q^{(m_{\partial})}_{n,j}(l)\varsigma_{l},
		\]
		where
		\[
		q^{(m_{\partial})}_{n,j}(l):=n\left[P_{m_{\partial}}\left(\frac{2\pi\i l_{j}}{n}\right)-1\right].
		\]
		Thus, by the Parseval equality,
		\begin{align*}
			\sum_{k\in\Zd}\left|\binn{\Phi_{j}(t)\varsigma_{k},\partial^{(m_{\partial})}_{n,j}\varsigma_{l}}\right|^{2} & =\sum_{k\in\Zd}\left|q^{(m_{\partial})}_{n,j}(l)\right|^{2}\left|\binn{\Phi_{j}(t)\varsigma_{k},\varsigma_{l}}\right|^{2}\\
			& =\left|q^{(m_{\partial})}_{n,j}(l)\right|^{2}\bnorm{\Phi_{j}(t)\varsigma_{l}}^{2}_{\L}=\left|q^{(m_{\partial})}_{n,j}(l)\right|^{2}\norm{\Phi_{j}(t)}^{2}_{\L}.
		\end{align*}
		
		Since 
		\[
		\left|q^{(m_{\partial})}_{n,j}(l)\right|^{2}\le C_{m_{\partial}}(1+|l|^{2})^{m_{\partial}},
		\]
		we obtain 
		\begin{align*}
			\bnorm{B(t)}^{2}_{\cL^{2}(L^{2}(\Td;\R^{d});\H[-I])} & \le C_{m_{\partial}}\sum^{d}_{j=1}\norm{\Phi_{j}(t)}^{2}_{\L}\sum_{l\in\Zd}(1+|l|^{2})^{-I+m_{\partial}}\\
			& \le C_{I,m_{\partial}}\norm{\Phi(t)}^{2}_{L^{2}(\Td;\R^{d})}
		\end{align*}
		due to the fact that $I-m_{\partial}>\frac{d}{2}$. This completes
		the proof of the lemma.
	\end{proof}
	
	The modified Ornstein--Uhlenbeck process will be defined as the mild
	solution to 
	\[
	d\xi_{t}=A\xi_{t}dt+B(t)\,dW_{t},\quad\xi_{0}=g,
	\]
	where $W$ is a cylindrical Wiener process on $L^{2}(\Td;\R^{d})$.
	Formally, this equation can be written as 
	\begin{align}
		\partial_{t}\xi_{t} & =\frac{1}{2}\Delta^{(2m_{\Delta})}_{n}\xi_{t}+\frac{1}{\sqrt{2}}\divv^{(m_{\partial})}_{n}\left(\Phi(t)\dot{W}_{t}\right),\quad t\ge0,\label{eq:SPDE_for_OU_process}\\
		\xi_{0} & =g,\nonumber 
	\end{align}
	where $\dot{W}$ denotes a $d$-dimensional space--time white noise.
	
	We first consider the solution with zero initial condition. Set $\cS_{t}:=\cS^{n,m_{\Delta}}_{t},$
	$t\ge0$.
	\begin{lem}
		\label{lem:stochastic_convolution}Let $I>\frac{d}{2}+m_{\partial}$
		and $\Phi$ be a locally bounded measurable function. Then the stochastic
		convolution 
		\begin{equation}
			Z_{t}\coloneqq\int^{t}_{0}\cS_{t-s}B(s)dW_{s},\quad t\ge0,\label{eq:stochastic_convolution}
		\end{equation}
		is well defined and admits an $\H[-I]$-valued continuous Gaussian
		version. Moreover, $(Z_{t})_{t\ge0}$ is the unique mild solution
		to (\ref{eq:SPDE_for_OU_process}) with initial condition $g=0$.
		For every $T>0$ there exists a constant $C_{I,m_{\partial},T}>0$ such that 
		\begin{equation}
			\E\left[\sup_{s\in[0,t]}\bnorm{Z_{s}}^{2}_{\H[-I]}\right]\le C_{I,m_{\partial},T}\int^{t}_{0}\bnorm{\Phi(s)}^{2}_{L^{2}(\Td;\R^{d})}ds,\quad t\in[0,T].\label{eq:uniform_control_of_OU_process}
		\end{equation}
	\end{lem}
	
	\begin{proof}
		By Lemma \ref{lem:modified_heat_semigroup}, $A=\frac{1}{2}\Delta^{(2m_{\Delta})}_{n}$
		generates the strongly continuous contraction semigroup $(\cS_{t})_{t\ge0}$
		on $\H[-I]$. Moreover, by Lemma \ref{lem:norm_of_diffusion_term_for_SPDE},
		the map $t\mapsto B(t)$ is measurable from $[0,\infty)$ to $\cL^{2}(L^{2}(\Td;\R^{d});\H[-I])$
		and locally bounded in the Hilbert--Schmidt norm. Hence the assumptions
		of \citep[Hypothesis 7.1]{Prato:2014} are satisfied.
		
		Therefore, by \citep[Theorem 7.2]{Prato:2014}, the stochastic convolution
		\[
		Z_{t}=\int^{t}_{0}\cS_{t-s}B(s)dW_{s}
		\]
		admits an $\H[-I]$-valued continuous version and is the unique mild
		solution to (\ref{eq:SPDE_for_OU_process}) with initial condition
		$g=0$. Since $B$ is deterministic, $(Z_{t})_{t\ge0}$ is a Gaussian
		process.
		
		Finally, (\ref{eq:uniform_control_of_OU_process}) follows from \citep[Proposition 7.3]{Prato:2014}
		and Lemma \ref{lem:norm_of_diffusion_term_for_SPDE}.
	\end{proof}
	
	\begin{rem}
		The condition $I>\frac{d}{2}+m_{\partial}$ is used only to ensure
		that $B(t)$ is Hilbert--Schmidt from $L^{2}(\Td;\R^{d})$ to $\H[-I]$.
		No smoothing property of $(\cS_{t})_{t\ge0}$ is used in the preceding
		construction. For fixed $n$, the operator $\Delta^{(2m_{\Delta})}_{n}$
		is of order $2m_{\Delta}$, and its semigroup has a regularizing effect.
		Exploiting this smoothing allows the stochastic convolution to be
		constructed in Sobolev spaces of higher regularity and, for $m_{\Delta}$
		sufficiently large, as a function-valued process.
	\end{rem}
	
	We will later consider solutions to (\ref{eq:SPDE_for_OU_process})
	started from $g\in\H[-I]$ at an arbitrary time $s\ge0$. For convenience,
	we set 
	\[
	\cS_{s,t}\coloneqq\cS^{n,m_{\Delta}}_{t-s},\quad0\le s\le t.
	\]
	
	\begin{prop}
		\label{prop:OU_started_at_s} Under the assumptions of Lemma \ref{lem:stochastic_convolution},
		for every $g\in\H[-I]$ and $s\ge0$, the process 
		\begin{equation}
			\xi_{s,t}(g)\coloneqq\cS_{s,t}g+\int^{t}_{s}\cS_{r,t}B(r)dW_{r},\quad t\ge s,\label{eq:OU_process}
		\end{equation}
		is the unique mild solution to (\ref{eq:SPDE_for_OU_process}) started
		from $g$ at time $s$.
		
		Moreover, for every $T>0$, $s\in[0,T]$ and $g\in\H[-I]$
		\[
		\E\left[\sup_{t\in[s,T]}\bnorm{\xi_{s,t}(g)}^{2}\right]\le2\norm g^{2}+C_{I,m_{\partial},T}\int^{T}_{s}\bnorm{\Phi(r)}^{2}_{L^{2}(\Td;\R^{d})}dr
		\]
		and for every $T,c,\delta>0$, 
		\[
		\sup_{\bnorm g_{\H[-I+\delta]}\le c}\sup_{s\in[0,T]}\E\left[\sup_{t\in[s,s+\eps]}\bnorm{\xi_{s,t}(g)-g}^{2}_{\H[-I]}\right]\to0
		\]
		as $\eps\to0+$.
	\end{prop}
	
	\begin{proof}
		The first claim of the statement follows from Lemmas \ref{lem:modified_heat_semigroup}
		and \ref{lem:stochastic_convolution}.
		
		For the second claim, using the contraction property of the semigroup,
		we have
		\begin{align*}
			\E\left[\sup_{t\in[s,T]}\bnorm{\xi_{s,t}(g)}^{2}_{\H[-I]}\right] & \le2\sup_{t\in[s,T]}\bnorm{\cS_{s,t}g}^{2}_{\H[-I]}+2\E\left[\sup_{t\in[s,T]}\bnorm{\int^{t}_{s}\cS_{r,t}B(r)\,dW_{r}}^{2}_{\H[-I]}\right]\\
			& \le2\bnorm g^{2}_{\H[-I]}+C_{T}\int^{T}_{s}\bnorm{\Phi(r)}^{2}_{L^{2}(\Td;\R^{d})}\,dr.
		\end{align*}
		Indeed, the estimate for the stochastic convolution follows from Lemma
		\ref{lem:stochastic_convolution} applied to the shifted cylindrical
		Wiener process
		\[
		W^{s}_{r}:=W_{s+r}-W_{s}
		\]
		and the function $r\mapsto\Phi(s+r)$.
		
		We next prove the last assertion. For $g\in\H[-I]$ and $s\in[0,T]$,
		we have 
		\begin{align*}
			\E\left[\sup_{t\in[s,s+\eps]}\bnorm{\xi_{s,t}(g)-g}^{2}_{\H[-I]}\right] & \le2R^{1}_{\eps}(g;s)+2R^{2}_{\eps}(s),
		\end{align*}
		where 
		\[
		R^{1}_{\eps}(g;s)\coloneqq\sup_{t\in[s,s+\eps]}\bnorm{\cS_{s,t}g-g}^{2}_{\H[-I]}
		\]
		and 
		\[
		R^{2}_{\eps}(s)\coloneqq\E\left[\sup_{t\in[s,s+\eps]}\bnorm{\int^{t}_{s}\cS_{r,t}B(r)dW_{r}}^{2}_{\H[-I]}\right].
		\]
		
		For the first term, 
		\begin{align*}
			R^{1}_{\eps}(g;s) & =\sup_{u\in[0,\eps]}\bnorm{(\cS_{u}-\id)g}^{2}_{\H[-I]}\\
			& \le\sup_{u\in[0,\eps]}\bnorm{\cS_{u}-\id}^{2}_{\cL_{1}(\H[-I+\delta];\H[-I])}\bnorm g^{2}_{\H[-I+\delta]}.
		\end{align*}
		Hence, by Lemma \ref{lem:operator_norm_continuity_modified_semigroup},
		\[
		\sup_{\bnorm g_{\H[-I+\delta]}\le c}\sup_{s\in[0,T]}R^{1}_{\eps}(g;s)\to0
		\]
		as $\eps\to0+$.
		
		For the stochastic term, using the shifted cylindrical Brownian motion
		\[
		W^{s}_{r}\coloneqq W_{s+r}-W_{s},
		\]
		we obtain 
		\[
		R^{2}_{\eps}(s)=\E\left[\sup_{u\in[0,\eps]}\bnorm{\int^{u}_{0}\cS_{u-r}B(s+r)dW^{s}_{r}}^{2}_{\H[-I]}\right].
		\]
		Therefore, by Lemma \ref{lem:stochastic_convolution}, 
		\[
		R^{2}_{\eps}(s)\le C_{I,m_{\partial},T+1}\int^{\eps}_{0}\bnorm{\Phi(s+r)}^{2}_{L^{2}(\Td;\R^{d})}\,dr
		\]
		for $\eps\le1$. Since $\Phi$ is locally bounded, 
		\[
		\sup_{s\in[0,T]}R^{2}_{\eps}(s)\le \eps C_{I,m_{\partial},T+1}\sup_{r\in[0,T+1]}\bnorm{\Phi(r)}^{2}_{L^{2}(\Td;\R^{d})}\to0
		\]
		as $\eps\to0+$.
		
		Combining the two estimates completes the proof.
	\end{proof}
	
	\begin{rem}
		\label{rem:convergence_to_zero_of_sup_xi}Similarly to the proof of
		Proposition \ref{prop:OU_started_at_s}, we can prove that for each
		$g\in\H[-I]$
		\[
		\sup_{s\in[0,T]}\E\left[\sup_{t\in[s,s+\eps]}\bnorm{\xi_{s,t}(g)-g}^{2}_{\H[-I]}\right]\to0
		\]
		as $\eps\to0+$.
	\end{rem}
	
	We finally introduce the time-dependent generator associated with
	the modified Ornstein--Uhlenbeck dynamics. Let $I>\frac{d}{2}+m_{\partial}$
	and $F\in\Cf^{2}_{\text{der}}(\H[-I-2m_{\Delta}])$. Since $\H[-I]\subset\H[-I-2m_{\Delta}],$
	the function $F$ and its derivatives can be evaluated at every $f\in\H[-I]$.
	Moreover, $DF(f)\in\H[I+2m_{\Delta}],$ and hence 
	\[
	\Delta^{(2m_{\Delta})}_{n}DF(f)\in\H[I].
	\]
	Therefore, the term $\binn{\Delta^{(2m_{\Delta})}_{n}DF(f),f}$ is
	well defined.
	
	Furthermore, by Lemma \ref{lem:existence_of_smooth_kernel}, for every
	$f\in\H[-I]$, the bilinear functional $D^{2}F(f)$ admits a kernel
	\[
	D^{2}F(f;\cdot)\in\Cf^{m_{\partial},m_{\partial}}\left((\Td)^{2}\right).
	\]
	Consequently, for $t\ge0$ we define 
	\begin{align}
		\cG^{\Phi}_{t}F(f)\coloneqq & \frac{1}{2}\binn{\Delta^{(2m_{\Delta})}_{n}DF(f),f}+\frac{1}{4}\sum^{d}_{j=1}\binn{\left(\partial^{(m_{\partial})}_{n,j}\right)^{\odot2}D^{2}F(f;\cdot),\Phi^{2}_{j}(t)},\quad f\in\H[-I].\label{eq:generator_G_Phi}
	\end{align}
	Here 
	\[
	\left(\partial^{(m_{\partial})}_{n,j}\right)^{\odot2}D^{2}F(f;x)\coloneqq\left.\partial^{(m_{\partial})}_{n,j,x_{1}}\partial^{(m_{\partial})}_{n,j,x_{2}}D^{2}F(f;x_{1},x_{2})\right|_{x_{1}=x_{2}=x}.
	\]
	
	\begin{lem}
		\label{lem:martingale_problem_modified_OU} Let $I>\frac{d}{2}+m_{\partial}$,
		$\Phi:[0,\infty)\to L^{2}(\Td;\R^{d})$ be measurable and locally
		bounded, and $F\in\Cf^{2}_{\text{der}}(\H[-I-2m_{\Delta}]).$ Then,
		for every $g\in\H[-I]$ and $s\ge0$, the process 
		\[
		F(\xi_{s,t}(g))-F(g)-\int^{t}_{s}\cG^{\Phi}_{r}F(\xi_{s,r}(g))dr,\quad t\ge s,
		\]
		is a continuous square-integrable martingale with respect to the filtration
		generated by $\left(W_{t}-W_{s}\right)_{t\ge s}$.
	\end{lem}
	
	\begin{proof}
		We apply Itô's formula in the Hilbert space $\H[-I-2m_{\Delta}]$.
		Since $\H[-I]\subset\H[-I-2m_{\Delta}]$ and 
		\[
		A=\frac{1}{2}\Delta^{(2m_{\Delta})}_{n}:\H[-I]\to\H[-I-2m_{\Delta}]
		\]
		is bounded, the process $\xi_{s,\cdot}(g)$ satisfies
		\[
		\xi_{s,t}(g)=g+\int^{t}_{s}A\xi_{s,r}(g)dr+\int^{t}_{s}B(r)dW_{r}
		\]
		as an $\H[-I-2m_{\Delta}]$-valued process, according to \citep[Theorem 3.1]{Gawarecki:2011}.
		By Itô's formula (see, e.g., \citep[Theorem 4.32]{Prato:2014}), we
		have 
		\begin{align*}
			F(\xi_{s,t}(g))= & F(g)+\int^{t}_{s}\binn{DF(\xi_{s,r}(g)),B(r)dW_{r}}\\
			& +\int^{t}_{s}\left[\binn{DF(\xi_{s,r}(g)),A\xi_{s,r}(g)}+\frac{1}{2}\Tr\left[B(r)^{*}D^{2}F(\xi_{s,r}(g))B(r)\right]\right]dr.
		\end{align*}
		Since $A=\frac{1}{2}\Delta^{(2m_{\Delta})}_{n}$ is self-adjoint and
		$DF(f)\in\H[I+2m_{\Delta}],$ we obtain 
		\[
		\binn{DF(f),Af}=\frac{1}{2}\binn{\Delta^{(2m_{\Delta})}_{n}DF(f),f},\quad f\in\H[-I].
		\]
		
		It remains to identify the trace term. Let $\{\tilde{\varsigma}_{k}\}_{k\in\Zd}$ be a real trigonometric basis obtained from the complex Fourier basis by taking sine and cosine components and chosen to be
		orthonormal in $L^{2}(\Td)$. We take $h_{k,j}=\tilde\varsigma_{k} e_{j}$,
		$k\in\Zd$, $j\in[d]$, and note that they form an orthonormal basis of $L^{2}(\Td;\R^{d})$.
		Hence 
		\begin{align*}
			\Tr\left[B(r)^{*}D^{2}F(f)B(r)\right] & =\sum^{d}_{j=1}\sum_{k\in\Zd}D^{2}F(f)\left[B(r)h_{k,j},B(r)h_{k,j}\right]\\
			& =\frac{1}{2}\sum^{d}_{j=1}\sum_{k\in\Zd}D^{2}F(f)\left[\left(\partial^{(m_{\partial})}_{n,j}\right)^{*}\bigl(\Phi_{j}(r)\tilde\varsigma_{k}\bigr),\left(\partial^{(m_{\partial})}_{n,j}\right)^{*}\bigl(\Phi_{j}(r)\tilde\varsigma_{k}\bigr)\right].
		\end{align*}
		Using the kernel representation of $D^{2}F(f)$ and moving the derivatives
		onto its kernel, we get 
		\begin{align*}
			&\Tr\left[B(r)^{*}D^{2}F(f)B(r)\right]\\
			& \qquad = \frac{1}{2}\sum^{d}_{j=1}\sum_{k\in\Zd}\int_{(\Td)^{2}}\partial^{(m_{\partial})}_{n,j,x}\partial^{(m_{\partial})}_{n,j,y}D^{2}F(f;x,y)\Phi_{j}(r,x)\Phi_{j}(r,y)\tilde\varsigma_{k}(x)\tilde\varsigma_{k}(y)dxdy.
		\end{align*}
		To justify the passage to the diagonal, let $P_{N}$ be the finite-rank
		orthogonal projections onto an increasing family of trigonometric
		modes. Since $B(r)$ is Hilbert--Schmidt and $D^{2}F(f)$ is bounded,
		$B(r)^{*}D^{2}F(f)B(r)$ is trace class. Hence its trace is the limit
		of the corresponding finite-dimensional traces. Using the kernel
		representation and Parseval's identity in these finite-dimensional
		traces and then passing to the limit yields
		\begin{align*}
			\Tr\left[B(r)^{*}D^{2}F(f)B(r)\right] & =\frac{1}{2}\sum^{d}_{j=1}\int_{\Td}\left(\partial^{(m_{\partial})}_{n,j}\right)^{\odot2}D^{2}F(f;x)\Phi_{j}^{2}(r,x)dx\\
			& =\frac{1}{2}\sum^{d}_{j=1}\binn{\left(\partial^{(m_{\partial})}_{n,j}\right)^{\odot2}D^{2}F(f;\cdot),\Phi_{j}^{2}(r)}.
		\end{align*}
		Consequently, 
		\[
		\binn{DF(f),Af}+\frac{1}{2}\Tr\left[B(r)^{*}D^{2}F(f)B(r)\right]=\cG^{\Phi}_{r}F(f)
		\]
		and, therefore, 
		\[
		F(\xi_{s,t}(g))-F(g)-\int^{t}_{s}\cG^{\Phi}_{r}F(\xi_{s,r}(g))dr=\int^{t}_{s}\binn{DF(\xi_{s,r}(g)),B(r)dW_{r}}.
		\]
		
		Finally, since $DF$ is bounded and $\Phi$ is locally bounded, Lemma~\ref{lem:norm_of_diffusion_term_for_SPDE}
		implies, for every $T>s$, 
		\[
		\E\int^{T}_{s}\bnorm{B(r)^{*}DF(\xi_{s,r}(g))}^{2}_{L^{2}(\Td;\R^{d})}dr<\infty.
		\]
		Hence the stochastic integral on the right-hand side is a square-integrable
		continuous martingale, which completes the proof.
	\end{proof}

	\section{Non-anticipative functionals and backward Kolmogorov equations\label{sec:non-anticipative_functionals}}
	
	In this section, we develop the framework of non-anticipative functionals
	needed to compare the fluctuation field with the modified Ornstein-Uhlenbeck
	process introduced in the previous section. After introducing suitable
	notions of horizontal and vertical differentiation and establishing
	the corresponding martingale problem, we study continuation functionals
	associated with the modified Ornstein-Uhlenbeck dynamics. The main
	goal is to derive a backward Kolmogorov equation for these functionals,
	which will be used in the subsequent analysis.
	
	\subsection{Path spaces and non-anticipative functionals}
	
	We begin by introducing the space of stopped paths and non-anticipative
	functionals, following \citep{Cont_Fournie2010,Cont_Fournie:2013}.
	Let $(E,d_{E})$ be a metric space and let $T>0$ be fixed. For an
	interval $K\subseteq[0,T]$, we denote by $D(K,E)$ the space of \cdl
	functions from $K$ to $E$, equipped with the uniform topology. If
	$K_{1}\subseteq K_{2}\subseteq[0,T]$ and $\gamma\in D(K_{2},E)$,
	we denote the restriction of $\gamma$ to $K_{1}$ by $\gamma^{K_{1}}$.
	
	For $0\le a<b\le T$, we define
	\[
	\Upsilon([a,b],E):=\bigcup_{t\in[a,b]}\left(\{t\}\times D([a,t],E)\right).
	\]
	We regard the spaces $D([a,t],E)$, $t\in[a,b]$, as disjoint, so
	that the terminal time $t$ is uniquely determined by an element $\gamma\in D([a,t],E)$.
	Accordingly, we will often identify $(t,\gamma)\in\Upsilon([a,b],E)$
	with $\gamma$ and simply write $\gamma\in\Upsilon([a,b],E)$.
	
	For $(t,\gamma),(r,\vartheta)\in\Upsilon([a,b],E)$, define 
	\begin{equation}
		d_{\Upsilon([a,b],E)}(\gamma,\vartheta):=\sup_{s\in[a,b]}d_{E}\bigl(\gamma_{s\wedge t},\vartheta_{s\wedge r}\bigr)+|t-r|.\label{eq:distance_d_infty}
	\end{equation}
	Then $\left(\Upsilon([a,b],E),d_{\Upsilon([a,b],E)}\right)$ is a
	metric space. Moreover, if $(E,d_{E})$ is complete, then $\Upsilon([a,b],E)$
	is complete as well.
	
	For measurability, we equip $D(K,E)$ with the cylindrical $\sigma$-algebra. We equip $\Upsilon([a,b],E)$ with the $\sigma$-algebra generated by the maps
	\[
	(t,\gamma)\mapsto t,
	\qquad
	(t,\gamma)\mapsto\gamma_{s\wedge t},
	\quad s\in[a,b].
	\]
	All continuity and differentiability requirements refer to the uniform metrics introduced above.
	\begin{defn}
		A measurable map 
		\[
		F:\Upsilon([a,b],E)\to\widetilde{E},\quad(t,\gamma)\mapsto F_{t}(\gamma),
		\]
		where $(\widetilde{E},d_{\widetilde{E}})$ is a metric space, is called
		a non-anticipative functional. If $\gamma\in D([a,t],E)$, we will
		also write $F(\gamma)$ instead of $F_{t}(\gamma)$.
	\end{defn}
	
	We next introduce the continuity properties of non-anticipative functionals
	that will be used below.
	\begin{defn}
		Let $F:\Upsilon([a,b],E)\to\widetilde{E}$ be a non-anticipative functional.
		\begin{enumerate}
			\item[(i)]  We say that $F$ is \emph{continuous at fixed time} $t\in[a,b]$
			if 
			\[
			D([a,t],E)\ni\gamma\longmapsto F_{t}(\gamma)\in\widetilde{E}
			\]
			is continuous with respect to the uniform metric.
			\item[(ii)] We say that $F$ is \emph{left-continuous} at $\gamma\in D([a,t],E)$,
			$t\in(a,b]$, if for every $\eps>0$ there exists $\delta>0$ such
			that 
			\[
			d_{\widetilde{E}}\bigl(F_{t}(\gamma),F_{r}(\vartheta)\bigr)<\eps
			\]
			whenever $a\le r\le t$, $\vartheta\in D([a,r],E)$, and 
			\[
			d_{\Upsilon([a,b],E)}(\gamma,\vartheta)<\delta.
			\]
			\item[(iii)] We say that $F$ is \emph{continuous} if it is continuous as a map
			from $\Upsilon([a,b],E)$ to $\widetilde{E}$.
		\end{enumerate}
	\end{defn}
	
	For $a\le t\le r\le b$, define the constant-extension operator 
	\[
	\Theta^{a}_{t,r}:D([a,t],E)\to D([a,r],E)
	\]
	by 
	\[
	(\Theta^{a}_{t,r}\gamma)_{s}:=\gamma_{s\wedge t},\quad s\in[a,r].
	\]
	Thus, $\Theta^{a}_{t,r}\gamma$ coincides with $\gamma$ on $[a,t]$
	and remains constant at $\gamma_{t}$ on $[t,r]$. We also set $\Theta_{t,r}:=\Theta^{0}_{t,r}$.
	
	For every $\gamma\in D([a,b],E)$, the map from $\{(t,r):a\le t\le r\le b\}$
	to $\Upsilon([a,b],E)$ defined by
	\[
	(t,r)\mapsto\Theta^{a}_{t,r}\gamma^{[a,t]}
	\]
	is measurable. Consequently, for every non-anticipative functional
	$F$, the map 
	\[
	(t,r)\mapsto F_{r}\bigl(\Theta^{a}_{t,r}\gamma^{[a,t]}\bigr)
	\]
	is also measurable.
	
	\subsection{Horizontal absolute continuity}
	
	To derive the path-dependent backward Kolmogorov equation used below,
	we need a notion of horizontal regularity which is sufficiently weak
	to accommodate the non-anticipative functionals arising in our comparison
	argument. In particular, we will consider functionals whose dependence
	on the path involves its values at finitely many fixed times. For
	such functionals, classical horizontal differentiability is too restrictive,
	since the horizontal derivative may fail to exist, or to possess the
	required regularity, at these prescribed times.
	
	We therefore introduce a weaker notion of horizontal regularity based
	on absolute continuity. More precisely, we require that, along every
	constant extension of a stopped path, the increments of the functional
	admit an integral representation in terms of a measurable non-anticipative
	functional. This allows the horizontal derivative to be defined only
	almost everywhere in time and provides precisely the regularity needed
	for the backward Kolmogorov equations considered below.
	\begin{defn}
		A real-valued non-anticipative functional $F$ is said to be \textit{horizontally
			absolutely continuous} if there exists a non-anticipative functional
		$G$ such that, for every $(t,\gamma)\in\Upsilon([0,T],E)$, the map
		\[
		s\mapsto G_{s}(\Theta_{t,s}\gamma)
		\]
		is Lebesgue integrable on $[t,T]$ and 
		\[
		F_{r}(\Theta_{t,r}\gamma)-F_{t}(\gamma)=\int^{r}_{t}G_{s}(\Theta_{t,s}\gamma)ds,\quad r\in[t,T].
		\]
		Any such functional $G$ is called a \textit{weak horizontal derivative}
		of $F$ and is denoted by 
		\[
		\cH_{t}F:=G_{t},\quad t\in[0,T].
		\]
	\end{defn}
	
	In general, a weak horizontal derivative is determined only almost
	everywhere in time along each constant extension. The next lemma shows
	that, under continuity at fixed times, any two weak horizontal derivatives
	nevertheless coincide almost everywhere along every \cdl path.
	\begin{lem}
		\label{lem:uniqueness_of_weak_derivatives} Let $F$ be a horizontally
		absolutely continuous non-anticipative functional, and let $G$ and
		$\widetilde{G}$ be two weak horizontal derivatives of $F$ which
		are continuous at each fixed time. Then, for every $\gamma\in D([0,T],E)$,
		\[
		G_{t}(\gamma^{[0,t]})=\widetilde{G}_{t}(\gamma^{[0,t]})
		\]
		for a.e. $t\in[0,T]$.
	\end{lem}
	
	\begin{proof}
		Fix $\gamma\in D([0,T],E)$. Since $G$ and $\widetilde{G}$ are both
		weak horizontal derivatives of $F$, for every $t\in[0,T]$ and $r\in[t,T]$,
		\[
		\int^{r}_{t}\left[G_{s}(\Theta_{t,s}\gamma^{[0,t]})-\widetilde{G}_{s}(\Theta_{t,s}\gamma^{[0,t]})\right]ds=0.
		\]
		Hence 
		\[
		G_{r}(\Theta_{t,r}\gamma^{[0,t]})=\widetilde{G}_{r}(\Theta_{t,r}\gamma^{[0,t]})
		\]
		for a.e. $r\in[t,T]$. In particular, for every $t\in[0,T]$, the
		set 
		\[
		B_{t}:=\left\{ r\in[t,T]:G_{r}(\Theta_{t,r}\gamma^{[0,t]})\neq\widetilde{G}_{r}(\Theta_{t,r}\gamma^{[0,t]})\right\} 
		\]
		has Lebesgue measure zero. By Fubini's theorem, for a.e. $r\in[0,T]$
		the set 
		\[
		B^{r}:=\left\{ t\in[0,r]:G_{r}(\Theta_{t,r}\gamma^{[0,t]})\neq\widetilde{G}_{r}(\Theta_{t,r}\gamma^{[0,t]})\right\} 
		\]
		has Lebesgue measure zero too.
		
		Fix such an $r$ at which $\gamma$ is continuous. Since $B^{r}$
		has Lebesgue measure zero, we may choose a sequence $t_{k}\uparrow r$
		with $t_{k}\notin B^{r}$. Moreover, 
		\[
		d_{\Upsilon([0,T],E)}\left(\Theta_{t_{k},r}\gamma^{[0,t_{k}]},\gamma^{[0,r]}\right)=\sup_{u\in[t_{k},r]}d_{E}(\gamma_{t_{k}},\gamma_{u})\longrightarrow0
		\]
		as $k\to\infty$, by the continuity of $\gamma$ at $r$. Since $G$
		and $\widetilde{G}$ are continuous at the fixed time $r$, we obtain
		\[
		G_{r}(\gamma^{[0,r]})=\lim_{k\to\infty}G_{r}(\Theta_{t_{k},r}\gamma^{[0,t_{k}]})=\lim_{k\to\infty}\widetilde{G}_{r}(\Theta_{t_{k},r}\gamma^{[0,t_{k}]})=\widetilde{G}_{r}(\gamma^{[0,r]}).
		\]
		Since a \cdl path has at most countably many discontinuities, the
		preceding identity holds for a.e. $r\in[0,T]$, which proves the claim.
	\end{proof}

	\subsection{Strong horizontal derivative\label{subsec:horizontal_derivative}}
	
	We next consider a connection of horizontal absolute continuity to
	the classical horizontal derivative of non-anticipative functionals.
	Let $0\le a<b\le T$ and let $F$ be a real-valued non-anticipative
	functional on $\Upsilon([a,b],E)$. Following \citep[Definition 3.1]{Cont_Fournie2010},
	for $\gamma\in D([a,t],E)$ and $t\in[a,b)$ we define the \textit{(strong)
		horizontal derivative} of $F$ by 
	\[
	\cD_{t}F(\gamma):=\lim_{r\downarrow t}\frac{F_{r}(\Theta^{a}_{t,r}\gamma)-F_{t}(\gamma)}{r-t},
	\]
	whenever the limit exists. We say that $F$ is horizontally differentiable
	if $\cD_{t}F(\gamma)$ exists for every $(t,\gamma)\in\Upsilon([a,b],E)$
	with $t<b$. Setting $\cD_{b}F(\gamma):=0$, we regard $\cD F$ as
	a non-anticipative functional on $\Upsilon([a,b],E)$.
	
	We also recall the notion of a boundedness preserving functional.
	\begin{defn}[Boundedness preserving functionals]
		Let $\widetilde{E}$ be a metric space. We denote by $\bB(\Upsilon([a,b],E);\widetilde{E})$
		the class of non-anticipative functionals
		\[
		F:\Upsilon([a,b],E)\to\widetilde{E}
		\]
		such that, for every compact set $U\subset E$, the set
		\[
		\left\{ F_{t}(\gamma):t\in[a,b],\;\gamma\in D([a,t],U)\right\} 
		\]
		is bounded in $\widetilde{E}$.
	\end{defn}
	
	The following result shows that, under mild additional assumptions,
	the strong horizontal derivative is also a weak horizontal derivative.
	\begin{lem}
		\label{lem:differentiability_and_absolute_continuity} Let $F$ be
		a real-valued, left-continuous, horizontally differentiable non-anticipative
		functional. Assume that, for every $(t,\gamma)\in\Upsilon([0,T],E)$,
		the map 
		\[
		s\mapsto\cD_{s}F(\Theta_{t,s}\gamma)
		\]
		is Lebesgue integrable on $[t,T]$. Then $F$ is horizontally absolutely
		continuous and $\cD F$ is a weak horizontal derivative of $F$. More
		precisely, for every $0\le t\le r\le T$ and $\gamma\in D([0,t],E)$,
		\[
		F_{r}(\Theta_{t,r}\gamma)-F_{t}(\gamma)=\int^{r}_{t}\cD_{s}F(\Theta_{t,s}\gamma)ds.
		\]
	\end{lem}
	
	\begin{proof}
		Fix $t\in[0,T)$ and $\gamma\in D([0,t],E)$, and define $\varphi:[t,T]\to\R$
		by 
		\[
		\varphi(s):=F_{s}(\Theta_{t,s}\gamma).
		\]
		We first show that $\varphi$ is left-continuous. Let $s_{k}\uparrow s$.
		Then 
		\[
		d_{\Upsilon([0,T],E)}\left(\Theta_{t,s_{k}}\gamma,\Theta_{t,s}\gamma\right)=s-s_{k}\longrightarrow0.
		\]
		Since $F$ is left-continuous, it follows that $\varphi(s_{k})\to\varphi(s)$.
		
		For $s\in[t,T)$, using 
		\[
		\Theta_{t,r}\gamma=\Theta_{s,r}(\Theta_{t,s}\gamma),\quad r\ge s,
		\]
		we obtain 
		\begin{align*}
			\varphi'_{+}(s) & =\lim_{r\downarrow s}\frac{F_{r}(\Theta_{t,r}\gamma)-F_{s}(\Theta_{t,s}\gamma)}{r-s}\\
			& =\cD_{s}F(\Theta_{t,s}\gamma).
		\end{align*}
		Since $\varphi$ is right differentiable, it is also right continuous
		and, thus, a continuous function. By the assumption of the lemma
		$\varphi_{+}'$ is Lebesgue integrable on $[t,T].$ Using \citep[Theorem 11]{Hagood_Thomson2006},
		the function $\varphi$ is absolutely continuous and
		\[
		\varphi(r)-\varphi(t)=\int^{r}_{t}\varphi'_{+}(s)\,ds=\int^{r}_{t}\cD_{s}F(\Theta_{t,s}\gamma)\,ds.
		\]
		This completes the proof of the lemma.
	\end{proof}
	
	\begin{cor}
		Let $F$ be a real-valued left-continuous, horizontally differentiable
		non-anticipative functional. If $\cD F\in\bB(\Upsilon([0,T],E))$, then $F$
		is horizontally absolutely continuous, and $\cD F$ is a weak horizontal
		derivative of $F$.
	\end{cor}
	
	\begin{proof}
		Let $K$ be the closure of $\{\gamma_{r}:r\in[0,t]\}$ in $E$. Since
		$\gamma$ is a \cdl function, the set $K$ is compact. Moreover,
		for every $s\in[t,T]$,
		\[
		\Theta_{t,s}\gamma\in D([0,s],K).
		\]
		Hence, by the boundedness preserving property of $\cD F$,
		\[
		\sup_{s\in[t,T]}\left|\cD_{s}F(\Theta_{t,s}\gamma)\right|<\infty.
		\]
		In particular, the map $s\mapsto\cD_{s}F(\Theta_{t,s}\gamma)$ is
		Lebesgue integrable on $[t,T]$. The claim therefore follows from
		\ref{lem:differentiability_and_absolute_continuity}.
	\end{proof}
	
	\begin{cor}
		Let $F$ be a real-valued, left-continuous, horizontally differentiable
		non-anticipative functional such that $\cD F\in\bB(\Upsilon([0,T],E))$.
		Assume that $\cD F$ is continuous at every fixed time. If $\cH F$
		is another weak horizontal derivative of $F$ which is continuous
		at every fixed time, then, for every $\vartheta\in D([0,T],E)$, 
		\[
		\cH_{t}F(\vartheta^{[0,t]})=\cD_{t}F(\vartheta^{[0,t]})
		\]
		for a.e. $t\in[0,T]$.
	\end{cor}
	
	\begin{proof}
		The corollary follows from Lemmas \ref{lem:differentiability_and_absolute_continuity}
		and \ref{lem:uniqueness_of_weak_derivatives}.
	\end{proof}

	\subsection{Vertical derivatives\label{subsec:vertical_derivatives}}
	
	We next introduce two notions of vertical differentiation. The first
	one is the usual vertical derivative of a non-anticipative functional
	\citep{Dupire:2019,Cont_Fournie2010,Cont_Fournie:2013,Bally_Cont:2016},
	which describes its sensitivity with respect to the terminal value
	of a stopped path. The second one will be used for functionals defined
	on the full path space and describes their sensitivity with respect
	to perturbations of the path after a given time. The latter notion
	will be needed below to study the vertical regularity of continuation
	functionals associated with the modified Ornstein-Uhlenbeck process.
	
	Let $(E,\norm{\cdot}_{E})$ be a Banach space. If $K_{1}\subseteq K_{2}$,
	we identify an $E$-valued function on $K_{1}$ with its extension
	to $K_{2}$ which is equal to zero on $K_{2}\backslash K_{1}$. In
	particular, we can write
	\[
	\gamma=\gamma^{[0,t)}+\gamma^{[t,T]}
	\]
	for all $\gamma\in D([0,T],E)$. For $(t,\gamma)\in\Upsilon([a,b],E)$,
	it will be convenient to write 
	\[
	F_{t}(\gamma^{[a,t)},x):=F_{t}\bigl(\gamma^{[a,t)}+x\I_{\{t\}}\bigr),\quad x\in E.
	\]
	In particular, 
	\[
	F_{t}(\gamma)=F_{t}(\gamma^{[a,t)},\gamma_{t}).
	\]
	For $t=a$, we use the convention 
	\[
	F_{a}(\gamma^{[a,a)},\gamma_{a}):=F_{a}(\gamma_{a}).
	\]
	
	\begin{defn}
		Let $0\le a<b\le T$ and $m\in\N$. For $(t,\gamma)\in\Upsilon([a,b],E)$,
		the $m$-th \textit{vertical derivative} $\nabla^{m}F_{t}(\gamma)$
		of a real-valued non-anticipative functional $F$ is defined as the
		$m$-th Fréchet derivative at $x=\gamma_{t}$ of the map 
		\[
		x\mapsto F_{t}(\gamma^{[a,t)},x).
		\]
	\end{defn}
	
	For $m\in\N$, we denote by $\Cf^{m}_{l}(\Upsilon([a,b],E))$ the
	space of continuous real-valued non-anticipative functionals which
	are $m$ times vertically differentiable and whose vertical derivatives
	$\nabla^{k}F$, $k\in[m]$, are bounded and left-continuous non-anticipative
	functionals. We also introduce the seminorm
	\[
	\bnorm F_{\Cf^{m}_{l}(\Upsilon([a,b],E))}:=\sum^{m}_{k=1}\sup_{t\in[a,b]}\sup_{\gamma\in D([a,t],E)}\bnorm{\nabla^{k}F_{t}(\gamma)}_{\cL_{k}(E)}.
	\]
	
	We will also need to differentiate functionals on the full path space
	with respect to perturbations of the path occurring after a prescribed
	time. This is motivated by the continuation functional introduced
	below: changing its terminal value at time $t$ changes the entire
	future Ornstein-Uhlenbeck trajectory, and the corresponding perturbation
	is propagated by the semigroup on $[t,T]$. It is therefore natural
	to differentiate the terminal path functional with respect to the
	whole future segment rather than only at a single time.
	\begin{defn}
		Let $F:D([a,b],E)\to\R$, $t\in[a,b]$, and $\gamma\in D([a,b],E)$.
		We say that $F$ is \textit{vertically path-differentiable} after
		$t$ at $\gamma$ if the map
		\[
		h\mapsto F(\gamma+h),\quad h\in D([t,b],E),
		\]
		where $h$ is extended by zero to $[a,b]$, is Fréchet differentiable
		at $0$. We denote its derivative by
		\[
		\nabla_{[t,b]}F(\gamma)\in\cL_{1}(D([t,b],E)).
		\]
		Equivalently,
		\[
		F(\gamma+h)-F(\gamma)=\nabla_{[t,b]}F(\gamma)[h]+o\left(\|h\|_{D([t,b],E)}\right)
		\]
		as $\|h\|_{D([t,b],E)}\to0$. We can similarly introduce $\nabla^{m}_{[t,b]}F(\gamma)$
		as a bounded multilinear operator on $D([t,b],E)$. 
	\end{defn}
	
	If $F$ is $m$ times Fréchet differentiable on $D([a,b],E)$, then
	for each $\ell\in[m]$ and $t\in[a,b]$, its $\ell$-th vertical path
	derivative after $t$ at $\gamma$ is defined by
	\[
	\nabla^{\ell}_{[t,b]}F(\gamma)[h_{1},\ldots,h_{\ell}]:=D^{\ell}F(\gamma)\left[h_{1},\ldots,h_{\ell}\right],\quad h_{i}\in D([t,b],E),\ i\in[\ell].
	\]
	Moreover,
	\[
	\sup_{t\in[a,b]}\bnorm{\nabla^{\ell}_{[t,b]}F(\gamma)}_{\cL_{\ell}(D([t,b],E))}=\bnorm{D^{\ell}F(\gamma)}_{\cL_{\ell}(D([a,b],E))}.
	\]
	
	For $0\le r\le t\le b\le T$, we define the constant extension operator
	\[
	\iota^{b}_{r,t}:D([t,b],E)\to D([r,b],E)
	\]
	by
	\[
	(\iota^{b}_{r,t}h)_{u}:=h_{u\vee t},\quad u\in[r,b].
	\]
	Clearly, $\iota^{b}_{r,t}$ is a bounded linear operator with norm
	one. We also set $\iota_{r,t}:=\iota^{T}_{r,t}$.
	\begin{defn}
		Let $0\le a<b\le T$ and $F:D([a,b],E)\to\R$.
		\begin{enumerate}
			\item[(i)]  Let $m\in\N$. We say that the family $\nabla^{m}_{[t,b]}F$, $t\in[a,b]$,
			is \textit{left-continuous} if, for every $t_{n}\uparrow t$ in $[a,b]$
			and every $\gamma^{n},\gamma\in D([a,b],E)$ such that
			\[
			\bnorm{\gamma^{n}-\gamma}_{D([a,b],E)}\to0,
			\]
			one has
			\[
			\nabla^{m}_{[t_{n},b]}F(\gamma^{n})\circ\left(\iota^{b}_{t_{n},t}\right)^{\otimes m}\to\nabla^{m}_{[t,b]}F(\gamma)
			\]
			in $\cL_{m}(D([t,b],E))$.
			\item[(ii)]  The family of vertical path derivatives $\nabla^{m}_{[t,b]}F,$
			$t\in[a,b]$, is said to be bounded if 
			\[
			\sup_{\gamma\in D([a,b],E)}\sup_{t\in[a,b]}\bnorm{\nabla^{m}_{[t,b]}F(\gamma)}_{\cL_{m}(D([t,b],E))}<\infty.
			\]
		\end{enumerate}
		The set of measurable continuous functions from $D([a,b],E)$ to $\R$ with bounded
		left-continuous vertical path derivatives up to the order $m\in\N$
		will be denoted by $\cP^{m}(D([a,b],E))$. We define the seminorm
		on $\cP^{m}(D([a,b],E))$ by
		\[
		\bnorm F_{\cP^{m}(D([a,b],E))}
		:=
		\sum^{m}_{k=1}
		\sup_{\gamma\in D([a,b],E)}
		\sup_{t\in[a,b]}
		\bnorm{
			\nabla^{k}_{[t,b]}F(\gamma)
		}_{\cL_{k}(D([t,b],E))}.
		\]
	\end{defn}
	
	We also denote by $\cP^{m}(\Upsilon([a,b],E))$ the space of non-anticipative
	functionals $F$ on $\Upsilon([a,b],E)$ such that
	\[
	F_{t}\in\cP^{m}(D([a,t],E))\quad\text{for every }t\in[a,b],
	\]
	and
	\[
	\norm F_{\cP^{m}(\Upsilon([a,b],E))}:=\sup_{t\in[a,b]}\norm{F_{t}}_{\cP^{m}(D([a,t],E))}<\infty.
	\]
	
	For an interval $K\subseteq[0,T]$ we define
	\[
	\jmath_{K}:E\to D(K,E)
	\]
	to be the constant-path embedding,
	\[
	(\jmath_{K}h)_{u}:=h,\quad u\in K.
	\]
	
	\begin{lem}
		\label{lem:regularity_of_map_G} Let $m\in\N$, $0\le t<b\le T$,
		$\gamma\in D([0,t],E)$, and $F\in\cP^{m}(D([0,b],E))$. For $(s,\vartheta)\in\Upsilon([t,b],E)$,
		define
		\[
		G_{s}(\vartheta):=F\left(\gamma^{[0,t)}+\Theta^{t}_{s,b}\vartheta\right).
		\]
		Then
		\[
		G\in\Cf^{m}_{l}(\Upsilon([t,b],E))
		\]
		and $G$ is horizontally differentiable with
		\[
		\cD_{s}G(\vartheta)=0.
		\]
		Moreover, for every $\ell\in[m]$,
		\begin{equation}
			\nabla^{\ell}G_{s}(\vartheta)[h_{1},\ldots,h_{\ell}]=\nabla^{\ell}_{[s,b]}F\left(\gamma^{[0,t)}+\Theta^{t}_{s,b}\vartheta\right)[\jmath_{[s,b]}h_{1},\ldots,\jmath_{[s,b]}h_{\ell}].\label{eq:derivative_of_G_via_F}
		\end{equation}
		Furthermore,
		\begin{equation}
			\bnorm G_{\Cf^{m}_{l}(\Upsilon([t,b],E))}\le\bnorm F_{\cP^{m}(D([0,b],E))}.\label{eq:control_of_G_via_F}
		\end{equation}
	\end{lem}
	
	\begin{proof}
		We first verify the continuity of $G$. Let $s_{n}\to s$ in $[t,b]$
		and let $\vartheta^{n}\in D([t,s_{n}],E)$ converge to $\vartheta\in D([t,s],E)$
		in $\Upsilon([t,b],E)$. Then
		\begin{align*}
			\bnorm{\Theta^{t}_{s_{n},b}\vartheta^{n}-\Theta^{t}_{s,b}\vartheta}_{D([t,b],E)} & =\sup_{u\in[t,b]}\bnorm{\vartheta^{n}_{u\wedge s_{n}}-\vartheta_{u\wedge s}}_{E}\\
			& \le d_{\Upsilon([t,b],E)}(\vartheta^{n},\vartheta)\to0.
		\end{align*}
		Consequently,
		\[
		\gamma^{[0,t)}+\Theta^{t}_{s_{n},b}\vartheta^{n}\to\gamma^{[0,t)}+\Theta^{t}_{s,b}\vartheta
		\]
		in $D([0,b],E)$. Since $F$ is continuous, it follows that
		\[
		G_{s_{n}}(\vartheta^{n})=F\left(\gamma^{[0,t)}+\Theta^{t}_{s_{n},b}\vartheta^{n}\right)\to F\left(\gamma^{[0,t)}+\Theta^{t}_{s,b}\vartheta\right)=G_{s}(\vartheta).
		\]
		Thus $G$ is continuous.
		
		By construction, 
		\[
		\Theta^{t}_{u,b}\Theta^{t}_{s,u}\vartheta=\Theta^{t}_{s,b}\vartheta,\quad t\le s\le u\le b,
		\]
		and hence 
		\[
		G_{u}(\Theta^{t}_{s,u}\vartheta)=G_{s}(\vartheta).
		\]
		Therefore 
		\[
		\cD_{s}G(\vartheta)=\lim_{u\to s+}\frac{G_{u}(\Theta^{t}_{s,u}\vartheta)-G_{s}(\vartheta)}{u-s}=0.
		\]
		
		We next consider the vertical derivatives of $G$. Fix $(s,\vartheta)\in\Upsilon([t,b],E)$
		and set
		\[
		\theta:=\gamma^{[0,t)}+\Theta^{t}_{s,b}(\vartheta).
		\]
		For $h\in E$, we have
		\[
		\Theta^{t}_{s,b}(\vartheta+h\I_{\{s\}})=\Theta^{t}_{s,b}\vartheta+\jmath_{[s,b]}h.
		\]
		Hence, by the vertical path differentiability of $F$ after $s$,
		$G_{s}$ is vertically differentiable and
		\[
		\nabla G_{s}(\vartheta)[h]=\nabla_{[s,b]}F(\theta)\left[\jmath_{[s,b]}h\right].
		\]
		More generally, for $\ell\in[m]$,
		\[
		\nabla^{\ell}G_{s}(\vartheta)[h_{1},\ldots,h_{\ell}]=\nabla^{\ell}_{[s,b]}F\left(\gamma^{[0,t)}+\Theta^{t}_{s,b}\vartheta\right)\left[\jmath_{[s,b]}h_{1},\ldots,\jmath_{[s,b]}h_{\ell}\right].
		\]
		
		We next verify the left-continuity of the vertical derivatives of
		$G$. Let $s_{n}\uparrow s$ in $[t,b]$ and let $\vartheta^{n}\in D([t,s_{n}],E)$
		converge to $\vartheta\in D([t,s],E)$ in $\Upsilon([t,b],E)$. Set
		\[
		\theta_{n}:=\gamma^{[0,t)}+\Theta^{t}_{s_{n},b}\vartheta^{n},\quad\theta:=\gamma^{[0,t)}+\Theta^{t}_{s,b}\vartheta.
		\]
		As above,
		\[
		\theta_{n}\to\theta\quad\text{in }D([0,b],E).
		\]
		Recall that
		\[
		\nabla G_{s_{n}}(\vartheta^{n})=\nabla_{[s_{n},b]}F(\theta_{n})\circ\jmath_{[s_{n},b]}.
		\]
		Since
		\[
		\jmath_{[s_{n},b]}=\iota^{b}_{s_{n},s}\circ\jmath_{[s,b]},
		\]
		we obtain
		\[
		\bnorm{\nabla G_{s_{n}}(\vartheta^{n})-\nabla G_{s}(\vartheta)}_{\cL_{1}(E)}\le\bnorm{\nabla_{[s_{n},b]}F(\theta_{n})\circ\iota_{s_{n},s}^b-\nabla_{[s,b]}F(\theta)}_{\cL_{1}(D([s,b],E))}.
		\]
		The right-hand side converges to zero by the left-continuity of the
		vertical path derivative of $F$. Hence $\nabla G$ is left-continuous.
		Similarly, one can show that $\nabla^{\ell}G$ is left-continuous
		for each $\ell\in[m]$. 
		
		Finally, since $\bnorm{\jmath_{[s,b]}}_{\cL_{1}(E,D([s,b],E))}=1$
		for every $s\in[0,b]$, (\ref{eq:derivative_of_G_via_F}) yields (\ref{eq:control_of_G_via_F}).
	\end{proof}
	
	The next statement is a chain rule for the weak horizontal derivative
	under a deterministic absolutely continuous translation of the path.
	\begin{lem}
		\label{lem:absolute_continuity_for_translation_of_path}Let $F$ be
		a horizontally absolutely continuous non-anticipative functional on
		$\Upsilon([0,T],E)$ such that $F\in\cP^{1}(\Upsilon([0,T],E)).$
		Assume that $\cH F$ is left-continuous and boundedness preserving.
		Let $c\in\R$, and let $\varrho:[0,T]\to E$ be absolutely continuous,
		with 
		\[
		\varrho_{t}=\varrho_{0}+\int^{t}_{0}\dot{\varrho}_{s}\,ds,\quad t\in[0,T],
		\]
		for some bounded Bochner integrable map $s\mapsto\dot{\varrho}_{s}$.
		Define the non-anticipative functional 
		\[
		F^{c,\varrho}_{t}(\gamma):=F_{t}\left(c\gamma+\varrho^{[0,t]}\right),\quad(t,\gamma)\in\Upsilon([0,T],E).
		\]
		Then $F^{c,\varrho}$ is horizontally absolutely continuous, and 
		\[
		\cH_{t}F^{c,\varrho}(\gamma)=\cH_{t}F\left(c\gamma+\varrho^{[0,t]}\right)+\nabla F_{t}\left(c\gamma+\varrho^{[0,t]}\right)\left[\dot{\varrho}_{t}\right]
		\]
		is a weak horizontal derivative of $F^{c,\varrho}$.
	\end{lem}
	
	\begin{proof}
		Fix $0\le t<T$ and $\gamma\in D([0,t],E)$, and define 
		\[
		f(r):=F^{c,\varrho}_{r}(\Theta_{t,r}\gamma)=F_{r}\left(c\Theta_{t,r}\gamma+\varrho^{[0,r]}\right),\quad r\in[t,T].
		\]
		We first show that $f$ is absolutely continuous.
		
		Let $t\le u<r\le T$ and set 
		\[
		\vartheta^{u}:=c\Theta_{t,u}\gamma+\varrho^{[0,u]}.
		\]
		Then 
		\begin{align}
			f(r)-f(u) & =F_{r}\left(c\Theta_{t,r}\gamma+\varrho^{[0,r]}\right)-F_{r}\left(\Theta_{u,r}\vartheta^{u}\right)\nonumber \\
			& +F_{r}\left(\Theta_{u,r}\vartheta^{u}\right)-F_{u}(\vartheta^{u}).\label{eq:decomposition_horizontal_translation}
		\end{align}
		By the horizontal absolute continuity of $F$, 
		\begin{equation}
			F_{r}\left(\Theta_{u,r}\vartheta^{u}\right)-F_{u}(\vartheta^{u})=\int^{r}_{u}\cH_{s}F\left(\Theta_{u,s}\vartheta^{u}\right)ds.\label{eq:horizontal_part_translation}
		\end{equation}
		
		We next consider the first difference on the right-hand side of (\ref{eq:decomposition_horizontal_translation}).
		Define 
		\[
		G^{r,u}_{s}(\vartheta):=F_{r}\left(c\Theta_{t,r}\gamma+\varrho^{[0,u)}+\Theta^{u}_{s,r}\vartheta\right),\quad(s,\vartheta)\in\Upsilon([u,r],E).
		\]
		Then 
		\[
		G^{r,u}_{u}(\varrho_{u})=F_{r}\left(\Theta_{u,r}\vartheta^{u}\right)
		\]
		and 
		\[
		G^{r,u}_{r}(\varrho^{[u,r]})=F_{r}\left(c\Theta_{t,r}\gamma+\varrho^{[0,r]}\right).
		\]
		By Lemma~\ref{lem:regularity_of_map_G}, $G^{r,u}$ is horizontally
		differentiable with 
		\[
		\cD_{s}G^{r,u}(\vartheta)=0,
		\]
		and its vertical derivative is given by 
		\[
		\nabla G^{r,u}_{s}(\vartheta)[h]=\nabla_{[s,r]}F_{r}\left(c\Theta_{t,r}\gamma+\varrho^{[0,u)}+\Theta^{u}_{s,r}\vartheta\right)\left[\jmath_{[s,r]}h\right],\quad h\in E.
		\]
		Take partitions $u=t_0<\cdots<t_N=r$ with mesh tending to zero, and approximate $\varrho$ by the step paths taking the value $\varrho_{t_{i+1}}$ on $[t_i,t_{i+1})$
		and $\varrho_r$ at $r$. Since $G^{r,u}$ is unchanged by constant extension,
		a telescoping sum reduces its increment along each such path to vertical increments at the times $t_i$. Applying the fundamental theorem of calculus to these increments and using
		$\varrho_{t_{i+1}}-\varrho_{t_i} =\int_{t_i}^{t_{i+1}}\dot{\varrho}_s\,ds$, we may pass to the limit by dominated convergence, using the boundedness and left-continuity of $\nabla G^{r,u}$. Consequently,
		\begin{equation}
			F_{r}\left(c\Theta_{t,r}\gamma+\varrho^{[0,r]}\right)-F_{r}\left(\Theta_{u,r}\vartheta^{u}\right)=\int^{r}_{u}\nabla G^{r,u}_{s}(\varrho^{[u,s]})[\dot{\varrho}_{s}]\,ds=\int^{r}_{u}A^{r}_{s}[\dot{\varrho}_{s}]\,ds,\label{eq:vertical_part_translation}
		\end{equation}
		where 
		\[
		A^{r}_{s}[h]:=\nabla_{[s,r]}F_{r}\left(c\Theta_{t,r}\gamma+\varrho^{[0,s)}+\jmath_{[s,r]}\varrho_{s}\right)\left[\jmath_{[s,r]}h\right],\quad h\in E.
		\]
		In particular, $A^{r}_{s}$ is independent of $u$.
		
		Combining (\ref{eq:decomposition_horizontal_translation}), (\ref{eq:horizontal_part_translation}),
		and (\ref{eq:vertical_part_translation}), we obtain 
		\begin{equation}
			f(r)-f(u)=\int^{r}_{u}\cH_{s}F\left(\Theta_{u,s}\vartheta^{u}\right)\,ds+\int^{r}_{u}A^{r}_{s}[\dot{\varrho}_{s}]\,ds.\label{eq:representation_translation_f}
		\end{equation}

		We next show that $f$ is Lipschitz continuous. Since $\gamma$ is
		c\`adl\`ag, its range has compact closure in $E$, while the image
		$\varrho([0,T])$ is compact by continuity. Hence
		\[
		K:=c\,\overline{\{\gamma_{s}:s\in[0,t]\}}+\varrho([0,T])
		\]
		is a compact subset of $E$. Every path
		$\Theta_{u,s}\vartheta^{u}$, $t\le u\le s\le T$, takes values in $K$.
		Since $\cH F$ is boundedness preserving, there exists a constant
		$C_{c,\gamma,\varrho,F}$ such that
		\[
		\sup_{t\le u\le s\le T}
		\left|\cH_{s}F\left(\Theta_{u,s}\vartheta^{u}\right)\right|
		\le C_{c,\gamma,\varrho,F}.
		\]
		Moreover, the boundedness of the vertical path derivative of $F$
		and of $\dot{\varrho}$ yields 
		\[
		\left|A^{r}_{s}[\dot{\varrho}_{s}]\right|\le\|F\|_{\cP^{1}(\Upsilon([0,T],E))}\sup_{v\in[0,T]}\|\dot{\varrho}_{v}\|_{E}.
		\]
		Consequently, by (\ref{eq:representation_translation_f}), 
		\[
		|f(r)-f(u)|\le C_{c,\gamma,\varrho,F,\dot{\varrho}}(r-u),\quad t\le u<r\le T.
		\]
		Thus $f$ is Lipschitz continuous and hence absolutely continuous.
		Consequently, $f'(s)$ exists for a.e. $s\in[t,T]$ and 
		\begin{equation}
			f(r)-f(u)=\int^{r}_{u}f'(s)\,ds,\quad t\le u\le r\le T.\label{eq:AC_translation_f}
		\end{equation}
		
		We finally identify $f'$. Fix $r\in(t,T)$ such that $f$ is differentiable
		at $r$ and $r$ is a Lebesgue point of $\dot{\varrho}$. Such $r$
		form a set of full Lebesgue measure by the absolute continuity of
		$f$ and the Lebesgue differentiation theorem for Bochner integrals
		\citep[Theorem 2.3.4]{Veraar:2016}.
		By (\ref{eq:representation_translation_f}),
		\begin{equation}
			f'(r)=\lim_{u\to r-}\frac{1}{r-u}\int^{r}_{u}\cH_{s}F\left(\Theta_{u,s}\vartheta^{u}\right)\,ds+\lim_{u\to r-}\frac{1}{r-u}\int^{r}_{u}A^{r}_{s}[\dot{\varrho}_{s}]\,ds.\label{eq:derivative_translation_f}
		\end{equation}
		Since $\varrho$ is continuous, 
		\[
		\sup_{s\in[u,r]}d_{\Upsilon([0,T],E)}\left((s,\Theta_{u,s}\vartheta^{u}),(r,\vartheta^{r})\right)\to0\quad\text{as }u\to r-.
		\]
		Hence, by the left-continuity of $\cH F$, 
		\[
		\sup_{s\in[u,r]}\left|\cH_{s}F\left(\Theta_{u,s}\vartheta^{u}\right)-\cH_{r}F(\vartheta^{r})\right|\to0\quad\text{as }u\to r-.
		\]
		Therefore, 
		\begin{equation}
			\lim_{u\to r-}\frac{1}{r-u}\int^{r}_{u}\cH_{s}F\left(\Theta_{u,s}\vartheta^{u}\right)\,ds=\cH_{r}F(\vartheta^{r}).\label{eq:limit_horizontal_translation}
		\end{equation}
		
		For the second term, set 
		\[
		A_{r}[h]:=\nabla F_{r}(\vartheta^{r})[h],\quad h\in E.
		\]
		By the left-continuity of the vertical path derivatives of $F$, 
		\[
		\sup_{s\in[u,r]}\|A^{r}_{s}-A_{r}\|_{\cL(E,\R)}\to0\quad\text{as }u\to r-.
		\]
		Consequently, 
		\begin{align*}
			& \left|\frac{1}{r-u}\int^{r}_{u}A^{r}_{s}[\dot{\varrho}_{s}]\,ds-A_{r}\left[\frac{\varrho_{r}-\varrho_{u}}{r-u}\right]\right|\\
			& \qquad\le\sup_{s\in[u,r]}\|A^{r}_{s}-A_{r}\|_{\cL(E,\R)}\frac{1}{r-u}\int^{r}_{u}\|\dot{\varrho}_{s}\|_{E}\,ds\longrightarrow0.
		\end{align*}
		Moreover, since $r$ is a Lebesgue point of $\dot{\varrho}$ and
		\[
		\varrho_{r}-\varrho_{u}=\int^{r}_{u}\dot{\varrho}_{s}\,ds,
		\]
		it follows that
		\[
		\frac{\varrho_{r}-\varrho_{u}}{r-u}=\frac{1}{r-u}\int^{r}_{u}\dot{\varrho}_{s}\,ds\to\dot{\varrho}_{r}\quad\text{as }u\to r-.
		\]
		Hence 
		\begin{equation}
			\lim_{u\to r-}\frac{1}{r-u}\int^{r}_{u}A^{r}_{s}[\dot{\varrho}_{s}]\,ds=\nabla F_{r}(\vartheta^{r})[\dot{\varrho}_{r}].\label{eq:limit_vertical_translation}
		\end{equation}
		Combining (\ref{eq:derivative_translation_f}), (\ref{eq:limit_horizontal_translation}),
		and (\ref{eq:limit_vertical_translation}), we obtain 
		\[
		f'(r)=\cH_{r}F(\vartheta^{r})+\nabla F_{r}(\vartheta^{r})[\dot{\varrho}_{r}]
		\]
		for a.e. $r\in[t,T]$.
		
		Since $f$ is absolutely continuous, (\ref{eq:AC_translation_f})
		therefore yields 
		\begin{align*}
			F^{c,\varrho}_{r}(\Theta_{t,r}\gamma)-F^{c,\varrho}_{t}(\gamma) & =\int^{r}_{t}\cH_{s}F\left(c\Theta_{t,s}\gamma+\varrho^{[0,s]}\right)\,ds\\
			&+\int^{r}_{t}\nabla F_{s}\left(c\Theta_{t,s}\gamma+\varrho^{[0,s]}\right)[\dot{\varrho}_{s}]\,ds.
		\end{align*}
		Thus $F^{c,\varrho}$ is horizontally absolutely continuous, with
		weak horizontal derivative 
		\[
		\cH_{t}F^{c,\varrho}(\gamma)=\cH_{t}F\left(c\gamma+\varrho^{[0,t]}\right)+\nabla F_{t}\left(c\gamma+\varrho^{[0,t]}\right)[\dot{\varrho}_{t}].
		\]
		The proof of the statement is complete.
	\end{proof}

	\subsection{Martingale problem for non-anticipative functionals \label{subsec:martingale_problem_for_non-anticipative_functionals}}
	
	We next derive a martingale problem for horizontally absolutely continuous
	non-anticipative functionals evaluated along a Markov jump process.
	This will provide the path-dependent analogue of the usual martingale
	problem and will later allow us to apply the backward Kolmogorov equation
	to the microscopic fluctuation field. The main point is that only
	horizontal absolute continuity, rather than strong horizontal differentiability,
	will be required.
	
	Let $(E,d_{E})$ be a Polish space, let $\lambda:E\to[0,\infty)$
	be a bounded Borel measurable function, and let $\mu$ be a transition
	probability kernel on $E$. We consider a càdlàg Markov jump process
	$(Y_{t})_{t\ge0}$ with generator
	\begin{equation}
		\cL f(x)=\lambda(x)\int_{E}\left(f(y)-f(x)\right)\mu(x,dy).\label{eq:generator_of_jump_process}
	\end{equation}
	In particular, for every bounded Borel measurable function $f:E\to\R$,
	the process
	\[
	f(Y_{t})-f(Y_{0})-\int^{t}_{0}\cL f(Y_{s})ds,\quad t\ge0,
	\]
	is an $(\F^{Y}_{t})$-martingale.
	\begin{rem}
		\label{rem:assumption_to_probability_kernel} Without loss of generality,
		we assume that $\mu(x,\{x\})=0$ for all $x\in E$. Indeed, setting
		$p(x):=\mu(x,\{x\})$, one may replace $\lambda(x)$ by $\widetilde{\lambda}(x):=\lambda(x)(1-p(x))$
		and, whenever $p(x)<1$, $\mu(x,\cdot)$ by
		\[
		\widetilde{\mu}(x,A):=\frac{\mu(x,A\setminus\{x\})}{1-p(x)}.
		\]
		If $p(x)=1$, then $\widetilde{\lambda}(x)=0$ and $\widetilde{\mu}(x,\cdot)$
		may be chosen arbitrarily. This leaves the generator $\cL$ unchanged
		and preserves boundedness of the jump rate.
	\end{rem}
	
	The following proposition gives the corresponding path-dependent martingale
	problem.
	\begin{prop}
		\label{prop:matingale_proplem_for_nonanticipative_functional} Let
		$(Y_{t})_{t\ge0}$ be a Markov jump process with generator (\ref{eq:generator_of_jump_process})
		and $T>0$. Let $F$ also be a left-continuous, horizontally absolutely
		continuous non-anticipative functional such that
		\begin{equation}
			\sup_{(t,\gamma)\in\Upsilon([0,T],E)}\int_{E}\left|F_{t}(\gamma^{[0,t)},y)-F_{t}(\gamma)\right|\mu(\gamma_{t},dy)<\infty.\label{eq:extra_asusmption_in_martingale_problem}
		\end{equation}
		Then
		
		\[
		F_{t}(Y^{[0,t]})-F_{0}(Y_{0})-\int^{t}_{0}\left[\cH_{s}F(Y^{[0,s]})+\cL F_{s}(Y^{[0,s)},\cdot)(Y_{s})\right]ds,\quad t\in[0,T],
		\]
		is a local $(\F^{Y}_{t})$-martingale.
	\end{prop}
	
	\begin{proof}
		For non-anticipative functionals with stronger horizontal regularity
		and vertical differentiability, a related statement follows from \citep[Proposition 6]{Cont_Fournie2010}.
		Here we give a direct proof adapted to horizontally absolutely continuous
		functionals and Markov jump processes with bounded jump rates.
		
		Since $\lambda$ is bounded, the process $Y$ has only finitely many
		jumps on $[0,T]$ almost surely. We first obtain a pathwise decomposition
		of $F_{t}(Y^{[0,t]})$ into its horizontal evolution between jumps
		and its increments at the jump times. Set $\tau_{0}:=0$ and, recursively,
		\[
		\tau_{i+1}:=\inf\{t>\tau_{i}:Y_{t}\neq Y_{\tau_{i}}\},
		\]
		with the usual convention $\inf\varnothing=\infty$. For fixed $t\in[0,T]$,
		let $k\in\N_{0}$ be such that
		\[
		\tau_{k}\le t<\tau_{k+1}.
		\]
		Between two consecutive jump times the process $Y$ is constant. Hence,
		for $r\in[\tau_{i-1},\tau_{i})$ 
		\[
		Y^{[0,r]}=\Theta_{\tau_{i-1},r}Y^{[0,\tau_{i-1}]}.
		\]
		By horizontal absolute continuity,
		\[
		F_{\tau_{i}}\left(\Theta_{\tau_{i-1},\tau_{i}}Y^{[0,\tau_{i-1}]}\right)-F_{\tau_{i-1}}(Y^{[0,\tau_{i-1}]})=\int^{\tau_{i}}_{\tau_{i-1}}\cH_{s}F(Y^{[0,s]})\,ds.
		\]
		Since $Y$ is constant on $[\tau_{i-1},\tau_{i})$, we get
		\[
		F_{\tau_{i}}\left(\Theta_{\tau_{i-1},\tau_{i}}Y^{[0,\tau_{i-1}]}\right)=F_{\tau_{i}}(Y^{[0,\tau_{i})},Y_{\tau_{i}-}).
		\]
		Consequently,
		\[
		F_{\tau_{i}}\bigl(Y^{[0,\tau_{i})},Y_{\tau_{i}-}\bigr)-F_{\tau_{i-1}}(Y^{[0,\tau_{i-1}]})=\int^{\tau_{i}}_{\tau_{i-1}}\cH_{s}F(Y^{[0,s]})ds.
		\]
		Similarly, 
		\[
		F_{t}(Y^{[0,t]})-F_{\tau_{k}}(Y^{[0,\tau_{k}]})=\int^{t}_{\tau_{k}}\cH_{s}F\bigl(Y^{[0,s]}\bigr)ds.
		\]
		Summing the horizontal increments and the jump increments $F_{\tau_{i}}(Y^{[0,\tau_{i}]})-F_{\tau_{i}}\bigl(Y^{[0,\tau_{i})},Y_{\tau_{i}-}\bigr)$,
		we obtain
		\begin{align*}
			F_{t}(Y^{[0,t]})-F_{0}(Y_{0})= & \sum^{k}_{i=1}\left[F_{\tau_{i}}\bigl(Y^{[0,\tau_{i})},Y_{\tau_{i}-}\bigr)-F_{\tau_{i-1}}(Y^{[0,\tau_{i-1}]})\right]\\
			& +\sum^{k}_{i=1}\left[F_{\tau_{i}}(Y^{[0,\tau_{i}]})-F_{\tau_{i}}\bigl(Y^{[0,\tau_{i})},Y_{\tau_{i}-}\bigr)\right]\\
			& +F_{t}(Y^{[0,t]})-F_{\tau_{k}}(Y^{[0,\tau_{k}]}).
		\end{align*}
		Thus,
		\begin{align*}
			F_{t}(Y^{[0,t]})-F_{0}(Y_{0})= & \int^{t}_{0}\cH_{s}F(Y^{[0,s]})\,ds+\sum_{0<s\le t}\left[F_{s}(Y^{[0,s]})-F_{s}(Y^{[0,s)},Y_{s-})\right].
		\end{align*}
		
		We next compensate the jump contribution. Define the integer-valued
		random measure 
		\[
		N(dt,dy):=\sum_{r:\,\Delta Y_{r}\neq0}\delta_{(r,Y_{r})}(dt,dy).
		\]
		Then, for every $t\in[0,T]$, 
		\[
		\sum_{0<s\le t}\left[F_{s}(Y^{[0,s]})-F_{s}(Y^{[0,s)},Y_{s-})\right]=\int_{(0,t]\times E}\Psi_{s}(y)N(ds,dy),
		\]
		where 
		\begin{equation}
			\Psi_{s}(y):=F_{s}(Y^{[0,s)},y)-F_{s}(Y^{[0,s)},Y_{s-}).\label{eq:predictable_process_Psi}
		\end{equation}
		By the canonical measurability of $F$, the map
		$(s,\omega,y)\mapsto\Psi_s(\omega,y)$ is
		$\mathcal P\otimes\mathcal B(E)$-measurable.
		
		We next identify the predictable compensator of $N$. By the standard
		construction of the jump process (see, e.g., \citep[Section 4.2]{Ethier:1986}),
		conditionally on $\F^{Y}_{\tau_{i}}$, the joint distribution of the
		holding time $\tau_{i+1}-\tau_{i}$ and the next state $Y_{\tau_{i+1}}$
		is
		\[
		G_{i}(du,dy)=\lambda(Y_{\tau_{i}})e^{-\lambda(Y_{\tau_{i}})u}\mu(Y_{\tau_{i}},dy)\,du,\quad u>0.
		\]
		Here we use the convention introduced in Remark \ref{rem:assumption_to_probability_kernel}
		that $\mu(x,\{x\})=0$. The corresponding survival function of the
		holding time is
		
		\[
		H_{i}([u,\infty))=e^{-\lambda(Y_{\tau_{i}})u},\quad u\ge0.
		\]
		Therefore, by \citep[Proposition 3.1]{Jacod1975}, the predictable
		compensator of $N$ is 
		\begin{align*}
			\nu(dt,dy) & =\sum^{\infty}_{i=0}\frac{G_{i}(d(t-\tau_{i}),dy)}{H_{i}([t-\tau_{i},\infty))}\I_{\{\tau_{i}<t\le\tau_{i+1}\}}\\
			& =\sum^{\infty}_{i=0}\lambda(Y_{\tau_{i}})\I_{\{\tau_{i}<t\le\tau_{i+1}\}}\mu(Y_{\tau_{i}},dy)\,dt\\
			& =\lambda(Y_{t-})\mu(Y_{t-},dy)\,dt.
		\end{align*}
		
		Consequently, 
		\[
		M_{t}:=\int_{(0,t]\times E}\Psi_{s}(y)\,(N-\nu)(ds,dy),\quad t\in[0,T],
		\]
		is a local $(\F^{Y}_{t})$-martingale. The integrability condition
		in the statement, together with the boundedness of $\lambda$, guarantees
		that the compensator integral is finite on $[0,T]$.
		
		Moreover, 
		\begin{align*}
			\int_{(0,t]\times E}\Psi_{s}(y)\,\nu(ds,dy) & =\int^{t}_{0}\lambda(Y_{s-})\int_{E}\left[F_{s}(Y^{[0,s)},y)-F_{s}(Y^{[0,s)},Y_{s-})\right]\mu(Y_{s-},dy)\,ds\\
			& =\int^{t}_{0}\cL F_{s}(Y^{[0,s)},\cdot)(Y_{s-})\,ds=\int^{t}_{0}\cL F_{s}(Y^{[0,s)},\cdot)(Y_{s})\,ds,
		\end{align*}
		where in the last equality we used that $Y_{s-}=Y_{s}$ for a.e. $s\in[0,T]$.
		Combining this identity with the pathwise decomposition above, we
		obtain
		\[
		F_{t}(Y^{[0,t]})-F_{0}(Y_{0})-\int^{t}_{0}\left[\cH_{s}F(Y^{[0,s]})+\cL F_{s}(Y^{[0,s)},\cdot)(Y_{s})\right]ds=M_{t},
		\]
		and the claim follows.
	\end{proof}

	\subsection{Regularity of the continuation functional}
	
	We now introduce the continuation functional associated with the modified
	Ornstein--Uhlenbeck dynamics. It will later be shown to solve the
	corresponding path-dependent backward Kolmogorov equation. The aim
	of this subsection is to establish the regularity properties needed
	for this identification.
	
	Throughout the remainder of Section \ref{sec:non-anticipative_functionals},
	we fix $n\ge3$ and usually omit it from the notation. However, we
	keep track of the dependence on $n$ in the estimates. As in Section
	\ref{subsec:Modified-Ornstein=002013Uhlenbeck-proc}, we assume that
	$m_{\Delta},m_{\partial}\in\N$, with $m_{\Delta}$ odd, and that
	\[
	\Phi=(\Phi_{j})_{j\in[d]}:[0,\infty)\to L^{2}(\Td;\R^{d})
	\]
	is measurable and locally bounded. Recall that, for $s\ge0$ and $g\in\H[-I]$,
	we denote by $\bigl(\xi_{s,t}(g)\bigr)_{t\ge s}$ the solution to
	the SPDE (\ref{eq:SPDE_for_OU_process}) starting from $g$ at time
	$s$, as defined in (\ref{eq:OU_process}).
	
	Let $T>0$ be fixed and let $F:D([0,T],\H[-I])\to\R$ be measurable.
	We define the \textit{continuation functional} associated with $F$
	by 
	\begin{equation}
		U_{t}(\gamma):=\E F\left(\gamma^{[0,t)}+\xi^{[t,T]}_{t,\cdot}(\gamma_{t})\right),\quad(t,\gamma)\in\Upsilon([0,T],\H[-I]),\label{eq:definition_of_U}
	\end{equation}
	whenever the expectation exists. Here, as before, the two functions
	appearing in the sum are extended by zero outside their respective
	domains. In particular, 
	\[
	U_{T}(\gamma)=F(\gamma),\quad\gamma\in D([0,T],\H[-I]).
	\]
	
	We first establish the continuity of $U$ and then investigate its
	vertical regularity.
	\begin{lem}
		\label{lem:continuity_of_U}Let $I>\frac{d}{2}+m_{\partial}$ and
		$F\in\cP^{1}(D([0,T],\H[-I]))$. Then the functional $U$ defined
		by (\ref{eq:definition_of_U}) is well-defined and continuous non-anticipative
		functional on $\Upsilon([0,T],\H[-I])$. 
	\end{lem}
	
	\begin{proof}
		We first show that $U_{t}(\gamma)$ is well-defined for every $(t,\gamma)\in\Upsilon([0,T],\H[-I])$.
		Since $F\in\cP^{1}(D([0,T],\H[-I]))$, its Fréchet derivative $\nabla_{[0,T]}F$
		is bounded. Hence, by the mean-value theorem, 
		\begin{align*}
			\E\left|F\left(\gamma^{[0,t)}+\xi^{[t,T]}_{t,\cdot}(\gamma_{t})\right)\right| & \le\E\left|F\left(\gamma^{[0,t)}+\xi^{[t,T]}_{t,\cdot}(\gamma_{t})\right)-F\left(\gamma^{[0,t)}\right)\right|+\left|F\left(\gamma^{[0,t)}\right)\right|\\
			& \le\bnorm F_{\cP^{1}(D([0,T],\H[-I]))}\E\left[\bnorm{\xi^{[t,T]}_{t,\cdot}(\gamma_{t})}_{D([0,T],\H[-I])}\right]+\left|F\left(\gamma^{[0,t)}\right)\right|\\
			& =\bnorm F_{\cP^{1}(D([0,T],\H[-I]))}\E\left[\sup_{s\in[t,T]}\bnorm{\xi_{t,s}(\gamma_{t})}_{\H[-I]}\right]+\left|F\left(\gamma^{[0,t)}\right)\right|<\infty,
		\end{align*}
		where the last inequality follows from Lemma \ref{lem:stochastic_convolution}.
		Thus, $U$ is well-defined. 
		
		The measurability of $U$ follows from Fubini's theorem. Thus, $U$
		is a non-anticipative functional.
		
		We next prove the continuity of $U$. Let $(s,\gamma),(r,\vartheta)\in\Upsilon([0,T],\H[-I])$
		and assume, without loss of generality, that $s\le r$. Set 
		\[
		\widetilde{\gamma}:=\gamma^{[0,s)}+\xi^{[s,T]}_{s,\cdot}(\gamma_{s}),\quad\widetilde{\vartheta}:=\vartheta^{[0,r)}+\xi^{[r,T]}_{r,\cdot}(\vartheta_{r}).
		\]
		Since the Fréchet derivative of $F$ is bounded, 
		\[
		|U_{s}(\gamma)-U_{r}(\vartheta)|\le\bnorm F_{\cP^{1}(D([0,T],\H[-I]))}\E\bnorm{\widetilde{\gamma}-\widetilde{\vartheta}}_{D([0,T],\H[-I])}.
		\]
		We estimate 
		\begin{align*}
			\bnorm{\widetilde{\gamma}-\widetilde{\vartheta}}_{D([0,T],\H[-I])}\le & \sup_{t\in[0,s)}\bnorm{\gamma_{t}-\vartheta_{t}}_{\H[-I]}\\
			& +\sup_{t\in[s,r)}\bnorm{\xi_{s,t}(\gamma_{s})-\vartheta_{t}}_{\H[-I]}\\
			& +\sup_{t\in[r,T]}\bnorm{\xi_{s,t}(\gamma_{s})-\xi_{r,t}(\vartheta_{r})}_{\H[-I]}.
		\end{align*}
		Using the representation of the modified Ornstein--Uhlenbeck process,
		for $t\in[s,r]$, 
		\begin{align*}
			\bnorm{\xi_{s,t}(\gamma_{s})-\vartheta_{t}}_{\H[-I]}\le & \bnorm{\cS_{s,t}(\gamma_{s})-\gamma_{s}}_{\H[-I]}+\bnorm{\gamma_{s}-\vartheta_{t}}_{\H[-I]}\\
			& +\bnorm{\int^{t}_{s}\cS_{u,t}B(u)dW_{u}}_{\H[-I]},
		\end{align*}
		where $B$ is defined by (\ref{eq:definition_of_B}), $\cS_{s,t}=\cS^{n,m_{\Delta}}_{s,t},$
		and $W$ is the cylindrical Wiener process on $L^{2}(\Td;\R^{d})$.
		Moreover, by the flow property of $\xi$ and Lemma (\ref{lem:modified_heat_semigroup}),
		for $t\in[r,T]$, 
		\begin{align*}
			\bnorm{\xi_{s,t}(\gamma_{s})-\xi_{r,t}(\vartheta_{r})}_{\H[-I]} & =\bnorm{\cS_{r,t}\left(\xi_{s,r}(\gamma_{s})-\vartheta_{r}\right)}_{\H[-I]}\\
			& \le\bnorm{\xi_{s,r}(\gamma_{s})-\vartheta_{r}}_{\H[-I]}.
		\end{align*}
		Consequently, 
		\begin{align}
			\E\bnorm{\widetilde{\gamma}-\widetilde{\vartheta}}_{D([0,T],\H[-I])}\le2 & d_{\Upsilon([0,T],\H[-I])}(\gamma,\vartheta)+\sup_{t\in[s,r]}\bnorm{\cS_{s,t}(\gamma_{s})-\gamma_{s}}_{\H[-I]}\label{eq:estimate_of_difference_in_U_continuity}\\
			& +\E\sup_{t\in[s,r]}\bnorm{\int^{t}_{s}\cS_{u,t}B(u)dW_{u}}_{\H[-I]}.\nonumber 
		\end{align}
		
		We note that the map
		\[
		[0,\infty)\times\H[-I]\ni(t,g)\mapsto\cS_{t}g\in\H[-I]
		\]
		is continuous. Indeed, this follows from the strong continuity and
		the contraction property of $(\cS_{t})_{t\ge0}$ (see Lemma \ref{lem:modified_heat_semigroup}).
		Thus, the second term on the right-hand side of
		(\ref{eq:estimate_of_difference_in_U_continuity})
		converges to zero as $r-s\to0$ by the strong continuity of the semigroup.
		By Lemma \ref{lem:stochastic_convolution}, the last term converges
		to zero as $r-s\to0$. Hence
		\[
		|U_{s}(\gamma)-U_{r}(\vartheta)|\to0
		\]
		whenever 
		\[
		d_{\Upsilon([0,T],\H[-I])}(\gamma,\vartheta)\to0.
		\]
		Thus, $U$ is continuous on $\Upsilon([0,T],\H[-I])$.
	\end{proof}
	
	\begin{lem}
		\label{lem:differentiability_of_U}Let $k\in\N$, $I>\frac{d}{2}+m_{\partial}$,
		and $\delta>0$. Assume that 
		\[
		F\in\cP^{k}(D([0,T],\H[-I-\delta])).
		\]
		Then 
		\[
		U\in\Cf^{k}_{l}(\Upsilon([0,T],\H[-I])).
		\]
		Moreover, for every $\ell\in[k]$, $(t,\gamma)\in\Upsilon([0,T],\H[-I])$,
		and $h_{1},\ldots,h_{\ell}\in\H[-I]$, 
		\[
		\nabla^{\ell}U_{t}(\gamma)[h_{1},\ldots,h_{\ell}]=\E\left[\nabla^{\ell}_{[t,T]}F\left(\gamma^{[0,t)}+\xi^{[t,T]}_{t,\cdot}(\gamma_{t})\right)\left[(\cS_{t,\cdot}h_{1})^{[t,T]},\ldots,(\cS_{t,\cdot}h_{\ell})^{[t,T]}\right]\right].
		\]
		Furthermore, 
		\begin{equation}
			\bnorm U_{\Cf^{k}_{l}(\Upsilon([0,T],\H[-I]))}\le\bnorm F_{\cP^{k}(D([0,T],\H[-I]))}.\label{eq:control_of_norm_of_U_via_F}
		\end{equation}
	\end{lem}
	
	\begin{rem}
		The additional loss of $\delta>0$ in the regularity assumption on
		$F$ is needed only to establish the left-continuity of the vertical
		derivatives of $U$ in operator norm. For vertical differentiability
		and the representation of the derivatives, it suffices to assume that
		$F\in\cP^{k}(D([0,T],\H[-I]))$.
	\end{rem}
	
	\begin{proof}
		By Lemma \ref{lem:continuity_of_U}, $U$ is a continuous non-anticipative
		functional. We first prove its vertical differentiability. Fix $(t,\gamma)\in\Upsilon([0,T],\H[-I])$
		and recall that 
		\[
		U_{t}(\gamma^{[0,t)},g)=\E F\left(\gamma^{[0,t)}+\xi^{[t,T]}_{t,\cdot}(g)\right),\quad g\in\H[-I].
		\]
		For every $\omega\in\Omega$, the map 
		\[
		\H[-I]\ni g\mapsto\xi^{[t,T]}_{t,\cdot}(\omega,g)\in D([t,T],\H[-I])
		\]
		is affine and, hence, Fréchet differentiable, with 
		\[
		D\xi^{[t,T]}_{t,\cdot}(\omega,g)[h]=(\cS_{t,\cdot}h)^{[t,T]},\quad h\in\H[-I].
		\]
		Moreover, by Lemma \ref{lem:modified_heat_semigroup}, 
		\[
		\bnorm{(\cS_{t,\cdot}h)^{[t,T]}}_{D([t,T],\H[-I])}\le\bnorm h_{\H[-I]}.
		\]
		
		Since $F\in\cP^{k}(D([0,T],\H[-I-\delta]))$, its restriction to $D([0,T],\H[-I])$
		belongs to $\cP^{k}(D([0,T],\H[-I]))$. Hence, by the chain rule,
		for every $\ell\in[k]$, 
		\begin{align*}
			& D^{\ell}\left[F\left(\gamma^{[0,t)}+\xi^{[t,T]}_{t,\cdot}(\omega,\cdot)\right)\right](g)[h_{1},\ldots,h_{\ell}]\\
			& \qquad=\nabla^{\ell}_{[t,T]}F\left(\gamma^{[0,t)}+\xi^{[t,T]}_{t,\cdot}(\omega,g)\right)\left[(\cS_{t,\cdot}h_{1})^{[t,T]},\ldots,(\cS_{t,\cdot}h_{\ell})^{[t,T]}\right].
		\end{align*}
		In particular, 
		\begin{align*}
			& \left|D^{\ell}\left[F\left(\gamma^{[0,t)}+\xi^{[t,T]}_{t,\cdot}(\omega,\cdot)\right)\right](g)[h_{1},\ldots,h_{\ell}]\right|\\
			& \qquad\le\sup_{\vartheta\in D([0,T],\H[-I])}\sup_{r\in[0,T]}\bnorm{\nabla^{\ell}_{[r,T]}F(\vartheta)}_{\cL_{\ell}(D([r,T],\H[-I]))}\prod^{\ell}_{j=1}\bnorm{h_{j}}_{\H[-I]},
		\end{align*}
		uniformly in $\omega$ and $g$.
		
		The random variables with values in the spaces of bounded multilinear
		operators are strongly measurable. Indeed, for fixed $t$ and $\gamma$, the random continuation
		paths take values in an affine translate of the separable space
		$C([t,T],\H[-I])$, while the corresponding vertical path derivatives
		of $F$ are continuous in the operator norm. Hence
		Proposition~\ref{prop:interchange_integral_and_derivative} can be
		applied iteratively. We conclude that
		$g\mapsto U_{t}(\gamma^{[0,t)},g)$ is $k$ times Fr\'echet
		differentiable and, for every $\ell\in[k]$,
		\begin{align}
			&\nabla^{\ell}U_{t}(\gamma)[h_{1},\ldots,h_{\ell}]\nonumber\\
			&\qquad=\E\bigg[ \nabla^{\ell}_{[t,T]}F\left(\gamma^{[0,t)}+\xi^{[t,T]}_{t,\cdot}(\gamma_{t})\right)\left[(\cS_{t,\cdot}h_{1})^{[t,T]},\ldots,(\cS_{t,\cdot}h_{\ell})^{[t,T]}\right]\bigg].\label{eq:expression_for_nabla_U}
		\end{align}
		
		We next prove the left-continuity of the vertical derivatives. We
		first consider $\nabla U$. Let $t_{n}\uparrow t$ and let $\gamma^{n}\in D([0,t_{n}],\H[-I])$
		converge to $\gamma\in D([0,t],\H[-I])$ in $\Upsilon([0,T],\H[-I])$.
		Set 
		\[
		\widetilde{\gamma}^{n}:=(\gamma^{n})^{[0,t_{n})}+\xi^{[t_{n},T]}_{t_{n},\cdot}(\gamma^{n}_{t_{n}}),\quad\widetilde{\gamma}:=\gamma^{[0,t)}+\xi^{[t,T]}_{t,\cdot}(\gamma_{t}).
		\]
		As in the proof of Lemma \ref{lem:continuity_of_U}, we have
		\[
		\E\bnorm{\widetilde{\gamma}^{n}-\widetilde{\gamma}}_{D([0,T],\H[-I])}\to0.
		\]
		Hence
		\[
		\widetilde{\gamma}^{n}\to\widetilde{\gamma}
		\]
		in probability in $D([0,T],\H[-I])$ and, by the continuous embedding
		$\H[-I]\hookrightarrow\H[-I-\delta]$, also in probability in $D([0,T],\H[-I-\delta])$.
		
		For $h\in\H[-I]$, the representation of the vertical derivative gives
		\begin{align}
			\nabla U_{t_{n}}(\gamma^{n})[h]-\nabla U_{t}(\gamma)[h] & =\E\left[\nabla_{[t_{n},T]}F(\widetilde{\gamma}^{n})\left[(\cS_{t_{n},\cdot}h)^{[t_{n},T]}-\iota_{t_{n},t}(\cS_{t,\cdot}h)^{[t,T]}\right]\right]\nonumber \\
			& +\E\left[\left(\nabla_{[t_{n},T]}F(\widetilde{\gamma}^{n})\circ\iota_{t_{n},t}-\nabla_{[t,T]}F(\widetilde{\gamma})\right)\left[(\cS_{t,\cdot}h)^{[t,T]}\right]\right].\label{eq:continuity_of_nabla_U}
		\end{align}
		Therefore, 
		\begin{align*}
			& \bnorm{\nabla U_{t_{n}}(\gamma^{n})-\nabla U_{t}(\gamma)}_{\cL_{1}(\H[-I])}\\
			& \qquad\le\bnorm F_{\cP^{1}(D([0,T],\H[-I-\delta]))}\bnorm{\cS^{[t_{n},T]}_{t_{n},\cdot}-\iota_{t_{n},t}\cS^{[t,T]}_{t,\cdot}}_{\cL_{1}(\H[-I],D([t_{n},T],\H[-I-\delta]))}\\
			& \qquad+\E\bnorm{\nabla_{[t_{n},T]}F(\widetilde{\gamma}^{n})\circ\iota_{t_{n},t}-\nabla_{[t,T]}F(\widetilde{\gamma})}_{\cL_{1}(D([t,T],\H[-I-\delta]))}.
		\end{align*}
		
		For the first term in (\ref{eq:continuity_of_nabla_U}), the semigroup
		property and its contraction property on $\H[-I-\delta]$ yield
		\[
		\bnorm{\cS^{[t_{n},T]}_{t_{n},\cdot}-\iota_{t_{n},t}\cS^{[t,T]}_{t,\cdot}}_{\cL_{1}(\H[-I],D([t_{n},T],\H[-I-\delta]))}\le\sup_{r\in[t_{n},t]}\bnorm{\id-\cS_{t_{n},r}}_{\cL_{1}(\H[-I],\H[-I-\delta])},
		\]
		which converges to zero as $n\to\infty$, by Lemma \ref{lem:operator_norm_continuity_modified_semigroup}.
		
		For the second term in (\ref{eq:continuity_of_nabla_U}), the left-continuity
		of the vertical path derivative of $F$ and the convergence $\widetilde{\gamma}^{n}\to\widetilde{\gamma}$
		in probability in $D([0,T],\H[-I-\delta])$ imply that the integrand
		converges to zero in probability. Since the vertical path derivative
		of $F$ is bounded, the dominated convergence theorem yields 
		\[
		\E\bnorm{\nabla_{[t_{n},T]}F(\widetilde{\gamma}^{n})\circ\iota_{t_{n},t}-\nabla_{[t,T]}F(\widetilde{\gamma})}_{\cL_{1}(D([t,T],\H[-I-\delta]))}\to0.
		\]
		Consequently, 
		\[
		\nabla U_{t_{n}}(\gamma^{n})\to\nabla U_{t}(\gamma)\quad\text{in }\cL_{1}(\H[-I]).
		\]
		
		The same argument, using the representation of $\nabla^{\ell}U$ and
		the left-continuity of $\nabla^{\ell}_{[t,T]}F$, applies to every
		$\ell\in[k]$. Thus all vertical derivatives $\nabla^{\ell}U$ are
		left-continuous. Hence 
		\[
		U\in\Cf^{k}_{l}(\Upsilon([0,T],\H[-I])).
		\]
		
		Finally the bound (\ref{eq:control_of_norm_of_U_via_F}) follows from
		(\ref{eq:expression_for_nabla_U}) and Lemma \ref{lem:modified_heat_semigroup}.
		This completes the proof of the statement.
	\end{proof}
	
	\begin{cor}
		\label{cor:existence_of_regular_kernel} Let $k,m\in\N$ and $I>\frac{d}{2}+(m_{\partial}\vee m).$
		Then, under the assumptions of Lemma~\ref{lem:differentiability_of_U},
		for every $\ell\in[k]$, $t\in[0,T]$, and $\gamma\in D([0,t],\H[-I])$,
		the multilinear functional $\nabla^{\ell}U_{t}(\gamma)$ admits a
		kernel 
		\[
		\nabla^{\ell}U_{t}(\gamma;\cdot)\in\Cf^{m,\ldots,m}((\Td)^{\ell}).
		\]
		Moreover, there exists a constant $C_{k,I,m}>0$ such that 
		\begin{align*}
			\sup_{t\in[0,T]}\sup_{\gamma\in D([0,t],\H[-I])}\bnorm{\nabla^{\ell}U_{t}(\gamma;\cdot)}_{\Cf^{m,\ldots,m}((\Td)^{\ell})}&\le C_{k,I,m}\bnorm U_{\Cf^{k}_{l}(\Upsilon([0,T],\H[-I]))}\\
			&\le C_{k,I,m}\bnorm F_{\cP^{k}(D([0,T],\H[-I]))}.
		\end{align*}
	\end{cor}
	
	\begin{proof}
		The statement directly follows from Lemmas \ref{lem:differentiability_of_U}
		and \ref{lem:existence_of_smooth_kernel}.
	\end{proof}
	
	We finish this section with the vertical path-differentiability property
	of the function $U$. For $0\le s\le r\le T$, define
	\[
	\Lambda_{s,r}:D([s,r],\H[-I])\to D([s,T],\H[-I])
	\]
	by
	\[
	\Lambda_{s,r}h=h^{[s,r)}+(\cS_{r,\cdot}h_r)^{[r,T]},\quad h\in D([s,r],\H[-I]).
	\]
	
	\begin{lem}
		\label{lem:path_regulariy_of_U_r} Let $k\in\N$ and $I>\frac{d}{2}+m_{\partial}$.
		Assume that 
		\[
		F\in\cP^{k}(D([0,T],\H[-I])).
		\]
		Then, 
		\[
		U\in\cP^{k}(\Upsilon([0,T],\H[-I])).
		\]
		Moreover, for every $\ell\in[k]$, $r\in[0,T],$ $s\in[0,r]$, $\gamma\in D([0,r],\H[-I])$,
		and $h_{1},\ldots,h_{\ell}\in D([s,r],\H[-I])$,
		\begin{equation}
			\nabla^{\ell}_{[s,r]}U_{r}(\gamma)[h_{1},\ldots,h_{\ell}]=\E\left[\nabla^{\ell}_{[s,T]}F\left(\gamma^{[0,r)}+\xi^{[r,T]}_{r,\cdot}(\gamma_{r})\right)[\Lambda_{s,r}h_{1},\ldots,\Lambda_{s,r}h_{\ell}]\right].\label{eq:path_derivative_of_U_r}
		\end{equation}
		Furthermore,
		\[
		\|U\|_{\cP^{k}(\Upsilon([0,T],\H[-I]))}\le\|F\|_{\cP^{k}(D([0,T],\H[-I]))}.
		\]
	\end{lem}
	
	\begin{proof}
		Fix $r\in[0,T]$. We first note that, for every $s\in[0,r]$, the
		map 
		\[
		\Lambda_{s,r}:D([s,r],\H[-I])\to D([s,T],\H[-I])
		\]
		is linear and bounded. Indeed, by Lemma \ref{lem:modified_heat_semigroup},
		\begin{align*}
			\bnorm{\Lambda_{s,r}h}_{D([s,T],\H[-I])} & =\max\left\{ \sup_{u\in[s,r]}\bnorm{h_{u}}_{\H[-I]},\sup_{u\in[r,T]}\bnorm{\cS_{r,u}h_{r}}_{\H[-I]}\right\} \le\bnorm h_{D([s,r],\H[-I])}.
		\end{align*}
		Hence 
		\begin{equation}
			\bnorm{\Lambda_{s,r}}_{\cL_{1}(D([s,r],\H[-I]),D([s,T],\H[-I]))}\le1.\label{eq:bound_E_sr}
		\end{equation}
		
		We next prove the vertical path differentiability of $U_{r}$. Fix
		$s\in[0,r]$, $\gamma\in D([0,r],\H[-I])$, and $h\in D([s,r],\H[-I])$,
		where $h$ is extended by zero to $[0,s)$. Using Proposition \ref{prop:OU_started_at_s},
		for every $u\in[r,T]$ we have 
		\[
		\xi_{r,u}(\gamma_{r}+h_{r})=\xi_{r,u}(\gamma_{r})+\cS_{r,u}h_{r}.
		\]
		Consequently, 
		\[
		(\gamma+h)^{[0,r)}+\xi^{[r,T]}_{r,\cdot}(\gamma_{r}+h_{r})=\gamma^{[0,r)}+\xi^{[r,T]}_{r,\cdot}(\gamma_{r})+\Lambda_{s,r}h=\widetilde{\gamma}^{r}+\Lambda_{s,r}h,
		\]
		where 
		\[
		\widetilde{\gamma}^{r}:=\gamma^{[0,r)}+\xi^{[r,T]}_{r,\cdot}(\gamma_{r}).
		\]
		Therefore, 
		\begin{equation}
			U_{r}(\gamma+h)=\E F\left(\widetilde{\gamma}^{r}+\Lambda_{s,r}h\right).\label{eq:Ur_perturbation}
		\end{equation}
		
		Since $F\in\cP^{k}(D([0,T],\H[-I]))$, for every fixed realization
		of $\widetilde{\gamma}^{r}$ the map 
		\[
		h\mapsto F\left(\widetilde{\gamma}^{r}+\Lambda_{s,r}h\right),\quad h\in D([s,r],\H[-I]),
		\]
		is $k$ times Fréchet differentiable. By the chain rule, for $\ell\in[k]$,
		\[
		D^{\ell}\left[F\left(\widetilde{\gamma}^{r}+\Lambda_{s,r}(\cdot)\right)\right](0)[h_{1},\ldots,h_{\ell}]=\nabla^{\ell}_{[s,T]}F(\widetilde{\gamma}^{r})[\Lambda_{s,r}h_{1},\ldots,\Lambda_{s,r}h_{\ell}].
		\]
		
		Moreover, by (\ref{eq:bound_E_sr}), 
		\[
		\left|\nabla^{\ell}_{[s,T]}F(\widetilde{\gamma}^{r})[\Lambda_{s,r}h_{1},\ldots,\Lambda_{s,r}h_{\ell}]\right|\le\bnorm F_{\cP^{k}(D([0,T],\H[-I]))}\prod^{\ell}_{j=1}\bnorm{h_{j}}_{D([s,r],\H[-I])},
		\]
		uniformly in the realization of $\widetilde{\gamma}^{r}$. 
		
		The derivative-valued random variables are strongly measurable by the
		same separability argument as in the proof of
		Lemma~\ref{lem:differentiability_of_U}. Hence
		Proposition~\ref{prop:interchange_integral_and_derivative} may be
		applied iteratively to (\ref{eq:Ur_perturbation}).
		This yields (\ref{eq:path_derivative_of_U_r}). Thus $U_{r}$ is vertically
		path-differentiable up to order $k$.
		
		We next verify the boundedness of these derivatives. From (\ref{eq:path_derivative_of_U_r})
		and (\ref{eq:bound_E_sr}), 
		\begin{align*}
			\bnorm{\nabla^{\ell}_{[s,r]}U_{r}(\gamma)}_{\cL_{\ell}(D([s,r],\H[-I]))} & \le\E\bnorm{\nabla^{\ell}_{[s,T]}F(\widetilde{\gamma}^{r})}_{\cL_{\ell}(D([s,T],\H[-I]))}\\
			& \le\sup_{\vartheta\in D([0,T],\H[-I])}\sup_{u\in[0,T]}\bnorm{\nabla^{\ell}_{[u,T]}F(\vartheta)}_{\cL_{\ell}(D([u,T],\H[-I]))}.
		\end{align*}
		Taking the supremum over $s\in[0,r]$ and $\gamma\in D([0,r],\H[-I])$,
		and then summing over $\ell\in[k]$, gives 
		\begin{equation}
			\bnorm{U_{r}}_{\cP^{k}(D([0,r],\H[-I]))}\le\bnorm F_{\cP^{k}(D([0,T],\H[-I]))}.\label{eq:path_norm_Ur}
		\end{equation}
		It remains to establish the left-continuity of the path derivatives.
		
		Fix $\ell\in[k]$ and let $s_{n}\uparrow s$ in $[0,r]$. Let $\gamma^{n},\gamma\in D([0,r],\H[-I])$
		satisfy 
		\[
		\bnorm{\gamma^{n}-\gamma}_{D([0,r],\H[-I])}\to0.
		\]
		Set 
		\[
		\widetilde{\gamma}^{n,r}:=(\gamma^{n})^{[0,r)}+\xi^{[r,T]}_{r,\cdot}(\gamma^{n}_{r}),\quad\widetilde{\gamma}^{r}:=\gamma^{[0,r)}+\xi^{[r,T]}_{r,\cdot}(\gamma_{r}).
		\]
		Using Proposition \ref{prop:OU_started_at_s}, we have 
		\[
		\xi_{r,u}(\gamma^{n}_{r})-\xi_{r,u}(\gamma_{r})=\cS_{r,u}(\gamma^{n}_{r}-\gamma_{r}),\quad u\in[r,T].
		\]
		Hence, by the contraction property, 
		\begin{align*}
			\bnorm{\widetilde{\gamma}^{n,r}-\widetilde{\gamma}^{r}}_{D([0,T],\H[-I])} & \le\max\left\{ \sup_{u\in[0,r)}\bnorm{\gamma^{n}_{u}-\gamma_{u}}_{\H[-I]},\sup_{u\in[r,T]}\bnorm{\cS_{r,u}(\gamma^{n}_{r}-\gamma_{r})}_{\H[-I]}\right\} \\
			& \le\bnorm{\gamma^{n}-\gamma}_{D([0,r],\H[-I])}.
		\end{align*}
		Therefore, 
		\begin{equation}
			\widetilde{\gamma}^{n,r}\to\widetilde{\gamma}^{r}\quad\text{in }D([0,T],\H[-I])\quad\text{a.s.}\label{eq:tilde_gamma_fixed_r_convergence}
		\end{equation}
		
		We also need the compatibility of the operators $\Lambda_{s,r}$ with
		the constant-extension maps. For $s_{n}\le s$ and $h\in D([s,r],\H[-I])$,
		it is easy to see that
		\begin{equation}
			\Lambda_{s_{n},r}\left(\iota^{r}_{s_{n},s}h\right)=\iota^{T}_{s_{n},s}\left(\Lambda_{s,r}h\right).\label{eq:E_iota_compatibility}
		\end{equation}
		
		Using (\ref{eq:path_derivative_of_U_r}) and (\ref{eq:E_iota_compatibility}),
		we obtain 
		\[
		\left(\nabla^{\ell}_{[s_{n},r]}U_{r}(\gamma^{n})\circ(\iota^{r}_{s_{n},s})^{\otimes\ell}\right)[h_{1},\ldots,h_{\ell}]=\E\Big[\nabla^{\ell}_{[s_{n},T]}F(\widetilde{\gamma}^{n,r})\big[\iota^{T}_{s_{n},s}\Lambda_{s,r}h_{1},\ldots,\iota^{T}_{s_{n},s}\Lambda_{s,r}h_{\ell}\big]\Big].
		\]
		Therefore, 
		\begin{align}
			& \bnorm{\nabla^{\ell}_{[s_{n},r]}U_{r}(\gamma^{n})\circ(\iota^{r}_{s_{n},s})^{\otimes\ell}-\nabla^{\ell}_{[s,r]}U_{r}(\gamma)}_{\cL_{\ell}(D([s,r],\H[-I]))}\nonumber \\
			& \qquad\le\E\bnorm{\nabla^{\ell}_{[s_{n},T]}F(\widetilde{\gamma}^{n,r})\circ(\iota^{T}_{s_{n},s})^{\otimes\ell}-\nabla^{\ell}_{[s,T]}F(\widetilde{\gamma}^{r})}_{\cL_{\ell}(D([s,T],\H[-I]))},\label{eq:estimate_for_path_derivative_for_U}
		\end{align}
		where we used again $\bnorm{\Lambda_{s,r}}\le1.$
		
		By (\ref{eq:tilde_gamma_fixed_r_convergence}) and the left-continuity
		of the family $\nabla^{\ell}_{[u,T]}F$, the integrand on the right-hand
		side of \ref{eq:estimate_for_path_derivative_for_U} converges to
		zero almost surely. Moreover, it is bounded by 
		\[
		2\sup_{\vartheta\in D([0,T],\H[-I])}\sup_{u\in[0,T]}\bnorm{\nabla^{\ell}_{[u,T]}F(\vartheta)}_{\cL_{\ell}(D([u,T],\H[-I]))}<\infty.
		\]
		Hence, by the dominated convergence theorem, the right hand side of
		\ref{eq:estimate_for_path_derivative_for_U} converges to zero. Thus
		the family $\nabla^{\ell}_{[s,r]}U_{r}$, $s\in[0,r]$, is left-continuous.
		
		Since this holds for every $\ell\in[k]$, all vertical path derivatives
		of $U_{r}$ up to order $k$ are bounded and left-continuous. Together
		with the continuity of $U_{r}$, which follows from Lemma \ref{lem:continuity_of_U},
		this proves 
		\[
		U_{r}\in\cP^{k}(D([0,r],\H[-I])).
		\]
		This completes the proof.
	\end{proof}

	\subsection{Backward Kolmogorov equation}
	
	We now turn to the path-dependent backward Kolmogorov equation associated
	with the modified Ornstein--Uhlenbeck dynamics. The aim of this subsection
	is to show that the continuation functional introduced in (\ref{eq:definition_of_U})
	is horizontally absolutely continuous and that its weak horizontal
	derivative is given by the negative of the corresponding Ornstein--Uhlenbeck
	generator.
	
	In this section, we fix $m_{\partial},m_{\Delta}\in\N$, where $m_{\Delta}$
	is odd. For $I>\frac{d}{2}+m_{\partial}$ and
	\begin{equation}\label{eq:function_Phi_for_Kolmogorov_equation}
		\Phi=(\Phi_{j})_{j\in[d]}:[0,\infty)\to L^{2}(\Td;\R^{d})
	\end{equation}
	measurable and locally bounded we extend the time-dependent generator
	$\cG^{\Phi}_{t}$ to sufficiently regular non-anticipative functionals,
	following (\ref{eq:generator_G_Phi}). For $0\le a<b\le T$, 
	\[
	F\in\Cf^{2}_{l}\left(\Upsilon([a,b],\H[-I-2m_{\Delta}])\right)\cup\cP^{2}\left(\Upsilon([a,b],\H[-I-2m_{\Delta}])\right)
	\]
	and $(t,\gamma)\in\Upsilon([a,b],\H[-I])$, we set
	\begin{align*}
		\cG^{\Phi}_{t}F(\gamma) & :=\frac{1}{2}\binn{\Delta^{(2m_{\Delta})}_{n}\nabla F_{t}(\gamma),\gamma_{t}}+\frac{1}{4}\sum^{d}_{j=1}\binn{\left(\partial^{(m_{\partial})}_{n,j}\right)^{\odot2}\nabla^{2}F_{t}(\gamma;\cdot),\Phi^{2}_{j}(t)}.
	\end{align*}
	
	The first term is well defined because $\nabla F_{t}(\gamma)\in\H[I+2m_{\Delta}]$,
	while the second one is well defined by restricting $\nabla^{2}F_{t}(\gamma)$ to $\H[-I]\times\H[-I]$
	and using its kernel representation together with the assumption $I>\frac{d}{2}+m_{\partial}$
	(see Lemma \ref{lem:existence_of_smooth_kernel}).

	\begin{rem}
		The map
		\[
		(t,\gamma)\mapsto\cG^{\Phi}_{t}F(\gamma)
		\]
		is a non-anticipative functional on $\Upsilon([a,b],\H[-I])$. Indeed, for fixed directions, the first and second vertical derivatives are measurable as pointwise limits of the corresponding first and second difference quotients of $F$. Together with Lemma~\ref{lem:existence_of_smooth_kernel}, the continuity of the relevant operators and pairings, and the measurability of $t\mapsto\Phi^2_j(t)$ in $L^1(\Td)$, this implies the measurability of both terms in $\cG^\Phi_tF(\gamma)$.
	\end{rem}

	Before we formulate the main statement of this section we will prove
	an auxiliary statement. 
	\begin{lem}
		\label{lem:regularity_for_G_Phi}Let $0\le a<b\le T$, $I>\frac{d}{2}+m_{\partial}$
		and let $\Phi$ be continuous. Then for each 
		\[
		F\in\Cf^{2}_{l}\left(\Upsilon([a,b],\H[-I-2m_{\Delta}])\right),
		\]
		the non-anticipative functional
		\[
		(t,\gamma)\mapsto\cG^{\Phi}_{t}F(\gamma)
		\]
		on $\Upsilon([a,b],\H[-I])$ is left-continuous. Moreover, there exists
		a constant $C_{m_{\partial},m_{\Delta},I}>0$ such that
		\begin{equation}
			\left|\cG^{\Phi}_{t}F(\gamma)\right|\le C_{m_{\partial},m_{\Delta},I}\left(\bnorm{\Phi(t)}^{2}_{L^{2}(\Td;\R^{d})}+\norm{\gamma_{t}}_{\H[-I]}\right)\norm F_{\Cf^{2}_{l}(\Upsilon([a,b],\H[-I-2m_{\Delta}]))}\label{eq:bound_of_generator}
		\end{equation}
		for all $(t,\gamma)\in\Upsilon([a,b],\H[-I])$.
	\end{lem}
	
	\begin{proof}
		By the uniform boundedness of 
		\[
		\Delta^{(2m_{\Delta})}_{n}:\H[I+2m_{\Delta}]\to\H[I]
		\]
		in $n$ and Lemma \ref{lem:existence_of_smooth_kernel}, we have 
		\begin{align*}
			\left|\cG^{\Phi}_{t}F(\gamma)\right| & \le\frac{1}{2}\bnorm{\Delta^{(2m_{\Delta})}_{n}\nabla F_{t}(\gamma)}_{\cL_{1}(\H[-I])}\norm{\gamma_{t}}_{\H[-I]}\\
			& +\frac{1}{4}\sum^{d}_{j=1}\bnorm{\Phi_{j}(t)}^{2}_{\L}\bnorm{\left(\partial^{(m_{\partial})}_{n,j}\right)^{\odot2}\nabla^{2}F_{t}(\gamma;\cdot)}_{\Cf(\Td)}\\
			& \le C_{m_{\partial},m_{\Delta},I}\left(\norm{\gamma_{t}}_{\H[-I]}+\bnorm{\Phi(t)}^{2}_{L^{2}(\Td;\R^{d})}\right)\norm F_{\Cf^{2}_{l}(\Upsilon([a,b],\H[-I-2m_{\Delta}]))}.
		\end{align*}
		Here we used 
		\[
		\bnorm{\Delta^{(2m_{\Delta})}_{n}\nabla F_{t}(\gamma)}_{\cL_{1}(\H[-I])}\le C_{m_{\Delta}}\bnorm{\nabla F_{t}(\gamma)}_{\cL_{1}(\H[-I-2m_{\Delta}])}
		\]
		and, by Lemma \ref{lem:existence_of_smooth_kernel}, 
		\begin{align}
			\bnorm{\left(\partial^{(m_{\partial})}_{n,j}\right)^{\odot2}\nabla^{2}F_{t}(\gamma;\cdot)}_{\Cf(\Td)} & \le C_{m_{\partial}}\bnorm{\nabla^{2}F_{t}(\gamma;\cdot)}_{\Cf^{m_{\partial},m_{\partial}}((\Td)^{2})}\le C_{m_{\partial},I}\bnorm{\nabla^{2}F_{t}(\gamma)}_{\cL_{2}(\H[-I])}.\label{eq:estimate_of_diffusion_term-1}
		\end{align}
		This proves (\ref{eq:bound_of_generator}).
		
		It remains to prove left-continuity. Set 
		\[
		G^{(1)}_{t}(\gamma):=\binn{\Delta^{(2m_{\Delta})}_{n}\nabla F_{t}(\gamma),\gamma_{t}}.
		\]
		Fix $t\in(a,b]$, $\gamma\in D([a,t],\H[-I])$, and let $\eps>0$.
		By the left-continuity of $\nabla F$, there exists $\delta>0$ such
		that 
		\[
		\bnorm{\nabla F_{t}(\gamma)-\nabla F_{r}(\vartheta)}_{\cL_{1}(\H[-I-2m_{\Delta}])}<\eps
		\]
		whenever $a\le r\le t$ and $d_{\Upsilon([a,b],\H[-I])}(\gamma,\vartheta)<\delta.$
		Here and below we use the continuous embedding $\H[-I]\hookrightarrow\H[-I-2m_{\Delta}]$.
		
		For such $(r,\vartheta)$, using again the uniform boundedness of
		$\Delta^{(2m_{\Delta})}_{n}$, we obtain 
		\begin{align*}
			\left|G^{(1)}_{t}(\gamma)-G^{(1)}_{r}(\vartheta)\right| & \le\left|\binn{\Delta^{(2m_{\Delta})}_{n}\nabla F_{t}(\gamma),\gamma_{t}-\vartheta_{r}}\right|\\
			& +\left|\binn{\Delta^{(2m_{\Delta})}_{n}\bigl(\nabla F_{t}(\gamma)-\nabla F_{r}(\vartheta)\bigr),\vartheta_{r}}\right|\\
			& \le C_{m_{\Delta}}\bnorm{\nabla F_{t}(\gamma)}_{\cL_{1}(\H[-I-2m_{\Delta}])}\bnorm{\gamma_{t}-\vartheta_{r}}_{\H[-I]}\\
			& +C_{m_{\Delta}}\bnorm{\nabla F_{t}(\gamma)-\nabla F_{r}(\vartheta)}_{\cL_{1}(\H[-I-2m_{\Delta}])}\bnorm{\vartheta_{r}}_{\H[-I]}\\
			& \le C_{m_{\Delta}}\bnorm{\nabla F_{t}(\gamma)}_{\cL_{1}(\H[-I-2m_{\Delta}])}d_{\Upsilon([a,b],\H[-I])}(\gamma,\vartheta)\\
			& +\eps C_{m_{\Delta}}\left(d_{\Upsilon([a,b],\H[-I])}(\gamma,\vartheta)+\bnorm{\gamma_{t}}_{\H[-I]}\right).
		\end{align*}
		By decreasing $\delta$ if necessary, the right-hand side of the estimate
		above can be made arbitrarily small. Hence $G^{(1)}$ is left-continuous.
		
		For $j\in[d]$, set 
		\[
		G^{(2),j}_{t}(\gamma):=\binn{\left(\partial^{(m_{\partial})}_{n,j}\right)^{\odot2}\nabla^{2}F_{t}(\gamma;\cdot),\Phi^{2}_{j}(t)}.
		\]
		Using (\ref{eq:estimate_of_diffusion_term-1}), for $a\le r\le t$
		we obtain 
		\begin{align*}
			\left|G^{(2),j}_{t}(\gamma)-G^{(2),j}_{r}(\vartheta)\right| & \le C_{m_{\partial},I}\bnorm{\nabla^{2}F_{t}(\gamma)}_{\cL_{2}(\H[-I])}\bnorm{\Phi^{2}_{j}(t)-\Phi^{2}_{j}(r)}_{L^{1}(\Td)}\\
			& \quad+C_{m_{\partial},I}\bnorm{\nabla^{2}F_{t}(\gamma)-\nabla^{2}F_{r}(\vartheta)}_{\cL_{2}(\H[-I])}\bnorm{\Phi_{j}(r)}^{2}_{\L}.
		\end{align*}
		Since the embedding $\H[-I]\hookrightarrow\H[-I-2m_{\Delta}]$ is
		continuous, restricting a bilinear form on $\H[-I-2m_{\Delta}]\times\H[-I-2m_{\Delta}]$
		to $\H[-I]\times\H[-I]$ yields
		\[
		\bnorm{\nabla^{2}F_{t}(\gamma)-\nabla^{2}F_{r}(\vartheta)}_{\cL_{2}(\H[-I])}\le C_{m_{\Delta}}\bnorm{\nabla^{2}F_{t}(\gamma)-\nabla^{2}F_{r}(\vartheta)}_{\cL_{2}(\H[-I-2m_{\Delta}])}.
		\]
		Moreover, Hölder's inequality yields 
		\[
		\bnorm{\Phi^{2}_{j}(t)-\Phi^{2}_{j}(r)}_{L^{1}(\Td)}\le\bnorm{\Phi_{j}(t)-\Phi_{j}(r)}_{\L}\bnorm{\Phi_{j}(t)+\Phi_{j}(r)}_{\L}.
		\]
		Hence, the continuity of $\Phi$ and the left-continuity of $\nabla^{2}F$
		imply that $G^{(2),j}$ is left-continuous. Consequently, $\cG^{\Phi}F$
		is left-continuous.
	\end{proof}
	
	\begin{thm}[Backward Kolmogorov equation]
		\label{thm:backward_kolmogorov_equation}Let $I>\frac{d}{2}+m_{\partial}$, the function $\Phi$, defined by \eqref{eq:function_Phi_for_Kolmogorov_equation}, be continuous
		and $F\in\cP^{2}\left(D([0,T],\H[-I-2m_{\Delta}])\right)$. Then
		\[
		U\in\cP^{2}\left(\Upsilon([0,T],\H[-I-2m_{\Delta}])\right),
		\]
		and the restriction of $U$ to $\Upsilon([0,T],\H[-I])$ is horizontally
		absolutely continuous. Moreover, its weak horizontal derivative satisfies
		\[
		\cH U(\gamma)+\cG^{\Phi}U(\gamma)=0,\quad\gamma\in\Upsilon([0,T],\H[-I]).
		\]
		Equivalently, for every $0\le r\le t\le T$ and $\gamma\in D([0,r],\H[-I])$,
		\begin{equation}
			U_{r}(\gamma)=U_{t}(\Theta_{r,t}\gamma)+\int^{t}_{r}\cG^{\Phi}_{s}U_{s}(\Theta_{r,s}\gamma)ds.\label{eq:absolute_continuity_for_U}
		\end{equation}
	\end{thm}
	
	\begin{proof}
		By Lemma \ref{lem:path_regulariy_of_U_r}, 
		\[
		U\in\cP^{2}\left(\Upsilon([0,T],\H[-I-2m_{\Delta}])\right),
		\]
		and
		\begin{equation}
			\norm{U}_{\cP^{2}(\Upsilon([0,T],\H[-I-2m_{\Delta}]))}
			\le
			\norm F_{\cP^{2}(D([0,T],\H[-I-2m_{\Delta}]))}.
			\label{eq:bound_for_U_in_cP_norm}
		\end{equation}
		
		Fix $0\le t<T$ and $\gamma\in D([0,t],\H[-I])$. Let $r\in(t,T]$.
		By the flow property of the modified Ornstein--Uhlenbeck process
		and the independence of the driving noise on $[t,r]$ and $[r,T]$
		for $r\in(t,T]$, we obtain 
		\begin{align}
			U_{t}(\gamma) & =\E F\left(\gamma^{[0,t)}+\xi^{[t,T]}_{t,\cdot}(\gamma_{t})\right)=\E F\left(\gamma^{[0,t)}+\xi^{[t,r)}_{t,\cdot}(\gamma_{t})+\xi^{[r,T]}_{r,\cdot}\bigl(\xi_{t,r}(\gamma_{t})\bigr)\right)\nonumber \\
			& =\E U_{r}\left(\gamma^{[0,t)}+\xi^{[t,r]}_{t,\cdot}(\gamma_{t})\right).\label{eq:dynamic_property_of_U-3}
		\end{align}
		
		We next will apply Itô's formula to the right-hand side of (\ref{eq:dynamic_property_of_U-3}).
		For this we define 
		\[
		G^{r,\gamma}_{s}(\vartheta):=U_{r}\left(\gamma^{[0,t)}+\Theta^{t}_{s,r}\vartheta\right),\quad(s,\vartheta)\in\Upsilon([t,r],\H[-I]).
		\]
		
		By Lemma \ref{lem:regularity_of_map_G}, applied with the terminal
		time $r$ and $F=U_{r}$, the functional $G^{r,\gamma}$ belongs to
		$\Cf^{2}_{l}(\Upsilon([t,r],\H[-I-2m_{\Delta}]))$ and is horizontally
		differentiable with
		\[
		\cD_{s}G^{r,\gamma}(\vartheta)=0.
		\]
		Moreover, for $\ell=1,2$, the vertical derivatives of $G^{r,\gamma}_{s}$
		are given by
		\[
		\nabla^{\ell}G^{r,\gamma}_{s}(\vartheta)[h_{1},\ldots,h_{\ell}]=\nabla^{\ell}_{[s,r]}U_{r}\left(\gamma^{[0,t)}+\Theta^{t}_{s,r}\vartheta\right)\left[\jmath_{[s,r]}h_{1},\ldots,\jmath_{[s,r]}h_{\ell}\right].
		\]
		
		Furthermore, 
		\begin{equation}
			\bnorm{G^{r,\gamma}}_{\Cf^{2}_{l}(\Upsilon([t,r],\H[-I-2m_{\Delta}]))}\le\bnorm{U_{r}}_{\cP^{2}(D([0,r],\H[-I-2m_{\Delta}]))}.\label{eq:estimate_for_G_r_gamma}
		\end{equation}
		
		Applying Itô's formula (see Proposition \ref{prop:Ito_formula}) to
		\[
		s\mapsto G^{r,\gamma}_{s}\left(\xi^{[t,s]}_{t,\cdot}(\gamma_{t})\right),
		\]
		and taking expectations, we obtain 
		\begin{align}
			U_{t}(\gamma) & =\E G^{r,\gamma}_{r}\left(\xi^{[t,r]}_{t,\cdot}(\gamma_{t})\right)=G^{r,\gamma}_{t}(\gamma_{t})+\E\int^{r}_{t}\cG^{\Phi}_{s}G^{r,\gamma}_{s}\left(\xi^{[t,s]}_{t,\cdot}(\gamma_{t})\right)ds\nonumber \\
			& =U_{r}(\Theta_{t,r}\gamma)+\E\int^{r}_{t}L^{r}_{s}(\gamma)ds,\label{eq:integral_representation_for_U-1-2}
		\end{align}
		where
		\[
		L^{r}_{s}(\gamma):=\cG^{\Phi}_{s}G^{r,\gamma}_{s}\left(\xi^{[t,s]}_{t,\cdot}(\gamma_{t})\right),\quad s\in[t,r].
		\]
		
		By Lemma \ref{lem:regularity_for_G_Phi}, (\ref{eq:estimate_for_G_r_gamma}),
		and (\ref{eq:bound_for_U_in_cP_norm}), 
		\begin{align*}
			\left|L^{r}_{s}(\gamma)\right| & \le C_{m_{\partial},m_{\Delta},I}\left(\bnorm{\Phi(s)}^{2}_{L^{2}(\Td;\R^{d})}+\bnorm{\xi_{t,s}(\gamma_{t})}_{\H[-I]}\right)\bnorm{G^{r,\gamma}}_{\Cf^{2}_{l}(\Upsilon([t,r],\H[-I-2m_{\Delta}]))}\\
			& \le C_{m_{\partial},m_{\Delta},I}\left(\bnorm{\Phi(s)}^{2}_{L^{2}(\Td;\R^{d})}+\bnorm{\xi_{t,s}(\gamma_{t})}_{\H[-I]}\right)\bnorm F_{\cP^{2}(D([0,T],\H[-I-2m_{\Delta}]))}.
		\end{align*}
		Consequently, 
		\begin{align}
			\E\left|L^{r}_{s}(\gamma)\right| & \le C_{m_{\partial},m_{\Delta},I}\left(\bnorm{\Phi(s)}^{2}_{L^{2}(\Td;\R^{d})}+\E\bnorm{\xi_{t,s}(\gamma_{t})}_{\H[-I]}\right)\bnorm F_{\cP^{2}(D([0,T],\H[-I-2m_{\Delta}]))}\nonumber \\
			& \le C_{m_{\partial},m_{\Delta},I,\Phi,\gamma}\bnorm F_{\cP^{2}(D([0,T],\H[-I-2m_{\Delta}]))}=:C_{m_{\partial},m_{\Delta},I,\Phi,\gamma,F}.\label{eq:boundedness_of_L_r}
		\end{align}
		Here the last estimate follows from Proposition \ref{prop:OU_started_at_s}
		and the boundedness of $\Phi$ on $[0,T]$.
		
		We now define 
		\[
		f(r):=U_{r}(\Theta_{t,r}\gamma),\quad r\in[t,T],
		\]
		and show that $f$ is absolutely continuous on $[t,T]$.
		
		Let $t\le a<b\le T$ and set 
		\[
		\vartheta^{a}:=\Theta_{t,a}\gamma\in D([0,a],\H[-I]).
		\]
		Then 
		\[
		\vartheta^{a}_{a}=\gamma_{t},\quad f(a)=U_{a}(\vartheta^{a}),\quad f(b)=U_{b}(\Theta_{a,b}\vartheta^{a}).
		\]
		Applying (\ref{eq:integral_representation_for_U-1-2}) with initial
		time $a$, terminal time $b$, and initial path $\vartheta^{a}$,
		we obtain 
		\[
		f(a)-f(b)=\int^{b}_{a}\E L^{b}_{s}(\vartheta^{a})ds.
		\]
		Since $\vartheta^{a}_{a}=\gamma_{t}$, estimate (\ref{eq:boundedness_of_L_r})
		yields 
		\[
		\sup_{s\in[a,b]}\E\left|L^{b}_{s}(\vartheta^{a})\right|\le C_{m_{\partial},m_{\Delta},I,\Phi,\gamma,F}.
		\]
		Therefore, 
		\[
		|f(b)-f(a)|\le C_{m_{\partial},m_{\Delta},I,\Phi,\gamma,F}(b-a).
		\]
		Hence $f$ is Lipschitz continuous on $[t,T]$ and thus absolutely
		continuous. Consequently, $f'(s)$ exists for a.e. $s\in[t,T]$ and
		\begin{equation}
			f(b)-f(a)=\int^{b}_{a}f'(s)ds\label{eq:absolute_continuity_for_f-1-2}
		\end{equation}
		for all $t\le a\le b\le T$.
		
		We finally identify the derivative of $f$. Fix $r\in(t,T)$ at which
		$f$ is differentiable. By (\ref{eq:integral_representation_for_U-1-2}),
		for every $u\in[t,r)$, 
		\[
		f(u)-f(r)=\int^{r}_{u}\E L^{r}_{s}(\vartheta^{u})\,ds,
		\]
		where $\vartheta^{u}=\Theta_{t,u}\gamma.$ Hence,
		\begin{equation}
			f'(r)=-\lim_{u\to r-}\frac{1}{r-u}\int^{r}_{u}\E L^{r}_{s}(\vartheta^{u})ds.\label{eq:expression_for_derivative_of_f-1-1}
		\end{equation}
		
		We claim that
		\begin{equation}
			\sup_{s\in[u,r]}\E\left|L^{r}_{s}(\vartheta^{u})-\cG^{\Phi}_{r}U_{r}(\Theta_{t,r}\gamma)\right|\to0\quad\text{as }u\to r-.\label{eq:convergence_of_L-1}
		\end{equation}
		For $u\le s\le r$, we have 
		\begin{align*}
			L^{r}_{s}(\vartheta^{u}) & =\cG^{\Phi}_{s}G^{r,\vartheta^{u}}_{s}\left(\xi^{[u,s]}_{u,\cdot}(\gamma_{t})\right)=\cG^{\Phi}_{s}G^{r,\gamma}_{s}\left(\jmath_{[t,u)}\gamma_{t}+\xi^{[u,s]}_{u,\cdot}(\gamma_{t})\right),
		\end{align*}
		whereas 
		\[
		\cG^{\Phi}_{r}U_{r}(\Theta_{t,r}\gamma)=\cG^{\Phi}_{r}G^{r,\gamma}_{r}\left(\jmath_{[t,r]}\gamma_{t}\right).
		\]
		Moreover, 
		\begin{align*}
			& d_{\Upsilon([t,r],\H[-I])}\left(\jmath_{[t,u)}\gamma_{t}+\xi^{[u,s]}_{u,\cdot}(\gamma_{t}),\jmath_{[t,r]}\gamma_{t}\right)\\
			& \qquad\le\sup_{v\in[u,s]}\bnorm{\xi_{u,v}(\gamma_{t})-\gamma_{t}}_{\H[-I]}+(r-s)\\
			& \qquad\le\sup_{v\in[u,r]}\bnorm{\xi_{u,v}(\gamma_{t})-\gamma_{t}}_{\H[-I]}+(r-u).
		\end{align*}
		Hence, by Remark \ref{rem:convergence_to_zero_of_sup_xi},
		\begin{equation}
			\sup_{s\in[u,r]}\E d^{2}_{\Upsilon([t,r],\H[-I])}\left(\jmath_{[t,u)}\gamma_{t}+\xi^{[u,s]}_{u,\cdot}(\gamma_{t}),\jmath_{[t,r]}\gamma_{t}\right)\to0\label{eq:convergence_for_d_xi}
		\end{equation}
		as $u\to r-$.
		
		Let $u_{k}\le s_{k}\le r$ be arbitrary sequences such that $u_{k}\to r-$.
		Then necessarily $s_{k}\to r-$, and (\ref{eq:convergence_for_d_xi})
		implies that 
		\[
		\jmath_{[t,u_{k})}\gamma_{t}+\xi^{[u_{k},s_{k}]}_{u_{k},\cdot}(\gamma_{t})\to\jmath_{[t,r]}\gamma_{t}
		\]
		in probability in $\Upsilon([t,r],\H[-I])$. Hence, by the left-continuity
		of $\cG^{\Phi}G^{r,\gamma}$, 
		\[
		L^{r}_{s_{k}}(\vartheta^{u_{k}})\to\cG^{\Phi}_{r}U_{r}(\Theta_{t,r}\gamma)
		\]
		in probability.
		
		On the other hand, by (\ref{eq:bound_of_generator}), (\ref{eq:estimate_for_G_r_gamma}),
		and (\ref{eq:bound_for_U_in_cP_norm}), 
		\begin{align*}
			\left|L^{r}_{s}(\vartheta^{u})\right| & \le C_{m_{\partial},m_{\Delta},I}\left(\bnorm{\Phi(s)}^{2}_{L^{2}(\Td;\R^{d})}+\bnorm{\xi_{u,s}(\gamma_{t})}_{\H[-I]}\right)\bnorm F_{\cP^{2}(D([0,T],\H[-I-2m_{\Delta}]))}.
		\end{align*}
		Since $\Phi$ is bounded on $[0,T]$ and, by Proposition \ref{prop:OU_started_at_s},
		\[
		\sup_{t\le u\le s\le r}\E\bnorm{\xi_{u,s}(\gamma_{t})}^{2}_{\H[-I]}<\infty,
		\]
		the family 
		\[
		\left\{ L^{r}_{s}(\vartheta^{u}):t\le u\le s\le r\right\} 
		\]
		is uniformly integrable. Consequently, 
		\[
		\E\left|L^{r}_{s_{k}}(\vartheta^{u_{k}})-\cG^{\Phi}_{r}U_{r}(\Theta_{t,r}\gamma)\right|\to0.
		\]
		Since the sequences $\{u_{k}\}_{k\ge1}$ and $\{s_{k}\}_{k\ge1}$
		were arbitrary, this proves (\ref{eq:convergence_of_L-1}).
		
		Therefore, 
		\[
		\left|\frac{1}{r-u}\int^{r}_{u}\E L^{r}_{s}(\vartheta^{u})\,ds-\cG^{\Phi}_{r}U_{r}(\Theta_{t,r}\gamma)\right|\le\sup_{s\in[u,r]}\E\left|L^{r}_{s}(\vartheta^{u})-\cG^{\Phi}_{r}U_{r}(\Theta_{t,r}\gamma)\right|\to0
		\]
		as $u\to r-$. Together with (\ref{eq:expression_for_derivative_of_f-1-1}),
		this yields 
		\[
		f'(r)=-\cG^{\Phi}_{r}U_{r}(\Theta_{t,r}\gamma)
		\]
		for a.e. $r\in[t,T]$.
		
		Since $f$ is absolutely continuous, we conclude that, for every $r\in[t,T]$,
		\[
		U_{r}(\Theta_{t,r}\gamma)-U_{t}(\gamma)=-\int^{r}_{t}\cG^{\Phi}_{s}U_{s}(\Theta_{t,s}\gamma)ds.
		\]
		Thus, the restriction of $U$ to $\Upsilon([0,T],\H[-I])$ is horizontally
		absolutely continuous, with weak horizontal derivative 
		\[
		\cH U=-\cG^{\Phi}U.
		\]
		In particular, 
		\[
		\cH U(\gamma)+\cG^{\Phi}U(\gamma)=0,\quad\gamma\in\Upsilon([0,T],\H[-I]),
		\]
		and (\ref{eq:absolute_continuity_for_U}) follows. This completes
		the proof.
	\end{proof}
	
	\begin{rem}
		The assumptions of the theorem may not, in general, imply that
		\[
		U\in\Cf^{2}_{l}\left(\Upsilon([0,T],\H[-I-2m_{\Delta}])\right),
		\]
		since the left-continuity of the vertical derivatives of $U$ relies
		on the convergence $\cS_{t}\to\id$ in the operator norm as $t\to0+.$
		However, if for some $\delta>0$,
		\[
		F\in\cP^{2}\left(D([0,T],\H[-I-2m_{\Delta}-\delta])\right),
		\]
		then Lemma \ref{lem:differentiability_of_U} applied with $I$ replaced
		by $I+2m_{\Delta}$, yields
		\[
		U\in\Cf^{2}_{l}\left(\Upsilon([0,T],\H[-I-2m_{\Delta}])\right).
		\]
		Thus, the additional loss of $\delta>0$ is used only to establish
		this stronger regularity property and is not needed for the path-dependent
		backward Kolmogorov equation itself.
	\end{rem}

	\section{Quantitative Gaussian approximation on path space}
	
	This section is devoted to the proof of the quantitative Gaussian
	approximation on path space. The argument combines several ingredients
	developed in the previous sections.
	
	We first compare the generator of the microscopic fluctuation field
	with that of the modified Ornstein--Uhlenbeck dynamics. The resulting
	estimate isolates a correlation correction term and quantifies all
	remaining approximation errors. We then establish a quantitative Gaussian
	approximation for the initial fluctuation field by Stein's method.
	
	These two estimates are combined with the backward Kolmogorov equation
	and a path-dependent Trotter-type argument to prove Theorem~\ref{thm:main_result}.
	In particular, we first obtain the approximation with the correlation
	correction and then estimate this term separately. We conclude the
	section with examples of admissible path-dependent test functions
	and with a lower bound showing that the correlation correction can
	indeed be of order $n^{-2}$.
	
	\subsection{Comparison of the microscopic and Gaussian generators \label{subsec:generator_comparison}}
	
	In this subsection, we compare the generator of the fluctuation field
	with that of the modified Ornstein--Uhlenbeck dynamics introduced
	in Section~\ref{sec:microscopic_and_modified_gaussian_dynamics}.
	The comparison is carried out for functions on a negative Sobolev
	space and provides the main generator estimate used in the proof of
	the quantitative Gaussian approximation.
	
	Let $m_{\tau}\in\N_{0}$ and $m_{\partial},m_{\Delta}\in\N$, with
	$m_{\Delta}$ odd, be fixed, and assume that 
	\[
	I>\frac{d}{2}+m_{\partial}.
	\]
	For $\varrho\in\Cf^{m_{\tau}}(\Td)$, set 
	\[
	\Phi^{n,\varrho}_{j}:=\sqrt{\Gamma^{(m_{\tau})}_{n,j}(\varrho)^{+}},\quad j\in[d],
	\]
	where $\Gamma^{(m_{\tau})}_{n,j}$ was defined in (\ref{eq:definition_of_Gamma_m_tau}).
	
	Motivated by the generator expansion in Proposition~\ref{prop:expansion_of_generator},
	we define, for $F\in\Cf^{2}_{\text{der}}(\H[-I-2m_{\Delta}])$, 
	\begin{equation}
		\cG^{n}_{\varrho}F(\xi):=\frac{1}{2}\left\langle \Delta^{(2m_{\Delta})}_{n}DF(\xi),\xi\right\rangle +\frac{1}{4}\sum^{d}_{j=1}\left\langle (\partial^{(m_{\partial})}_{n,j})^{\odot2}D^{2}F(\xi;\cdot),\Gamma^{(m_{\tau})}_{n,j}(\varrho)^{+}\right\rangle .\label{eq:generator_of_OU_process_with_Gamma}
	\end{equation}
	
	Let now 
	\[
	[0,T]\ni t\mapsto\varrho_{t}\in\Cf^{m_{\tau}}(\Td)
	\]
	be continuous, and set 
	\begin{equation}
		\Phi^{n}(t):=\left(\Phi^{n,\varrho_{t}}_{j}\right)_{j\in[d]},\quad t\in[0,T].\label{eq:definition_of_Phi_n}
	\end{equation}
	Then $\Phi^{n}$ is a continuous $L^{2}(\Td;\R^{d})$-valued function
	and 
	\[
	\cG^{n}_{\varrho_{t}}=\cG^{\Phi^{n}}_{t},\quad t\in[0,T],
	\]
	where $\cG^{\Phi}$ was defined in (\ref{eq:generator_G_Phi}). Consequently,
	$\left(\cG^{n}_{\varrho_{t}}\right)_{t\in[0,T]}$ is the family of
	generators associated with the modified Ornstein--Uhlenbeck dynamics
	\begin{align}
		\partial_{t}\xi_{t} & =\frac{1}{2}\Delta^{(2m_{\Delta})}_{n}\xi_{t}+\frac{1}{\sqrt{2}}\divv^{(m_{\partial})}_{n}\left(\Phi^{n}(t)\dot{W}_{t}\right),\quad t\in[0,T],\label{eq:SPDE_for_OU_process_with_Gamma}\\
		\xi_{0} & =g.\nonumber 
	\end{align}
	
	\begin{rem}
		The coefficient $\Gamma^{(m_{\tau})}_{n,j}(\varrho)$ is the higher-order
		approximation of the microscopic covariance $\Gamma_{n,j}(\rho)$.
		Unlike the shift operator $\tau_{n,j}$, its truncated approximation
		$\tau^{(m_{\tau})}_{n,j}$ is a differential operator and does not
		preserve positivity or the interval $[0,1]$. Moreover, the modified
		profile $\varrho$ need not satisfy a maximum principle. Consequently,
		$\Gamma^{(m_{\tau})}_{n,j}(\varrho)$ need not be non-negative. Since
		it determines the diffusion coefficient in (\ref{eq:SPDE_for_OU_process_with_Gamma}),
		we therefore use $\left(\Gamma^{(m_{\tau})}_{n,j}(\varrho)\right)^{+}$
		in the definition of $\cG^{n}_{\varrho}$. This choice is not essential
		for the estimates below. The same arguments apply with $\left|\Gamma^{(m_{\tau})}_{n,j}(\varrho)\right|$
		in place of $\left(\Gamma^{(m_{\tau})}_{n,j}(\varrho)\right)^{+}$.
		
		We introduce the correction operator 
		\[
		\cC^{n}F(\bar{\zeta}):=\frac{1}{2n^{d}}\sum^{d}_{j=1}\binn{\left(\partial^{(m_{\partial})}_{n,j}\right)^{\odot2}D^{2}F(\bar{\zeta};\cdot),\zeta\tau_{n,j}\zeta}_{n},\quad\zeta\in\Ln,
		\]
		for $F\in\Cf^{2}_{\text{der}}(\H[-I])$ and $I>\frac{d}{2}+m_{\partial}$.
		The next result collects the error terms arising from the comparison
		of the fluctuation generator with the modified Ornstein--Uhlenbeck
		generator. Recall that $ \cL^{n}_{\rho}$ was defined in (\ref{eq:definition_of_cL_rho}).
	\end{rem}
	
	\begin{prop}
		\label{prop:comparison_of_generators}Let $m_{\Delta},m_{\partial},m_{W},M_{1}\in\N$,
		$I>\frac{d}{2}+m_{\partial}$, and $m_{\tau}\in\N_{0}$ be such that
		\[
		M_{1}>\frac{d}{2}+\max\left\{ \lceil I\rceil+2m_{\Delta}+2,\,m_{\partial}+\lceil I\rceil,\,m_{\partial}+m_{W}\right\} ,
		\]
		and let $F\in\Cf^{3}_{\text{der}}(\H[-M_{1}])$. Then, for every $\eta\in\{0,1\}^{\Tnd}$,
		$\rho\in[0,1]^{\Tnd}$, $\zeta:=n^{\frac{d}{2}}(\eta-\rho)$ and $\varrho\in\Cf^{M_{2}}(\Td)$,
		where 
		\[
		M_{2}:=\max\{m_{W},m_{\tau}+1\},
		\]
		one has 
		\[
		\left|\left(\cL^{n}_{\rho}-\cG^{n}_{\varrho}+\cC^{n}\right)F(\bar{\zeta})\right|\le C_{M_{1}}E_{n}\bnorm F_{\Cf^{3}_{\text{der}}(\H[-M_{1}])},
		\]
		where 
		\begin{equation}
			E_{n}:=E^{n}_{\Delta}+E^{n}_{\partial}+E^{n}_{\mathrm{fluc}}+E^{n}_{\mathrm{prof}}+E^{n}_{\mathrm{disc}}+E^{n}_{+}+E^{n}_{\tau}+R^{n},\label{eq:general_error_E_n}
		\end{equation}
		\begin{align*}
			E^{n}_{\Delta} & :=\frac{1}{n^{2m_{\Delta}}}\bnorm{\bar{\zeta}}_{\H[-I]},\qquad E^{n}_{\partial}:=\frac{1}{n^{m_{\partial}}},\\
			E^{n}_{\mathrm{fluc}} & :=\frac{1}{n^{\frac{d}{2}}}\left(1+\norm{\ex_{n}\rho}_{\Cf^{\lceil I\rceil}(\Td)}\right)\norm{\bar{\zeta}}_{\H[-I]},\\
			E^{n}_{\mathrm{prof}} & :=\left(1+\bnorm{\ex_{n}\rho}_{\Cf(\Td)}+\norm{\varrho}_{\Cf(\T^{d})}\right)\bnorm{\ex_{n}\rho-\varrho}_{\Cf(\Td)},\\
			E^{n}_{\mathrm{disc}} & :=\frac{1}{n^{m_{W}}}\bnorm{\varrho}_{\Cf^{m_{W}}(\Td)}\left(1+\bnorm{\varrho}_{\Cf^{m_{W}}(\Td)}\right)\\
			E^{n}_{+} & :=\left(\bnorm{\varrho^{-}}_{\Cf(\Td)}+\bnorm{(1-\varrho)^{-}}_{\Cf(\Td)}\right)\left(1+\norm{\varrho}_{\Cf(\Td)}\right)\\
			E^{n}_{\tau} & :=\frac{1}{n^{m_{\tau}+1}}\norm{\varrho}_{\Cf^{m_{\tau}+1}(\Td)}\left(1+\norm{\varrho}_{\Cf(\Td)}\right)
		\end{align*}
		and
		\[
		R^{n}:=\frac{1}{n^{\frac{d}{2}+1}}.
		\]
	\end{prop}
	
	\begin{proof}
		Let $\eta\in\{0,1\}^{\Tnd}$, $\rho\in[0,1]^{\Tnd}$, $\zeta=n^{d/2}(\eta-\rho)$,
		and $\varrho\in\Cf^{M_{2}}(\Td)$. Set 
		\[
		M_{0}:=\max\left\{ \lceil I\rceil+2m_{\Delta}+2,\,m_{\partial}+\lceil I\rceil,\,m_{\partial}+m_{W}\right\} .
		\]
		Since $M_{1}>\frac{d}{2}+M_{0}$, Lemma \ref{lem:existence_of_smooth_kernel}
		implies that, for every $\ell\in[3]$, the Fréchet derivative $D^{\ell}F(\bar{\zeta})$
		admits a kernel 
		\[
		D^{\ell}F(\bar{\zeta};\cdot)\in\Cf^{M_{0},\ldots,M_{0}}((\Td)^{\ell}),
		\]
		satisfying 
		\[
		\bnorm{D^{\ell}F(\bar{\zeta};\cdot)}_{\Cf^{M_{0},\ldots,M_{0}}((\Td)^{\ell})}\le C_{M_{1}}\bnorm{D^{\ell}F(\bar{\zeta})}_{\cL_{\ell}(\H[-M_{1}])}.
		\]
		In particular, setting 
		\[
		\varphi_{j}:=\left(\partial^{(m_{\partial})}_{n,j}\right)^{\odot2}D^{2}F(\bar{\zeta};\cdot),\quad j\in[d],
		\]
		we have 
		\[
		\varphi_{j}\in\Cf^{\max\{\lceil I\rceil,m_{W}\}}(\Td).
		\]
		
		Using Proposition~\ref{prop:expansion_of_generator}, we decompose
		the difference of the two generators as 
		\begin{align*}
			& \left|\cL^{n}_{\rho}F(\bar{\zeta})-\cG^{n}_{\varrho}F(\bar{\zeta})+\frac{1}{2}\sum^{d}_{j=1}\binn{\left(\partial^{(m_{\partial})}_{n,j}\right)^{\odot2}D^{2}F(\bar{\zeta};\cdot),(\eta-\rho)\tau_{n,j}(\eta-\rho)}_{n}\right|\\
			& \qquad\le\frac{1}{2}|S_{1}|+\frac{1}{4}\sum^{d}_{j=1}\left(|S^{j}_{2}|+|S^{j}_{3}|\right)+R,
		\end{align*}
		where 
		\begin{align*}
			S_{1} & :=\binn{\Delta_{n}DF(\bar{\zeta})-\Delta^{(2m_{\Delta})}_{n}DF(\bar{\zeta}),\bar{\zeta}},\\
			S^{j}_{2} & :=\binn{\left[\left(\partial^{(m_{\partial})}_{n,j}\right)^{\odot2}-\partial^{\odot2}_{n,j}\right]D^{2}F(\bar{\zeta};\cdot),(\tau_{n,j}\eta-\eta)^{2}}_{n},\\
			S^{j}_{3} & :=\binn{\varphi_{j},(\tau_{n,j}\eta-\eta)^{2}}_{n}-\binn{\varphi_{j},\left(\Gamma^{(m_{\tau})}_{n,j}(\varrho)\right)^{+}}+2\binn{\varphi_{j},(\eta-\rho)\tau_{n,j}(\eta-\rho)}_{n},
		\end{align*}
		and 
		\[
		R\le\frac{C_{M_{1}}}{n^{\frac{d}{2}+1}}\bnorm F_{\Cf^{3}_{\text{der}}(\H[-M_{1}])}.
		\]
		The terms $S_{1}$ and $S^{j}_{2}$ correspond to the approximation
		of the drift and covariance operators, respectively, whereas $S^{j}_{3}$
		collects the errors arising from the covariance structure. The term
		$R$ is the third-order Taylor remainder.
		
		Using Lemma~\ref{lem:approximation_of_discrete_derivatives}, we
		obtain 
		\begin{align*}
			|S_{1}| & \le\bnorm{\Delta_{n}DF(\bar{\zeta})-\Delta^{(2m_{\Delta})}_{n}DF(\bar{\zeta})}_{\H[I]}\bnorm{\bar{\zeta}}_{\H[-I]}\\
			& \le\bnorm{\Delta_{n}DF(\bar{\zeta};\cdot)-\Delta^{(2m_{\Delta})}_{n}DF(\bar{\zeta};\cdot)}_{\Cf^{\lceil I\rceil}(\Td)}\bnorm{\bar{\zeta}}_{\H[-I]}\\
			& \le\frac{C_{M_{1}}}{n^{2m_{\Delta}}}\bnorm{DF(\bar{\zeta};\cdot)}_{\Cf^{\lceil I\rceil+2m_{\Delta}+2}(\Td)}\bnorm{\bar{\zeta}}_{\H[-I]}\\
			& \le\frac{C_{M_{1}}}{n^{2m_{\Delta}}}\bnorm{DF(\bar{\zeta})}_{\cL_{1}(\H[-M_{1}])}\bnorm{\bar{\zeta}}_{\H[-I]}.
		\end{align*}
		
		Similarly, since $(\tau_{n,j}\eta-\eta)^{2}\in\{0,1\}^{\Tnd}$, 
		\begin{align*}
			|S^{j}_{2}| & \le\bnorm{\left[\left(\partial^{(m_{\partial})}_{n,j}\right)^{\odot2}-\partial^{\odot2}_{n,j}\right]D^{2}F(\bar{\zeta};\cdot)}_{\Cf((\Td)^{2})}\\
			& \le\bnorm{\left[\left(\partial^{(m_{\partial})}_{n,j}\right)^{\otimes2}-\partial^{\otimes2}_{n,j}\right]D^{2}F(\bar{\zeta};\cdot)}_{\Cf((\Td)^{2})}.
		\end{align*}
		By the triangle inequality and Lemma~\ref{lem:approximation_of_discrete_derivatives},
		\begin{align*}
			|S^{j}_{2}| & \le\bnorm{\left[\left(\partial^{(m_{\partial})}_{n,j}\right)^{\otimes2}-\partial_{n,j}\otimes\partial^{(m_{\partial})}_{n,j}\right]D^{2}F(\bar{\zeta};\cdot)}_{\Cf((\Td)^{2})}\\
			& +\bnorm{\left[\partial_{n,j}\otimes\partial^{(m_{\partial})}_{n,j}-\partial^{\otimes2}_{n,j}\right]D^{2}F(\bar{\zeta};\cdot)}_{\Cf((\Td)^{2})}\\
			& \le\frac{C_{M_{1}}}{n^{m_{\partial}}}\bnorm{D^{2}F(\bar{\zeta};\cdot)}_{\Cf^{m_{\partial}+1,m_{\partial}+1}((\Td)^{2})}\\
			& \le\frac{C_{M_{1}}}{n^{m_{\partial}}}\bnorm{D^{2}F(\bar{\zeta})}_{\cL_{2}(\H[-M_{1}])}.
		\end{align*}
		
		The terms $S^{j}_{3}$ collect the errors arising from the covariance
		structure. To distinguish their different origins, we further decompose
		them into five contributions
		\begin{align*}
			S^{j}_{3} & =\binn{\varphi_{j},\eta+\tau_{n,j}\eta-2\eta\tau_{n,j}\eta}_{n}-\binn{\varphi_{j},\Gamma^{(m_{\tau})}_{n,j}(\varrho)^{+}}\\
			& +2\binn{\varphi_{j},(\eta-\rho)\tau_{n,j}(\eta-\rho)}_{n}\\
			& =\binn{\varphi_{j},\Gamma_{n,j}(\eta)-\Gamma_{n,j}(\rho)+2(\eta-\rho)\tau_{n,j}(\eta-\rho)}_{n}\\
			& +\binn{\varphi_{j},\Gamma_{n,j}(\rho)-\Gamma_{n,j}(\varrho)}_{n}\\
			& +\left[\binn{\varphi_{j},\Gamma_{n,j}(\varrho)}_{n}-\binn{\varphi_{j},\Gamma_{n,j}(\varrho)}\right]\\
			& +\binn{\varphi_{j},\Gamma_{n,j}(\varrho)-\Gamma_{n,j}(\varrho)^{+}}\\
			& +\binn{\varphi_{j},\Gamma_{n,j}(\varrho)^{+}-\Gamma^{(m_{\tau})}_{n,j}(\varrho)^{+}}\\
			& =:S^{j}_{3,1}+\cdots+S^{j}_{3,5},
		\end{align*}
		where $\Gamma_{n,j}$ is extended to functions on $\Td$ by the same
		formula as in (\ref{eq:definition_of_Gamma}).
		
		The terms $S^{j}_{3,1},\ldots,S^{j}_{3,5}$ correspond, respectively,
		to the fluctuation error, the profile approximation error, the noise
		discretization error, the positivity correction error, and the shift
		approximation error. We estimate them separately. Among these terms,
		only $S^{j}_{3,1}$ depends on the microscopic fluctuations. The remaining
		contributions are deterministic approximation errors.
		
		We start with the fluctuation error $S^{j}_{3,1}$. The cancellation
		of the quadratic fluctuation contribution is precisely the reason
		for introducing the correction term in the statement. Using the identity
		\[
		\eta=\rho+n^{-d/2}\zeta,
		\]
		we rewrite $S^{j}_{3,1}$ as 
		\begin{align*}
			S^{j}_{3,1} & =\binn{\varphi_{j},\Gamma_{n,j}\left(\rho+\frac{1}{n^{d/2}}\zeta\right)-\Gamma_{n,j}(\rho)+2(\eta-\rho)\tau_{n,j}(\eta-\rho)}_{n}\\
			& =\frac{1}{n^{\frac{d}{2}}}\binn{\varphi_{j},\zeta+\tau_{n,j}\zeta}_{n}-\frac{2}{n^{\frac{d}{2}}}\binn{\varphi_{j},\rho\tau_{n,j}\zeta+\zeta\tau_{n,j}\rho}_{n}.
		\end{align*}
		Consequently, 
		\begin{align*}
			|S^{j}_{3,1}| & \le\frac{1}{n^{\frac{d}{2}}}\left|\binn{\varphi_{j},\bar{\zeta}+\overline{\tau_{n,j}\zeta}}\right|\\
			& +\frac{2}{n^{\frac{d}{2}}}\left|\binn{\varphi_{j}\ex_{n}\rho,\overline{\tau_{n,j}\zeta}}\right|+\frac{2}{n^{\frac{d}{2}}}\left|\binn{\varphi_{j}\ex_{n}(\tau_{n,j}\rho),\bar{\zeta}}\right|\\
			& \le\frac{2}{n^{\frac{d}{2}}}\left[\bnorm{\varphi_{j}}_{\Cf^{\lceil I\rceil}(\Td)}+\bnorm{\varphi_{j}\ex_{n}\rho}_{\Cf^{\lceil I\rceil}(\Td)}+\bnorm{\varphi_{j}\ex_{n}(\tau_{n,j}\rho)}_{\Cf^{\lceil I\rceil}(\Td)}\right]\bnorm{\bar{\zeta}}_{\H[-I]}\\
			& \le\frac{C_{M_{1}}}{n^{\frac{d}{2}}}\left(1+\bnorm{\ex_{n}\rho}_{\Cf^{\lceil I\rceil}(\Td)}\right)\bnorm{\varphi_{j}}_{\Cf^{\lceil I\rceil}(\Td)}\bnorm{\bar{\zeta}}_{\H[-I]}\\
			& \le\frac{C_{M_{1}}}{n^{\frac{d}{2}}}\left(1+\bnorm{\ex_{n}\rho}_{\Cf^{\lceil I\rceil}(\Td)}\right)\bnorm{D^{2}F(\bar{\zeta})}_{\cL_{2}(\H[-M_{1}])}\bnorm{\bar{\zeta}}_{\H[-I]}.
		\end{align*}
		
		By Lemma~\ref{lem:bounds_for_Gamma}~(i) with $\ell=0$, the profile
		approximation error $S^{j}_{3,2}$ is bounded by 
		\begin{align*}
			|S^{j}_{3,2}| & =\left|\binn{\varphi_{j},\Gamma_{n,j}(\ex_{n}\rho)-\Gamma_{n,j}(\varrho)}_{n}\right|\\
			& \le\bnorm{\varphi_{j}}_{\Cf(\Td)}\bnorm{\Gamma_{n,j}(\ex_{n}\rho)-\Gamma_{n,j}(\varrho)}_{\Cf(\Td)}\\
			& \le C_{M_{1}}\bnorm{D^{2}F(\bar{\zeta})}_{\cL_{2}(\H[-M_{1}])}\left(1+\bnorm{\ex_{n}\rho}_{\Cf(\Td)}+\bnorm{\varrho}_{\Cf(\Td)}\right)\bnorm{\ex_{n}\rho-\varrho}_{\Cf(\Td)}.
		\end{align*}
		Thus, the profile approximation error is controlled by the accuracy
		with which $\varrho$ approximates $\ex_{n}\rho$.
		
		The noise discretization error $S^{j}_{3,3}$ is estimated using Lemma
		\ref{lem:bounds_for_Gamma}~(iii): 
		\begin{align*}
			|S^{j}_{3,3}| & =\left|\binn{\varphi_{j},\Gamma_{n,j}(\varrho)}_{n}-\binn{\varphi_{j},\Gamma_{n,j}(\varrho)}\right|\\
			& \le\frac{C_{M_{1}}}{n^{m_{W}}}\bnorm{\varphi_{j}}_{\Cf^{m_{W}}(\Td)}\bnorm{\varrho}_{\Cf^{m_{W}}(\Td)}\left(1+\bnorm{\varrho}_{\Cf^{m_{W}}(\Td)}\right)\\
			& \le\frac{C_{M_{1}}}{n^{m_{W}}}\bnorm{D^{2}F(\bar{\zeta})}_{\cL_{2}(\H[-M_{1}])}\bnorm{\varrho}_{\Cf^{m_{W}}(\Td)}\left(1+\bnorm{\varrho}_{\Cf^{m_{W}}(\Td)}\right).
		\end{align*}
		This is the only term in which the discrete inner product is approximated
		by its continuous counterpart. Consequently, the noise discretization
		error is controlled by the $C^{m_{W}}$-regularity of $\varrho$.
		
		Using Lemma~\ref{lem:bounds_for_Gamma}~(iv), we estimate the positivity
		correction error $S^{j}_{3,4}$ as follows: 
		\begin{align*}
			|S^{j}_{3,4}| & =\left|\binn{\varphi_{j},\Gamma_{n,j}(\varrho)-\Gamma_{n,j}(\varrho)^{+}}\right|\\
			& \le\bnorm{\varphi_{j}}_{\Cf(\Td)}\bnorm{\Gamma_{n,j}(\varrho)-\Gamma_{n,j}(\varrho)^{+}}_{\Cf(\Td)}\\
			& =\bnorm{\varphi_{j}}_{\Cf(\Td)}\bnorm{\Gamma_{n,j}(\varrho)^{-}}_{\Cf(\Td)}\\
			& \le C_{M_{1}}\bnorm{D^{2}F(\bar{\zeta})}_{\cL_{2}(\H[-M_{1}])}\left(\bnorm{\varrho^{-}}_{\Cf(\Td)}+\bnorm{(1-\varrho)^{-}}_{\Cf(\Td)}\right)\left(1+\bnorm{\varrho}_{\Cf(\Td)}\right).
		\end{align*}
		
		Finally, we estimate the shift approximation error $S^{j}_{3,5}$.
		By Lemma~\ref{lem:bounds_for_Gamma}~(ii) with $\ell=0$ and the
		Lipschitz continuity of the map $x\mapsto x^{+}$, we obtain 
		\begin{align*}
			|S^{j}_{3,5}| & =\left|\binn{\varphi_{j},\Gamma_{n,j}(\varrho)^{+}-\Gamma^{(m_{\tau})}_{n,j}(\varrho)^{+}}\right|\\
			& \le\bnorm{\varphi_{j}}_{\Cf(\Td)}\bnorm{\Gamma_{n,j}(\varrho)^{+}-\Gamma^{(m_{\tau})}_{n,j}(\varrho)^{+}}_{\Cf(\Td)}\\
			& \le\bnorm{\varphi_{j}}_{\Cf(\Td)}\bnorm{\Gamma_{n,j}(\varrho)-\Gamma^{(m_{\tau})}_{n,j}(\varrho)}_{\Cf(\Td)}\\
			& \le\frac{C_{M_{1}}}{n^{m_{\tau}+1}}\bnorm{D^{2}F(\bar{\zeta})}_{\cL_{2}(\H[-M_{1}])}\bnorm{\varrho}_{\Cf^{m_{\tau}+1}(\Td)}\left(1+\bnorm{\varrho}_{\Cf(\Td)}\right).
		\end{align*}
		
		Combining the estimates for $S_{1}$, $S^{j}_{2}$, $S^{j}_{3,1},\ldots,S^{j}_{3,5}$,
		$j\in[d]$, together with the bound on the remainder $R$, completes
		the proof.
	\end{proof}
	
	We conclude this subsection with a direct consequence of Proposition~\ref{prop:comparison_of_generators}
	and the definition of the space $\cP^{3}\left(\Upsilon([0,T],\H[-M_{1}])\right)$.
	For a non-anticipative functional $F$, we define 
	\begin{equation}
		\cL^{n}_{\rho}F(\gamma):=\cL^{n}_{\rho}F(\gamma^{[0,t)};\cdot)(\gamma_{t}),\quad\text{and}\quad\cC^{n}F(\gamma):=\cC^{n}F(\gamma^{[0,t)};\cdot)(\gamma_{t}),\label{eq:extension_of_cC_and_cL}
	\end{equation}
	for $(t,\gamma)\in\Upsilon([0,T],\H[-M_{1}])$.
	\begin{cor}
		\label{cor:pathwise_comparison_of_generators} Let $I,m_{\Delta},m_{\partial},m_{W},m_{\tau},M_{1},M_{2}$
		be as in Proposition~\ref{prop:comparison_of_generators}, let $T>0$,
		and let 
		\[
		F\in\cP^{3}\left(\Upsilon([0,T],\H[-M_{1}])\right).
		\]
		Then, for every 
		\[
		\eta\in D\left([0,T],\{0,1\}^{\Tnd}\right),\quad\rho\in D\left([0,T],[0,1]^{\Tnd}\right),
		\]
		and 
		\[
		\varrho\in D\left([0,T],\Cf^{M_{2}}(\Td)\right),
		\]
		setting 
		\[
		\zeta:=n^{\frac{d}{2}}(\eta-\rho),\quad(\bar{\zeta})_{t}:=\overline{\zeta_{t}},
		\]
		one has 
		\[
		\sup_{t\in[0,T]}\left|\left(\cL^{n}_{\rho_{t}}-\cG^{n}_{\varrho_{t}}+\cC^{n}\right)F\left(\bar{\zeta}^{[0,t]}\right)\right|\le C_{M_{1}}\sup_{t\in[0,T]}E_{n}(t)\,\bnorm F_{\cP^{3}\left(\Upsilon([0,T],\H[-M_{1}])\right)},
		\]
		where $E_{n}(t)$ is given by (\ref{eq:general_error_E_n}) with $(\rho,\bar{\zeta},\varrho)$
		replaced by $(\rho_{t},\bar{\zeta}_{t},\varrho_{t})$.
	\end{cor}
	
	\begin{proof}
		Fix $t\in[0,T]$ and $\gamma\in D([0,t],\H[-M_{1}])$.
		The terminal section
		\[
		f_{t,\gamma}(g):=F_{t}(\gamma^{[0,t)},g),
		\quad g\in\H[-M_{1}],
		\]
		belongs to $\Cf^{3}_{\text{der}}(\H[-M_{1}])$, and
		\[
		\bnorm{f_{t,\gamma}}_{\Cf^{3}_{\text{der}}(\H[-M_{1}])}
		\le
		\bnorm{F_{t}}_{\cP^{3}(D([0,t],\H[-M_{1}]))}
		\le
		\bnorm F_{\cP^{3}(\Upsilon([0,T],\H[-M_{1}]))}.
		\]
		Applying Proposition~\ref{prop:comparison_of_generators} to this
		section with $\gamma=\bar{\zeta}^{[0,t]}$, we obtain
		\[
		\left|
		\left(\cL^{n}_{\rho_{t}}-\cG^{n}_{\varrho_{t}}+\cC^{n}\right)
		F\left(\bar{\zeta}^{[0,t]}\right)
		\right|
		\le
		C_{M_{1}}E_{n}(t)\,
		\bnorm F_{\cP^{3}(\Upsilon([0,T],\H[-M_{1}]))}.
		\]
		Taking the supremum over $t\in[0,T]$ completes the proof.
	\end{proof}

	\subsection{Gaussian approximation of the initial fluctuations}
	
	In this subsection, we compare the law of the initial fluctuation
	field with its Gaussian approximation. Since the particles are initially
	independent, the problem reduces to a quantitative central limit theorem
	for an empirical measure-valued random variable. Our approach is based
	on Stein's method using an exchangeable-pair construction. The argument
	follows the proof of \citep[Proposition 4.1]{Gess_Konarovskyi2026},
	adapted to the empirical measure framework considered here.
	
	We will use a finite-dimensional version of Stein's method. We first
	recall the following result from \citep[Theorem 3]{Meckes:2009}.
	Let $(X,X')$ be an exchangeable pair of random variables in $\R^{q}$,
	that is, 
	\[
	\law(X,X')=\law(X',X),
	\]
	such that 
	\[
	\E[X'-X\mid X]=-\Lambda X
	\]
	and 
	\[
	\E\left[(X'-X)(X'-X)^{T}\mid X\right]=2\Lambda\Sigma+E'
	\]
	for some invertible matrix $\Lambda$, a symmetric non-negative definite
	matrix $\Sigma$, and a $\sigma(X)$-measurable random matrix $E'$.
	Then, for every $f\in\Cf^{3}(\R^{q})$, 
	\begin{align}
		&\left|\E f(X)-\E f(Z)\right|\nonumber\\
		&\qquad\le\left\Vert \Lambda^{-1}\right\Vert _{op}\bigg[\frac{1}{4}\sup_{z\in\R^{q}}\bnorm{D^{2}f}_{HS}\E\bnorm{E'}_{HS}+\frac{1}{9}\sup_{z\in\R^{q}}\bnorm{D^{3}f}_{op}\E\left|X'-X\right|^{3}\bigg]\label{eq:berry_essen_bound}
	\end{align}
	where $Z$ is a centered Gaussian random variable in $\R^{q}$ with
	covariance matrix $\Sigma$. Here $\|\cdot\|_{HS}$ and $\|\cdot\|_{op}$
	denote the Hilbert--Schmidt and operator norms, respectively. We
	also denote by $\|\cdot\|_{\cL_{k}(\H[-I])}$ the norm on the space
	of bounded $k$-linear forms on $\H[-I]$.
	
	We now state the quantitative Gaussian approximation for the initial
	fluctuation field.
	\begin{prop}
		\label{prop:berry_essen_bound}Let $m_{0}\in\N$, $I>\frac{d}{2}+m_{0}$,
		and $M_{0}\in\N$ be such that 
		\[
		M_{0}>\frac{d}{2}+2\lceil I/2\rceil.
		\]
		Let $\varrho\in\Cf^{m_{0}}(\Td;[0,1])$, and let $\xi$ be a centered Gaussian
		random variable in $\H[-I]$ satisfying 
		\[
		\E\left[\binn{\xi,\varphi}\binn{\xi,\psi}\right]=\binn{\varrho(1-\varrho)\varphi,\psi},\quad\varphi,\psi\in\H[I].
		\]
		Let $\eta^{n}(x)$, $x\in\Tnd$, be independent Bernoulli random variables
		with parameters $\varrho(x)$, respectively, and set 
		\[
		\rho^{n}:=\varrho\big|_{\Tnd},\quad\zeta^{n}:=n^{\frac{d}{2}}\left(\eta^{n}-\rho^{n}\right).
		\]
		Then, for every $F\in\Cf^{3}_{\text{der}}(\H[-M_{0}])$, 
		\[
		\left|\E F\left(\bar{\zeta}^{n}\right)-\E F(\xi)\right|\le C_{M_{0},m_{0}}\|F\|_{\Cf^{3}_{\text{der}}(\H[-M_{0}])}\left[\frac{1}{n^{d/2}}+\frac{1}{n^{m_{0}}}\left(1+\|\varrho\|^{2}_{\Cf^{m_{0}}(\Td)}\right)\right].
		\]
	\end{prop}
	
	\begin{proof}
		For simplicity of notation, we omit the superscript $n$ in $\zeta^{n}$,
		$\eta^{n}$, and $\rho^{n}$. Since the version of Stein's method
		used here is finite-dimensional, we first project the random fields
		onto finitely many Fourier modes. Fix an integer $M\ge n$ and set
		\[
		\pr_{M}f=\sum_{k\in\Z^{d}_{M}}\binn{f,\varsigma_{k}}\varsigma_{k},\quad f\in\H[-I].
		\]
		We will estimate 
		\[
		\left|\E F\left(\pr_{M}\bar{\zeta}\right)-\E F\left(\pr_{M}\xi\right)\right|.
		\]
		
		To apply (\ref{eq:berry_essen_bound}) to the Fourier coefficients
		of $\pr_{M}\bar{\zeta}$, we work with a real trigonometric basis
		$\{\tilde{\varsigma}_{k}\}_{k\in\Zd}$, obtained from the complex
		Fourier basis by taking sine and cosine components and chosen to be
		orthonormal in $L^{2}(\Td)$. Then 
		\[
		\pr_{M}\bar{\zeta}=\sum_{k\in\Z^{d}_{M}}\binn{\bar{\zeta},\tilde{\varsigma}_{k}}\tilde{\varsigma}_{k}=\sum_{k\in\Z^{d}_{M}}\binn{\zeta,\tilde{\varsigma}_{k}}_{n}\tilde{\varsigma}_{k}.
		\]
		Introduce 
		\[
		\tilde{\varsigma}_{k}':=\left(1+|k|^{2}\right)^{I/2}\tilde{\varsigma}_{k},\quad\tilde{\varsigma}_{k}'':=\left(1+|k|^{2}\right)^{-I/2}\tilde{\varsigma}_{k},\quad k\in\Zd.
		\]
		Then $\{\tilde{\varsigma}_{k}'\}_{k\in\Zd}$ is an orthonormal basis
		of $\H[-I]$, and 
		\[
		\pr_{M}\bar{\zeta}=\sum_{k\in\Z^{d}_{M}}\binn{\bar{\zeta},\tilde{\varsigma}_{k}''}\tilde{\varsigma}_{k}'=\sum_{k\in\Z^{d}_{M}}\binn{\zeta,\tilde{\varsigma}_{k}''}_{n}\tilde{\varsigma}_{k}'.
		\]
		We identify $\pr_{M}\bar{\zeta}$ with its coordinate vector with
		respect to this basis and define 
		\[
		X_{k}:=\binn{\zeta,\tilde{\varsigma}_{k}''}_{n},\quad k\in\Z^{d}_{M}.
		\]
		Thus, $X=(X_{k})_{k\in\Z^{d}_{M}}$ is a random vector in $\R^{\Z^{d}_{M}}$
		satisfying 
		\[
		|X|^{2}=\bnorm{\pr_{M}\bar{\zeta}}^{2}_{\H[-I]}.
		\]
		
		We now construct an exchangeable pair $(X,X')$. Let $\tilde{\zeta}$
		be an independent copy of $\zeta$, and let $U$ be uniformly distributed
		on $\Tnd$, independently of $\zeta$ and $\tilde{\zeta}$. Define
		\[
		\zeta'=\zeta-\zeta(U)\I_{\{U\}}+\tilde{\zeta}(U)\I_{\{U\}},
		\]
		and 
		\[
		X_{k}':=\binn{\zeta',\tilde{\varsigma}_{k}''}_{n}=X_{k}+\frac{1}{n^{d}}\left(\tilde{\zeta}(U)-\zeta(U)\right)\tilde{\varsigma}_{k}''(U),\quad k\in\ZdM,
		\]
		where $\I_{A}$ denotes the indicator function of $A$.
		
		We next introduce a finite-dimensional test function. Define $f:\R^{\ZdM}\to\R$
		by 
		\[
		f(z):=F(\kappa(z)),\quad\kappa(z):=\sum_{k\in\ZdM}z_{k}\tilde{\varsigma}_{k}',\quad z=(z_{k})_{k\in\ZdM}.
		\]
		Then 
		\[
		F(\pr_{M}\bar{\zeta})=f(X).
		\]
		Moreover, $f\in\Cf^{3}(\R^{\ZdM})$, and, for $k,l\in\ZdM$, 
		\[
		\partial_{z_{k}}f(z)=DF(\kappa(z))[\tilde{\varsigma}_{k}']
		\]
		and 
		\[
		\partial_{z_{k}}\partial_{z_{l}}f(z)=D^{2}F(\kappa(z))[\tilde{\varsigma}_{k}',\tilde{\varsigma}_{l}'].
		\]
		
		Since $M\ge n$, the Fourier coefficients contained in $X$ determine
		the lattice field $\zeta$. In particular, $\zeta$ is $\sigma(X)$-measurable.
		Therefore, for every $k\in\ZdM$, 
		\[
		\E[X_{k}'-X_{k}\mid X]=-\frac{1}{n^{d}}X_{k}.
		\]
		Furthermore, for $k,l\in\ZdM$, 
		\[
		\E\left[(X_{k}'-X_{k})(X_{l}'-X_{l})\mid X\right]=\frac{1}{n^{d}}\binn{\rho(1-\rho)\tilde{\varsigma}_{k}'',\tilde{\varsigma}_{l}''}_{n}+\frac{1}{n^{2d}}\binn{\zeta^{2}\tilde{\varsigma}_{k}'',\tilde{\varsigma}_{l}''}_{n}.
		\]
		
		The covariance matrix $\Sigma=(\Sigma_{k,l})_{k,l\in\ZdM}$ of the
		Gaussian vector 
		\[
		Z:=\left(\binn{\pr_{M}\xi,\tilde{\varsigma}_{k}''}\right)_{k\in\ZdM}=\left(\binn{\xi,\tilde{\varsigma}_{k}''}\right)_{k\in\ZdM}
		\]
		is given by 
		\[
		\Sigma_{k,l}=\binn{\varrho(1-\varrho)\tilde{\varsigma}_{k}'',\tilde{\varsigma}_{l}''}.
		\]
		Consequently, 
		\[
		\E\left[(X_{k}'-X_{k})(X_{l}'-X_{l})\mid X\right]=\frac{2}{n^{d}}\Sigma_{k,l}+\frac{1}{n^{d}}E^{*}_{k,l},
		\]
		where 
		\begin{align*}
			E^{*}_{k,l} & :=\binn{\left((\eta-\rho)^{2}-\rho(1-\rho)\right)\tilde{\varsigma}_{k}'',\tilde{\varsigma}_{l}''}_{n}\\
			& +2\left[\binn{\rho(1-\rho)\tilde{\varsigma}_{k}'',\tilde{\varsigma}_{l}''}_{n}-\binn{\varrho(1-\varrho)\tilde{\varsigma}_{k}'',\tilde{\varsigma}_{l}''}\right]\\
			& =:E^{*}_{1,k,l}+2E^{*}_{2,k,l}.
		\end{align*}
		In particular, $E^{*}$ is $\sigma(X)$-measurable.
		
		Applying (\ref{eq:berry_essen_bound}) with 
		\[
		\Lambda:=\frac{1}{n^{d}}\mathrm{Id},\quad E':=\frac{1}{n^{d}}E^{*},
		\]
		where $\mathrm{Id}$ denotes the identity matrix on $\R^{\ZdM}$,
		we obtain 
		\begin{align}
			\left|\E F(\pr_{M}\bar{\zeta})-\E F(\pr_{M}\xi)\right| & =\left|\E f(X)-\E f(Z)\right|\nonumber \\
			& \le\frac{1}{4}\sup_{z\in\R^{\ZdM}}\|\nabla^{2}f(z)\|_{HS}\E\|E^{*}\|_{HS}\label{eq:berry_essen_bound_application}\\
			&+\frac{n^{d}}{9}\sup_{z\in\R^{\ZdM}}\|\nabla^{3}f(z)\|_{op}\E|X'-X|^{3}.\nonumber 
		\end{align}
		
		We now estimate the three quantities appearing on the right-hand side
		of (\ref{eq:berry_essen_bound_application}): the Hessian term, the
		third-order term, and the covariance error $E^{*}$.
		
		Set $m:=\left\lceil \frac{I}{2}\right\rceil .$ By Lemma~\ref{lem:existence_of_smooth_kernel},
		$D^{2}F(\kappa(z))$ admits a kernel 
		\[
		D^{2}F(\kappa(z);\cdot)\in\Cf^{2m,2m}((\Td)^{2}).
		\]
		Hence, 
		\begin{align*}
			\|\nabla^{2}f(z)\|^{2}_{HS} & =\sum_{k,l\in\ZdM}\left(D^{2}F(\kappa(z))[\tilde{\varsigma}_{k}',\tilde{\varsigma}_{l}']\right)^{2}\\
			& \le\sum_{k,l\in\Zd}(1+|k|^{2})^{I}(1+|l|^{2})^{I}\left[\int_{\Td}\int_{\Td}D^{2}F(\kappa(z);x,y)\tilde{\varsigma}_{k}(x)\tilde{\varsigma}_{l}(y)dxdy\right]^{2}\\
			& \le C_{I}\sum_{k,l\in\Zd}\left[\int_{\Td}\int_{\Td}\bigl((1-\Delta)^{m}\otimes(1-\Delta)^{m}\bigr)D^{2}F(\kappa(z);x,y)\tilde{\varsigma}_{k}(x)\tilde{\varsigma}_{l}(y)dxdy\right]^{2}\\
			& =C_{I}\int_{\Td}\int_{\Td}\left|\bigl((1-\Delta)^{m}\otimes(1-\Delta)^{m}\bigr)D^{2}F(\kappa(z);x,y)\right|^{2}dxdy\\
			& \le C_{I}\|D^{2}F(\kappa(z);\cdot)\|^{2}_{\Cf^{2m,2m}((\Td)^{2})}\\
			& \le C_{M_{0}}\|D^{2}F(\kappa(z))\|^{2}_{\cL_{2}(\H[-M_{0}])}.
		\end{align*}
		
		We next estimate the third-order term in (\ref{eq:berry_essen_bound_application}).
		For $a_{1},a_{2},a_{3}\in\R^{\ZdM}$, 
		\[
		D^{3}f(z)[a_{1},a_{2},a_{3}]=D^{3}F(\kappa(z))[\kappa(a_{1}),\kappa(a_{2}),\kappa(a_{3})].
		\]
		Since $\kappa:\R^{\ZdM}\to\H[-I]$ is an isometry onto its range and
		$\H[-I]\hookrightarrow\H[-M_{0}]$, we have 
		\[
		\|\kappa(a)\|_{\H[-M_{0}]}\le\|\kappa(a)\|_{\H[-I]}=|a|.
		\]
		Consequently, 
		\[
		\|D^{3}f(z)\|_{op}\le\|D^{3}F(\kappa(z))\|_{\cL_{3}(\H[-M_{0}])}.
		\]
		
		Furthermore, 
		\begin{align*}
			\E|X'-X|^{3} & =\frac{1}{n^{3d}}\E\left[|\tilde{\zeta}(U)-\zeta(U)|^{3}\left(\sum_{k\in\ZdM}|\tilde{\varsigma}_{k}''(U)|^{2}\right)^{3/2}\right]\\
			& =\frac{1}{n^{3d}}\E\left[|\tilde{\zeta}(U)-\zeta(U)|^{3}\left(\sum_{k\in\ZdM}\frac{|\tilde{\varsigma}_{k}(U)|^{2}}{(1+|k|^{2})^{I}}\right)^{3/2}\right]\\
			& \le\frac{C_{I,d}}{n^{3d/2}},
		\end{align*}
		where we used $I>d/2$, the uniform boundedness of the trigonometric
		basis, and 
		\[
		|\tilde{\zeta}(U)-\zeta(U)|\le2n^{d/2}.
		\]
		
		It remains to estimate the covariance error $E^{*}=E^{*}_{1}+2E^{*}_{2}$,
		which we decompose into the random fluctuation contribution $E^{*}_{1}$
		and the deterministic discretization error $E^{*}_{2}$. By Jensen's
		inequality, 
		\[
		\E\|E^{*}\|_{HS}\le\left(\E\|E^{*}_{1}\|^{2}_{HS}\right)^{1/2}+2\|E^{*}_{2}\|_{HS}.
		\]
		
		We first estimate $E^{*}_{1}$. Since the random variables 
		\[
		(\eta(x)-\rho(x))^{2}-\rho(x)(1-\rho(x)),\quad x\in\Tnd,
		\]
		are independent, centered, and uniformly bounded, we obtain 
		\begin{align*}
			\E\|E^{*}_{1}\|^{2}_{HS} & =\sum_{k,l\in\ZdM}\E\left[\binn{\left((\eta-\rho)^{2}-\rho(1-\rho)\right)\tilde{\varsigma}_{k}'',\tilde{\varsigma}_{l}''}^{2}_{n}\right]\\
			& =\frac{1}{n^{2d}}\sum_{k,l\in\ZdM}\frac{1}{(1+|k|^{2})^{I}}\frac{1}{(1+|l|^{2})^{I}}\\
			& \qquad\qquad\cdot\E\left[\left(\sum_{x\in\Tnd}\left((\eta(x)-\rho(x))^{2}-\rho(x)(1-\rho(x))\right)\tilde{\varsigma}_{k}(x)\tilde{\varsigma}_{l}(x)\right)^{2}\right]\\
			& =\frac{1}{n^{2d}}\sum_{k,l\in\ZdM}\frac{1}{(1+|k|^{2})^{I}}\frac{1}{(1+|l|^{2})^{I}}\\
			& \qquad\qquad\cdot\sum_{x\in\Tnd}\E\left[\left((\eta(x)-\rho(x))^{2}-\rho(x)(1-\rho(x))\right)^{2}\tilde{\varsigma}^{2}_{k}(x)\tilde{\varsigma}^{2}_{l}(x)\right]\\
			& \le\frac{C}{n^{d}}\sum_{k,l\in\Zd}\frac{1}{(1+|k|^{2})^{I}}\frac{1}{(1+|l|^{2})^{I}}\\
			& \le\frac{C_{I,d}}{n^{d}},
		\end{align*}
		where in the last inequality we used $I>d/2$.
		
		We next estimate $E^{*}_{2}$. By Lemma~\ref{lem:approximation_for_ex},
		\begin{align*}
			& \left|\binn{\rho(1-\rho)\tilde{\varsigma}_{k},\tilde{\varsigma}_{l}}_{n}-\binn{\varrho(1-\varrho)\tilde{\varsigma}_{k},\tilde{\varsigma}_{l}}\right|\\
			& \qquad=\left|\binn{\varrho(1-\varrho)\tilde{\varsigma}_{k}\tilde{\varsigma}_{l},\varsigma_{0}}_{n}-\binn{\varrho(1-\varrho)\tilde{\varsigma}_{k}\tilde{\varsigma}_{l},\varsigma_{0}}\right|\\
			& \qquad\le\frac{C_{m_{0},d}}{n^{m_{0}}}\left\Vert \varrho(1-\varrho)\tilde{\varsigma}_{k}\tilde{\varsigma}_{l}\right\Vert _{\Cf^{m_{0}}(\Td)}\\
			& \qquad\le\frac{C_{m_{0},d}}{n^{m_{0}}}(1+|k|^{m_{0}})(1+|l|^{m_{0}})\left(1+\|\varrho\|^{2}_{\Cf^{m_{0}}(\Td)}\right).
		\end{align*}
		Consequently, since $I>\frac{d}{2}+m_{0}$, 
		\begin{align*}
			\|E^{*}_{2}\|^{2}_{HS} & \le\frac{C_{m_{0},d}}{n^{2m_{0}}}\sum_{k,l\in\ZdM}\frac{(1+|k|^{m_{0}})^{2}(1+|l|^{m_{0}})^{2}}{(1+|k|^{2})^{I}(1+|l|^{2})^{I}}\left(1+\|\varrho\|^{2}_{\Cf^{m_{0}}(\Td)}\right)^{2}\\
			& \le\frac{C_{I,m_{0},d}}{n^{2m_{0}}}\left(1+\|\varrho\|^{2}_{\Cf^{m_{0}}(\Td)}\right)^{2}.
		\end{align*}
		
		Combining the estimates for $\nabla^{2}f$, $\nabla^{3}f$, $X'-X$,
		$E^{*}_{1}$, and $E^{*}_{2}$, we obtain 
		\begin{align*}
			& \left|\E F\left(\pr_{M}\bar{\zeta}\right)-\E F\left(\pr_{M}\xi\right)\right|\\
			& \quad\le C_{M_{0}}\sup_{g\in\H[-M_{0}]}\|D^{2}F(g)\|_{\cL_{2}(\H[-M_{0}])}\left[\frac{1}{n^{d/2}}+\frac{1}{n^{m_{0}}}\left(1+\|\varrho\|^{2}_{\Cf^{m_{0}}(\Td)}\right)\right]\\
			& \quad+\frac{C_{M_{0}}}{n^{d/2}}\sup_{g\in\H[-M_{0}]}\|D^{3}F(g)\|_{\cL_{3}(\H[-M_{0}])}.
		\end{align*}
		The right-hand side is independent of $M$. Moreover, 
		\[
		\pr_{M}\bar{\zeta}\to\bar{\zeta}\quad\text{and}\quad\pr_{M}\xi\to\xi\quad\text{in }\H[-I]
		\]
		as $M\to\infty$. Since $M_{0}>I$, the embedding $\H[-I]\hookrightarrow\H[-M_{0}]$
		is continuous, and hence the same convergences hold in $\H[-M_{0}]$.
		
		Since $DF$ is bounded, $F$ has at most linear growth on $\H[-M_{0}]$.
		Moreover, the families 
		\[
		\left\{ \|\pr_{M}\bar{\zeta}\|_{\H[-M_{0}]}\right\} _{M\ge n}\quad\text{and}\quad\left\{ \|\pr_{M}\xi\|_{\H[-M_{0}]}\right\} _{M\ge n}
		\]
		are bounded in $L^{2}(\Omega)$. Consequently, $\{F(\pr_{M}\bar{\zeta})\}_{M\ge n}$
		and $\{F(\pr_{M}\xi)\}_{M\ge n}$ are uniformly integrable. Letting
		$M\to\infty$ therefore yields 
		\[
		\left|\E F(\bar{\zeta})-\E F(\xi)\right|\le C_{M_{0},m_{0}}\|F\|_{\Cf^{3}_{\text{der}}(\H[-M_{0}])}\left[\frac{1}{n^{d/2}}+\frac{1}{n^{m_{0}}}\left(1+\|\varrho\|^{2}_{\Cf^{m_{0}}(\Td)}\right)\right],
		\]
		which proves the claim.
	\end{proof}

	\subsection{Proof of Theorem \ref{thm:main_result}}
	
	We split the proof of Theorem \ref{thm:main_result} into two parts.
	First, using the backward Kolmogorov equation and a path-dependent
	Trotter-type argument, we prove (\ref{eq:main_estimate_with_correlaton_term}),
	which compares the SSEP fluctuation process with its Gaussian approximation
	up to a correlation correction term. The second part estimates this
	correction term and thereby yields (\ref{eq:main_estimate_without_correlation}).
	
	\subsubsection{Proof of estimate (\ref{eq:main_estimate_with_correlaton_term})}
	
	We now turn to the proof of Theorem~\ref{thm:main_result}. Throughout
	this section, we fix $T>0$ and the parameters appearing in the statement
	of the theorem. Let $\rhoini\in\Cf^{M_{2}}(\Td;[0,1])$ be the initial
	density profile. For each $n\ge3$, let $(\eta^{n}_{t})_{t\ge0}$
	be the SSEP on $\Tnd$ such that $\eta^{n}_{0}(x)$, $x\in\Tnd$,
	are independent Bernoulli random variables with parameters $\rhoini(x)$.
	We set 
	\[
	\rho^{n}_{t}:=\E\eta^{n}_{t},\quad t\ge0,
	\]
	so that $(\rho^{n}_{t})_{t\ge0}$ solves the discrete heat equation
	(\ref{eq:discrete_heat_equation}) with initial condition 
	\[
	\rho^{n}_{0}=\rhoini|_{\Tnd}.
	\]
	The corresponding fluctuation field is 
	\[
	\zeta^{n}_{t}:=n^{d/2}\left(\eta^{n}_{t}-\rho^{n}_{t}\right),\quad t\ge0,
	\]
	and we denote by $\bar{\zeta}^{n}_{t}$ its canonical embedding into
	$\H[-I]$ via the empirical distribution.
	
	We further recall that $(\varrho^{n}_{t})_{t\ge0}$ is the solution
	to the modified heat equation (\ref{eq:modified_heat_equation-1}).
	In particular, $(\varrho^{n}_{t})_{t\ge0}$ provides a higher-order
	approximation of the smooth interpolation $(\ex_{n}\rho^{n}_{t})_{t\ge0}$.
	More precisely, by Lemma~\ref{lem:bound_for_varrho_and_rho}, there
	exists a constant $C_{m_{\Delta},\ell,T}>0$ such that 
	\begin{equation}
		\sup_{t\in[0,T]}\bnorm{\ex_{n}\rho^{n}_{t}-\varrho^{n}_{t}}_{\Cf^{\ell}(\Td)}\le\frac{C_{m_{\Delta},\ell,T}}{n^{2m_{\Delta}}}\bnorm{\rhoini}_{\Cf^{\ell+d+2m_{\Delta}+3}(\Td)}\label{eq:estimate_of_difference_of_ex_nrho^n_and_varrho_n}
	\end{equation}
	for all $n\ge3$.
	
	Let $(\xi^{n}_{t})_{t\ge0}$ be the Gaussian process defined by (\ref{eq:SPDE_for_OU_process_with_Gamma-1}),
	with coefficients determined by $(\varrho^{n}_{t})_{t\ge0}$. Its
	initial condition $\xiini$ is a centered Gaussian random variable
	in $\H[-I]$ with covariance 
	\[
	\E\left[\binn{\xiini,\varphi}\binn{\xiini,\psi}\right]=\binn{\rhoini(1-\rhoini)\varphi,\psi},\quad\varphi,\psi\in\H[I].
	\]
	
	Fix 
	\[
	F\in\cP^{3}\left(D([0,T],\H[-M_{1}])\right).
	\]
	For $(t,\gamma)\in\Upsilon([0,T],\H[-M_{1}])$, define 
	\[
	U^{n}_{t}(\gamma):=\E F\left(\gamma^{[0,t)}+\left(\xi^{n}_{t,\cdot}(\gamma_{t})\right)^{[t,T]}\right),
	\]
	where $(\xi^{n}_{t,s}(g))_{s\ge t}$ denotes the solution to (\ref{eq:SPDE_for_OU_process_with_Gamma-1})
	started from $g$ at time $t$.
	
	By Lemma \ref{lem:path_regulariy_of_U_r}, 
	\[
	U^{n}\in\cP^{3}\left(\Upsilon([0,T],\H[-M_{1}])\right).
	\]
	and
	\[
	\|U^{n}\|_{\cP^{3}(\Upsilon([0,T],\H[-M_{1}]))}\le\|F\|_{\cP^{3}(D([0,T],\H[-M_{1}]))}.
	\]
	
	Since $M_{1}-2m_{\Delta}\ge I>\frac{d}{2}+m_{\partial}$, Theorem
	\ref{thm:backward_kolmogorov_equation}, applied with $I=M_{1}-2m_{\Delta}$,
	shows that the restriction of $U^{n}$ to 
	\[
	\Upsilon([0,T],\H[-M_{1}+2m_{\Delta}])
	\]
	is horizontally absolutely continuous and satisfies the backward Kolmogorov
	equation 
	\begin{equation}
		\cH U^{n}+\cG^{\Phi^{n}}U^{n}=0.\label{eq:kolmogorov_equation}
	\end{equation}
	Here $\Phi^{n}$ is defined by (\ref{eq:definition_of_Phi_n}) with
	$\varrho_{t}=\varrho^{n}_{t}$, $t\in[0,T]$, so that 
	\[
	\cG^{\Phi^{n}}_{t}=\cG^{n}_{\varrho^{n}_{t}}.
	\]
	
	Our next goal is to combine the path-dependent martingale problem
	from Proposition~\ref{prop:matingale_proplem_for_nonanticipative_functional}
	with the backward Kolmogorov equation to derive a path-dependent Trotter
	formula.
	
	Choose $\delta>0$ such that 
	\begin{equation}
		M_{1}-\delta>\frac{d}{2}+m_{\partial}
		+\max\{2m_{\Delta},d+1\}.\label{eq:choise_of_delta}
	\end{equation}
	By Lemma~\ref{lem:differentiability_of_U}, applied with $I=M_{1}-\delta$,
	we have 
	\[
	U^{n}\in\Cf^{2}_{l}\left(\Upsilon([0,T],\H[-M_{1}+\delta])\right).
	\]
	Hence, Lemma~\ref{lem:regularity_for_G_Phi}, applied with $I=M_{1}-2m_{\Delta}-\delta$,
	shows that 
	\[
	\cG^{\Phi^{n}}U^{n}=-\cH U^{n}
	\]
	is left-continuous and boundedness preserving on $\Upsilon([0,T],\H[-M_{1}+2m_{\Delta}+\delta]).$
	
	For simplicity of notation, we write $\bar{\rho}^{n,[0,t]}$, $\bar{\eta}^{n,[0,t]}$,
	and $\bar{\zeta}^{n,[0,t]}$ instead of $(\bar{\rho}^{n})^{[0,t]}$,
	$(\bar{\eta}^{n})^{[0,t]}$, and $(\bar{\zeta}^{n})^{[0,t]}$, respectively.
	For 
	\[
	(t,\gamma)\in\Upsilon([0,T],\H[-M_{1}+2m_{\Delta}+\delta]),
	\]
	define 
	\[
	\widetilde{U}^{n}_{t}(\gamma):=U^{n}_{t}\left(n^{d/2}\left(\gamma^{[0,t]}-\bar{\rho}^{n,[0,t]}\right)\right).
	\]
	We verify that $\widetilde{U}^{n}$ satisfies the assumptions of Proposition~\ref{prop:matingale_proplem_for_nonanticipative_functional}.
	
	Since 
	\[
	\bar{\rho}^{n}_{t}=\bar{\rho}^{n}_{0}+\frac{1}{2}\int^{t}_{0}\overline{\Delta_{n}\rho^{n}_{s}}\,ds,
	\]
	and $s\mapsto\overline{\Delta_{n}\rho^{n}_{s}}$ is bounded and Bochner
	integrable, Lemma~ \ref{lem:absolute_continuity_for_translation_of_path}
	shows that $\widetilde{U}^{n}$ is horizontally absolutely continuous,
	with weak horizontal derivative 
	\begin{align}
		\cH_{t}\widetilde{U}^{n}(\gamma) & =\cH_{t}U^{n}\left(n^{d/2}\left(\gamma^{[0,t]}-\bar{\rho}^{n,[0,t]}\right)\right)\nonumber\\
		&-\frac{n^{d/2}}{2}\nabla U^{n}_{t}\left(n^{d/2}\left(\gamma^{[0,t]}-\bar{\rho}^{n,[0,t]}\right)\right)\left[\overline{\Delta_{n}\rho^{n}_{t}}\right]\label{eq:horizontal_derivative_U_tilde} \\
		& =\cH_{t}U^{n}\left(n^{d/2}\left(\gamma^{[0,t]}-\bar{\rho}^{n,[0,t]}\right)\right)\nonumber\\
		&-\frac{n^{d/2}}{2}\binn{\Delta_{n}\nabla U^{n}_{t}\left(n^{d/2}\left(\gamma^{[0,t]}-\bar{\rho}^{n,[0,t]}\right);\cdot\right),\rho^{n}_{t}}_{n}.\nonumber
	\end{align}
	Moreover, by the vertical chain rule, 
	\[
	\nabla\widetilde{U}^{n}_{t}(\gamma)=n^{d/2}\nabla U^{n}_{t}\left(n^{d/2}\left(\gamma^{[0,t]}-\bar{\rho}^{n,[0,t]}\right)\right),
	\]
	and hence $\nabla\widetilde{U}^{n}$ is bounded. Since the state space
	of the SSEP is finite, the boundedness of $\nabla\widetilde{U}^{n}$
	and the mean value theorem imply 
	\[
	\sup_{(t,\gamma)}\int\left|\widetilde{U}^{n}_{t}(\gamma^{[0,t)},y)-\widetilde{U}^{n}_{t}(\gamma)\right|\mu(\gamma_{t},dy)<\infty.
	\]
	Thus, condition~(\ref{eq:extra_asusmption_in_martingale_problem})
	is satisfied.
	
	Applying Proposition~ \ref{prop:matingale_proplem_for_nonanticipative_functional}
	to the embedded SSEP $(\bar{\eta}^{n}_{t})_{t\ge0}$ and the functional
	$\widetilde{U}^{n}$, we obtain the local martingale 
	\begin{align*}
		M^{n}_{t}&:= \widetilde{U}^{n}_{t}(\bar{\eta}^{n,[0,t]})-\widetilde{U}^{n}_{0}(\bar{\eta}^{n}_{0})\\
		&-\int^{t}_{0}\left[\cH_{s}\widetilde{U}^{n}(\bar{\eta}^{n,[0,s]})+\cL^{0,n}\widetilde{U}^{n}_{s}(\bar{\eta}^{n,[0,s)};\bar{\cdot})(\eta^{n}_{s})\right]ds,\quad t\in[0,T].
	\end{align*}
	Using~(\ref{eq:horizontal_derivative_U_tilde}), (\ref{eq:definition_of_cL_rho}),
	and (\ref{eq:extension_of_cC_and_cL}), we obtain 
	\begin{align*}
		M^{n}_{t} &= U^{n}_{t}(\bar{\zeta}^{n,[0,t]})-U^{n}_{0}(\bar{\zeta}^{n}_{0})\\
		&-\int^{t}_{0}\left[\cH_{s}U^{n}(\bar{\zeta}^{n,[0,s]})+\cL^{n}_{\rho^{n}_{s}}U^{n}_{s}(\bar{\zeta}^{n,[0,s]})\right]ds,\quad t\in[0,T].
	\end{align*}
	Finally, the backward Kolmogorov equation~(\ref{eq:kolmogorov_equation})
	and the identity $\cG^{\Phi^{n}}_{s}=\cG^{n}_{\varrho^{n}_{s}}$ yield
	\begin{align}
		M^{n}_{t}= & U^{n}_{t}(\bar{\zeta}^{n,[0,t]})-U^{n}_{0}(\bar{\zeta}^{n}_{0})-\int^{t}_{0}\left(\cL^{n}_{\rho^{n}_{s}}-\cG^{n}_{\varrho^{n}_{s}}\right)U^{n}_{s}(\bar{\zeta}^{n,[0,s]})\,ds,\quad t\in[0,T].\label{eq:martingale_U_main}
	\end{align}

	For each fixed $n$, set
	\[
	K_{n}:=
	\left\{
	n^{d/2}\left(\bar{\eta}-\bar{\rho}^{n}_{t}\right):
	\eta\in\{0,1\}^{\Tnd},\ t\in[0,T]
	\right\}.
	\]
	Since the configuration space is finite and
	$t\mapsto\bar{\rho}^{n}_{t}$ is continuous, $K_{n}$ is compact in
	$\H[-M_{1}+2m_{\Delta}+\delta]$. Moreover,
	$\bar{\zeta}^{n}_{t}\in K_{n}$ for every $t\in[0,T]$.
	Therefore, the boundedness-preserving property of $U^{n}$,
	Lemma~\ref{lem:regularity_for_G_Phi}, and
	(\ref{eq:martingale_U_main}) imply that $M^{n}$ is bounded on
	$[0,T]$. Consequently, $(M^{n}_{t})_{t\in[0,T]}$ is a martingale.
	Taking expectations in (\ref{eq:martingale_U_main}) at time $T$,
	we obtain 
	\begin{align}
		\E U^{n}_{T}(\bar{\zeta}^{n,[0,T]})-\E U^{n}_{0}(\bar{\zeta}^{n}_{0})=\E\int^{T}_{0}\left(\cL^{n}_{\rho^{n}_{s}}-\cG^{n}_{\varrho^{n}_{s}}\right)U^{n}_{s}(\bar{\zeta}^{n,[0,s]})\,ds.\label{eq:trotter_preliminary}
	\end{align}
	
	By the definition of $U^{n}$, 
	\[
	U^{n}_{T}(\gamma)=F(\gamma),\quad\gamma\in D([0,T],\H[-M_{1}]),
	\]
	whereas 
	\[
	\E U^{n}_{0}(\xiini)=\E F(\xi^{n,[0,T]}).
	\]
	Therefore, adding and subtracting $\E U^{n}_{0}(\xiini)$ in (\ref{eq:trotter_preliminary}),
	we obtain 
	\begin{align}
		\E F(\bar{\zeta}^{n,[0,T]})-\E F(\xi^{n,[0,T]})&=\E\left[U^{n}_{0}(\bar{\zeta}^{n}_{0})-U^{n}_{0}(\xiini)\right]\nonumber\\
		&+\E\int^{T}_{0}\left(\cL^{n}_{\rho^{n}_{s}}-\cG^{n}_{\varrho^{n}_{s}}\right)U^{n}_{s}(\bar{\zeta}^{n,[0,s]})\,ds.\label{eq:path_dependent_trotter_formula}
	\end{align}
	
	Since 
	\[
	U^{n}_{0}\in\Cf^{3}_{\text{der}}(\H[-M_{1}])
	\]
	and 
	\[
	\norm{U^{n}_{0}}_{\Cf^{3}_{\text{der}}(\H[-M_{1}])}\le\norm F_{\cP^{3}(D([0,T],\H[-M_{1}]))},
	\]
	Proposition~\ref{prop:berry_essen_bound}, applied with $M_{0}=M_{1}$
	and $m_{0}=m_{\partial}$, yields 
	\begin{align*}
		\left|\E U^{n}_{0}(\bar{\zeta}^{n}_{0})-\E U^{n}_{0}(\xiini)\right| & \le C_{M_{1},m_{\partial}}\norm F_{\cP^{3}(D([0,T],\H[-M_{1}]))}\left[\frac{1}{n^{d/2}}+\frac{1}{n^{m_{\partial}}}\left(1+\norm{\rhoini}^{2}_{\Cf^{M_{2}}(\Td)}\right)\right].
	\end{align*}
	Applying Corollary~\ref{cor:pathwise_comparison_of_generators}
	with $m_{W}=d+1$, we obtain
	\[
	\sup_{s\in[0,T]}\left|\E\left[\left(\cL^{n}_{\rho^{n}_{s}}-\cG^{n}_{\varrho^{n}_{s}}+\cC^{n}\right)U^{n}_{s}(\bar{\zeta}^{n,[0,s]})\right]\right|\le C_{M_{1}}\sup_{s\in[0,T]}\E[E_{n}(s)]\norm F_{\cP^{3}(D([0,T],\H[-M_{1}]))}.
	\]
	Combining this estimate with (\ref{eq:path_dependent_trotter_formula}),
	we obtain 
	\begin{align}
		& \left|\E F(\bar{\zeta}^{n,[0,T]})-\E F(\xi^{n,[0,T]})+\int^{T}_{0}\E\cC^{n}U^{n}_{s}(\bar{\zeta}^{n,[0,s]})\,ds\right|\nonumber \\
		& \qquad\le C_{M_{1},m_{\partial},T}\left[\frac{1}{n^{d/2}}+\frac{1}{n^{m_{\partial}}}\left(1+\norm{\rhoini}^{2}_{\Cf^{M_{2}}(\Td)}\right)+\sup_{s\in[0,T]}\E[E_{n}(s)]\right]\label{eq:trotter_estimate}\\
		&\qquad\qquad\cdot\norm F_{\cP^{3}(D([0,T],\H[-M_{1}]))}.\nonumber
	\end{align}
	
	It remains to estimate $\sup_{t\in[0,T]}\E[E_{n}(t)].$
	By \citep[Lemma~2.5]{Gess_Konarovskyi2026}, for every $t\in[0,T]$,
	\begin{align*}
		\E\left[\bnorm{\bar{\zeta}^{n}_{t}}^{2}_{-I}\right] & \le C_{I}\E\left[\bnorm{\ex_{n}\zeta^{n}_{t}}^{2}_{\H[-I]}\right]\\
		& \le C_{I}(1+t)\left(1+\bnorm{\nabla_{n}\rho^{n}_{0}}^{2}_{\Cf(\Tnd)}\right)\\
		& \le C_{I}(1+t)\left(1+\bnorm{\rhoini}^{2}_{\Cf^{1}(\Td)}\right).
	\end{align*}
	Moreover, Corollary~\ref{cor:regularity_of_modified_semigroup} implies
	that, for every $\ell\in\N_{0}$, 
	\begin{equation}
		\norm{\varrho^{n}_{t}}_{\Cf^{\ell}(\Td)}\le C_{\ell}\norm{\rhoini}_{\Cf^{\ell+1+\lceil d/2\rceil}(\Td)},\quad t\in[0,T].\label{eq:estimate_for_varrho_n}
	\end{equation}
	By Lemma~\ref{lem:bound_for_varrho_and_rho}, 
	\begin{equation}
		\bnorm{(\varrho^{n}_{t})^{-}}_{\Cf(\Td)}+\bnorm{(1-\varrho^{n}_{t})^{-}}_{\Cf(\Td)}\le\frac{C_{m_{\Delta},T}}{n^{2m_{\Delta}}}\bnorm{\rhoini}_{\Cf^{2m_{\Delta}+d+3}(\Td)}.\label{eq:negative_part_varrho_n}
	\end{equation}
	Furthermore, combining (\ref{eq:estimate_for_varrho_n}) with (\ref{eq:estimate_of_difference_of_ex_nrho^n_and_varrho_n}),
	we obtain 
	\begin{equation}
		\bnorm{\ex_{n}\rho^{n}_{t}}_{\Cf^{\ell}(\Td)}\le C_{m_{\Delta},\ell,T}\bnorm{\rhoini}_{\Cf^{\ell+d+2m_{\Delta}+3}(\Td)},\quad t\in[0,T].\label{eq:estimate_for_ex_n_rho_n}
	\end{equation}
	
	Using these estimates in the definition of $E_{n}$, we obtain, uniformly
	in $t\in[0,T]$, 
	\begin{align*}
		\E E^{n}_{\Delta}(t) & \le\frac{C_{I,T}}{n^{2m_{\Delta}}}\left(1+\bnorm{\rhoini}_{\Cf^{1}(\Td)}\right),\\
		\E E^{n}_{\mathrm{fluc}}(t) & \le\frac{C_{I,T}}{n^{d/2}}\left(1+\bnorm{\rhoini}^{2}_{\Cf^{\lceil I\rceil+d+2m_{\Delta}+3}(\Td)}\right),\\
		E^{n}_{\mathrm{prof}}(t) & \le\frac{C_{m_{\Delta},T}}{n^{2m_{\Delta}}}\left(1+\bnorm{\rhoini}^{2}_{\Cf^{d+2m_{\Delta}+3}(\Td)}\right),\\
		E^{n}_{\mathrm{disc}}(t)&\le\frac{C}{n^{d+1}}
		\left(
		1+\bnorm{\rhoini}^{2}_{\Cf^{d+2+\lceil d/2\rceil}(\Td)}
		\right),\\
		E^{n}_{+}(t) & \le\frac{C_{m_{\Delta},T}}{n^{2m_{\Delta}}}\left(1+\bnorm{\rhoini}^{2}_{\Cf^{2m_{\Delta}+d+3}(\Td)}\right),\\
		E^{n}_{\tau}(t) & \le\frac{C}{n^{m_{\tau}+1}}\left(1+\bnorm{\rhoini}^{2}_{\Cf^{m_{\tau}+2+\lceil d/2\rceil}(\Td)}\right).
	\end{align*}
	Finally, 
	\[
	E^{n}_{\partial}(t)=\frac{1}{n^{m_{\partial}}},\qquad R^{n}=\frac{1}{n^{d/2+1}}.
	\]
	
	Consequently, by the choice of $M_{2}$, 
	\[
	\sup_{t\in[0,T]}\E[E_{n}(t)]\le C\left(\frac{1}{n^{d/2}}+\frac{1}{n^{2m_{\Delta}}}+\frac{1}{n^{m_{\partial}}}+\frac{1}{n^{m_{\tau}+1}}\right)\left(1+\norm{\rhoini}^{2}_{\Cf^{M_{2}}(\Td)}\right).
	\]
	Setting 
	\[
	\alpha:=\min\left\{ \frac{d}{2},\,2m_{\Delta},\,m_{\partial},\,m_{\tau}+1\right\} ,
	\]
	we conclude that 
	\[
	\sup_{t\in[0,T]}\E[E_{n}(t)]\le\frac{C}{n^{\alpha}}\left(1+\norm{\rhoini}^{2}_{\Cf^{M_{2}}(\Td)}\right).
	\]
	
	Combining this estimate with~(\ref{eq:trotter_estimate}) yields
	\begin{align*}
		&\left|\E F(\bar{\zeta}^{n,[0,T]})-\E F(\xi^{n,[0,T]})+\int^{T}_{0}\E\cC^{n}U^{n}_{s}(\bar{\zeta}^{n,[0,s]})\,ds\right|\\
		&\qquad\le\frac{C}{n^{\alpha}}\left(1+\norm{\rhoini}^{2}_{\Cf^{M_{2}}(\Td)}\right)\norm F_{\cP^{3}(D([0,T],\H[-M_{1}]))},
	\end{align*}
	which proves the claim.
	
	\subsubsection{Estimate of correlation correction term}
	
	We now estimate the correlation correction term 
	\[
	\int^{T}_{0}\E\cC^{n}U^{n}_{t}(\bar{\zeta}^{n,[0,t]})\,dt=\frac{1}{2n^{d}}\sum^{d}_{j=1}\int^{T}_{0}\E\binn{\left(\partial^{(m_{\partial})}_{n,j}\right)^{\odot2}\nabla^{2}U^{n}_{t}\left(\bar{\zeta}^{n,[0,t]};\cdot\right),\zeta^{n}_{t}\tau_{n,j}\zeta^{n}_{t}}_{n}\,dt.
	\]
	For $t\in[0,T]$ and $j\in[d]$, set 
	\[
	Z_{n,j}(t):=\left(\partial^{(m_{\partial})}_{n,j}\right)^{\odot2}\nabla^{2}U^{n}_{t}\left(\bar{\zeta}^{n,[0,t]};\cdot\right).
	\]
	Applying Corollary~\ref{cor:existence_of_regular_kernel} with $k=2$,
	$m_{\text{Cor.}}=d+1+m_\partial$, and $I_{\text{Cor.}}=M_{1}-\delta$, where $\delta>0$ is chosen as in~(\ref{eq:choise_of_delta}),
	we obtain that 
	\[
	Z_{n,j}(t)\in\Cf^{d+1}(\Td)
	\]
	almost surely and 
	\[
	\bnorm{Z_{n,j}(t)}_{\Cf^{d+1}(\Td)}\le C_{M_{1}}\bnorm{U^{n}}_{\Cf^{2}_{l}(\Upsilon([0,T],\H[-M_{1}+\delta]))}\le C_{M_{1}}\bnorm F_{\cP^{3}(D([0,T],\H[-M_{1}]))}
	\]
	for every $t\in[0,T]$, almost surely.
	
	Next, let 
	\[
	I_{1}:=\frac{d}{2}+1>\frac{d}{2}.
	\]
	Following the proof of \citep[Lemma 2.6]{Gess_Konarovskyi2026}, we
	obtain 
	\begin{align}
		& \left|\frac{1}{n^{d}}\E\binn{Z_{n,j}(t),\zeta^{n}_{t}\tau_{n,j}\zeta^{n}_{t}}_{n}\right|^{2}\nonumber \\
		& \quad=\frac{1}{n^{2d}}\left|\sum_{k\in\Znd}\E\left[\binn{Z_{n,j}(t),\varsigma_{k}}_{n}\binn{\varsigma_{k},\zeta^{n}_{t}\tau_{n,j}\zeta^{n}_{t}}_{n}\right]\right|^{2}\nonumber \\
		& \quad\le\frac{1}{n^{2d}}\sum_{k\in\Znd}(1+|k|^{2})^{I_{1}}\E\left[\left|\binn{Z_{n,j}(t),\varsigma_{k}}_{n}\right|^{2}\right]\label{eq:estimate_for_correlation_correction}\\
		& \qquad\qquad\cdot\sum_{k\in\Znd}(1+|k|^{2})^{-I_{1}}\E\left[\left|\binn{\varsigma_{k},\zeta^{n}_{t}\tau_{n,j}\zeta^{n}_{t}}_{n}\right|^{2}\right].\nonumber 
	\end{align}
	
	Using 
	\[
	\binn{Z_{n,j}(t),\varsigma_{k}}_{n}=\binn{\ex_{n}Z_{n,j}(t),\varsigma_{k}},\quad k\in\Znd,
	\]
	and Corollary~\ref{cor:bound_of_ex}, we obtain 
	\begin{align*}
		& \sum_{k\in\Znd}(1+|k|^{2})^{I_{1}}\E\left[\left|\binn{Z_{n,j}(t),\varsigma_{k}}_{n}\right|^{2}\right]\\
		& \qquad\le\E\left[\bnorm{\ex_{n}Z_{n,j}(t)}^{2}_{\H[I_{1}]}\right]\le C_{I_{1}}\E\left[\bnorm{Z_{n,j}(t)}^{2}_{\Cf^{d+1}(\Td)}\right]\\
		& \qquad\le C_{M_{1}}\bnorm F^{2}_{\cP^{3}(D([0,T],\H[-M_{1}]))},
	\end{align*}
	
	Moreover, since $I_{1}>d/2$, 
	\[
	\sum_{k\in\Znd}(1+|k|^{2})^{-I_{1}}\le C_{I_{1}},
	\]
	uniformly in $n$. Hence 
	\[
	\frac{1}{n^{2d}}\sum_{k\in\Znd}(1+|k|^{2})^{-I_{1}}\E\left[\left|\binn{\varsigma_{k},\zeta^{n}_{t}\tau_{n,j}\zeta^{n}_{t}}_{n}\right|^{2}\right]\le\frac{C_{I_{1}}}{n^{2d}}\max_{k\in\Znd}\E\left[\left|\binn{\varsigma_{k},\zeta^{n}_{t}\tau_{n,j}\zeta^{n}_{t}}_{n}\right|^{2}\right].
	\]
	
	We next estimate, for $k\in\Znd$, 
	\begin{align*}
		& \frac{1}{n^{2d}}\E\left[\left|\binn{\varsigma_{k},\zeta^{n}_{t}\tau_{n,j}\zeta^{n}_{t}}_{n}\right|^{2}\right]\\
		& \quad=\frac{1}{n^{2d}}\E\left[\binn{\varsigma_{k},\zeta^{n}_{t}\tau_{n,j}\zeta^{n}_{t}}_{n}\binn{\zeta^{n}_{t}\tau_{n,j}\zeta^{n}_{t},\varsigma_{k}}_{n}\right]\\
		& \quad=\frac{1}{n^{2d}}\sum_{x,y\in\Tnd}\varsigma_{k}(x)\varsigma_{-k}(y)\E\Big[\left(\eta^{n}_{t}(x)-\rho^{n}_{t}(x)\right)\left(\eta^{n}_{t}(x+e^{n}_{j})-\rho^{n}_{t}(x+e^{n}_{j})\right)\\
		& \hspace{4cm}\cdot\left(\eta^{n}_{t}(y)-\rho^{n}_{t}(y)\right)\left(\eta^{n}_{t}(y+e^{n}_{j})-\rho^{n}_{t}(y+e^{n}_{j})\right)\Big].
	\end{align*}
	
	By \citep[Theorem 2.1]{Ferrari_Presutti:1991}\footnote{The proof of the corresponding estimate on $\mathbb Z^d$ carries over to the discrete torus, with constants uniform in $n$ on bounded time intervals.} and \citep[Theorem 6.1]{Jara_Landim:2008},
	for $d=1$, and by \citep[Lemma 3.1]{Xu2025} for $d\ge2$, the
	expectation in the last display is bounded, whenever $x,x+e^{n}_{j},y,y+e^{n}_{j}$
	are pairwise distinct, by 
	\[
	\eps^{(c)}_{n}:=C_{T}\left(1+\bnorm{\rhoini}^{4}_{\Cf^{1}(\Td)}\right)\begin{cases}
		n^{-2}, & d=1,\\[1mm]
		n^{-3}\ln^{3}(1+n^{2}T), & d=2,\\[1mm]
		n^{-4}, & d\ge3.
	\end{cases}
	\]
	
	On the other hand, the number of pairs $(x,y)\in(\Tnd)^{2}$ for which
	$x,x+e^{n}_{j},y,y+e^{n}_{j}$ are not pairwise distinct is bounded
	by $Cn^{d}$. Since 
	\[
	|\eta^{n}_{t}(x)-\rho^{n}_{t}(x)|\le1,
	\]
	we therefore obtain 
	\[
	\frac{1}{n^{2d}}\E\left[\left|\binn{\varsigma_{k},\zeta^{n}_{t}\tau_{n,j}\zeta^{n}_{t}}_{n}\right|^{2}\right]\le\eps^{(c)}_{n}+\frac{C}{n^{d}}\le\frac{C_{T}}{n^{d\wedge4}}\left(1+\bnorm{\rhoini}^{4}_{\Cf^{1}(\Td)}\right).
	\]
	
	Combining the preceding estimates with (\ref{eq:estimate_for_correlation_correction})
	and taking square roots, we obtain 
	\[
	\sup_{t\in[0,T]}\left|\frac{1}{n^{d}}\E\binn{Z_{n,j}(t),\zeta^{n}_{t}\tau_{n,j}\zeta^{n}_{t}}_{n}\right|\le\frac{C_{T,M_{1}}}{n^{\frac{d}{2}\wedge2}}\left(1+\bnorm{\rhoini}^{2}_{\Cf^{1}(\Td)}\right)\bnorm F_{\cP^{3}(D([0,T],\H[-M_{1}]))}.
	\]
	
	Consequently, recalling the definition of $Z_{n,j}$ and summing over
	$j\in[d]$, we obtain 
	\[
	\left|\int^{T}_{0}\E\cC^{n}U^{n}_{t}(\bar{\zeta}^{n,[0,t]})\,dt\right|\le\frac{C_{T,M_{1}}}{n^{\frac{d}{2}\wedge2}}\left(1+\bnorm{\rhoini}^{2}_{\Cf^{1}(\Td)}\right)\bnorm F_{\cP^{3}(D([0,T],\H[-M_{1}]))}.
	\]
	Combining this estimate with (\ref{eq:main_estimate_with_correlaton_term})
	proves (\ref{eq:main_estimate_without_correlation}). This completes
	the proof of Theorem \ref{thm:main_result}.
	
	\subsection{Some examples of test functions}
	
	Although Theorem~\ref{thm:main_result} is formulated for the class
	of test functions $\cP^{3}(D([0,T],\H[-M_{1}]))$, we state the
	following examples for a general Sobolev space $\H[I]$.
	\begin{example}
		\label{exa:cylindrical_function}Let $I\in\R$, $k,N\in\N$, and 
		\[
		0\le t_{1}<\ldots<t_{N}\le T,
		\]
		and let $G\in\Cf^{k}_{\text{der}}((\H[I])^{N})$. Define 
		\[
		F(\gamma):=G(\gamma_{t_{1}},\ldots,\gamma_{t_{N}}),\quad\gamma\in D([0,T],\H[I]).
		\]
		Then 
		\[
		F\in\cP^{k}(D([0,T],\H[I])).
		\]
		
		Indeed, $F$ is continuous with respect to the uniform topology. Moreover,
		for $p\in[N]$ and $t\in(t_{p-1},t_{p}]$, where $t_{0}:=0$, we have
		\[
		\nabla_{[t,T]}F(\gamma)[h]=\sum^{N}_{i=p}D_{i}G(\gamma_{t_{1}},\ldots,\gamma_{t_{N}})[h_{t_{i}}],\quad h\in D([t,T],\H[I]).
		\]
		More generally, for every $\ell\in[k]$, 
		\[
		\nabla^{\ell}_{[t,T]}F(\gamma)[h^{(1)},\ldots,h^{(\ell)}]=\sum^{N}_{i_{1},\ldots,i_{\ell}=p}D_{i_{1}}\cdots D_{i_{\ell}}G(\gamma_{t_{1}},\ldots,\gamma_{t_{N}})[h^{(1)}_{t_{i_{1}}},\ldots,h^{(\ell)}_{t_{i_{\ell}}}].
		\]
		Since the derivatives of $G$ up to order $k$ are bounded, the vertical
		path derivatives of $F$ are bounded. The left continuity simply follows
		from the form of the vertical path derivatives. Hence 
		\[
		F\in\cP^{k}(D([0,T],\H[I])).
		\]
	\end{example}
	
	We next consider functionals depending on time averages of the path.
	\begin{example}
		Let $I\in\R$, $k\in\N$, let $\nu$ be a finite measure on $[0,T]$,
		and let $G\in\Cf^{k}_{\text{der}}(\H[I])$. Since every $\gamma\in D([0,T],\H[I])$
		is bounded and strongly measurable, it is Bochner integrable with
		respect to $\nu$. We define 
		\[
		F(\gamma):=G\left(\int_{[0,T]}\gamma_{s}\,\nu(ds)\right),\quad\gamma\in D([0,T],\H[I]).
		\]
		Then 
		\[
		F\in\cP^{k}(D([0,T],\H[I])).
		\]
		
		Indeed, for $\gamma,\widetilde{\gamma}\in D([0,T],\H[I])$, 
		\[
		\bnorm{\int_{[0,T]}(\gamma_{s}-\widetilde{\gamma}_{s})\,\nu(ds)}_{\H[I]}\le\nu([0,T])\bnorm{\gamma-\widetilde{\gamma}}_{D([0,T],\H[I])}.
		\]
		Hence the map $F$ is continuous due to the continuity of $G$.
		
		For $t\in[0,T]$, 
		\[
		\nabla_{[t,T]}F(\gamma)[h]=DG\left(\int_{[0,T]}\gamma_{s}\,\nu(ds)\right)\left[\int_{[t,T]}h_{s}\,\nu(ds)\right],\quad h\in D([t,T],\H[I]).
		\]
		More generally, for every $\ell\in[k]$, 
		\begin{align*}
			& \nabla^{\ell}_{[t,T]}F(\gamma)[h^{(1)},\ldots,h^{(\ell)}]\\
			& \qquad=D^{\ell}G\left(\int_{[0,T]}\gamma_{s}\,\nu(ds)\right)\left[\int_{[t,T]}h^{(1)}_{s}\,\nu(ds),\ldots,\int_{[t,T]}h^{(\ell)}_{s}\,\nu(ds)\right].
		\end{align*}
		Since $\nu$ is finite and the derivatives of $G$ are bounded, all
		vertical path derivatives of $F$ are bounded.
		
		We finally verify their left-continuity. Let $t_{n}\uparrow t$ and
		$\gamma^{n}\to\gamma$ in $D([0,T],\H[I])$. First, 
		\[
		\bnorm{\int_{[0,T]}(\gamma^{n}_{s}-\gamma_{s})\,\nu(ds)}_{\H[I]}\le\nu([0,T])\bnorm{\gamma^{n}-\gamma}_{D([0,T],\H[I])}\longrightarrow0.
		\]
		Moreover, for $h\in D([t,T],\H[I])$, 
		\begin{align*}
			& \bnorm{\int_{[t_{n},T]}(\iota_{t_{n},t}h)_{s}\,\nu(ds)-\int_{[t,T]}h_{s}\,\nu(ds)}_{\H[I]}=\nu([t_{n},t))\bnorm{h_{t}}_{\H[I]}\le\nu([t_{n},t))\bnorm h_{D([t,T],\H[I])}.
		\end{align*}
		Since $[t_{n},t)\downarrow\emptyset$ and $\nu$ is finite, 
		\[
		\nu([t_{n},t))\to0.
		\]
		Combining these observations with the continuity of $D^{\ell}G$ gives,
		for every $\ell\in[k]$, 
		\[
		\nabla^{\ell}_{[t_{n},T]}F(\gamma^{n})\circ(\iota_{t_{n},t})^{\otimes\ell}\to\nabla^{\ell}_{[t,T]}F(\gamma)
		\]
		in $\cL_{\ell}(D([t,T],\H[I]))$. Thus the vertical path derivatives
		are left-continuous, and consequently 
		\[
		F\in\cP^{k}(D([0,T],\H[I])).
		\]
	\end{example}

	\subsection{Estimate of correlation correction from below}
	
	We conclude this section by showing that the correlation correction
	may indeed be of order $n^{-2}$ for smooth test functionals with
	bounded second and third vertical path derivatives. Thus, the $n^{-2}$
	order of the correlation correction is, in this sense, sharp. The
	main ingredient is a lower bound on the nearest-neighbor two-point
	correlations, which we establish first.
	
	We define the two-point correlation function by 
	\[
	V^{n}_{t}(x,y):=\E\left[\left(\eta^{n}_{t}(x)-\rho^{n}_{t}(x)\right)\left(\eta^{n}_{t}(y)-\rho^{n}_{t}(y)\right)\right],\quad x,y\in\Tnd.
	\]
	Following the proof of the main theorem in \citep[p. 32]{Ravishankar:1992},
	we have, for $x,y\in\Tnd$ with $x\not= y$,
	\begin{align}
		V^{n}_{t}(x,y) & =-\frac{1}{2}\int^{t}_{0}\sum^{d}_{j=1}\sum_{u\in\Tnd}P^{n}_{t-s}(x,y;u,u+e^{n}_{j})\left(\partial_{n,j}\rho^{n}_{s}(u)\right)^{2}\,ds\nonumber \\
		& -\frac{1}{2}\int^{t}_{0}\sum^{d}_{j=1}\sum_{u\in\Tnd}P^{n}_{t-s}(x,y;u,u-e^{n}_{j})\left(\partial_{n,j}\rho^{n}_{s}(u-e^{n}_{j})\right)^{2}\,ds,\label{eq:representation_for_V_n}
	\end{align}
	where $P^{n}_{t}(x,y;u,v)$ denotes the transition probability of
	the two-particle exclusion process from $(x,y)$ to $(u,v)$ in time
	$t$. In particular, 
	\begin{equation}
		V^{n}_{t}(x,y)\le0,\quad x,y\in\Tnd, \quad x\not=y,\quad t\ge0.\label{eq:negativity_of_V}
	\end{equation}
	
	The representation (\ref{eq:representation_for_V_n}) yields the following
	local lower bound.
	\begin{lem}
		\label{lem:two-point-correlation-function} Fix $t>0$, $j\in[d]$,
		$n\ge3$, and $x\in\Tnd$. Assume that there exist $\delta>0$ and
		$c\in(0,n^{2}t]$ such that 
		\[
		\left(\partial_{n,j}\rho^{n}_{t-s}(x)\right)^{2}\ge\delta,\quad s\in[0,c/n^{2}].
		\]
		Then 
		\[
		\left|V^{n}_{t}(x,x+e^{n}_{j})\right|\ge\frac{\delta}{2(2d-1)n^{2}}\left(1-e^{-(2d-1)c}\right).
		\]
	\end{lem}
	
	\begin{proof}
		Taking $u=x$ and fixing the direction $j$ in (\ref{eq:representation_for_V_n}),
		we obtain 
		\begin{align*}
			\left|V^{n}_{t}(x,x+e^{n}_{j})\right| & \ge\frac{1}{2}\int^{t}_{0}P^{n}_{t-s}(x,x+e^{n}_{j};x,x+e^{n}_{j})\left(\partial_{n,j}\rho^{n}_{s}(x)\right)^{2}\,ds\\
			& =\frac{1}{2}\int^{t}_{0}P^{n}_{s}(x,x+e^{n}_{j};x,x+e^{n}_{j})\left(\partial_{n,j}\rho^{n}_{t-s}(x)\right)^{2}\,ds.
		\end{align*}
		Let 
		\[
		\sigma:=\inf\left\{ s\ge0:(X_{s},Y_{s})\neq(x,x+e^{n}_{j})\right\} .
		\]
		Since the particles are at neighboring sites, there are $4d-2$ admissible
		jumps, each occurring at rate $n^{2}/2$. Hence 
		\[
		\sigma\sim\mathrm{Exp}((2d-1)n^{2})
		\]
		and, consequently, 
		\[
		P^{n}_{s}(x,x+e^{n}_{j};x,x+e^{n}_{j})\ge\p_{x,x+e^{n}_{j}}\{\sigma>s\}=e^{-(2d-1)n^{2}s}.
		\]
		Since $c/n^{2}\le t$, the assumption yields 
		\[
		\left|V^{n}_{t}(x,x+e^{n}_{j})\right|\ge\frac{\delta}{2}\int^{c/n^{2}}_{0}e^{-(2d-1)n^{2}s}\,ds=\frac{\delta}{2(2d-1)n^{2}}\left(1-e^{-(2d-1)c}\right),
		\]
		which proves the claim.
	\end{proof}
	
	We now use the preceding estimate to show that, for every nonconstant
	initial profile, the correlation correction is of order at least $n^{-2}$
	for a suitable smooth test functional. Notice that, in the following
	proposition, boundedness is only required for the second and third
	vertical path derivatives.
	\begin{prop}
		\label{prop:bound_from_belov_of_correlation_correction} Let the assumptions
		of Theorem~\ref{thm:main_result} be satisfied and assume that $\rhoini$
		is not constant. Then there exist $F\in\Cf^{3}(D([0,T],\H[-M_1]))$,
		whose vertical path derivatives $\nabla^{\ell}_{[t,T]}F$, $t\in[0,T]$,
		$\ell\in[3]$, are left-continuous, such that $\nabla^{2}_{[t,T]}F$
		and $\nabla^{3}_{[t,T]}F$ are uniformly bounded, while $\nabla_{[t,T]}F$
		has at most linear growth, and a constant $c>0$ such that 
		\[
		\left|\int^{T}_{0}\E\cC^{n}U^{n}_{t}(\bar{\zeta}^{n,[0,t]})\,dt\right|\ge\frac{c}{n^{2}}
		\]
		for all sufficiently large $n$. 
	\end{prop}
	
	\begin{proof}
		Since $\rhoini$ is not constant, there exist $x_{0}\in\Td$ and $i\in[d]$
		such that $\partial_{i}\rhoini(x_{0})\neq0$. Hence, by continuity,
		there exist $r>0$ and $\delta>0$ such that 
		\begin{equation}
			\left(\partial_{i}\rhoini(y)\right)^{2}\ge\delta,\quad y\in B_{r}(x_{0}).\label{eq:assumption_on_rhoini}
		\end{equation}
		We choose $\varphi\in\Cf^{\infty}(\Td)$ such that, for some $\kappa>0$,
		\begin{equation}
			\left(\partial_{i}\varphi(y)\right)^{2}\ge\kappa,\quad y\in B_{r}(x_{0}).\label{eq:assumption_on_phi}
		\end{equation}
		For $t_{0}\in(0,T]$, to be specified below, define 
		\[
		F(\gamma):=\frac{1}{2}\binn{\varphi,\gamma_{t_{0}}}^{2},\quad\gamma\in D([0,T],\H[-M_1]).
		\]
		
		A direct calculation shows that $F$ has left-continuous vertical path
		derivatives. More precisely, for $t\in[0,T]$ and
		$h,h^{(1)},h^{(2)}\in D([t,T],\H[-M_1])$,
		\[
		\nabla_{[t,T]}F(\gamma)[h]
		=
		\binn{\varphi,\gamma_{t_{0}}}
		\binn{\varphi,h_{t_{0}}}
		\I_{[0,t_{0}]}(t),
		\]
		\[
		\nabla^{2}_{[t,T]}F(\gamma)
		\left[h^{(1)},h^{(2)}\right]
		=
		\binn{\varphi,h^{(1)}_{t_{0}}}
		\binn{\varphi,h^{(2)}_{t_{0}}}
		\I_{[0,t_{0}]}(t),
		\]
		and
		\[
		\nabla^{3}_{[t,T]}F(\gamma)=0.
		\]
		Thus, $\nabla_{[t,T]}F$ has at most linear growth, whereas
		$\nabla^{2}_{[t,T]}F$ and $\nabla^{3}_{[t,T]}F$ are uniformly bounded.
		The left-continuity of these derivatives follows directly from the
		above expressions.
		
		We next compute the second vertical derivative of the continuation
		functional. For $t\le t_{0}$,
		\[
		U^{n}_{t}(\gamma)
		=
		\frac{1}{2}\E\left[
		\binn{\varphi,\xi^{n}_{t,t_{0}}(\gamma_{t})}^{2}
		\right].
		\]
		By the linear dependence of the modified Ornstein--Uhlenbeck dynamics
		on its initial condition,
		\[
		\xi^{n}_{t,t_{0}}(\gamma_{t})
		=
		\cS^{n,m_{\Delta}}_{t,t_{0}}\gamma_{t}
		+
		Z^{n}_{t,t_{0}},
		\]
		where $Z^{n}_{t,t_{0}}$ does not depend on $\gamma_{t}$. Therefore,
		direct differentiation gives
		\[
		\nabla^{2}U^{n}_{t}(\gamma)[h^{(1)},h^{(2)}]
		=
		\binn{\varphi,\cS^{n,m_{\Delta}}_{t,t_{0}}h^{(1)}}
		\binn{\varphi,\cS^{n,m_{\Delta}}_{t,t_{0}}h^{(2)}}
		\I_{[0,t_{0}]}(t).
		\]
		Since $\cS^{n,m_{\Delta}}_{t,t_{0}}$ is self-adjoint, for $t\in[0,t_{0}]$,
		the kernel of $\nabla^{2}U^{n}_{t}(\gamma)$ is defined by
		\[
		\nabla^{2}U^{n}_{t}(\gamma;x,y)=\varphi^{n}_{t}(x)\varphi^{n}_{t}(y),\quad x,y\in\Td,
		\]
		where 
		\[
		\varphi^{n}_{t}:=\cS^{n,m_{\Delta}}_{t,t_{0}}\varphi.
		\]
		Consequently, 
		\[
		\left(\partial^{(m_{\partial})}_{n,j}\right)^{\odot2}\nabla^{2}U^{n}_{t}(\gamma;\cdot)(x)=\left(\partial^{(m_{\partial})}_{n,j}\varphi^{n}_{t}(x)\right)^{2}=:\psi^{n}_{t,j}(x)\ge0
		\]
		for $j\in[d]$ and $x\in\Td$. Hence, by (\ref{eq:negativity_of_V}),
		for $t\in[0,t_{0}]$, 
		\[
		\E\left[\cC^{n}U^{n}_{t}(\bar{\zeta}^{n,[0,t]})\right]=\frac{1}{2n^{d}}\sum^{d}_{j=1}\sum_{x\in\Tnd}\psi^{n}_{t,j}(x)V^{n}_{t}(x,x+e^{n}_{j})\le0.
		\]
		Moreover, the expression above vanishes for $t>t_{0}$. Therefore,
		\begin{align}
			& \left|\int^{T}_{0}\E\left[\cC^{n}U^{n}_{t}(\bar{\zeta}^{n,[0,t]})\right]dt\right|=-\frac{1}{2n^{d}}\sum^{d}_{j=1}\sum_{x\in\Tnd}\int^{t_{0}}_{0}\psi^{n}_{t,j}(x)V^{n}_{t}(x,x+e^{n}_{j})\,dt.\label{eq:correction_lower_bound_start}
		\end{align}
		
		Recall that $(\rho^{n}_{t})_{t\ge0}$ is the solution to (\ref{eq:discrete_heat_equation})
		started from $\rhoini|_{\Tnd}$. Let $(\tilde{\varrho}_{t})_{t\ge0}$
		solve 
		\[
		\partial_{t}\tilde{\varrho}_{t}=\frac{1}{2}\Delta\tilde{\varrho}_{t},\quad\tilde{\varrho}_{0}=\rhoini.
		\]
		By the standard convergence of the discrete heat semigroup and its
		spatial derivatives, 
		\begin{equation}
			\sup_{t\in[0,T]}\sup_{x\in\Tnd}\left|\partial_{n,i}\rho^{n}_{t}(x)-\partial_{i}\tilde{\varrho}_{t}(x)\right|\to0,\quad n\to\infty.\label{eq:convergence_gradient_discrete_heat}
		\end{equation}
		Since $\partial_{i}\tilde{\varrho}$ is continuous and $\tilde{\varrho}_{0}=\rhoini$,
		it follows from (\ref{eq:assumption_on_rhoini}) that there exists
		$t'_{0}\in(0,T]$ such that 
		\[
		\left(\partial_{i}\tilde{\varrho}_{t}(x)\right)^{2}\ge\frac{3\delta}{4},\quad t\in[0,t'_{0}],\quad x\in B_{r}(x_{0}).
		\]
		Consequently, by (\ref{eq:convergence_gradient_discrete_heat}), there
		exists $N_{1}\in\N$ such that 
		\begin{equation}
			\left(\partial_{n,i}\rho^{n}_{t}(x)\right)^{2}\ge\frac{\delta}{2},\quad t\in[0,t'_{0}],\quad x\in B_{r}(x_{0})\cap\Tnd,\label{eq:lower_bound_discrete_gradient}
		\end{equation}
		for all $n\ge N_{1}$.
		
		We next choose $t_{0}$. Since 
		\[
		\varphi^{n}_{t}=\cS^{n,m_{\Delta}}_{t,t_{0}}\varphi=\cS^{n,m_{\Delta}}_{t_{0}-t}\varphi,
		\]
		the uniform convergence of $\cS^{n,m_{\Delta}}_{s}$ to the identity
		(see Lemma \ref{lem:operator_norm_continuity_modified_semigroup})
		and the Sobolev embedding imply that 
		\[
		\sup_{n\ge3}\sup_{t\in[0,t_{0}]}\bnorm{\varphi^{n}_{t}-\varphi}_{C^{1}(\Td)}\to0\quad\text{as }t_{0}\to0+.
		\]
		Hence, by (\ref{eq:assumption_on_phi}), we may choose $t_{0}\in(0,t'_{0}]$
		such that 
		\[
		\left(\partial_{i}\varphi^{n}_{t}(x)\right)^{2}\ge\frac{3\kappa}{4},\quad t\in[0,t_{0}],\quad x\in B_{r}(x_{0}),
		\]
		for every $n\ge3$. By the definition of $\partial^{(m_{\partial})}_{n,i}$
		and the uniform regularity of $(\varphi^{n}_{t})_{t\in[0,t_{0}]}$,
		there exists $N_{2}\in\N$ such that
		\begin{equation}
			\psi^{n}_{t,i}(x)=\left(\partial^{(m_{\partial})}_{n,i}\varphi^{n}_{t}(x)\right)^{2}\ge\frac{\kappa}{2},\quad t\in[0,t_{0}],\quad x\in B_{r}(x_{0})\cap\Tnd,\label{eq:lower_bound_psi}
		\end{equation}
		for all $n\ge N_{2}$.
		
		For all sufficiently large $n$, we have 
		\[
		\frac{1}{n^{2}}\le\frac{t_{0}}{2}.
		\]
		Let $t\in[t_{0}/2,t_{0}]$ and $x\in B_{r}(x_{0})\cap\Tnd$. Then,
		for $s\in[0,n^{-2}]$, 
		\[
		t-s\in[0,t_{0}],
		\]
		and hence, by (\ref{eq:lower_bound_discrete_gradient}), 
		\[
		\left(\partial_{n,i}\rho^{n}_{t-s}(x)\right)^{2}\ge\frac{\delta}{2}.
		\]
		Lemma~\ref{lem:two-point-correlation-function} therefore yields
		\begin{equation}
			-V^{n}_{t}(x,x+e^{n}_{i})=\left|V^{n}_{t}(x,x+e^{n}_{i})\right|\ge\frac{c_{0}}{n^{2}},\label{eq:lower_bound_V_for_correction}
		\end{equation}
		for all sufficiently large $n$, uniformly in $t\in[t_{0}/2,t_{0}]$
		and $x\in B_{r}(x_{0})\cap\Tnd$, where 
		\[
		c_{0}:=\frac{\delta}{4(2d-1)}\left(1-e^{-(2d-1)}\right)>0.
		\]
		
		Keeping only the term $j=i$ in (\ref{eq:correction_lower_bound_start})
		and using (\ref{eq:lower_bound_psi}) and (\ref{eq:lower_bound_V_for_correction}),
		we obtain 
		\begin{align*}
			& \left|\int^{T}_{0}\E\left[
			\cC^{n}U^{n}_{t}\left(\bar{\zeta}^{n,[0,t]}\right)
			\right]dt\right|\\
			& \qquad\ge\frac{1}{2n^{d}}\sum_{x\in B_{r}(x_{0})\cap\Tnd}\int^{t_{0}}_{t_{0}/2}\psi^{n}_{t,i}(x)\left|V^{n}_{t}(x,x+e^{n}_{i})\right|\,dt\\
			& \qquad\ge\frac{\kappa c_{0}t_{0}}{8n^{d+2}}\left|B_{r}(x_{0})\cap\Tnd\right|.
		\end{align*}
		Since 
		\[
		\frac{1}{n^{d}}\left|B_{r}(x_{0})\cap\Tnd\right|\to\Leb(B_{r}(x_{0}))>0,
		\]
		the statement follows.
	\end{proof}

	For the quadratic functional considered above, we can estimate the weak error directly. Its Gaussian continuation remains quadratic up to the observation time, which allows us to use the standard time-dependent Dynkin formula.
	
	\begin{prop}[Quadratic weak error]
		\label{prop:quadratic_weak_error}
		Let the parameters, the initial profile $\rhoini$ and the processes $\bar\zeta^n$ and $\xi^n$ be as in
		Theorem~\ref{thm:main_result}. Fix $t_0\in(0,T]$ and
		$\varphi\in\Cf^\infty(\Td)$, and define
		\[
		F(\gamma):=\frac12\binn{\varphi,\gamma_{t_0}}^2,\qquad
		\gamma\in D([0,T],\H[-M_1]).
		\]
		Let $U^n$ be the Gaussian continuation of $F$, and set
		\[
		C_n(F):=\int_0^T\E\cC^nU_t^n(\bar\zeta^{n,[0,t]})\,dt.
		\]
		There exists a constant $C$, depending only on
		$d,T,I,m_\Delta,m_\partial,m_\tau,M_1,M_2$ and $\varphi$, such that
		\begin{align}
			&\left|\E F(\bar\zeta^{n,[0,T]})-\E F(\xi^{n,[0,T]})+C_n(F)\right|\nonumber\\
			&\qquad\le C\left(1+\|\rhoini\|_{\Cf^{M_2}(\Td)}^2\right)
			\left(n^{-2m_\Delta}+n^{-m_\partial}+n^{-(m_\tau+1)}\right)
			\label{eq:quadratic_corrected_error}
		\end{align}
		for all $n\ge3$. In particular, for
		$\beta:=\min\{2m_\Delta,m_\partial,m_\tau+1\}$ and $d\ge3$,
		\[
		\left|\E F(\bar\zeta^{n,[0,T]})-\E F(\xi^{n,[0,T]})\right|\le C\left(1+\|\rhoini\|_{\Cf^{M_2}(\Td)}^2\right)
		\left(n^{-\beta}+n^{-2}\right),\qquad n\ge3.
		\]
	\end{prop}
	
	\begin{rem}\label{rem:bound_from_belov}
		If $\rhoini$ is nonconstant and $\beta>2$, the choice of $t_0$
		and $\varphi$ in
		Proposition~\ref{prop:bound_from_belov_of_correlation_correction}
		satisfies
		\begin{equation}
			\left|
			\E F(\bar\zeta^{n,[0,T]})-\E F(\xi^{n,[0,T]})
			\right|
			\ge \frac{c}{2n^2}
			\label{eq:quadratic_weak_error_lower_bound}
		\end{equation}
		for all sufficiently large $n$, where $c>0$ is the constant
		in that proposition. Indeed, the correlation correction has
		magnitude at least $cn^{-2}$, whereas
		\eqref{eq:quadratic_corrected_error} is $o(n^{-2})$.
		This lower bound holds in every dimension. For $d\ge3$,
		the upper bound in Proposition~\ref{prop:quadratic_weak_error}
		therefore shows that the direct weak error for this observable
		has exact order $n^{-2}$.
	\end{rem}

	\begin{proof}[Proof of Proposition~\ref{prop:quadratic_weak_error}]
		The first derivative of $F$ need not be bounded, so we cannot apply
		Theorem~\ref{thm:main_result} directly. However, its quadratic form
		allows us to repeat the comparison argument, with the terms of order
		$n^{-d/2}$ vanishing in every dimension. Throughout the proof, $C$
		may depend on the parameters listed in the statement.
		
		For $0\le t\le t_0$, set
		\[
		\varphi_t^n:=\cS^{n,m_\Delta}_{t,t_0}\varphi,\qquad
		\psi_{t,j}^n:=\left(\partial_{n,j}^{(m_\partial)}\varphi_t^n\right)^2,
		\qquad j\in[d].
		\]
		The modified semigroup is self-adjoint, and the stochastic convolution
		is centered. Hence $U_t^n(\gamma)=u_t^n(\gamma_t)$, where
		\begin{equation}
			u_t^n(g):=\frac12\binn{\varphi_t^n,g}^2
			+\frac14\sum_{j=1}^d\int_t^{t_0}
			\binn{\psi_{s,j}^n,\left(\Gamma_{n,j}^{(m_\tau)}(\varrho_s^n)\right)^+}\,ds.
			\label{eq:quadratic_state_continuation}
		\end{equation}
		The factor $1/4$ follows from the factor $1/2$ in $F$ and the
		noise covariance in \eqref{eq:SPDE_for_OU_process_with_Gamma-1}.
		In particular,
		\[
		Du_t^n(g)=\binn{\varphi_t^n,g}\varphi_t^n,\qquad
		D^2u_t^n=\varphi_t^n\otimes\varphi_t^n,\qquad D^3u_t^n=0.
		\]
		Thus the second derivative of $u^n_t$ is deterministic and $Du_t^n(0)=0$.
		For $t>t_0$, the continuation is
		$U_t^n(\gamma)=\frac12\binn{\varphi,\gamma_{t_0}}^2$ and is independent
		of the terminal state, so its generator and correction terms vanish.
		
		Since
		$\partial_t\varphi_t^n=-\frac12\Delta_n^{(2m_\Delta)}\varphi_t^n$,
		differentiating \eqref{eq:quadratic_state_continuation} gives
		\begin{equation}\label{eq:backward_equation_for_u}
			\partial_tu_t^n+\cG^n_{\varrho_t^n}u_t^n=0,\qquad 0\le t<t_0.
		\end{equation}
		
		To apply Dynkin's formula, define, on the finite configuration space,
		\[
		v_t^n(\eta):=u_t^n\bigl(n^{d/2}(\bar\eta-\bar\rho_t^n)\bigr).
		\]
		The identity $\partial_t\rho_t^n=\frac12\Delta_n\rho_t^n$ and
		\eqref{eq:definition_of_cL_rho} imply
		\[
		(\partial_t+\cL^{n,0})v_t^n(\eta)
		=\left(\partial_tu_t^n+\cL^n_{\rho_t^n}u_t^n\right)
		\bigl(n^{d/2}(\bar\eta-\bar\rho_t^n)\bigr).
		\]
		For each fixed $n$, the function $v_t^n$ and its time derivative are
		bounded on $[0,t_0]$ times the configuration space. Applying Dynkin's
		formula up to $t_0$, using the equation (\ref{eq:backward_equation_for_u}) and
		$\E u_0^n(\xiini)=\E F(\xi^{n,[0,T]})$, and adding $C_n(F)$, we obtain
		\begin{equation}
			\begin{aligned}
				\E F(\bar\zeta^{n,[0,T]})&-\E F(\xi^{n,[0,T]})+C_n(F)\\
				&=\E u_0^n(\bar\zeta_0^n)-\E u_0^n(\xiini)+\int_0^{t_0}\E\left[
				(\cL^n_{\rho_t^n}-\cG^n_{\varrho_t^n}+\cC^n)
				u_t^n(\bar\zeta_t^n)\right]dt.
			\end{aligned}
			\label{eq:quadratic_trotter_identity}
		\end{equation}
		The Gaussian expectations are finite by the second-moment estimates
		for the modified dynamics. Thus this identity does not require the
		path-dependent backward equation or the bounded-first-derivative
		assumption in Theorem~\ref{thm:main_result}.
		
		We first estimate the initial term. Independence of the initial
		occupations and the covariance of $\xiini$ give the exact identity
		\[
		\E u_0^n(\bar\zeta_0^n)-\E u_0^n(\xiini)
		=\frac12\left(
		\binn{\rhoini(1-\rhoini)(\varphi_0^n)^2,1}_n
		-\binn{\rhoini(1-\rhoini)(\varphi_0^n)^2,1}\right).
		\]
		By Corollary~\ref{cor:regularity_of_modified_semigroup}, every fixed
		$\Cf^\ell$ norm of $\varphi_t^n$ is bounded uniformly in $n$ and
		$0\le t\le t_0\le T$. Lemma~\ref{lem:approximation_for_ex}, with
		$m=M_2$ and $k=0$, therefore yields
		\[
		\left|\E u_0^n(\bar\zeta_0^n)-\E u_0^n(\xiini)\right|
		\le Cn^{-M_2}\left(1+\|\rhoini\|_{\Cf^{M_2}(\Td)}^2\right).
		\]
		Only the covariance quadrature error remains at the initial time.
		
		We next compute the expected generator difference in
		\eqref{eq:quadratic_trotter_identity}. Write
		\[
		A_n:=\Delta_n-\Delta_n^{(2m_\Delta)},\qquad
		\widetilde\psi_{t,j}^n:=\left(\partial_{n,j}\varphi_t^n\right)^2,
		\]
		and define the deterministic function on $\Tnd$
		\[
		Q_{t,j}^n(x):=\E\left[
		\left(\eta_t^n(x+e_j^n)-\eta_t^n(x)\right)^2\right].
		\]
		Since the occupation variables take values in $\{0,1\}$,
		\[
		0\le Q_{t,j}^n\le1,\qquad
		Q_{t,j}^n(x)=\Gamma_{n,j}(\rho_t^n)(x)-2V_t^n(x,x+e_j^n).
		\]
		The microscopic Taylor expansion is exact because $D^3u_t^n=0$.
		Using the last identity and the definition of $\cC^n$, we obtain
		\begin{equation}
			\begin{aligned}
				\E\left[
				(\cL^n_{\rho_t^n}-\cG^n_{\varrho_t^n}+\cC^n)
				u_t^n(\bar\zeta_t^n)\right]&=\frac12\E\left[
				\binn{\varphi_t^n,\bar\zeta_t^n}
				\binn{A_n\varphi_t^n,\bar\zeta_t^n}\right]\\
				&\qquad\quad+\frac14\sum_{j=1}^d
				\left(\binn{\widetilde\psi_{t,j}^n-\psi_{t,j}^n,Q_{t,j}^n}_n
				+B_{t,j}^n\right),
			\end{aligned}
			\label{eq:quadratic_expected_generator_difference}
		\end{equation}
		where
		\[
		B_{t,j}^n:=
		\binn{\psi_{t,j}^n,\Gamma_{n,j}(\rho_t^n)}_n
		-\binn{\psi_{t,j}^n,
			\left(\Gamma_{n,j}^{(m_\tau)}(\varrho_t^n)\right)^+}.
		\]
		Here the quadratic covariance contribution is canceled by $\cC^n$.
		The terms linear in the centered field have zero expectation because
		$\psi_{t,j}^n$ is deterministic. Thus no fluctuation error of order
		$n^{-d/2}$ occurs in this identity.
		
		For the drift term, the fixed-time second-moment estimate used in
		the proof of Theorem~\ref{thm:main_result} gives
		\[
		\sup_{t\le T}\E\|\bar\zeta_t^n\|_{\H[-I]}^2
		\le C\left(1+\|\rhoini\|_{\Cf^1(\Td)}^2\right).
		\]
		Consequently, Lemma~\ref{lem:approximation_of_discrete_derivatives}
		and the uniform smoothness of $\varphi_t^n$ imply
		\begin{align*}
			\left|\E\left[
			\binn{\varphi_t^n,\bar\zeta_t^n}
			\binn{A_n\varphi_t^n,\bar\zeta_t^n}\right]\right|
			&\le \|\varphi_t^n\|_{\H[I]}\|A_n\varphi_t^n\|_{\H[I]}
			\E\|\bar\zeta_t^n\|_{\H[-I]}^2\\
			&\le Cn^{-2m_\Delta}
			\left(1+\|\rhoini\|_{\Cf^1(\Td)}^2\right).
		\end{align*}
		The same approximation lemma, applied to the discrete derivative,
		gives
		\[
		\left|\binn{\widetilde\psi_{t,j}^n-\psi_{t,j}^n,Q_{t,j}^n}_n\right|
		\le\|\widetilde\psi_{t,j}^n-\psi_{t,j}^n\|_{\Cf(\Td)}
		\le Cn^{-m_\partial}.
		\]
		
		It remains to estimate the deterministic covariance difference
		$B_{t,j}^n$. The estimates of $S^j_{3,2},\ldots,S^j_{3,5}$ in the proof
		of Proposition~\ref{prop:comparison_of_generators}, with
		$\varphi_j=\psi_{t,j}^n$ and quadrature order $m_W=\beta$, together
		with the profile estimates in the proof of Theorem~\ref{thm:main_result}, give
		\begin{align*}
			|B_{t,j}^n|\le Cn^{-\beta}
			\left(1+\|\rhoini\|_{\Cf^{M_2}(\Td)}^2\right).
		\end{align*}
		This choice of quadrature order is permitted by the uniform smoothness
		of $\psi_{t,j}^n$ and the assumption $M_2\ge\beta+1+\lceil d/2\rceil$.
		Combining these bounds with
		\eqref{eq:quadratic_expected_generator_difference} and
		\eqref{eq:quadratic_trotter_identity}, and using $M_2\ge\beta$, proves
		\eqref{eq:quadratic_corrected_error}.
		
		We now estimate the correlation correction from above when $d\ge3$.
		The deterministic Hessian gives
		\begin{equation}
			C_n(F)=\frac1{2n^d}\sum_{j=1}^d\sum_{x\in\Tnd}
			\int_0^{t_0}\psi_{t,j}^n(x)V_t^n(x,x+e_j^n)\,dt.
			\label{eq:quadratic_correlation_representation}
		\end{equation}
		Using the torus version of \citep[Lemma 3.1]{Xu2025} for two-point
		correlations, we get
		\[
		\sup_{t\le T}\sup_{\substack{x,y\in\Tnd\\x\ne y}}
		|V_t^n(x,y)|
		\le C_{d,T}n^{-2}(1+\|\rhoini\|_{\Cf^1(\Td)}^2).
		\]
		Together with the uniform bound on $\psi_{t,j}^n$ and
		\eqref{eq:quadratic_correlation_representation}, this gives
		\[
		|C_n(F)|\le Cn^{-2}(1+\|\rhoini\|_{\Cf^1(\Td)}^2).
		\]
		Thus, the claimed upper bound follows
		from \eqref{eq:quadratic_corrected_error}.
	\end{proof}

	\appendix
	
	\section{Auxiliary facts on differential calculus}
	
	We recall that $\cL(X,Y)$ denotes the space of bounded linear operators
	from $X$ to $Y$ equipped with the operator norm $\norm{\cdot}_{\cL(X,Y)}$.
	\begin{prop}
		\label{prop:interchange_integral_and_derivative}Let $X,Y$ be Banach
		spaces, $U$ be an open set in $X$ and $(\Omega,\F,Q)$ be a measure
		space. For a function $F:\Omega\times U\to Y$ assume that the following
		assumptions are satisfied:
		\begin{enumerate}
			\item For $Q$-a.e. $\omega$ the map $z\mapsto F(\omega,z)$ is Frechet
			differentiable on $U$.
			\item For every $x\in U$, $F(\cdot,x)$ is $Q$-strongly measurable and
			Bochner integrable.
			\item There exists a $Q$-integrable function $G:\Omega\to\R$ such that
			\[
			\sup_{x\in U}\norm{DF(\cdot,x)}_{\cL(X,Y)}\le G\quad Q\text{-a.e.}
			\]
		\end{enumerate}
		Then $DF(\cdot,x)[h]$ is $Q$-strongly measurable for each $x\in U$
		and $h\in X$, and is Bochner integrable. Moreover, the function 
		\[
		\bar{F}(x)=\int_{\Omega}F(\omega,x)Q(d\omega),\quad x\in U,
		\]
		is Frechet differentiable on $U$ and
		\[
		D\bar{F}(x)[h]=\int_{\Omega}DF(\omega,x)[h]Q(d\omega).
		\]
	\end{prop}
	
	\begin{proof}
		By \citep[Corollary 1.1.23]{Veraar:2016}, $DF(\cdot,x)[h]$ is $Q$-strongly
		measurable as the $Q$-almost everywhere limit
		\[
		DF(\cdot,x)[h]:=\lim_{n\to\infty}n\left(F\left(\cdot,x+\frac{h}{n}\right)-F(\cdot,x)\right)
		\]
		of $Q$-strongly measurable functions. The Bochner integrability of
		$DF(\cdot,x)[h]$ follows from the assumption 3. and \citep[Proposition 1.2.2]{Veraar:2016}
		because
		\[
		\int_{\Omega}\bnorm{DF(\omega,x)[h]}_{Y}Q(d\omega)\le\norm h_{X}\int_{\Omega}G(\omega)Q(d\omega)<\infty.
		\]

		We next note that the linear map
		\[
		h\mapsto\int_{\Omega}DF(\omega,x)[h]Q(d\omega)
		\]
		is bounded, due to
		\begin{align*}
			\bnorm{\int_{\Omega}DF(\omega,x)[h]Q(d\omega)}_{Y}&\le\int_{\Omega}\bnorm{DF(\omega,x)}_{\cL(X,Y)}Q(d\omega)\norm h_{X}\\
			&\le\norm h_{X}\int_{\Omega}G(\omega)Q(d\omega).
		\end{align*}
		We denote for fixed $x\in U$
		\[
		R(\omega,h):=\frac{1}{\norm h_{X}}\left[F(\omega,x+h)-F(\omega,x)-DF(\omega,x)[h]\right].
		\]
		Since $F(\omega,\cdot)$ is Frechet differentiable on $U$ for $Q$-a.e.
		$\omega$, we conclude
		\[
		R(\cdot,h)\to0\quad\text{as}\ h\to0
		\]
		$Q$-a.e. Moreover, by the mean-value theorem (see \citep[Proposition 3.3.1]{Cartan:1971}),
		we have
		\[
		\bnorm{F(\omega,x+h)-F(\omega,x)}_{Y}\le\norm h_{X}\sup_{y\in U}\norm{DF(\omega,y)}_{\cL(X,Y)}\le G(\omega)\norm h_{X}
		\]
		for $Q$-a.e. $\omega.$ Thus,
		\[
		\bnorm{R(\cdot,h)}_{Y}\le2G(\cdot)
		\]
		$Q$-a.e. By the dominated convergence theorem (see \citep[Proposition 1.2.5]{Veraar:2016}),
		we get
		\begin{align*}
			&\lim_{h\to0}\frac{1}{\norm h_{X}}
			\bnorm{\bar{F}(x+h)-\bar{F}(x)
				-\int_{\Omega}DF(\omega,x)[h]Q(d\omega)}_{Y}\\
			&\qquad=\lim_{h\to0}
			\bnorm{\int_{\Omega}R(\omega,h)Q(d\omega)}_{Y}\le\lim_{h\to0}
			\int_{\Omega}\bnorm{R(\omega,h)}_{Y}Q(d\omega).
		\end{align*}
		This completes the proof of the proposition.
	\end{proof}
	
	\begin{rem}
		\label{rem:measurability_for_separable_spaces}If the space $Y$ is
		separable, then the $Q$-strong measurability of $F$ in Proposition
		\ref{prop:interchange_integral_and_derivative} can be replaced by
		a usual measurability according to the Pettis measurability theorem
		(see \citep[Theorem 1.1.6 and Corollary 1.1.10]{Veraar:2016}).
	\end{rem}

	\section{Functional Itô formula}\label{sec:functional_itos_formula}
	
	In this section, we record a Hilbert-space extension of the functional
	Itô formula for continuous semimartingales. It follows from the finite-dimensional
	result by a standard approximation argument, and we only provide the
	main details.
	
	Let $H$ and $E$ be separable Hilbert spaces, and let $(\Omega,\F,(\F_{t})_{t\in[0,T]},\p)$
	be a filtered probability space. Let $(W_{t})_{t\in[0,T]}$ be a cylindrical
	Wiener process on $H$, let $\Phi$ be an $L_{HS}(H,E)$-valued predictable
	process satisfying 
	\[
	\int^{T}_{0}\|\Phi_{t}\|^{2}_{L_{HS}(H,E)}\,dt<\infty\quad\text{\ensuremath{\p}-a.s.},
	\]
	and let $\Psi$ be an $E$-valued predictable process such that 
	\[
	\int^{T}_{0}\|\Psi_{t}\|_{E}\,dt<\infty\quad\text{\ensuremath{\p}-a.s.}
	\]
	Here $L_{HS}(H,E)$ denotes the space of Hilbert-Schmidt operators
	from $H$ to $E$. If $\xi_{0}$ is an $\F_{0}$-measurable $E$-valued
	random variable, we set 
	\[
	\xi_{t}:=\xi_{0}+\int^{t}_{0}\Psi_{s}\,ds+\int^{t}_{0}\Phi_{s}\,dW_{s},\quad t\in[0,T].
	\]
	
	We use the notions of horizontal and vertical derivatives introduced
	in Sections \ref{subsec:horizontal_derivative} and \ref{subsec:vertical_derivatives}.
	The following result is the Hilbert-space counterpart of the functional
	Itô formula for continuous finite-dimensional semimartingales obtained
	in \citep{Cont_Fournie2010,Cont_Fournie:2013}.
	\begin{prop}
		\label{prop:Ito_formula} Let $F$ be a non-anticipative functional
		on $\Upsilon([0,T],E)$ which is horizontally differentiable and twice vertically
		differentiable. Assume that $F$, $\nabla F$, and $\nabla^{2}F$
		are left-continuous, $\cD F$ is continuous at each fixed time, and $\cD F$, $\nabla F$ and $\nabla^{2}F$
		are boundedness preserving. Then, for every $t\in[0,T]$, $\p$-a.s.,
		\begin{align*}
			F_{t}(\xi^{[0,t]}) & =F_{0}(\xi_{0})+\int^{t}_{0}\cD_{s}F(\xi^{[0,s]})\,ds+\int^{t}_{0}\nabla F_{s}(\xi^{[0,s]})[\Psi_{s}]\,ds\\
			& +\int^{t}_{0}\nabla F_{s}(\xi^{[0,s]})[\Phi_{s}\,dW_{s}]+\frac{1}{2}\int^{t}_{0}\Tr\left(\Phi^{*}_{s}\nabla^{2}F_{s}(\xi^{[0,s]})\Phi_{s}\right)ds.
		\end{align*}
	\end{prop}
	
	\begin{proof}
		We only provide a sketch of the standard finite-dimensional approximation
		argument. Let $(h_{k})_{k\in\N}$ and $(e_{k})_{k\in\N}$ be orthonormal
		bases of $H$ and $E$, respectively, and denote by $P^{H}_{n}$ and
		$P^{E}_{n}$ the corresponding orthogonal projections onto the first
		$n$ basis vectors. Set 
		\[
		\xi^{n}_{t}:=P^{E}_{n}\xi_{0}+\int^{t}_{0}P^{E}_{n}\Psi_{s}\,ds+\int^{t}_{0}P^{E}_{n}\Phi_{s}P^{H}_{n}\,dW_{s}.
		\]
		Identifying $P^{E}_{n}E$ with $\R^{n}$ and restricting $F$ to
		$\Upsilon([0,T],P^{E}_{n}E)$, the finite-dimensional functional
		Itô formula \citep[Theorem~4.1]{Cont_Fournie:2013} yields
		\begin{align*}
			F_{t}(\xi^{n,[0,t]}) & =F_{0}(\xi^{n}_{0})+\int^{t}_{0}\cD_{s}F(\xi^{n,[0,s]})\,ds+\int^{t}_{0}\nabla F_{s}(\xi^{n,[0,s]})[P^{E}_{n}\Psi_{s}]\,ds\\
			& \quad+\int^{t}_{0}\nabla F_{s}(\xi^{n,[0,s]})[P^{E}_{n}\Phi_{s}P^{H}_{n}\,dW_{s}]\\
			& \quad+\frac{1}{2}\int^{t}_{0}\Tr\left((P^{E}_{n}\Phi_{s}P^{H}_{n})^{*}\nabla^{2}F_{s}(\xi^{n,[0,s]})(P^{E}_{n}\Phi_{s}P^{H}_{n})\right)\,ds.
		\end{align*}
		
		The standard convergence properties of orthogonal projections and
		stochastic integrals imply
		\[
		\sup_{t\in[0,T]}\|\xi^{n}_{t}-\xi_{t}\|_{E}\to0
		\]
		in probability. Using the tightness of
		$\{\xi^{n}\}_{n\in\N}\cup\{\xi\}$ in $C([0,T],E)$, we may localize
		to compact sets; the left-continuity and boundedness-preserving
		properties of $F$ and its derivatives then allow us to pass to the
		limit in the deterministic terms by the dominated convergence theorem.
		Moreover,
		\[
		P^{E}_{n}\Phi_{s}P^{H}_{n}\to\Phi_{s}\quad\text{in }L_{HS}(H,E)
		\]
		for almost every $s$, and hence the stochastic integrals converge
		by the Burkholder--Davis--Gundy inequality. Finally, 
		\[
		\Tr\left((P^{E}_{n}\Phi_{s}P^{H}_{n})^{*}\nabla^{2}F_{s}(\xi^{n,[0,s]})(P^{E}_{n}\Phi_{s}P^{H}_{n})\right)\to\Tr\left(\Phi^{*}_{s}\nabla^{2}F_{s}(\xi^{[0,s]})\Phi_{s}\right),
		\]
		while, after localization, the absolute values of these terms are
		bounded by $C\|\Phi_{s}\|^{2}_{L_{HS}(H,E)}.$ The dominated convergence
		theorem therefore applies also to the trace term. Passing to the limit
		in the finite-dimensional formula proves the claim.
	\end{proof}

	\section*{Acknowledgments}
	
	Benjamin Gess acknowledges support from the Max Planck Society through the Research Group ``Stochastic Analysis in the Sciences''. Vitalii Konarovskyi thanks TU Berlin for its warm hospitality, where part of this research was carried out.
	
	\section*{Funding}
	
	The work of the first author was funded by the European Union (ERC, FluCo, grant agreement No. 101088488). Views and opinions expressed are, however, those of the author(s) only and do not necessarily reflect those of the European Union or the European Research Council. Neither the European Union nor the granting authority can be held responsible for them.
	
	The second author was supported by the Deutsche Forschungsgemeinschaft (DFG, German Research Foundation) through CRC/TRR 388 ``Rough Analysis, Stochastic Dynamics and Related Fields'', Project A11.
	
	\section*{Declaration on the use of generative AI}
	The authors have developed this project since 2024 without AI assistance. No AI was used in the development of the central ideas and main proofs. Generative AI tools were used only in the latest stage to improve the language, exposition, organization, and presentation of the manuscript.\\
	In addition to this, the results and proofs in Appendix B and Proposition 5.9 were developed through iterative mathematical discussions with ChatGPT, closely guided by arguments already available to the authors.\\
	The authors have independently verified all mathematical statements and proofs in the final manuscript, including those developed with AI assistance. They take full responsibility for the mathematical content, references, and final text of the article.

\end{document}